\documentclass[preprint,12pt]{elsarticle}
\usepackage{amsfonts}
\usepackage{amsthm}
\usepackage{amsmath}
\usepackage[all]{xy}
\usepackage{titlesec}
\usepackage{dsfont}
\usepackage{amssymb}
\usepackage{extarrows}
\usepackage{mathrsfs}
\usepackage{arydshln}
\usepackage{multirow}
\usepackage{color}
\usepackage{newtxmath}

\usepackage[bookmarksnumbered, colorlinks, plainpages]{hyperref}
in \textwidth 6 true in

\journal{XXXX}
\newtheorem{theorem}{Theorem}[section]
\newtheorem{proposition}[theorem]{Proposition}
\newtheorem{lemma}[theorem]{Lemma}
\newtheorem{corollary}[theorem]{Corollary}
\newtheorem{definition}[theorem]{Definition}
\newtheorem{remark}[theorem]{Remark}
\theoremstyle{conjecture}
\newtheorem{conjecture}[theorem]{Conjecture}
\newtheorem*{note}{Note}

\newcommand{\lr}[1]{\ensuremath{\left \langle #1  \right \rangle }}

\newcommand{\Z}{\ensuremath{\mathbb{Z}}}

\newcommand{\fhe}[0]{\ensuremath{{\scriptstyle\circ}}}

\begin{document}

\begin{frontmatter}

\title{The unstable homotopy groups of 2-cell complexes}

\author{Zhongjian Zhu}
\ead{20160118@wzu.edu.cn}
\address{College of Mathematics and Physics, Wenzhou University, Wenzhou 325035, China}

\begin{abstract}
In this paper,  we develop the new method,  initiated by 
B. Gray (1972), to compute the unstable homotopy groups of the mapping cone, especially for $2$-cell complex  $X=S^m\cup_{\alpha} e^{n}$.  By Gray's work mentioned above or the traditional method given by I.M.James (1957) which were  widely used in previous related work to compute $\pi_{i}(X)$, the dimension $i\leq 2n+m-4$.  By our method, we can compute $\pi_{i}(X)$ for $i>2n+m-4$. We use this different technique to generalize J.Wu's work, at  Mem. of AMS, on homotopy groups of mod $2$ Moore spaces to higher dimensional homotopy groups of mod $2^r$ Moore spaces $P^{n}(2^r)$ for all $r\geq 1$.  This practice shows that  the technique given here is a new general method to compute the unstable homotopy groups of CW complexes with higher dimension.
\end{abstract}

\begin{keyword}
Relative James construction\sep fibration sequence\sep homotopy group\sep Moore space

\MSC 55P10\sep55Q05\sep55Q15\sep 55Q52
\end{keyword}
\end{frontmatter}

\section{Introduction}
\label{intro}

Calculating the unstable homotopy groups of finite CW-complexes is a fundamental and difficult problem in algebraic topology.  This problem is reduced to compute $\pi_{i}(Y^{\ast})$, where $Y^{\ast}=Y\cup_fe^{n}$, $f$ is the attaching map of the cell $e^n$. There are two traditional methods to compute the homotopy groups $\pi_{i}(Y^{\ast})$ for $i\leq 2n+m-4$ with $m-1$ connected space $Y$. The first is  applying Jame's method, i.e., the  Theorem 2.1 of \cite{James I M},  to the  homotopy exact sequences of a pair $(Y^{\ast}, Y)$.  By this method, a lot of related work has been done on 2-cell complexes, such as mod $2^r$ Moore spaces when $r=1,2,3$ and suspended  complex projective  spaces.
 Let $P^{n}(p^r)$ denote  $n$-dimensional  elementary Moore spaces, whose only nontrivial reduced homology is  $\overline{H}_{n-1}(P^{n}(p^r))=\Z_{p^r}$ ($\Z_{k}:=\Z/k\Z$). J. Mukai computed  $\pi_{i}(P^{n}(2))$ for some $i\leq 3n-5$ and $3\leq n\leq 7$ in \cite{Mukai2,Mukai3}. In 1999, J.Mukai and T.Shinpo computed the  $\pi_{i} (P^{n}(4))$ in the range $i=2n-3,2n-4$ and $n\leq 24$ \cite{Mukai};  In 2007, X.G.Liu computed  $\pi_{i}(P^{n}(8)) $  in the range $i=2n-3$ for $2\leq n\leq 20$ and $i=2n-2$ for $3\leq n\leq 7$ \cite{X.G.Liu}. Some  homotopy groups of suspended  complex projective  plane $C_{\eta}^{n}:=\Sigma^{n-4}\mathbb{C}P^2, n\geq 5$,  was also computed in \cite{MukaiS1}. The second method is using the fibration sequence $\Omega\Sigma X\xrightarrow{\partial} F\rightarrow Y\cup_{f}CX\xrightarrow{pinch} \Sigma X$, where $ Y\cup_{f}CX$ is the mapping cone of the map $f$, which is also denoted by $C_f$.
In 1972, B.Gray constructed the relative James construction $J(X,A)$ for a CW-pair $(X,A)$ and inclusion $A\stackrel{i}\hookrightarrow X$, which is filtrated by $J_{r}(X,A) (r=1,2,\dots)$  and proved that $J(X,A)$ is homotopy equivalent to the homotopy fiber of the pinch map $X\cup_{i}CA\rightarrow \Sigma A$ \cite{Gray}.
So the homotopy type of the above homotpy fiber  $F$ is homotopy equivalent to Gray's relative James construction $J(M_f,X)$ which is filtrated by a sequence of subspace $J_r(M_f,X), r=1,2,\dots$, where $M_f$ is the mapping  cylinder  of $f$.  When $X=\Sigma X'$, $Y=\Sigma Y'$, the homotopy type of  the  second filtration $J_2(M_f,X)$ of $J(M_f,X)$ was also obtained by Gray in \cite{Gray}. It has the homotopy type $Y\cup_{\gamma_2}C(Y \wedge X')$,   where $\gamma=[id_Y, f]$ is the generalized Whitehead product. By the recently developed work of Yang, Mukai and Wu \cite{JXYang}  on the homotopy analysis of the connection map $\Omega\Sigma X\xrightarrow{\partial} F$,  Gray's method to compute homotopy groups of CW-complex becomes more operable. As applications, they compute homotopy groups of suspended quaternionic projective
plane, i.e., $\pi_{i}(\Sigma^{k} \mathbb{H}P^2)$ for $i\leq k+15$ and  we (the author and Pan) compute the  first two dimensional unstable homotopy groups of indecomposable $\mathbf{A}_3^2$-complexes, i.e.,  $2$-connected  finite CW-complexes  with  dimensions of cells less than or equal to $5$ \cite{ZP23}.

So far, the only general way to compute $\pi_{i}(Y^{\ast})$ for $i>2n+m-4$ is given by Wu in 2003 \cite{WJ Proj plane}. In it,  a $p$-local functorial decomposition $\Omega\Sigma X\simeq A^{min}(X)\times \Omega(\bigvee_{n=2}^{\infty} Q_{n}^{max}(X))$ is given,  which is helpful to calculate homotopy groups of $\Sigma X$. $Q_{n}^{max}(X)$ is a wedge summand of $\Sigma X^{\wedge n}$, where $X^{\wedge n}$ is a $n$-fold self-smash product of $X$. It is an amazing method since it connects  the algebraic representation theory and the homotopy theory, by which he got a lot of unstable homotopy groups $\pi_{i}(P^{n}(2))(n\geq 3)$ where $i$ can be larger than  $3n-5$ ( i.e., $2n+m-4$ with $m=n-1$). However it contains so high technique that the homotopy type of $A^{min}(X)$ or skeletons of it is difficult to get \cite{Amin}.  $Q_{n}^{max}(X)$ is also hard to be obtained  for $X$ with more than $2$ cells since the homotopy wedge decomposition of $X^{\wedge n}$ is not easy \cite{ZP,ZLP,ZPhyperbolicity}. 

  Another way to compute $\pi_{i}(Y^{\ast})$ for $i>2n+m-4$ is to get information of homotopy type of the filtration $J_r(M_f,X)$ of the Gray's relative James construction $J(M_f,X)$ for  the pinch map $ Y\cup_{f}CX\xrightarrow{pinch} \Sigma X$, where $X=\Sigma X',Y=\Sigma Y'$ are suspensions.  By our previous work we show that the  attaching map  $\gamma_r: \Sigma^{r-1}Y'\wedge X'^{\wedge (r-1)}\rightarrow J_{r-1}(M_f,X)$ is an element of  set of $r$-th order Whitehead products defined by G.J.Porter \cite{Porter}, that is 

\begin{theorem}[Theorem 3.4 of \cite{ZJ}]\label{Theorem of ZhuJin}
	Let $X\xrightarrow{f}Y$ be a map of simply connected CW complexes, $X=\Sigma X'$, $Y=\Sigma Y'$. Then 
	$J_n(M_f,X)$ has the homotopy type  $Y\cup_{\gamma_2}C(\Sigma Y'\wedge X')\cup_{\gamma_3}\cdots \cup_{\gamma_n}C(\Sigma^{n-1} Y'\wedge X'^{\wedge{n-1}})$, $\gamma_r$ is an element of $r$-th order Whitehead products in $[j^{r-1}_{Y}, j^{r-1}_{Y}f, \cdots, j^{r-1}_{Y}f]$ where $j^{r-1}_{Y}: Y\hookrightarrow Y\cup_{\gamma_2}C(\Sigma Y'\wedge X')\cup_{\gamma_3}\cdots \cup_{\gamma_{r-1}}C(\Sigma^{r-2} Y'\wedge X'^{\wedge{(r-2)}})$ is the canonical inclusion for $r=2,\dots n$.
\end{theorem}

As an application,  we used Theorem \ref{Theorem of ZhuJin} to compute the homotopy groups $\pi_{k}(P^{3}(2^r))$ ($k=5,6$) for all positive integers $r$ in \cite{ZJ}. However the above calculation relies on a result of one special third order Whitehead products computed in \cite{Porter Postnikov}, previously we did not   very sure that Theorem \ref{Theorem of ZhuJin} had wide applications in computing high dimensional homotopy groups of mapping cones because of that $\gamma_r (r\geq 3)$ is an element of  $r$-th order Whitehead products  $[j^{r-1}_{Y}, j^{r-1}_{Y}f, \cdots, j^{r-1}_{Y}f]$, which is a set that may contain more than one element so that it will not be enough to determine  $\gamma_r$.

In this paper, we find that by combining Theorem \ref{Theorem of ZhuJin} and some homotopy  property, especially the naturality  of $\gamma_r$, we can give the detailed  description of  $\gamma_3$ for a lot of  2-cell complexes. This enable us to develop the method of Gray to compute the  $\pi_{i}(Y^{\ast})$ for $i>2n+m-4$. We use this new method to generalize J.Wu's work \cite{WJ Proj plane} on homotopy groups of mod $2$ Moore spaces to higher dimensional homotopy groups of mod $2^r$ Moore spaces $P^{n}(2^r)$ for all $r\geq 1$.  The following is our main result. 

\begin{theorem}\label{Thm:Main thm} The unstable homotopy groups of  $\pi_{i}(P^{n}(2^r))$ for $n=4,5,6$ and $2n-2\leq i\leq 4n-6$ are given as follows.
		$$\footnotesize{\begin{tabular}{|c|c|c|c|c|c|}
			\multicolumn{6}{c}{}\\
			\hline
			$i=$\rule[10pt]{0pt}{2pt}	&$6$ &$7$ &$8$&$9$&$10$\\
				\hline
		$\pi_{i}(P^{4}(2^r))$\rule[10pt]{0pt}{2pt}	&$(2)^{1+\vartheta_r}\!+\!2^{r+1}, r\!\leq\! 2$&$(2)^2\!+\!4^{\vartheta_r}$&$(2)^3, r\!=\! 1$&$(2)^{2+\vartheta_r}\!+\!2^{m_{r}^3}$&$(2)^{2}\!+\!4, r\!=\!1$ \\
				                &$2\!+\!4\!+\!2^{r},r\!\geq \!3$\rule[10pt]{0pt}{2pt}& & $(2)^2\!+\!(4)^2, r\!=\! 2$&&$(2)^{2}\!+\!2^{m_{r}^3}\!+\!2^r, r\!\geq\! 2$\\
				                 && &$(2)^4\!+\!2^r, r\!\geq \! 3$  &&\\
				\hline
		\end{tabular}~}$$
	$$Table~1$$
	$$\footnotesize{\begin{tabular}{|c|c|c|c|}
		\multicolumn{4}{c}{}\\
		\hline
		$i=$\rule[10pt]{0pt}{2pt}	&$8$ &$9$ &$10$\\
		\hline
		$\pi_{i}(P^{5}(2^r))$\rule[10pt]{0pt}{2pt}	&$(2)^{1+\vartheta_r}\!+\!2^{m_{r}^3}$&$(2)^{2}\!+\! 4^{\vartheta_r},r\!\leq \!2$&$2+4, r\!=\!1$ \\
		&& $(2)^{4},r\!\geq \!3$& $(2)^{2}\!+\!2^{m_{r+1}^3},r\!\geq \!2$\\
		\hline
			$11$\rule[10pt]{0pt}{2pt}	&$12$ &$13$ &$14$\\
		\hline
		 $2\!+\!2^{m_{r+1}^3}\!+\!2^r$\rule[10pt]{0pt}{2pt}	&$(2)^{3}$ &$(2)^{2}\!+\!(4)^{2},r\!=\!1$ &$(2)^{7}\!+\!4,r\!=\!1$\\
				& &$(2)^{6},r\!\geq\! 2$ &$(2)^{7}\!+\!(4)^3,r\!=\! 2$\\
			& & &$(2)^{6}\!+\!4\!+\!(8)^2\!+\!2^r,r\!\geq\! 3$\\
		\hline
	\end{tabular}~}$$
$$Table~2$$
	$$\footnotesize{\begin{tabular}{|c|c|c|c|c|c}
		\multicolumn{6}{c}{}\\
	\cline{1-5}
		$i=$\rule[10pt]{0pt}{2pt}	&$10$ &$11$ &$12$&$13$&\\
			\cline{1-5}
	$\pi_{i}(P^{6}(2^r))$\rule[10pt]{0pt}{2pt}	&$8,r\!=\!1$ &$(2)^2$ &$(2)^{2+\vartheta_r}$&$(2)^3,\!r=1\!$&\\
                        	&$2+2^{r+1},r\!\geq \!2$ & &&$2\!+\!4\!+\!2^{m_{r}^3},\!r\geq 2\!$&\\
			\cline{1-5}
	$14$\rule[10pt]{0pt}{2pt}	&$15$ &$16$ &$17$&$18$&\\
			\cline{1-5}
$(2)^{4}\!+\!4$	\rule[10pt]{0pt}{2pt}&$(2)^{4}\!+\!4,$&$(2)^{6}\!+\!4$ &$(2)^{5}\!+\!(8)^2$&$(2)^{5}\!+\!8$&$r\!=\!1$\\
$(2)^{4}\!+\!2^{m_{r}^3}\!+\!2^{r},$	&$(2)^{5}\!+\!2^{m_{r}^3}$ &$(2)^{6}\!+\!(2^{m_{r}^3})^2$ &$(2)^{4}\!+\!4\!+\!(2^{m_{r}^3})^2\!+\!2^{m_{r+1}^4}$&$(2)^{4}\!+\!2^r\!+\!2^{m_{r}^3}\!+\!2^{m_{r+1}^4}$&$r\!\geq\! 2$\\
	 	\cline{1-5}
	\end{tabular}~}$$
$$Table~3$$
In the above tables,  an integer $k$ indicates a cyclic group $\Z_k:=\Z/k\Z$ of order
$k$,  the symbol $``+"$ the direct sum of the groups and $(2)^t$ indicates the direct sum of $t$-copies
of  $\Z_2$. $m_a^b:=min\{a,b\}$ indicates the minimal integer of  $a$ and $b$. $\vartheta_1 = 0$ and $\vartheta_r =1$ for $r \geq 2$.
\end{theorem}

\begin{theorem}\label{Thm:Main thm2} The first $9$ dimensional unstable homotopy groups of  $\pi_{i}(P^{n}(2^r))$ for $n=7,8$ are given as follows.
	$$\footnotesize{\begin{tabular}{|c|c|c|c|c|}
		\multicolumn{5}{c}{}\\
		\cline{1-5}
		$i=$\rule[10pt]{0pt}{2pt}	&$11$ &$12$ &$13$&$14$\\
		\cline{1-5}
		$\pi_{i}(P^{7}(2^r))$\rule[10pt]{0pt}{2pt}	&$2^r$ &$2$ &$(2)^{2}+\!2^{m_{r-1}^2}$&$(2)^2\!+\!2^{m_{r+1}^3}\!+\!2^{m_{r}^3}$\\
		\cline{1-5}
		$15$\rule[10pt]{0pt}{2pt}	&$16$ &$17$ &$18$&$19$\\
		\cline{1-5}
		$2\!+\!(4)^{2}\!+\!8,r\!=\!1$	\rule[10pt]{0pt}{2pt}&$(2)^4+4,r=1$&$2\!+\!4\!+\!(2^{m_{r}^3})^{3}$&$(2)^{3}\!+\!2^{m_{r}^3}\!+\!2^{m_{r}^4}$ &$(2)^{2}\!+\!4,r\!=\!1$\\
		$(2)^{2}\!+\!2^{m_{r}^3},r\!\geq \!2$	&$(2)^{5}\!+\!2^{m_{r}^3},r\!\geq \!2$& & &$(2)^{3}\!+\!2^{m_{r}^4},r\!\geq\! 2$\\
		\cline{1-5}
	\end{tabular}~}$$
$$Table~4$$
	$$\footnotesize{\begin{tabular}{|c|c|c|c|c|}
		\multicolumn{5}{c}{}\\
		\cline{1-5}
		$i=$\rule[10pt]{0pt}{2pt}	&$13$ &$14$ &$15$&$16$\\
		\cline{1-5}
		$\pi_{i}(P^{8}(2^r))$\rule[10pt]{0pt}{2pt}	&$2$ &$2\!+\!2^{r-1}\!+\!2^{r+1},r\!\leq \!3$ &$(2)^{3+\vartheta_r}\!+\!2^{m_{r}^3}$&$(2)^3\!+\!(4)^2\!,r\!=\!1$\\
		\rule[10pt]{0pt}{2pt}	&&$2\!+\!8\!+\!2^{r},r\!\geq \!4$ &&$(2)^7\!+\!4,r\!= \!2$\\
			\rule[10pt]{0pt}{2pt}	&& &&$(2)^9,r\!\geq \!3$\\
		\cline{1-5}
		$17$\rule[10pt]{0pt}{2pt}	&$18$ &$19$ &$20$&$21$\\
		\cline{1-5}
		$(2)^{3}\!+\!(4)^{2},r\!=\!1$	\rule[10pt]{0pt}{2pt}&$(2)^2\!+\!2^{m_{r-1}^3}\!$&$2\!+\!2^{m_{r}^3}$&$(2)^{1+\vartheta_r}\!+\!2^{r}$ &$(2)^3\!+\!2^{m_{r}^2}$\\
		$(2)^{6}\!+\!2^{m_{r-1}^3}\!+\!8,r\!\geq \!2$	&$+2^{m_{r}^3}\!+\!2^{m_{r+1}^3}$& & &$+2^{m_{r-1}^3}\!+\!2^{m_{r+1}^4}$\\
		\cline{1-5}
	\end{tabular}~}$$
$$Table~5$$
where for $r=1$, $\Z_{2^{m_{r-1}^3}}=\Z_{2^{m_{r-1}^2}}=\Z_{2^0}=\Z_{1}=0$.
\end{theorem}

\begin{remark}
	~~
	
 $\bullet$~	In Theorem \ref{Thm:Main thm}, $\pi_{i}(P^{4}(2^r))$ for $i=6,7$ are given by \cite{ZP23};  $\pi_{8}(P^{5}(2^r))$ is given by \cite{JZhtpygps}.
	
	$\bullet$~ $2n-3$ dimensional homotopy groups of $P^{n}(2^r), n=4,5,6$ (which may be unstable) are not listed in the Tables of Theorem \ref{Thm:Main thm}, since $\pi_{5}(P^{4}(2^r))=\pi^s_{5}(P^{4}(2^r))\cong \left\{
	\begin{array}{ll}
		\Z_4, & \hbox{$r=1$;} \\
		\Z_2\oplus\Z_2, & \hbox{$r\geq 2$}
	\end{array}
	\right. 
	$ is given in \cite[Part IV]{Bau1985};  $\pi_{7}(P^{5}(2^r))\cong \Z_{2^{m_{r-1}^2}}\oplus \Z_{2^{r+1}}\oplus \Z_2$ and  $\pi_{9}(P^{6}(2^r))\cong \Z_{2^{m_{r}^3}}\oplus \Z_2$ are given in \cite{ZP23} and \cite{JZhtpygps} respectively.

	$\bullet$~ Some homotopy groups of  mod $2$ Moore spaces in the above Tables, that is $\pi_{i}(P^{4}(2))$ $i\leq 10$,  $\pi_{i}(P^{5}(2))$ $i\leq 13$,  $\pi_{i}(P^{6}(2))$ $i\leq 14$, $\pi_{i}(P^{7}(2))$ $i\leq 15$ and $\pi_{i}(P^{8}(2))$ $i\leq 16$ are also given by \cite{WJ Proj plane}.
	
	$\bullet$~ We also generalize the results of  $\pi_{i}(P^{n}(2))$ for $n=9,10,11$ in \cite{WJ Proj plane} to  $\pi_{i}(P^{n}(2^r))$ for any $r\geq 1$ in Theorem \ref{thm: pi(P9,10,11)} without proof.
\end{remark}

There are two advantages of the method given in this paper to compute the unstable homotopy groups. Firstly, it is easy to  get the generators of homotopy groups by the calculation process. 
Secondly, we are able to calculate the homotopy groups of $P^{n}(2^r)$ for all $r\geq 1$ simultaneously. 

 Some new techniques are given here to determine the 	connection homomorphism 
$\partial_{k\ast}: \pi_{k+1}(X)\rightarrow \pi_{k}(F)$ and solve the extension problem of the short exact sequence of homotopy groups 
$$	0\rightarrow  Coker (\partial_{k\ast})\rightarrow \pi_{k}(C_f)\rightarrow Ker(\partial_{k-1\ast})\rightarrow  0$$
 which is induced by fibration sequence  $\Omega\Sigma X\xrightarrow{\partial} F\rightarrow Y\cup_{f}CX\xrightarrow{pinch} \Sigma X$, where  $Ker(h)$ and $Coker(h)$ denote the Kernel and Cokernal of a homomorphism $h: G\rightarrow H$ of abelian groups respectively;  $f_{m\ast}$ denotes the induced homomorphism $f_{\ast}: \pi_{m}(X)\rightarrow \pi_{m}(Y)$ for a map $f:X\rightarrow Y$ of spaces.  We also show that the homotopy type of $J_3(M_f, X)$ for some other spaces such as suspended complex or quaternionic projective plane can be obtained by our method. This enables us to calculate the homotopy groups of them with dimension lager than that in the previous work. This gives us reason to believe that the method given in this paper has wide applications in computing high dimensional homotopy groups of CW-complexes, at least for $2$-cell complexes.

By the way, 
recently the research on the suspension homotopy of manifolds becomes a popular topic
\cite{So6dim,R.Huang 2,R.Huang Li,LiZhu,Theriault}, however the homology groups of  manifolds considered  by them   have no 2-torsion. And the known  unstable homotopy groups of mod $q^r$ ($q$ is odd prime) Moore spaces are very important in their methods. So the above results of unstable homotopy groups of mod $2^r$ Moore spaces have a potential application, that is to classify the homotopy types oqhe suspension of non-simpliy connected manifolds whose homology groups are allowed to have 2-torsion.

The paper is arranged as follows. Section \ref{sec: 2} introduces some concepts and properties of higher order Whitehead products and relative James construction;   Section \ref{section: extension problem} gives some homotopy properties of the connection map $\Omega\Sigma X\xrightarrow{\partial} F_p$ and some q which are useful to solve the extension problem of the short exact sequence.   Section \ref{sec:htyp Moore 456} and Section \ref{sec:htyp Moore78} calculate the unstable homotopy groups of mod $2^r$ Moore spaces in Theorem \ref{Thm:Main thm} and Theorem \ref{Thm:Main thm2} respectively. The last section, i.e., Section \ref{sec:some remarks} gives  some remarks to the method given in this paper.

\section{Higher order Whitehead products and relative James construction}
\label{sec: 2}
In this paper,  all spaces and maps are in the category of pointed topological spaces and maps (i.e. continuous functions) preserving base point. And we always use $*$ and $0$ to denote the basepoint and the constant map mapping to the basepoint respectively. Without special mention, all spaces are CW-complexes and all the space pairs are CW-pairs. We denote $A\hookrightarrow X$  as an inclusion map.
By abuse of notion, we will not distinguish the notions between a map and its homotopy class in many cases. 

 For the maps $f_i:X_i\rightarrow X_{i+1} (i=1,2,\cdots, n)$,  $f_nf_{n-1}\cdots f_1: X_1\rightarrow X_n$ denotes the composition of $f_i  (i=1,2,\cdots, n)$; Moreover, if  $[X_1, X_n]$ is a abelian group, for example $X_1$ is an double suspension, then for $k\in \Z$,  $kf_nf_{n-1}\cdots f_1:=k(f_nf_{n-1}\cdots f_1)$ denotes the $k$ times of $f_nf_{n-1}\cdots f_1$;

\subsection{Higher order Whitehead products}
\label{subsec: Higher order Whitehead products}

Let $T_{r}(X_1,X_2,\dots,X_n)$ be the subset of the Cartesian product $X_1\times X_2\times\dots\times X_n$, consisting of those $n$-tuples with at least $r$ co-ordinates at a base point. Thus  $T_{n-1}(X_1,X_2,\dots,X_n)=X_1\vee X_2\vee\dots\vee X_n$, $T_{1}(X_1,X_2,\dots,X_n)$ is the  ``fat wedge", and  $X_1\times X_2\times\dots\times X_n/T_{1}(X_1,X_2,\dots,X_n)=X_1\wedge X_2\wedge\dots\wedge X_n$. From Theorem 1.2 and Theorem 2.1 of \cite{Porter},
there is a principle cofibration
\begin{align}
	\Sigma^{n-1} X_1\wedge \dots\wedge  X_n \xrightarrow{W_n} T_{1}(\Sigma X_1,\dots,\Sigma  X_n)\rightarrow \Sigma X_1\times \dots\times\Sigma X_n
	\label{Cof for Cartesian Pordu.}
\end{align}
where the map $W_n$ is natural.

Given a map $f: T_{1}(\Sigma X_1,\dots,\Sigma  X_n)\rightarrow X$, $n\geq 2$, define
\begin{align}
	W(f):=f_{\ast}(W_n)=fW_n\in [\Sigma^{n-1} X_1\wedge \dots\wedge  X_n, X]
	\nonumber
\end{align}
the $n$-th order Whitehead product, which depends only upon the homotopy class of $f$ \cite{Porter}.

Let $k_j:\Sigma X_j \hookrightarrow T_{1}(\Sigma X_1,\dots,\Sigma  X_n), j=1,2,\dots, n$, be the canonical injections.
\begin{definition}
	The set of $n$-th order Whitehead products of $f_j:\Sigma X_j\rightarrow X$, $j=1,\dots,n$, is
	\begin{align}
		[f_1,\dots, f_n]:=\{W(f)|f: T_{1}(\Sigma X_1,\dots,\Sigma  X_n)\rightarrow X, fk_j\simeq f_j, j=1,\dots,n\}.
		\nonumber
	\end{align}
\end{definition}
\begin{remark}\label{Remark of Def HOWP}
	1) $[f_1,\dots, f_n]$ is a subset of $[\Sigma^{n-1} X_1\wedge \dots\wedge  X_n, X]$, and it is perhaps empty. We also have
	$ [f_1,\dots, f_n]:=\{W(f):=fW_n~|~f: T_{1}(\Sigma X_1,\dots,\Sigma  X_n)\rightarrow X\}$ for all $f$ extending the wedge sum map $(f_1,\dots,f_n): \Sigma X_1\vee \dots\vee \Sigma X_n\rightarrow X$ up to homotopy.
	
	2)$[f_1,\dots, f_n]$ depends only upon the homotopy classes of $f_i$ $(i=1,\dots,n)$;
	
	3) Hardie\cite{Hardie} has also given the definition of $[f_1,\dots, f_n]$ when all  $X_i$ are spheres (called the $n$-th order spherical Whitehead product). $[f_1,f_2,f_3]$ is, in this case, the Zeeman product studied by Hardie\cite{Hardie Zeeman}. When $X_1, X_2$ are arbitrary, $[f_1,f_2]$ is the ``generalized Whitehead product" studied by Arkowitz\cite{Arkowitz}.
	
	4) From the Theorem 2.5 \cite{Porter},   if $X$ is an H-space, then $[f_1,\dots, f_n]=\{0 \}$; if $\Sigma$ is the suspension homomorphism, then $\Sigma [f_1,\dots, f_n]=\{0\}$.
\end{remark}

The following naturality of higher order Whitehead products comes from Theorem 2.1 of \cite{Porter} and Corollary 2.4 of \cite{ZJ}.
\begin{proposition}\label{Prop: Naturality HOWP}(Naturality)
	Let $f_i:A_i\rightarrow B_i$, $h_i:\Sigma B_i\rightarrow X$, $i=1,\dots,n$, $g:X\rightarrow Y$ and $\varphi: T_1(\Sigma B_1,\dots, \Sigma B_n)\rightarrow X$, then
	\begin{itemize}
		\item [(a)] $(\Sigma^{n-1}(f_1\wedge \dots\wedge f_n ))^{\ast}[h_1,\dots,h_n]\subset [h_1\Sigma f_1,\dots,h_n \Sigma f_n]$
		\item [(b)]$g_\ast[h_1,\dots,h_n]\subset [gh_1,\dots,gh_n]$.
		\item [(c)] If $B_j=\Sigma B'_j$ for some $j$, then $k[h_1,\dots,h_j,\dots,h_n]\subset [h_1,\dots,kh_j,\dots h_n]$ for any integer $k$;
	\end{itemize}
		If $f_i:A_i\rightarrow B_i$ and  $g:X\rightarrow Y$ are homotopy equivalences, then
	\begin{itemize}
		\item [(d)] $(\Sigma^{n-1}(f_1\wedge \dots\wedge f_n ))^{\ast}[h_1,\dots,h_n]=[h_1\Sigma f_1,\dots,h_n \Sigma f_n]$;
		\item [(e)]$g_\ast[h_1,\dots,h_n]=[gh_1,\dots,gh_n]$.
	\end{itemize}
\end{proposition}

The following lemma  about the property of third order Whitehead product is from the Theorem 4.2 of \cite{Golasinski deMelo}, and the special case for spheres $X_i (i=1,2,3)$ are given by Hardie \cite{Hardie Zeeman}.
\begin{lemma}\label{lemma for [a,b,c]}
	$[f_1,f_2,f_3]$ with $f_i:\Sigma X_i\rightarrow X$ $(i=1,2,3)$  is a coset of the following subgroup of  group $[\Sigma^2X_1\wedge X_2\wedge X_3, X]$
	\begin{align}
		[[\Sigma^2X_2\wedge X_3, X],  f_1]+[[\Sigma^2X_1\wedge X_3, X],  f_2]+[[\Sigma^2X_1\wedge X_2, X],  f_3]. \nonumber
	\end{align}
\end{lemma}

\subsection{Relative James construction}
\label{subsec: Relative James construction}

Let  $(X,A)$ be a pair of  spaces with base point $*\in A$, and suppose that $A$ is closed in $X$. In \cite{Gray}, B.Gray constructed a space $(X,A)_{\infty}$ analogous to the James construction, which is denoted by us as $J(X,A)$ in \cite{ZP23} to parallel with the the absolute James construction $J(X)$ \cite{JamesReduceProdu}. In fact, $J(X,A)$ is the subspace of $J(X)$ of words for which letters after the first are in $A$.  Especially,  $J(X,X)=J(X)$. As parallel with the familiar symbol $J_{r}(X)$ which is the $r$-th filtration of $J(X)$, we denote the  $r$-th filtration of $J(X,A)$ by  $J_{n}(X,A):=J(X,A)\cap J_{r}(X)$, which is denoted by Gray as  $(X, A)_r$ in  \cite{Gray}.
For example,
\begin{align}
	&J_{1}(X,A)=X, J_{2}(X,A)=(X\times A)/((a,\ast)\thicksim (\ast,a)), \nonumber\\
	&J_{3}(X,A)=(X\times A\times A )/((x,\ast, a)\thicksim  (x, a, \ast);  (a, a',\ast)\thicksim (\ast,a,a')\thicksim (a,\ast,a')).\nonumber
\end{align} In fact there is a  pushout diagram for $r\geq 2$:
$$ \footnotesize{\xymatrix{
		X\times A^{\wedge (n-1)} \ar[r]^{\Pi_{r}} &  J_{r}(X, A) \\
		T_1(X,A,\cdots,A) \ar@{^{(}->}[u]\ar[r] & J_{r-1}(X, A)\ar@{^{(}->}[u]^{I_{r-1}} } }$$
where  $\Pi_{r}$ and $I_{r-1}$  are the projection and  the inclusion respectively and both of them are natural. Moreover, $J_{n}(X,A)/J_{n-1}(X,A)$ is naturally homeomorphic to $(X\times A^{n-1})/T_1(X,A,\cdots,A)=X\wedge A^{\wedge (n-1)}$.

It is well known that there is a natural weak homotopy equivalence  $\omega:J(X)\rightarrow \Omega\Sigma X$, which is a homotopy equivalence when $X$ is a finite CW-complex, and satisfies
$\footnotesize{\xymatrix{
		&X\ar@/_0.5pc/[rr]_{\Omega\Sigma}\ar@{^{(}->}[r]& J(X) \ar[r]^{\omega}& \Omega\Sigma X } }$,
where $X\xrightarrow{\Omega\Sigma} \Omega\Sigma X$, also denoted by $E_X$,  is the inclusion $x\mapsto \psi$ where $\psi: S^1\rightarrow S^1\wedge X, t\mapsto  t\wedge x .$

In the following text,  the ``commutative" means ``homotopy commutative".

Let $X\xrightarrow{f}Y$ be a map. There is a cofibration sequence $X\xrightarrow{f}Y\xrightarrow{i}C_f\xrightarrow{p} \Sigma X$.  We always use $C_f$, $F_f$ and  $M_f$ to denote the maping cone ( or say, cofibre ),  homotopy fibre and  mapping cylinder of $f$, $C_f\xrightarrow{p} \Sigma X$ the pinch map and $\Omega\Sigma X\xrightarrow{\partial}F_p\xrightarrow{\tau_p} C_f\xrightarrow{p}\Sigma X$ the homotopy fibration sequence induced by $p$ respectively. Suppose that $X=\Sigma X', Y=\Sigma Y'$ are suspensions.
Applying  Theorems  of \cite{Gray} to $(M_f,X)$, we get 
\begin{itemize}
	\item [(i)]  $F_{p}\simeq J(M_f,X)$;
	\item [(ii)] $\Sigma J(M_f,X)\simeq \bigvee_{k\geq 0}(\Sigma Y\wedge X^{\wedge k})$;$\Sigma J_{k}(M_f,X)\simeq \bigvee_{i= 0}^{k-1}(\Sigma Y\wedge X^{\wedge i})$;
	\item [(iii)]  $J_2(M_f,X)\simeq Y\cup_{\gamma_2}C(Y\wedge X')$, where $\gamma_2=[id_Y, f]$ is the generalized Whitehead product.
\end{itemize}

We denote both the inclusion  $Y\hookrightarrow J_{2}(M_f,X)$ and the composition of the inclusions
$Y\hookrightarrow J_{2}(M_f,X)\hookrightarrow J(M_f,X) \simeq  F_{p}$ by $j_{p}$ without ambiguous.   As pointed out in  the proof of Lemma 4.1.of \cite{Gray},  $j_p$ can be chosen
naturally as the lift of the inclusion $Y\xrightarrow{i}C_f$, i.e., 
\begin{align}
	\tau_p j_p\simeq i \label{equ:taujp=i}
\end{align}
Moreover the relative James construction is natural in the sense of the following 
\begin{lemma}[Lemma 2.2 of \cite{ZP23}]\label{Lem: natural J(Mf,X)}
	Suppose the left  diagram  is commutative\\
	$\footnotesize{\xymatrix{
			X\ar[d]^{\mu_X}\ar[r]^-{f} & Y\ar[d]^{\mu_Y}\\
			X_1\ar[r]^-{f_1} & Y_1}}$;~~~~~~$ \footnotesize{\xymatrix{
			F_p\simeq  J(M_f,X)\ar[d]^{J(\widehat{\mu},\mu_X)}\ar[r]&M_f/X\simeq C_f  \ar[d]^{\bar{\mu}}\ar[r]^-{p}& \Sigma X\ar[d]^{\Sigma\mu_X}\\
			F_{p_1}\simeq J(M_{f_1},X_1)\ar[r]& M_{f_1}/X_1 \simeq C_{f_1} \ar[r]^-{p_1}&\Sigma X'}}$
	
	then it induces the right commutative diagrams on fibrations, where $\widehat{\mu}$ satisfies
	
	$ \footnotesize{\xymatrix{
			&Y\ar@/^0.5pc/[rrr]^{\mu_Y}\ar@{^{(}->}[r]_{\simeq}& M_f \ar[r]_{\widehat{\mu}}& M_{f_1}\ar[r]_{\simeq}& Y_1} }$. 
	
	Let~$J_{r}(M_f,X)\xrightarrow{J(\widehat{\mu},\mu_X)|_{J_{r}(M_f,X)}=J_{r}(\widehat{\mu},\mu_X)} J_{r}(M_{f_1},X_1)$ $(r\geq 1)$.
	Then we have the following homotopy commutative diagram
		$$ \footnotesize{\xymatrix{
			J_{n-1}(M_f,X)\ar[r]^{I_{n-1}}\ar[d]^-{J_{n-1}(\widehat{\mu},\mu_X)}& J_{n}(M_f,X)\ar[r]^-{proj.}\ar[d]^-{J_{n}(\widehat{\mu},\mu_X)}&	J_{n}(M_f,X)/J_{n-1}(M_f,X)\ar[r]^-{\simeq}\ar[d]^-{\overline{J_{n}(\widehat{\mu},\mu_X)}} &	Y\wedge X^{\wedge ^{n-1}}\ar[d]^{\mu_Y\wedge\mu_X^{\wedge ^{n-1}}} \\
			J_{n-1}(M_{f_1},X_1)\ar[r]^{I_{n-1}}& J_{n}(M_{f_1},X_1)\ar[r]^-{proj.}&	J_{n}(M_{f_1},X_1)/J_{n-1}(M_{f_1},X_1)\ar[r]^-{\simeq}&	Y_1\wedge X_1^{\wedge ^{n-1}}
} }$$
\end{lemma}

\subsection{homotopy of $\gamma_n$ and Hopf invariant}
\label{subsect: homotopy of gamman and hopf inva}

Let $j^r_{X}:X=J_{1}(X,A)\hookrightarrow J_{r}(X,A)$ be the canonical inclusion. By Theorem \ref{Theorem of ZhuJin}, there is a cofibration sequence 
\begin{align*}
		\Sigma^{n-1}Y'\wedge X'^{\wedge(n-1)}\xrightarrow{\gamma_n} J_{n-1}(M_f, X)\xrightarrow{I_{n-1}} J_{n}(M_f, X)
\end{align*}
where $\gamma_n$ is an element of $n$-th order Whitehead products  $[j^{n-1}_{Y}, j^{n-1}_{Y}f, \cdots, j^{n-1}_{Y}f]$. Thus from Remark \ref{Remark of Def HOWP}, 
\begin{align*}
	\Sigma\gamma_n=0.
\end{align*}

By \cite[Remark 3.5]{ZJ} $\gamma_n$ is natural in the sense of  the following
\begin{lemma}\label{lem: Natural gamma_n}
	For $\mu_X=\Sigma\mu'_X: X=\Sigma X'\rightarrow X_1=\Sigma X'_1$, $\mu_Y=\Sigma\mu'_Y: Y=\Sigma Y'\rightarrow Y_1=\Sigma Y'_1$, there is the following commutative diagram
	\begin{align}
		\small{\xymatrix{
				\Sigma^{n-1}Y'\wedge X'^{\wedge(n-1)} \ar[r]^-{\gamma_n}\ar[d]^{\Sigma^{n-1}\mu'_Y\wedge \mu_X'^{\wedge(n-1)}}&J_{n-1}(Y,X)\ar[r]^{I_{n-1}}\ar[d]^{J_{n-1}(\mu_Y,\mu_X)}&J_{n}(Y,X)\ar[d]^{J_{n}(\mu_Y,\mu_X)}\\
				\Sigma^{n-1}Y_1'\wedge X_1'^{\wedge(n-1)}  \ar[r]^-{\gamma_n}& J_{n-1}(Y_1,X_1)\ar[r]^{I_{n-1}} & J_{n}(Y_1,X_1)
		} }. \nonumber
	\end{align}
\end{lemma}

Define the $n$-th relative James-Hopf invariant
$$J(X,A)\xrightarrow{H_{n}} J(X\wedge A^{\wedge(n-1)}),~x_1x_2\dots x_t\mapsto\prod\limits_{1\leq i_1<i_2<\dots<i_n\leq t}(x_{i_{1}}\wedge x_{i_{2}}\wedge\dots\wedge x_{i_{n}}  )$$ which are natural for pairs.

By abuse of notion, the second James-Hopf invariant $H_2$ also denotes the composition of the maps 
\begin{align}
	\Omega\Sigma X \xrightarrow{\simeq} J(X)=J(X,X)\xrightarrow{H_2} J(X\wedge X)\xrightarrow{\simeq} \Omega\Sigma (X\wedge X)   \label{def:James H2}
\end{align}
where $X$ is a CW-complex and let $H_2': F_p\simeq J(M_f,X)\xrightarrow{H_2} J(M_f\wedge X)\xrightarrow{\simeq} J(Y\wedge X)\simeq \Omega\Sigma (Y\wedge X)$. Then by Lemmas 2.5 and 2.7 of \cite{ZP23}, we get 

\begin{lemma}\label{Lem: compute H2} 
	Let $X\xrightarrow{f}Y$ be a map. Then the following diagram is homotopy commutative
	$$ \footnotesize{\xymatrix{
			\Omega\Sigma X\ar[d]_{H_2 } \ar[r]^-{\partial} & F_p \ar[d]^{H'_2}\simeq J(M_f, X) & \ar@{_{(}->}[l]J_{2}(M_f,X)\ar[r]^-{pinch} &J_{2}(M_f,X)/ J_{1}(M_f,X)=M_f\wedge X\ar@{^{(}->}[d]\\
			\Omega\Sigma (X\wedge X) \ar[r]^-{\Omega\Sigma(f\wedge id_{X})} & \Omega\Sigma (Y\wedge X) &&J(Y\wedge X)\simeq J(M_f\wedge X)\ar[ll]_-{\omega}~~. } }$$
\end{lemma}

Let $X,Y$ be connected CW-complexes.
Recall that the Hilton-Milnor theorem states that there is a homotopy equivalence 
\[\Omega \Sigma (X\vee Y)\simeq \Omega\Sigma X\times \Omega\Sigma Y\times \Omega\Sigma \big(\bigvee_{i\geq 1}X^{\wedge i}\wedge Y\big).\]
Let $j_1\colon X\to X\vee Y$ and $j_2\colon Y\to X\vee Y$ be the canonical inclusion into wedge sum, respectively.
Recall the following duality isomorphism 
\[\Omega_0\colon [\Sigma X,Y]\to [X,\Omega Y],\quad\Omega_0f=(\Omega f)\circ E_X.\]
For maps $f\colon \Sigma X\to Z,g\colon \Sigma Y\to Z$, the Whitehead product 
\[[f,g]\colon \Sigma X\wedge Y\to Z\] is the duality of the Samelson product 
\[[\Omega_0 f,\Omega_0g]^S\colon X\wedge Y\to \Omega Z; \]
that is, $\Omega_0[f,g]=[\Omega_0f,\Omega_0g]^S$.
Let $X$ be a connected CW co-$H$-space.  Recall that the  Hopf-Hilton invariant $H$ (\cite{Hopfinvariants1967}, Section 5.3 of \cite{NeisendorferBook}) is the composition 
\[H\colon \Omega \Sigma X\xrightarrow{\Omega \Sigma (j_1+j_2)} \Omega \Sigma (X\vee X)\xrightarrow{p_{[j_1,j_2]}}\Omega\Sigma (X\wedge X),\]
where $p_{[j_1,j_2]}$ is the canonical projection after pre-composing with the homotopy equivalence given by the Hilton-Milnor theorem. 

By Lemma 3.12 of \cite{Steer1963}, the Hopf-Hilton invariant $H$ and the second James-Hopf invariant $H_2$ (absolute) are homotopic.

\begin{lemma} \label{lem:pinch gamma3}
	Let	$(X,A)$  be a CW-pair with $X=\Sigma X'$, $A=\Sigma A'$. $i_A: A\hookrightarrow X$ is the inclusion map, then $p_{X\wedge A}\gamma_3\simeq 0$, where $p_{X\wedge A}: J_2(X,A)\rightarrow J_2(X,A)/X\cong X\wedge A $ is the canonical projection. Moreover, if $[id_X, i_A]=0$, then $J_2(X,A)\simeq X\vee X\wedge A$. $J_3(X,A)\simeq (X\vee X\wedge A)\cup_{\gamma_3}C(X'\wedge A\wedge A)$ such that the $\gamma_3$  also satisfies 
	\begin{align*}
	p_{X}\gamma_3=[p_X, \mu p_{X\wedge A}]H_2(\gamma_3)~~\text{for some}~~\mu\in [X\wedge A, X]
	\end{align*}
	where
	$p_X: X\vee X\wedge A \rightarrow X$ is the canonical projection of the first wedge summand. 
\end{lemma}
\begin{proof}
	Let $j_X:X\hookrightarrow J_2(X,A)$ be the canonical inclusion. From Theorem \ref{Theorem of ZhuJin},  $\gamma_3\in [j_X, j_Xi_A, j_Xi_A]$.  Then from Proposition \ref{Prop: Naturality HOWP},  $p_{X\wedge A}\gamma_3\in p_{X\wedge A}[j_X, j_Xi_A, j_Xi_A]\subset  [p_{X\wedge A}j_X, p_{X\wedge A}j_Xi_A, p_{X\wedge A}j_Xi_A]=[0,0,0]$. Since $0\in [0,0,0]$ by  \cite[ Corollary 2.6]{Golasinski deMelo} and it is a  coset of the subgroup $\{0\}$ by Lemma \ref{lemma for [a,b,c]},  $[0,0,0]=\{0\}$.  Hence $p_{X\wedge A}\gamma_3=0$.
	
If $[id_X, i_A]=0$, then $J_2(X,A)\simeq X\vee X\wedge A$.	By Corollary 5.9 of \cite{Gray}, $X$ is a wedge summand of $J(X,A)$, i.e., there is a projection $P_{\infty}:J(X,A)\rightarrow X$, such that $P_{\infty}I_X=id_X$ where $I_X: X=J_1(X,A)\hookrightarrow J(X,A)$ is the canonical inclusion which is also the following composition 
	\begin{align*}
		\footnotesize{\xymatrix{
				&X\ar@/^1pc/[rrr]^{I_X}\ar@{^{(}->}[r]_{j_X}& J_2(X,A) \ar@{^{(}->}[r]_{I_2}& J_3(X,A)\ar@{^{(}->}[r]& J(X,A)} }
	\end{align*}
	Let $P_n: J_{n}(X,A)\hookrightarrow J(X,A) \xrightarrow{P_{\infty}} X$. $P_2j_X=id_X$ implies that $P_2=(id_X, \mu)=p_X+ \mu p_{X\wedge A}$ for some $\mu\in [X\wedge A, X]$.
	By the following exact sequence 
	\begin{align*}
		[J_3(X,A),X]\xrightarrow{I_2^{\ast}}[X\vee X\wedge A,X]\xrightarrow{\gamma_3^{\ast}}[X'\wedge A\wedge A, X]
	\end{align*}
	we get $I_{2}^{\ast}(P_3)=P_2$. Thus $\gamma_3^{\ast}(P_2)=P_2\gamma_3=0$, i.e., 
	$(p_X+ \mu p_{X\wedge A})\gamma_3=0$. 
	From the (4.11) of \cite{Hopfinvariants1967}, we have 
	$p_X\gamma_3+\mu p_{X\wedge A}\gamma_3+[\mu p_{X\wedge A}, p_X]H_2(\gamma_3)=0$, i.e., 
	$p_X\gamma_3=[p_X, \mu p_{X\wedge A}]H_2(\gamma_3)$.
	
	We complete the proof of this Lemma.
\end{proof}
The above lemma  easily gives the following corollary
\begin{corollary}\label{Cor:proj gamma3}
Let $X\xrightarrow{f}Y$ be a map with $X=\Sigma X'$, $Y=\Sigma Y'$. Then $p_{Y\wedge X}\gamma_3\simeq 0$, where $p_{Y\wedge X}: J_2(M_f, X)\rightarrow J_2(M_f, X)/Y\simeq Y\wedge X$ is the canonical projection. Moreover, if $[id_Y, f]=0$, then $J_2(M_f,X)\simeq Y\vee Y\wedge X$. $J_3(M_f,X)\simeq (Y\vee Y\wedge X)\cup_{\gamma_3}C(Y'\wedge X\wedge X)$ such that the $\gamma_3$  also satisfies 
\begin{align*}
	p_{Y}\gamma_3=[p_Y, \mu p_{Y\wedge X}]H_2(\gamma_3)~~\text{for some}~~\mu\in [Y\wedge X, Y]
\end{align*}
where
$p_Y: Y\vee Y\wedge X\rightarrow Y$ is the canonical projection of the first wedge summand. 
\end{corollary}

Here we give the following conjecture which holds  in all the special cases we have encountered and a lot of maps satisfy the condition of the conjecture, such as the attaching map of the top cell in the even dimensional mod $2^r$ Moore space, suspended complex   projective
plane  $\Sigma \mathbb{C}P^2$, quaternionic projective
plane  $\Sigma \mathbb{H}P^2$ and so on.

\begin{conjecture} Let $X\xrightarrow{f}Y$ be a map with $X=\Sigma X'$, $Y=\Sigma Y'$.  If $[id_Y, f]=0$, that is $J_2(M_f,A)\simeq Y\vee Y\wedge X$, then 
	$\gamma_3=[j_Yf, j_{Y\wedge X}]:Y'\wedge X\wedge X\rightarrow   Y\vee Y\wedge X$, where $j_Y$ and $j_{Y\wedge X}$ are canonical inclusions of the corresponding wedge summands of   $J_2(M_f,X)\simeq  X\vee X\wedge A$.
\end{conjecture}

Let $j_i:X_i\hookrightarrow \vee_{i=1}^{n} X_i$, $q_i:\vee_{i=1}^{n} X_i\twoheadrightarrow X_i$ (resp.  $j_i: A_i\hookrightarrow \oplus_{i=1}^{n} A_i$, $q_i:  \oplus_{i=1}^{n} A_i\twoheadrightarrow A_i$, $A_i$ are abelian groups) are inclusion and projection of the $i$-th wedge summand $X_i$ (rep. direct summand $A_i$). If $X_i$ is a sphere $S^m$, denote $j_i, q_i$ by  $j^m_i, q^m_i$ respectively.    Let $(f_1,f_2,\cdots, f_n): \vee_{i=1}^{n} X_i\rightarrow Y$ be the map such that $(f_1,f_2,\cdots, f_n)j_i\simeq f_i (i=1,2,\cdots, n)$ and $g_1\vee g_2\vee \cdots \vee g_n$ denote the map $(j_1g_1, j_2g_2,\cdots, j_ng_n): \vee_{i=1}^{n} X_i\rightarrow \vee_{i=1}^{n} Y_i$.

Consider the cofibration sequence $S^{n+k}\xrightarrow{\alpha} S^{n}\rightarrow C_{\alpha}\xrightarrow{p} S^{n+k+1}$. There is the cofibration sequence
\begin{align}
	&S^{2n+k-1}\xrightarrow{\gamma_2} S^{n}\xrightarrow{j_p} J_2(M_{\alpha},S^{n+k} )\xrightarrow{\bar p} S^{2n+k},~~ \gamma_2=[\iota_n, \alpha].\nonumber\\
	&Sk_{5n+2k-3}(F_{\bar p})\simeq J_2(M_{\gamma_2},S^{2n+k-1} )\simeq S^n\cup_{[\iota_n,[\iota_n, \alpha]]} e^{3n+k-1}. \label{equ: SkFbarp}
\end{align}
Moreover if $[\iota_n,[\iota_n, \alpha]]=0$, then there is  the inclusion of the wedge summand $\bar\beta_{3n+k-1}: S^{3n+k-1}\hookrightarrow  J_2(M_{\gamma_2},S^{2n+k-1} )\simeq S^n\vee S^{3n+k-1}$. Let $\beta_{3n+k-1}$ be the composition of the following canonical inclusions 
\begin{align*}
	\beta_{3n+k-1}:~	S^{3n+k-1}\stackrel{\bar\beta_{3n+k-1}}\hookrightarrow  J_2(M_{\gamma_2},S^{2n+k-1} )\hookrightarrow F_{\bar p}\stackrel{\bar\tau}\hookrightarrow J_2(M_{\alpha},S^{n+k} ).
\end{align*}
\begin{lemma}\label{lem: J3 alpha}
	Let $S^{n+k}\xrightarrow{\alpha} S^{n}\rightarrow C_{\alpha}\xrightarrow{p} S^{n+k+1}$ be the cofibration sequence, with $k\geq 0, n\geq 2$ and $k< n-1$.
	
   	If $[\iota_n,[\iota_n, \alpha]]=0$, i.e., $J_2(M_{\gamma_2},S^{2n+k-1})\simeq S^n\vee S^{3n+k-1}$, then 
	\begin{align*}
		\gamma_3=\beta_{3n+k-1}\theta+j_{p}\lambda : S^{3n+2k-1}\rightarrow J_2(M_{\alpha},S^{n+k} ), \theta \in \pi_{3n+2k-1}(S^{3n+k-1}), \lambda\in \pi_{3n+2k-1}(S^{n}). 
	\end{align*}
	Especially,	if $[\iota_n, \alpha]= 0$, i.e.,  $J_{2}(M_{\alpha}, S^{n+k})\simeq S^n\vee S^{2n+k}$, then $	\gamma_3\simeq [j_1^n, j_2^{2n+k}]\theta+j_{1}^n\lambda $.
\end{lemma}

\begin{proof}
The fibration sequence $\Omega S^{2n+k}\xrightarrow{\bar\partial} F_{\bar p}\xrightarrow{\bar\tau} J_2(M_{\alpha},S^{n+k} )\xrightarrow{\bar p} S^{2n+k}$ induces the exact sequence with commutative square 
	\begin{align*} \small{\xymatrix{
				\pi_{3n+2k}( S^{2n+k})\ar[r]^-{ \bar{\partial}_{\ast}}&\pi_{3n+2k-1}(F_{\bar{p}})\ar[r]^-{\bar{\tau}_{\ast}}&\pi_{3n+2k-1}(J_2(M_{\alpha},S^{n+k} ))\ar[r]^-{\bar{p}_{\ast}}&\pi_{3n+k-1}(S^{2n+k}) }}  
	\end{align*}
By (\ref{equ: SkFbarp}), $	\pi_{3n+2k-1}(F_{\bar{p}})=j_{\bar p\ast}\pi_{3n+2k-1}(S^n)\oplus \bar\beta_{3n+k-1\ast}\pi_{3n+2k-1}(S^{3n+k-1}),  k<n-1$, 
where  $\bar\tau j_{\bar p}\simeq j_p$ by (\ref{equ:taujp=i}).
 	From Corollary \ref{Cor:proj gamma3}, $\bar p\gamma_3\simeq 0$, that is, 
	\begin{align*}
		\gamma_3\in \bar{\tau}_{\ast}(\pi_{3n+2k-1}(F_{\bar{p}}))=j_{p}\fhe \pi_{3n+2k-1}(S^n)+\beta_{3n+k-1}\fhe\pi_{3n+2k-1}(S^{3n+k-1}).
	\end{align*}
	This implies the first conclusion of this lemma.
\end{proof}

\begin{remark}\label{remark gamma3}
For the case $[\iota_n, \alpha]\simeq 0$ in Lemma \ref{lem: J3 alpha}, if $\Sigma : \pi_{3n+2k-1}(S^{n})\rightarrow \pi_{3n+2k}(S^{n+1})$ is an injection, then by $\Sigma\gamma_3=0$, we get $j_1^{n+1}\Sigma \lambda=0$ which implies   $ \lambda=0$. Hence 
	$\gamma_3=[j_1^n, j_2^{2n+k}]\theta$ for some $\theta \in \pi_{3n+2k-1}(S^{3n+k-1})$.
\end{remark}

For a map $f\colon X\to Y$, denote by $f_E$ the composition 
\[f_E=E_Y\circ f=(\Omega\Sigma f)\circ E_X=\Omega_0(\Sigma f)\colon X\to \Omega \Sigma Y.\]

	\begin{lemma}\label{lem:H_2-SmSn}
	Let $H_\ast\colon \pi_{m+n}(\Omega\Sigma (S^m\vee S^n))\to \pi_{m+n}(\Omega\Sigma (S^m\vee S^n)^{\wedge 2})$ be the homomorphism induced by the second James-Hopf invariant. Then  
	\[H_\ast([(j^{m}_1)_E,(j^{n}_2)_E]^S)=(j_1^m\wedge j_2^n)_E- (j_2^n\wedge j_1^m )_E.\]
	Thus we have 
	\begin{align*}
		& H_{2}([j_1^{m+1}, j_2^{n+1}])=\Sigma j_1^m\wedge j_2^n- \Sigma j_2^n\wedge j_1^m, \\
		\text{where}~~&	 H_2:\pi_{n+m+1}(S^{n+1}\vee S^{m+1})\rightarrow \pi_{n+m+1}(\Sigma (S^m\vee S^n)^{\wedge 2})  
	\end{align*}
	is the induced by the  second James-Hopf invariant.
	
\end{lemma}

\begin{proof}
	There are isomorphisms
	\begin{align*}
		H_{m+n}(S^{m+n})&\xrightarrow{\tau_{m,n}}H_{m+n}(S^m\wedge S^n)\cong H_m(S^m)\otimes H_n(S^n):\\
		\iota_{m+n}&\mapsto (-1)^k\cdot \iota_m\otimes\iota_n,~k\in\Z.
	\end{align*}
	In the following we simplify $j^{m}_1, j_{2}^n$ by $j_1, j_2$ respectively.
	By the Bott-Samelson theorem,
	\begin{align*}
		H_{m+n}(\Omega\Sigma (S^m\vee S^n))&\cong T\langle (j_1)_{E\ast}(\iota_m), (j_2)_{E\ast}(\iota_n)\rangle,\\
		H_{m+n}(\Omega\Sigma (S^m\vee S^n)^{\wedge 2})&\cong T\langle (j_t\wedge j_s)_{E\ast}(\iota_{m}\otimes\iota_{n})~|~t,s=1,2\rangle,
	\end{align*}
	where $T\langle x_1,\cdots,x_m\rangle$ is the free tensor algebra $T(V)$ and $V$ is the free abelian group generated by $x_1,\cdots,x_m$.
	
	Let $h_{X}\colon \pi_{m+n}(\Omega\Sigma X)\to H_{m+n}(\Omega\Sigma X)$ be the Hurewicz homomorphism with $X=S^m\vee S^n$. We have 
	\begin{align*}
		h_X&([(j_1)_E,(j_2)_E]^S)=(-1)^{k}[(j_1)_E,(j_2)_E]^S_\ast(\iota_m\otimes\iota_n)\\
		&=(-1)^{k}\big((j_1)_{E\ast}(\iota_m)\otimes (j_2)_{E\ast}(\iota_n)-(j_2)_{E\ast}(\iota_n)\otimes (j_1)_{E\ast}(\iota_m)\big),\\
		h_{X^{\wedge 2}}&((j_t\wedge j_s)_E)=(-1)^k(j_t\wedge j_s)_{E\ast}(\iota_m\otimes \iota_n),~~t,s=1,2.
	\end{align*}
	Set $H_\ast ([(j_1)_E,(j_2)_E]^S)=\sum_{t=1,s=1}^{m,n}x_{ts}\cdot  (j_t\wedge j_s)_E$, $x_{ts}\in\Z$. Then 
	\begin{align*}
		h_{X^{\wedge 2}} H_\ast&([(j_1)_E,(j_2)_E]^S)= h_{X^{\wedge 2}} \big( \sum_{t,s}x_{ts}\cdot (j_t\wedge j_s)_E\big)\\
		&=(-1)^k \sum_{t,s}x_{ts}\cdot (j_t\wedge j_s)_{E\ast}(\iota_m\otimes \iota_n),\\[1ex]
		H_\ast h_X&([(j_1)_E,(j_2)_E]^S)\\
		&= (-1)^{k}H_\ast\big((j_1)_{E\ast}(\iota_m)\otimes (j_2)_{E\ast}(\iota_n)-(j_2)_{E\ast}(\iota_n)\otimes (j_1)_{E\ast}(\iota_m)\big)\\
		&= (-1)^k \big(\overline{(j_1)_{E\ast}(\iota_m)\otimes (j_2)_{E\ast}(\iota_n)}-\overline{(j_2)_{E\ast}(\iota_n)\otimes (j_1)_{E\ast}(\iota_m)}\big)\\
		&=(-1)^{k}\big((j_1\wedge j_2)_{E\ast}(\iota_m\otimes\iota_n)-(j_2\wedge j_1)_{E\ast}(\iota_n\otimes\iota_m)\big),
	\end{align*}
	where the second to last equality holds by \cite[Lemma 5.2.5]{NeisendorferBook}.
	Thus from the equality 
	\[H_\ast h_X([(\iota_1)_E,(\iota_2)_E]^S)=h_{X^{\wedge 2}} H_\ast([(\iota_1)_E,(\iota_2)_E]^S) \] 
	we get $x_{11}=x_{22}=0$, $x_{12}=1, x_{21}=-1$.
\end{proof}

\section{homotopy of $\Omega\Sigma X\xrightarrow{\partial} F_p$ and extension problem}
\label{section: extension problem}

In  this section we give some lemmas for the homotopy properties of the connection map  $\Omega\Sigma X\xrightarrow{\partial} F_p$ and for the  extension problem of the short exact sequence $0\rightarrow  Coker (\partial_{k\ast}) \rightarrow \pi_{k}(C_f)\rightarrow Ker(\partial_{k-1\ast})\rightarrow  0$.

The following  lemma  comes from \cite[Lemma 2.3]{ZP23}, which generalizes the Lemma 4.4.1 of \cite{JXYang}.

\begin{lemma}\label{partial calculate1}
	Let $f: X\rightarrow Y$ be a map , then for the following fibration sequence,
	$$\Omega\Sigma X\xrightarrow{\partial} J(M_{f}, X)\simeq F_p \rightarrow C_f\xrightarrow{p} \Sigma X$$
	there exists homotopy-commutative diagram,
	$$ \footnotesize{\xymatrix{
			X\ar[d]_{\Omega\Sigma } \ar[r]^{f} & Y \ar@{_{(}->}[d]_{j_p}\\
			\Omega\Sigma X \ar[r]^-{\partial} & J(M_f, X)\simeq F_p } }$$
\end{lemma}

\begin{lemma}\label{lem: partial(asigamb)}
	Let $f\colon X\to Y$ be a map, $\beta\in \pi_{n}(S^m)$. Then 
	$\partial_{n\ast}(\alpha\Sigma \beta)=\partial_{m\ast}(\alpha)\beta$ for any $\alpha\in \pi_{m+1}(\Sigma X)$, i.e.,  the following diagram commutes
	$$ \footnotesize{\xymatrix{
			\pi_{m+1}(\Sigma X)\ar[d]_{(\Sigma\beta)^{\ast} } \ar[r]^{\partial_{m\ast}} & \pi_{m}(F_p) \ar[d]^{\beta_{\ast}}\\
			\pi_{n+1}(\Sigma X)\ar[r]^{\partial_{n\ast}} &  \pi_{n}(F_p)~~.} }$$
\end{lemma}
\begin{proof}
	The lemma is obtained by the folloing equations
	\begin{align*}
		\partial_{m\ast}(\alpha\Sigma\beta)=&\partial\Omega_{0}(\alpha\Sigma \beta)=\partial\Omega(\alpha\Sigma \beta)E_{S^{n}}\simeq \partial\Omega(\alpha)(\Omega\Sigma \beta)E_{S^{n}} \\
		=& \partial\Omega(\alpha)E_{S^m}\beta=\partial\Omega_{0}(\alpha)\beta=\partial_{n\ast}(\alpha)\beta.
	\end{align*}
\end{proof}

\begin{lemma}\label{lemm Suspen Fp}
	Let $S^{m}\xrightarrow{\alpha} S^{n}\rightarrow C_{\alpha}\xrightarrow{p} S^{m+1}$ be the cofibration sequence, with $m\geq n\geq 2$. Consider the fibration sequence  $\Omega S^{m+k+1}\xrightarrow{\partial^k}F_{\Sigma^kp} \xrightarrow{i^k}\Sigma^k C_{\alpha}\xrightarrow{\Sigma^kp}S^{m+k+1}$.
	The map $E^k:F_{p}\rightarrow \Omega^{k}F_{\Sigma^k p}$ which is the restriction of $\Omega^{k}\Sigma^k: C_\alpha\rightarrow \Omega^{k}\Sigma^k C_\alpha$ on the fiber  satisfies the  following left homotopy commutative diagram of fibration sequences
	\begin{align}
		\footnotesize{\xymatrix{
				\Omega S^{m+1} \ar[r]^-{\partial}\ar[d]^{\Omega(\Omega^{k}\Sigma^k)}&F_p\ar[r]^{i}\ar[d]^{E^k}& C_\alpha\ar[d]^{\Omega^{k}\Sigma^k}\ar[r]^{p}& S^{m+1}\ar[d]_{\Omega^{k}\Sigma^k}\\
				\Omega^{k+1}S^{m+k+1}  \ar[r]^-{\Omega^{k}\partial^k}&\Omega^{k}F_{\Sigma^kp} \ar[r]^{\Omega^{k}i^{k}} &\Omega^{k}\Sigma^k C_\alpha \ar[r]^-{\Omega^{k}\Sigma^k p}& \Omega^{k}S^{m+k+1};
		} }~\footnotesize{\xymatrix{
				S^n \ar@{^{(}->}[r]^-{j_p}\ar@{^{(}->}[d]_{a\Omega^k\Sigma^k } & F_p \ar[d]_{E^k}\\
				\Omega^k S^{n+k} \ar[r]^-{\Omega^kj_{\Sigma^kp}} & \Omega^kF_{\Sigma^kp} } } \nonumber
	\end{align}
	where	$\Omega^k\Sigma^k$ represents the $k$-fold canonical inclusions.
	Then $E^k$ satisfies the above right homotopy commutative diagram with $a=\pm 1$ for $m>n$ and $(a,t)=1$ for $m=n$ and $\alpha=t\iota_n$.
\end{lemma}
\begin{proof}
	Since $Sk_{n}(\Omega^{k}F_{\Sigma^kp})=S^n\hookrightarrow \Omega^{k}F_{\Sigma^kp}$ is homotopic to the composition of the maps  $S^n\xrightarrow{\Omega^{k}\Sigma^k}  \Omega^{k}S^{n+k}\xrightarrow{\Omega^kj_{\Sigma^kp}}\Omega^kF_{\Sigma^kp}$, there is an integer $a$ such that the right  square of  the lemma is homotopy commutative. Then the condition of $a$ in the lemma is obtained by observing the  commutative diagram of $n$-dimensional homotopy groups which is induced by the following homotopy commutative diagram
	\begin{align*}
		\footnotesize{\xymatrix{
				S^n \ar@/^0.5pc/[rr]^{i}\ar[r]_-{j_p}\ar[d]_{\Omega(\Omega^{k}\Sigma^k)}&F_p\ar[r]_{i}\ar[d]^{E^k}& C_\alpha\ar[d]^{\Omega^{k}\Sigma^k}\\
				\Omega^{k}S^{n+k} \ar@/_0.8pc/[rr]_{\Omega^{k}\Sigma^k i} \ar[r]^-{\Omega^{k}j_{\Sigma^kp}}&\Omega^{k}F_{\Sigma^kp} \ar[r]^{\Omega^{k}i^{k}} &\Omega^{k}\Sigma^k C_\alpha.
		} }
	\end{align*}
\end{proof}
If $k$ is large enough in above lemma, we get the following corollary
\begin{corollary}\label{cor: suspen Fp Moore} For $a\in \Z$ is given in Lemma \ref{lemm Suspen Fp},
	there is the following two commutative squares
	\begin{align*}
		\xymatrix{
			\pi_{u+1}(S^{m+1}) \ar[r]^-{\partial_{u\ast}}\ar[d]_{\Sigma^{\infty} } & \pi_{u}(F_p) \ar[d]_{E^{\infty}}\\
			\pi^s_{u}(S^{m})\ar[r]^-{\alpha_{\ast}} & 	\pi^s_{u}(S^{n});} ~~	\xymatrix{
			\pi_{u}(S^{n}) \ar[r]^-{j_{p\ast}}\ar[dr]_{a\Sigma^{\infty} } & \pi_{u}(F_p) \ar[d]_{E^{\infty}}\\
			& 	\pi^s_{u}(S^{n}).}
	\end{align*}
\end{corollary}

\begin{lemma}\label{lem skFp nosusupension}
	Let $S^{m}\xrightarrow{f} C_{\alpha}\rightarrow C_{f}\xrightarrow{p} S^{m+1}$ be the cofibration sequence, where $C_\alpha=S^n\cup_{\alpha}e^{n+k+1}$ with  $k\geq 1,n\geq 2, m>n, k+2$.  $F_p$ is the homotopy fiber of map $p$.   Then 
	\begin{align*}
		Sk_{n+m}F_p\simeq C_{\alpha}\cup_{[i_\alpha, f]} e^{n+m}
	\end{align*} 
	where $i_\alpha: S^n\hookrightarrow C_\alpha$ is the canonical inclusion.
\end{lemma}
\begin{proof}
	$F_p\simeq J(M_f, S^m)$ where  $J_{1}(M_f, S^m)=M_f\simeq C_\alpha$. $\Sigma J_{2}(M_f, S^m)\simeq \Sigma C_\alpha\vee \Sigma C_\alpha\wedge S^m=S^{n+1}\cup  e^{n+k+2}\cup e^{n+m+1}\cup e^{n+k+m+2}$. By the natural inclusion $J_{1}(M_f, S^m)\hookrightarrow J_{2}(M_f, S^m)$, we have 
	\begin{align*}
		&Sk_{n+m-1}F_p\simeq Sk_{n+m-1}J_{2}(M_f, S^m)\simeq J_{1}(M_f, S^m)\simeq C_\alpha;\\
		&Sk_{n+k+m}F_p\simeq Sk_{n+m}J_{2}(M_f, S^m)\simeq  C_\alpha\cup_{\beta}e^{n+m}, \text{~for some~} \beta\in \pi_{n+m-1}(C_\alpha).
	\end{align*}
	
	We have the following  two commutative diagrams for cofibration sequences left and fibration sequences right respectively.
	
	$  \small{\xymatrix{
			S^m \ar@{=}[d]\ar[r]^-{f}&C_\alpha\ar[r] & C_f\ar[r]^-{p}& S^{m+1}\ar@{=}[d]\\
			S^m\ar[r]^-{j_2^m}&S^n\vee S^m\ar[u]_{(i_\alpha, f)} \ar[r]^-{\iota_n\vee i_C}& S^n\vee C_{j_2^m}\ar[u]^{\bar{\omega}}\ar[r]^-{p_{\vee}}& S^{m+1}
	} }
	$, ~~~~$ \small{\xymatrix{
			\Omega S^{m+1} \ar@{=}[d]\ar[r]&F_{p}\ar[r] &C_f \\
			\Omega S^{m+1} \ar[r]&F_{p_{\vee}}\ar[u]^{\omega} \ar[r]& S^n\vee C_{j_2^m}\ar[u]^{\bar{\omega}}
	} }
	$,
\newline	where the homotopy fiber $F_{p_{\vee}}\simeq J(M_{j_2^m}, S^m)$ and 
	$J_2(M_{j_2^m}, S^m)\simeq (S^n\vee S^m)\cup_{[id, j_2^m]} C (S^n\vee S^m)\wedge S^{m-1}\simeq (S^n\vee S^m)\cup_{([j_1^n, j_2^m], [j_2^m, j_2^m])}C(S^{n+m-1}\vee S^{2m-1})$.
	
	Consider the following homotopy commutative diagram of the cofibration sequences
	
	$  \small{\xymatrix{
			J_{1}(M_f, S^m)\simeq C_\alpha\ar[r]^{j_p} &J_{2}(M_f, S^m)\ar[r]& C_\alpha\wedge S^{m}\\
			J_{1}(M_{j_2^m}, S^m)\simeq S^n\vee S^m\ar[u]_{\omega |_{J_{1}}=(i_\alpha, f)} \ar[r]^-{j_{p_{\vee}}}&J_{2}(M_{j_2^m}, S^m)\ar[u]^{\omega |_{J_{2}}}\ar[r]&  (S^n\vee S^m)\wedge S^{m}\ar[u]_{(i_\alpha, f)\wedge \iota_m}.
	} }
	$
	
	It is easy to get that the induced homomorphism  $H_{n+m}((i_\alpha, f)\wedge \iota_m)$	is isomorphic and so is $H_{n+m}(\omega |_{J_{2}})$.
	
	We can decompose the following two inclusions $j_p$ and $j_{p_{\vee}}$ by 
	\begin{align*}
		&j_p: C_\alpha\xrightarrow{j'_p} Sk_{n+m}J_{2}(M_f, S^m)\hookrightarrow J_{2}(M_f, S^m), \\
		&j_{p_{\vee}}: S^n\vee S^m \xrightarrow{j'_{p_{\vee}}} Sk_{n+m}J_{2}(M_{j_2^m}, S^m)\hookrightarrow J_{2}(M_{j_2^m}, S^m)
	\end{align*}

	We have  the following homotopy commutative diagram of the cofibration sequences 
	
	$\small{\xymatrix{
			S^{n+m-1}\ar[r]^-{\beta} & J_1(M_{f},S^m)\simeq C_\alpha \ar[r]^-{j'_p}&Sk_{n+m}J_2(M_{f},S^m)\ar[r]^-{proj} &S^{n+m}\\
			S^{n+m-1}\ar[u]_{\theta} \ar[r]^-{~~[j_1^n, j_2^m]~~}& J_1(M_{j_2^m},S^m)\simeq S^n\vee S^m\ar[u]^{\omega|_{J_1}}\ar[r]^-{j'_{p_{\vee}}}& Sk_{n+m}J_2(M_{j_2^m},S^m)\ar[u]^{(\omega|_{J_2})|_{Sk_{n+m}}}\ar[r]^-{proj}&S^{n+m}\ar[u]^{\Sigma \theta}
	} }
	$

	where $(\omega|_{J_2})|_{Sk_{n+m}}$ is the restriction of $\omega|_{J_2}$ to the $(n+m)$-skeleton and  there is a  $\theta$ making the left square homotopy commutative because the following sequence 
	\newline
	$\pi_{n+m-1}(S^{n+m-1})\xrightarrow{\beta_{\ast}} \pi_{n+m-1}(C_\alpha)\xrightarrow{j'_{p\ast}}\pi_{n+m-1}(Sk_{n+m}J_2(M_{f},S^m))$	is exact by Theorem 1.16 of \cite{Cohen} and $j'_{p\ast}(\omega|_{J_1}[j_1^n, j_2^m])=0$. 
	
	The isomorphism  $H_{n+m}(\omega |_{J_{2}})$ implies that $H_{n+m}((\omega|_{J_2})|_{Sk_{n+m}})$  is isomorphic. Thus we get $H_{n+m}(\Sigma \theta)$ is also isomorphic and so is the $H_{n+m-1}(\theta)$. Hence $\theta =\pm \iota_{n+m-1}$. Then 
	\begin{align*}
		&\beta=\pm (\omega |_{J_{1}})[j_1^n, j_2^m]=\pm (i_\alpha, f)[j_1^n, j_2^m]=\pm[i_\alpha, f].
	\end{align*}
	Since the sign $\pm$ in $\beta=\pm[i_\alpha, f]$ does not effect the homotopy type of $Sk_{n+m}J_{2}(M_f, S^m)\simeq  C_\alpha\cup_{\beta}e^{n+m}$, we choose   $\beta=[i_\alpha, f]$. It completes the proof of this lemma.
	
\end{proof}

\begin{note}
	In above lemma, $C_\alpha$  may not be a suspension, so we can not use the Gray's Theorem to get $J_2(M_f, S^m)\simeq C_\alpha\cup_{[id, f]}C(C_\alpha\wedge S^{m-1})$.
\end{note}

The following two lemmas will be used to solve the extension problem of the short exact sequence of  the homotopy groups.

\begin{lemma}[\text{\cite[Lemma 2.7]{JZhtpygps}}]\label{lem: split Suspen Cf}
	Let $X\xrightarrow{f}Y\stackrel{i}\hookrightarrow C_f\xrightarrow{p}\Sigma X$ be a cofibration sequence of CW complexes. Consider the fibration sequence  $\Omega\Sigma^{k+1}X\xrightarrow{\partial^k}F_{\Sigma^kp} \xrightarrow{i^k}\Sigma^k C_f\xrightarrow{\Sigma^kp}\Sigma^{k+1}X$.
	Assume that the following restriction of $\Sigma^k: \pi_{n}(\Sigma X)\rightarrow \pi_{n+k}(\Sigma^{k+1}X)$ is isomorphic
	\begin{align}
		Ker(\partial_{n\!-\!1\ast}\!\!:\!\pi_{n}(\Sigma X)\!\rightarrow \!\pi_{n\!-\!1}(F_p))\xrightarrow{\Sigma^k}  Ker(\partial^k_{n\!+\!k\!-\!1\ast}\!\!:\!\pi_{n+k}(\Sigma^{k+1} X)\!\rightarrow \pi_{n\!+\!k-\!1}(F_{\Sigma^kp})) \nonumber
	\end{align}
	If $0\rightarrow Coker \partial_{n\ast}\rightarrow \pi_{n}(C_f)\rightarrow  Ker \partial_{n-1\ast} \rightarrow 0$
	is split, then so is the short exact sequence
	\begin{align*}
		0\rightarrow Coker \partial^k_{n+k\ast}\rightarrow \pi_{n}(\Sigma^k C_f)\rightarrow  Ker \partial^k_{n+k-1\ast}\rightarrow 0.
	\end{align*}
\end{lemma}

\begin{lemma}\label{lem L3(X)split}
	For  a $2$-local suspension space $X$ which is $2$-cell complexes such that the reduced homology $\bar H_{\ast}(X,\Z_2)$ of dim $2$ with generators $u$ and $v$ such that 
	$|u|<|v|$, where $|x|$ denotes the degree of $x\in\bar H_{\ast}(X,\Z_2)$. Then 
	$\pi_{i}(\Sigma^{1+|u|+|v|}X)$ is a direct summand of $\pi_{i}(\Sigma X)$ for any positive integer $i$.
\end{lemma}
\begin{proof}
	By Proposition 3.1 of \cite{ChenWu}, $L_3(X)\simeq \Sigma^{|u|+|v|}X$, where   $L_3(X)$ is defined in Notation 1.5 of \cite{Wucombina}. Then by Corollary 4.1 of \cite{Wucombina},
	\begin{align*}
		\Omega\Sigma X\simeq \Omega\Sigma^{1+|u|+|v|}X\times \text{other space}.
	\end{align*}
	This implies the Lemma \ref{lem L3(X)split}.
\end{proof}

\begin{lemma}\label{lem split}
	Let $S^{n+k}\xrightarrow{\alpha} S^{n}\xrightarrow{i} C_{\alpha}\xrightarrow{p} S^{n+k+1}$ be the cofibration sequence, with $k\geq 0, n\geq 2$. For any $\Sigma \theta\in \pi_{m}(S^{n+k+1}) $ with  $\alpha\theta =0$,  there is an element  $-\mathbf{x}$ in the Toda bracket $\{i,\alpha, \theta\}\subset \pi_{m}(C_{\alpha})$, such that $p(\mathbf{x})=\Sigma \theta$. 
\end{lemma}
\begin{proof}
	From Corollary 2.5.1 of \cite{JXYang},  $\iota_{n+k+1}\in \{p,i,\alpha\}$. By Proposition 1.4 of \cite{Toda}, $-\Sigma \theta\in -\{p,i,\alpha\}\Sigma \theta=p\{i,\alpha, \theta\}$. This implies the existence of $\mathbf{x}$.
\end{proof}

\begin{lemma}\label{lem:Toda bracket} $r$ is a positive integer. Assume $2^r\beta=0$ for an $\beta\in \pi_{m+k}(S^{m}), k>0$. Then for $\alpha\in \pi_{m+1}(S^{n_1})$, $\gamma\in \pi_{n_2}(S^{m+k})$, 
	\begin{align*}
		\{2^r\alpha, \Sigma\beta, 2^r\Sigma\gamma\}\ni \left\{
		\begin{array}{ll}
		\alpha\Sigma\beta\eta_{m+k+1}\Sigma^2\gamma, & \hbox{$r=1$;} \\
			0, & \hbox{$r\geq 2$.}
		\end{array}
		\right.
	\end{align*}
\end{lemma}
\begin{proof}
\begin{align*}
		&\{2^r\alpha, \Sigma\beta, 2^r\Sigma\gamma\}=	\{\alpha(2^r\iota_{m+1}), \Sigma\beta, (2^r\iota_{m+k+1})\Sigma\gamma\}\supset 	\alpha\{2^r\iota_{m+1}, \Sigma\beta, 2^r\iota_{m+k+1}\}\Sigma^2\gamma\\
		&\supset 	\alpha\{2^r\iota_{m+1}, \Sigma\beta, 2^r\iota_{m+k+1}\}_{1}\Sigma^2\gamma\ni \left\{
		\begin{array}{ll}
			\alpha\Sigma\beta\eta_{m+k+1}\Sigma^2\gamma, & \hbox{$r=1$;} \\
			0, & \hbox{$r\geq 2$}
		\end{array}
		\right.
\end{align*}	
where the above  first $``\supset"$ comes from ~\text{\cite[Proposition 1.2]{Toda}}; the second  $``\supset"$ comes from \text{\cite[(1.15)]{Toda}} and the last $``\ni"$ comes from  \text{\cite[Corollary 3.7]{Toda}}. 
\end{proof}

\section{Skeletons of the homotopy fiber}
\label{sec:attaching maps}
For an abelian group $A$, $A^{\oplus k}$ denotes the direct sum of $k$-copies of $A$.  $\lr{\alpha_1, \alpha_2,\cdots, \alpha_s}$ denotes the subgroup of $A$ which is generated by $\alpha_1, \alpha_2,\cdots, \alpha_s\in A$. Let $A_1$ be a subgroup of $A$. Then ``$a\equiv b$ mod $A_1$"  means $a-b\in A_1$  and ``$a\in H$ mod $A_1$" or ``$ H\ni a$ mod $A_1$" means $H$, which contains an element $a$,  is a coset of subgroup $A_1$.  For a map $f:X\rightarrow Y$ and $\alpha\in [Z,Y]$, if there is an $\tilde{\alpha}\in [Z,X]$ such that $f_{\ast}(\tilde{\alpha})=p\tilde{\alpha}=\alpha$, then $\tilde{\alpha}$ is called the lift of $\alpha$ by $f$.

$\bullet~$ For $P^{k+1}(2^r)$, there is a canonical cofibration sequence
\begin{align}
	S^k\xrightarrow{2^{r}\iota_k} S^{k}\xrightarrow{i_{k}} P^{k+1}(2^r)\xrightarrow{p_{k+1}}  S^{k+1} \label{Cofiberation for Mr^k}
\end{align}
Note that for $r=0$,  $P^{k+1}(2^0)=\ast$ is  contractible.
We simplify the notation $F_{p_{k+1}}$ by $F_{k+1}$, which is the homotopy fiber of the pinch map $p_{k+1}$ above, by denoting $J_{i,k}^r:=J_{i}(M_{2^r\iota_k}, S^k)$, we have 
\begin{align*}
	Sk_{3k-1}(F_{k+1})\simeq J_{2,k}^r\simeq S^k\cup_{\gamma^r_2}CS^{2k-1},~~\gamma_2^r=2^r[\iota_k, \iota_k];\\
	Sk_{4k-1}(F_{k+1})\simeq J_{3,k}^r\simeq J_{2,k}^r\cup_{\gamma^r_3} CS^{3k-1}
\end{align*}
We have the natural inclusion  $I^r_{i-1}: J_{i-1,k}^r\rightarrow J_{i,k}^r$.  and  denote the natural  inclusion $J_{i,k}^r \rightarrow F_{k+1}$ by $ \check I^r_{i}$. So we have $\check  I^r_{i}I^r_{i-1}=\check I^r_{i-1} $.

Let $\beta_{k}=j_{p_{k+1}}:S^k\hookrightarrow J_{2,k}^r$ and $\bar{p}_{2k}: J_{2,k}^r\rightarrow S^{2k}$ be the canonical pinch map. 

Moreover if $\gamma_2^r=2^r[\iota_k,\iota_k]\simeq 0$,   then there is a canonical inclusion 
$\beta_{2k}: S^{2k} \hookrightarrow J_{2,k}^r\simeq S^{k}\vee S^{2k}$ and in this case, $\beta_{k}=j_{1}^k$, $\bar{p}_{2k}=q_{2}^{2k}$.

Let $\Omega S^{k+1}\xrightarrow{\partial^{k+1,r}}  F_{k+1}\xrightarrow{\tau_{k+1}} P^{k+1}(2^r)\xrightarrow{p_{k+1}}  S^{k+1}$ be the homotopy fibration sequence
which implies the following   exact sequence with homotopy commutative squares by Lemma \ref{partial calculate1}
\begin{align} \small{\xymatrix{
			\pi_{m+1}(S^{k+1})\ar[r]^-{ \partial^{k+1,r}_{m\ast}}&\pi_{m}(F_{k+1})\ar[r]^-{\tau_{k+1\ast}}&\pi_{m}(P^{k+1}(2^r))\ar[r]^-{p_{k+1\ast}}&\pi_{m}(S^{k+1})\ar[r]^-{ \partial^{k+1,r}_{m-1\ast}}& \pi_{m-1}(F_{k+1})\\
			\pi_{m}(S^k)\ar[r]^-{(2^r\iota_{k})_{\ast}}\ar[u]_{\Sigma } &	\pi_{m}(S^k)\ar[u]_{\beta_{k\ast}}&&	\pi_{m-1}(S^k)\ar[r]^-{(2^r\iota_{k})_{\ast}}\ar[u]_{\Sigma} &	\pi_{m-1}(S^k)\ar[u]_{\beta_{k\ast}} }}  \label{exact:hgps piPk+1(2r)}
\end{align}
For any $\alpha \in Ker(\partial^{k+1,r}_{m-1\ast})$,  We always denote the $\widetilde{\alpha}\in \pi_{m}(P^{k+1}(2^r))$ a lift of $\alpha$ by $p_{k+1}$, i.e.,  $p_{k+1\ast}(\widetilde{\alpha})=\alpha$.

$\bullet~$ For $J_{2,k}^r$, there are the following cofibration and fibration sequences 
\begin{align*}
	\small{	S^{2k-1}\!\xrightarrow{\gamma^r_2}S^{k}\!\xrightarrow{\beta_k}\!J_{2,k}^r\!\xrightarrow{\bar{p}_{2k}}\!S^{2k}; ~~
		\Omega S^{2k}\!\xrightarrow{\bar{\partial}^{2k,r}}\!F_{\bar{p}_{2k}}\!\xrightarrow{\bar{\tau}_{2k}}\!J_{2,k}^r\!\xrightarrow{\bar{p}_{2k}}\!S^{2k}.} 
\end{align*}
If  $[\iota_k, [\iota_k, \iota_k]]=0$ (this holds under 2-localization or for odd $k$ by  \cite[Theorem 8.8 in Chapter XI]{GTM61} ), then   $[\iota_k, \gamma_2]=0$. Denoting $J_{2}(M_{\gamma_2}, S^{2k-1})$ by $J_{2,k}^{\gamma,r}$ and $I_2^{\gamma, r}:J_{2,k}^{\gamma,r}\rightarrow F_{\bar{p}_{2k}}$ the canonical inclusion.
\begin{align}
	Sk_{5k-3}(F_{\bar{p}_{2k}})\simeq J_{2,k}^{\gamma,r}\simeq  S^k\vee S^{3k-1}. \label{equ: Sk_5k-3 barFp2k}
\end{align}

Let $\bar{\beta_{i}}: S^i\hookrightarrow S^k\vee S^{3k-1}\simeq J_{2,k}^{\gamma,r}\simeq  Sk_{5k-3}(F_{\bar{p}_{2k}})\hookrightarrow  F_{\bar{p}_{2k}} (i=k, 3k-1)$, where the first map is the canonical inclusions of the corresponding wedge summands. 

Note that $\bar{\tau}_{2k}\bar\beta_{k}\simeq \beta_k$. We also denote $\beta_{3k-1}:=\bar{\tau}_{2k}\bar\beta_{3k-1}$

Similar as the definition of $j_p$ in Section \ref{subsec: Relative James construction}, we also denote the composition
$S^i\xrightarrow{\beta_i} J_{2,k}^r\xrightarrow{\check I^r_2} F_{p}$ by $\beta_i$ without ambiguous for $i=k,2k,3k-1$.  

Let $Im(h):=h(G)$ denote the Image of the homomorphism  $h: G\rightarrow H$ of abelian groups.
\begin{lemma}\label{lem: exact seq pi_m(J2(2rl4))}
	Let $m\leq 4k-3$,	there is a short exact sequence 
	\begin{align*}
		\small{	0\rightarrow \frac{\bar\beta_{k\ast}(\pi_{m}(S^{k}))}{\bar\beta_{k\ast}\gamma^r_{2\ast}(\pi_{m}(S^{2k\!-\!1}))}\oplus G_{m}^{k}\xrightarrow{\bar{\tau}_{2k\ast}}\pi_{m}(J_{2,k}^r)\xrightarrow{\bar{p}_{2k\ast}}\Sigma(Ker(\gamma^r_{2\ast}))\rightarrow 0} 
	\end{align*}
	where the right $\gamma^r_{2\ast}$ is the induced homomorphism on $\pi_{m-1}(-)$ by $\gamma^r_2$ and
	$G_{m}^k=\left\{
	\begin{array}{ll}
		0, & \hbox{$m\leq 3k\!-\!2$;} \\
		\bar\beta_{3k\!-\!1\ast}(\pi_{m}(S^{3k\!-\!1})), & \hbox{$3k\!-\!2\leq m\leq 4k-3$.}
	\end{array}
	\right.$ 
\end{lemma}
\begin{proof}
	By  Lemma \ref{partial calculate1}, we have exact sequence with commutative squares
	\begin{align*} \footnotesize{\xymatrix{
				\pi_{m+1}( S^{2k})\ar[r]^-{ \bar{\partial}^{2k,r}_{m\ast}}&\pi_{m}(F_{\bar{p}_{2k}})\ar[r]^-{\bar{\tau}_{2k\ast}}&\pi_{m}(J_{2,k}^r)\ar[r]^-{\bar{p}_{2k\ast}}&\pi_{m}(S^{2k})\ar[r]^-{ \bar{\partial}^{2k,r}_{m-1\ast}}& \pi_{m-1}(F_{\bar{p}_{2k}})\\
				\pi_{m}(S^{2k\!-\!1})\ar[r]^-{\gamma^r_{2\ast}}\ar@{->>}[u]_{ \Sigma } &\pi_{m}(S^{k})\ar[ur]_-{\beta_{k\ast}}\ar@{^{(}->}[u]_{ \bar{\beta}_{k\ast}}&&\pi_{m-1}(S^{2k-1})\ar[r]^-{\gamma^r_{2\ast}}\ar[u]_{\cong\Sigma } &	\pi_{m-1}(S^{k})\ar@{^{(}->}[u]_{ \bar{\beta}_{k\ast}} }}  
	\end{align*}
	where the  suspension homomorphisms $\Sigma$  in the left and right squares are epimorphism  and isomorphism respectively by the  Freudenthal Suspension Theorem. Note that 
	\begin{align*}
		\pi_{m}(F_{\bar{p}_{2k}})=\bar\beta_{k\ast}(\pi_{m}(S^{k}))\oplus \bar\beta_{3k-1\ast}(\pi_{m}(S^{3k-1})) (m\leq 4k-3),
	\end{align*}
	Thus we complete the proof of this lemma by the following
	\begin{align*}
		&Coker(\bar{\partial}^{2k,r}_{m\ast})=\frac{\pi_{m}(F_{\bar{p}_{2k}})}{Im(\bar\beta_{k\ast}\gamma^r_{2\ast})}=\frac{\bar\beta_{k\ast}(\pi_{m}(S^{k}))}{\bar\beta_{k\ast}\gamma^r_{2\ast}(\pi_{m}(S^{2k\!-\!1}))}\oplus \bar\beta_{3k-1\ast}(\pi_{m}(S^{3k-1}));\\
		&Ker(\bar{\partial}^{2k,r}_{m-1\ast})=\Sigma(Ker(\gamma^r_{2\ast})).	
	\end{align*}	
\end{proof}
\begin{remark}\label{remark: pi3k-1(J2(2rl4)) }
	Take $m=3k-1$ in Lemma \ref{lem: exact seq pi_m(J2(2rl4))}, we get short exact sequence 
	\begin{align*}
		\small{	0\rightarrow \frac{\bar\beta_{k\ast}(\pi_{3k\!-\!1}(S^{k}))}{\bar\beta_{k\ast}\gamma^r_{2\ast}(\pi_{3k\!-\!1}(S^{2k\!-\!1}))}\oplus \Z_{(2)}\{\bar{\beta}_{3k\!-\!1}\}\xrightarrow{\bar{\tau}_{2k\ast}}\pi_{3k\!-\!1}(J_{2,k}^{r})\xrightarrow{\bar{p}_{2k\ast}}\Sigma(Ker(\gamma^r_{2\ast}))\rightarrow 0} 
	\end{align*}
	where the right $\gamma^r_{2\ast}$ is the induced homomorphism on $\pi_{3k-2}(-)$ by $\gamma^r_2$.
\end{remark}

Note that in the following text, for any $\alpha \in \Sigma(Ker(\gamma^r_{2\ast}))$,  we always denote the $\widehat{\alpha}\in \pi_{3k-1}(J_{2}(2^r\iota_k))$ a lift of $\alpha$ by $\bar{p}_{2k}$, i.e.,  $\bar{p}_{2k\ast}(\widehat{\alpha})=\alpha$.

$\bullet~$ For $J_{3,k}^{r}$, there are the following cofibration and fibration sequences 
\begin{align*}
	\small{S^{3k-1}\!\xrightarrow{\gamma^r_3}\!J_{2,k}^{r}\!\xrightarrow{I_2^r}\! J_{3,k}^{r}\!\xrightarrow{\check{p}_{3k}}\!S^{3k};~~
		\Omega S^{3k}\xrightarrow{\check{\partial}^{3k,r} }F_{\check{p}_{3k}}\!\xrightarrow{\check \tau_{3k}}\!J_{3,k}^{r}\!\xrightarrow{\check{p}_{3k}}\!S^{3k}.} 
\end{align*}
We denote $I_{2}^r\beta_{i}$   by $\check\beta_{i}$ for $i=k,2k, 3k-1$.

Then from Lemma \ref{lem skFp nosusupension}, we get 
\begin{align}
	Sk_{4k-1}(F_{\check{p}_{3k}})\simeq J_{2,k}^{r}\cup_{[\beta_{k}, \gamma^r_3]}e^{4k-1}. \label{equ: Sk4k-1F_3k}
\end{align}
\begin{lemma}\label{lem: exact seq pi_m(J3(2rl4))}
	Let $m\leq 4k-2$, there is a short exact sequence 
	\begin{align*}
		\small{	0\rightarrow \frac{\pi_{m}(J_{2,k}^r)}{\gamma^r_{3\ast}(\pi_{m}(S^{3k\!-\!1}))+L_{m}^k}\xrightarrow{I_{2\ast}^r}\pi_{m}(J_{3,k}^r)\xrightarrow{\check{p}_{3k\ast}}\Sigma(Ker(\gamma^r_{3\ast}))\rightarrow 0} 
	\end{align*}
	where the right $\gamma_{3\ast}$ is the induced homomorphism on $\pi_{m-1}(-)$ by $\gamma_3$ and \newline
	$L_{m}^k=\left\{
	\begin{array}{ll}
		0, & \hbox{$m\leq 4k\!-\!3$;} \\
		\lr{[\beta_{k},\gamma^r_3]}, & \hbox{$ m=4k-2$.}
	\end{array}
	\right.$ 
\end{lemma}

Note that in the following text, for any $\alpha \in \Sigma_{\ast}(Ker(\gamma^r_{3\ast}))$,  we always denote the $\widearc{\alpha}\in \pi_{m}(J_{3,k}^r)$ a lift of $\alpha$ by $\check{p}_{3k}$, i.e.,  $\check{p}_{3k\ast}(\widearc{\alpha})=\alpha$.

\begin{proof}
	From (\ref{equ: Sk4k-1F_3k}), we have the isomorphism 	
	\begin{align*}
		\pi_{m}(J_{2,k}^r)/L_{m}^k\xrightarrow{j_{\check{p}_{3k}\ast}~\cong } \pi_{m}(F_{\check{p}_{3k}}) (m\leq 4k-2).
	\end{align*}
	Then this lemma is obtained by the following exact sequence with commutative triangle and squares for $m\leq 4k-2$
	\begin{align*} \small{\xymatrix{
				\pi_{m+1}( S^{3k})\ar[r]^-{ \check{\partial}^{3k,r}_{m\ast}}&\pi_{m}(F_{\check{p}_{3k}})\ar[r]^-{\check{\tau}_{3k\ast}}&\pi_{m}(J_{3,k}^r)\ar[r]^-{\check{p}_{3k\ast}}&\pi_{m}(S^{3k})\ar[r]^-{ \check{\partial}^{2k,r}_{m-1\ast}}& \pi_{m-1}(F_{\check{p}_{3k}})\\
				\pi_{m}(S^{3k\!-\!1})\ar[r]^-{\gamma^r_{3\ast}}\ar[u]_{ \cong\Sigma } &\pi_{m}(J_{2}(2^r\iota_k))\ar[ur]_-{I^r_{2\ast}}\ar@{->>}[u]_{ j_{\check{p}_{3k}\ast}}&&\pi_{m-1}(S^{3k-1})\ar[r]^-{\gamma^r_{3\ast}}\ar[u]_{\cong\Sigma } &	\pi_{m-1}(J_{2,k}^r)\ar[u]_{ \cong j_{\check{p}_{3k}\ast}} }}  
	\end{align*}
	where the  suspension homomorphisms $\Sigma$  in squares are  isomorphisms  by the  Freudenthal Suspension Theorem; $j_{\check{p}_{3k}\ast}$  are isomorphisms except the left one  for $m=4k-2$ which is just an epimorphism.	
\end{proof}

If we need to emphasize the dependency of the above spaces or maps on the variable $r$, we will replace the $F_{k+1}, i_k, p_{k+1}$, $\tau_{k+1}$, $\beta_{i} (i=k,2k,3k-1)$, $\bar{p}_{2k}$, $\bar\tau_{2k}$, $\bar\beta_{i} (i=k, 3k-1)$, $\check{p}_{3k}$, $\check\beta_{i} (i=k,2k,3k-1)$ by   $F^r_{k+1}$, $i^r_k$, $p^r_{k+1}$, $\tau^r_{k+1}$, $\beta^r_{i}$,   $\bar{p}_{2k}^r$, $\bar\tau^r_{2k}$, $\bar\beta^r_{i} $, $\check{p}_{3k}^r$,  $\check\beta^r_{i}$ respectively.

We have the following homotopy commutative diagrams of cofibrations (left) and fibrations (right)
\begin{align}
	\small{\xymatrix{
			S^k\ar[d]^{2^{M_{s\!-\!t}^0} \iota_k} \ar[r]^-{2^{s}\iota_k}&S^k \ar[d]^{2^{M_{t\!-\!s}^0} \iota_k}\ar[r] & P^{k+1}(2^{s}) \ar[d]_{\bar{\chi}^{s}_{t}}\ar[r]^{p^{s}_{k+1}}& S^{k+1}\ar[d]_{2^{M_{s\!-\!t}^0} \iota_{k+1}} \\
			S^k \ar[r]^-{2^{t}\iota_k}&S^k \ar[r]& P^{k+1}(2^{t})\ar[r]^-{p^{t}_{k+1}}& S^{k+1}
	} }
	, \small{\xymatrix{
			\Omega S^{k+1}\ar[d]^{\Omega(2^{M_{s\!-\!t}^0} \iota_{k+1})}  \ar[r]^-{\partial^{k+1,s}}&F^{s}_{k+1}\ar[d]^{\chi_{t}^{s}}\ar[r]^-{\tau_{k+1}^{s}} &P^{k+1}(2^{s}) \ar[d]^{\bar{\chi}_{t}^{s}}\\
			\Omega S^{k+1} \ar[r]^-{\partial^{k+1,t} }&F_{k+1}^{t} \ar[r]^{\tau^{t}_{k+1}}& P^{k+1}(2^{t})
	} }  \label{diam Moore s to t}
\end{align}
where  $M_a^b:=Max\{r,s\}$ be the maximal  number of integers $a$ and $b$.

Let $\bar g^{s}_{t}:=(\chi^{s}_{t})|_{J_2}$, there is a map $ g^{s}_{t}: F_{\bar{p}_{2k}^s}\rightarrow   F_{\bar{p}_{2k}^t}$ such that the following diagrams homotopy commutes
\begin{align}
	\footnotesize{\xymatrix{
			J_{2,k}^{\gamma,s}\ar[r]^-{I_2^{\gamma,s}}\ar[d]^{( g^s_t)|_{J_2}}&F_{\bar p_{2k}^s}\ar[d]^{ g^s_t}\ar[r]^-{\bar\tau_{2k}^s}&J_{2,k}^s\ar[d]^{\bar g^s_t}\ar[r]^-{I_{2}^s}& J_{3,k}^s\ar[d]^{(\chi^s_t) |_{J_3}}\ar[r]^-{\check I_{3}^{s}}&F^s_{k+1}\ar[d]^{\chi_{t}^{s}}\ar[r]^-{\tau_{k+1}^s} &P^{k+1}(2^s) \ar[d]^{\bar{\chi}_{t}^{s}}\\
			J_{2,k}^{\gamma,t}\ar[r]^-{I_2^{\gamma,t}} &F_{\bar p_{2k}^t}\ar[r]^-{\bar\tau_{2k}^t}&J_{2,k}^t \ar[r]^-{I_{2}^t}& J_{3,k}^t\ar[r]^-{\check I_{3}^{t}}&F_{k+1}^{t} \ar[r]^{\tau^{t}_{k+1}}& P^{k+1}(2^{t})
	} } \label{diam big}
\end{align}

From Lemma \ref{Lem: natural J(Mf,X)} and \ref{lem: Natural gamma_n}
\begin{align}
	&	\footnotesize{\xymatrix{
			S^{2k-1}\ar[d]^{2^{|s-t|}\iota_{2k\!-\!1}}\ar[r]^-{\gamma_2^s}&	S^k \ar[d]^{2^{M_{t\!-\!s}^0} \iota_k}\ar[r]^-{\beta_{k}^s}&J_{2,k}^s\ar[d]_{\bar g^s_t}\ar[r]^-{\bar{p}^s_{2k}} &  S^{2k} \ar[d]_{2^{|s-t|}\iota_{2k}}\\
			S^{2k-1} \ar[r]^-{\gamma_2^t} &	S^k \ar[r]^-{\beta_{k}^t}&J_{2,k}^t\ar[r]^-{\bar{p}^t_{2k}} &  S^{2k}
	}}
	, \footnotesize{\xymatrix{
			S^{3k-1}\ar[d]^{ 2^{M^{t\!-\!s}_{\!2(s\!-\!t)}}\!\iota_{3k\!-\!1}}\ar[r]^-{\gamma^s_3}&J_{2,k}^s\ar[d]^{\bar g^s_t }\ar[r]^-{I_{2}^s}& J_{3,k}^{r}\ar[d]_{(\chi^s_t) |_{J_3}}\ar[r]^-{\check p_{3k}^s}& S^{3k}\ar[d]\\
			S^{3k-1}\ar[r]^-{\gamma^t_3}&J_{2,k}^t \ar[r]^-{I_{2}^t}&  J_{3,k}^t\ar[r]^-{\check p_{3k}^t} & S^{3k}
	} } \label{diam 1: J2(2slk) to J2(2tlk)}\\
	& \text{So we have~~~~~~}	\bar g_{t}^s\beta_{k}^s\simeq 2^{M_{t-s}^0} \beta_{k}^t;~~~ 2^{M^{t\!-\!s}_{\!2(s\!-\!t)}} \gamma^t_3\simeq\bar g_{t}^s \gamma^s_3. \label{Equ:gamma30,3r}\\
	& \footnotesize{\xymatrix{
			\Omega S^{2k}\ar[d]^{\Omega(2^{s}\!\iota_{2k})}\ar[r]^-{\bar\partial^{2k,s}}&F_{\bar p^s_{2k}}\ar[d]^{g_{t}^s}\ar[r]^-{\bar\tau_{2k}^s} &J_{2,k}^s \ar[d]^{\bar g_{t}^s}\\
			\Omega S^{2k} \ar[r]^-{\bar\partial^{2k,t}}&F_{\bar p^t_{2k}} \ar[r]^-{\bar\tau_{2k}^t}& J_{2,k}^t
	} };~~~	\footnotesize{\xymatrix{
			S^k \ar[d]_{2^{M_{t\!-\!s}^{0}}\!\iota_{k}}\ar[r]^-{\bar\beta_{k}^s}&J_{2,k}^{\gamma,s}\!=\!S^k\vee S^{3k-1}\ar[d]_{(g^s_t) |_{J_2}}\ar[r]^-{q^{3k-1}_{2}} &   S^{3k-1} \ar[d]_{2^{M_{2(t\!-\!s)}^{s\!-\!t}}\!\iota_{3k\!-\!1}}\\
			S^k \ar[r]^-{\bar\beta_{k}^t}&J_{2,k}^{\gamma,t}\!=\!S^k\vee S^{3k-1}\ar[r]^-{q^{3k-1}_{2}} &  S^{3k-1}  
	}}   \label{diam 2: J2(2slk) to J2(2tlk)}\\
	&\bar g_{t}^s\circ \bar\tau_{2k}^s= \bar\tau_{2k}^tg_{t}^s;  \label{equ1 g^s_t}  \\
	&	(g^s_t) |_{J_2}=2^{M_{t\!-\!s}^{0}}j_1^kq_1^k+ 2^{M_{2(t\!-\!s)}^{s\!-\!t}}j_2^{3k-1}q_{2}^{3k-1}+j_1^k\theta q_{2}^{3k-1},  ~~\theta\in \pi_{3k-1}(S^k).  \label{equ2  g^s_t}
\end{align}

\begin{lemma}\label{lem:gamma3 k odd}
	For $k$ is odd, $r\geq 1$, then 
	\begin{itemize}
		\item [(i)] $J_{2,k}^r\simeq S^k\vee S^{2k}$;
		\item [(ii)] $\gamma_3^r=a_{r}[j_1^k, j_{2}^{2k}]+a_{r}j_1^k[\iota_{k},\mu_r]\in \pi_{3k-1}(S^k\vee S^{2k})$ for some $a_r\in  \Z$, $\mu_r\in \pi_{2k}(S^k)$;
		\item  [(iii)] Moreover, if $k=3,7$, then $\gamma_3^r=\pm 2^r[j_1^k, j_{2}^{2k}]$.
	\end{itemize}
\end{lemma}
\begin{proof}
	For $k$ is odd, Theorem 8.8 of \cite{GTM61} implies $2[\iota_{k},\iota_{k}]=0$. $J_{2,k}^r\simeq S^k\vee S^{2k}$ $(r\geq 1)$.
	
	For(ii),
	there is an isomorphism 
	\begin{align}
		\pi_{3k\!-\!1}( S^{3k-1})\!\oplus\!	\pi_{3k\!-\!1}(S^k)\!\oplus\! \pi_{3k\!-\!1}( S^{2k})  \cong\pi_{3k-1}(S^k\vee S^{2k}) \nonumber 
	\end{align}
	By the well known definition of the above isomorphism and from Lemma \ref{lem: J3 alpha}, we  assume 
	\begin{align*}
		\gamma_3^r= a_r[j_1^k, j_2^{2k}]+j_1^{k}\lambda_r, ~~ a_r\in \Z, \lambda_r \in\pi_{3k-1}(S^k).
	\end{align*}
 Moreover from Corollary \ref{Cor:proj gamma3}, there is an element $\mu_r\in \pi_{2k}(S^k)$ such that 
	\begin{align*}
		&\lambda_r =q_{1}^{k}\gamma^r_3=[q_1^k, \mu_r q_{2}^{2k}]H_2(\gamma^r_3)\\
		&=[\iota_k, \mu_r](\Sigma q_{1}^{k-1}\wedge q_{2}^{2k-1})(H_2(a_r[j_1^k, j_2^{2k}])+H_2(j_1^{k}\lambda_r))  \\
		&=a_r[\iota_k, \mu_r](\Sigma q_{1}^{k-1}\wedge q_{2}^{2k-1})(\Sigma j_{1}^{k-1}\wedge j_{2}^{2k-1}-\Sigma j_{2}^{2k-1}\wedge j_{1}^{k-1})   ~~(\text{by Lemma \ref{lem:H_2-SmSn}} )\\
		&+[\iota_k, \mu_r](\Sigma q_{1}^{k-1}\wedge q_{2}^{2k-1})(\Sigma j_1^{k-1}\wedge j_1^{k-1})H_{2}(\lambda_r)\\
		&=a_r[\iota_k, \mu_r].
	\end{align*} 
	This completes the proof of (ii) of this lemma.

 For (iii), it is well known that $[\iota_k,\iota_k]=0 ~(\text{odd}~k\geq 2)$ if and only if $k=3,7$. Thus $J_{2,k}^0\simeq S^k\vee S^{2k}$ for $k=3,7$. From Proposition 5.9 and Theorem 7.7 of \cite{Toda}, 
	$\Sigma: \pi_{3k-1}(S^{k})\rightarrow \pi_{3k}(S^{k+1})$ are injections, for any $r\geq 0$,
	$Sk_{4k-1}(F_{k+1})\simeq J_{3,k}^r\simeq (S^k\vee S^{2k})\cup_{\gamma_3^r} CS^{3k-1}$ with $\gamma_3^r=a_{r}[j_1^k, j_{2}^{2k}]\in \pi_{3k-1}(S^k\vee S^{2k})$ for some $a_r\in  \Z$ by Remark \ref{remark gamma3}.
	So by the exact sequence (\ref{exact:hgps piPk+1(2r)}) for $r=0$, we get
	\begin{align}
		\pi_{3k}(S^{k+1})\cong \pi_{3k-1}(F_{k+1}^0)\cong \frac{\pi_{3k-1}(S^k)\oplus \pi_{3k-1}( S^{2k})\oplus \Z\{[j_1^k, j_{2}^{2k}]\}}{\lr{(0,~0,~a_0[j_1^k, j_{2}^{2k}])}}. \label{equ1 pi3k(Sk+1)}
	\end{align}
	By the Hopf bundles $S^{k}\rightarrow S^{2k+1}\rightarrow S^{k+1}$ $(k=3,7)$, there is an isomorphism 
	\begin{align}
		\pi_{3k}(S^{k+1})\cong \pi_{3k-1}(S^{k})\oplus  \pi_{3k}(S^{2k+1})\cong \pi_{3k-1}(S^{k})\oplus  \pi_{3k-1}(S^{2k}). \label{equ2 pi3k(Sk+1)}
	\end{align}
	Comparing (\ref{equ1 pi3k(Sk+1)}) and (\ref{equ2 pi3k(Sk+1)}), we get $a_0=\pm 1$, i.e., $\gamma_3^0=\pm[j_1^k, j_{2}^{2k}]$
	
	We assume $\bar g_{0}^r= j^k_1 q^k_1 + 2^r j^{2k}_2 q^{2k}_2+j^k_1 \theta q^{2k}_2\in [ S^k\vee S^{2k},  S^k\vee S^{2k}]$ for some $\theta\in \pi_{2k}(S^k)$.
	By noting that  $q^{k}_1\gamma_{3}^r\simeq 0$, $q^{2k}_2\gamma_{3}^r\simeq 0$, from (\ref{Equ:gamma30,3r}),
	\begin{align}
		2^{2r} \gamma^0_3= &\bar g_{0}^r\circ\gamma^r_3=( j^k_1 q^k_1 + 2^r j^{2k}_2 q^{2k}_2+j^k_1 \lambda q^{2k}_2)\gamma^r_3\nonumber\\
		&=(j^k_1 q^k_1 +2^r j^{2k}_2 q^{2k}_2)\gamma^r_3+j^k_1 \theta q^{2k}_2 \gamma^r_3 \pm
		[ j^k_1 q^k_1+ 2^r j^{2k}_2 q^{2k}_2,j^k_1\theta q^{2k}_2]H_2(\gamma^r_3) \nonumber\\
		&	( j^k_1 q^k_1 +2^r j^{2k}_2 q^{2k}_2)\gamma^r_3=\pm a_r [j^k_1 q^k_1 ,  2^r j^{2k}_2 q^{2k}_2]H_2([j^k_1, j_2^{2k}])\nonumber\\
		&=\pm 2^{r}a_r [ j^k_1,  j^{2k}_2 ](\Sigma q^{k-1}_1\wedge  q^{2k-1}_2)(\Sigma j_1^{k-1}\wedge j_2^{2k-1}- \Sigma j_2^{2k-1}\wedge j_1^{k-1}) \nonumber\\
		&=\pm 2^{r}a_r [ j^k_1,  j^{2k}_2 ].\nonumber\\
		\Rightarrow~~~ &  (2^{2r}\pm 2^ra_r) [ j^k_1,  j^{2k}_2 ]=\pm[j^k_1 q^k_1 + 2^r j^{2k}_2 q^{2k}_2,j^k_1 \theta q^{2k}_2]H_2(\gamma^r_3). \label{equ []H(gamma3r)}
	\end{align}
	Since $\theta$ is a torsion element, so is the right term of (\ref{equ []H(gamma3r)}). However the left term of (\ref{equ []H(gamma3r)}) is torsion free or zero. We get
	\begin{align}
		(2^{2r}\pm 2^ra_r) [ j^k_1,  j^{2k}_2 ]=\pm[j^k_1 q^k_1 + 2^r j^{2k}_2 q^{2k}_2,j^k_1 \theta q^{2k}_2]H_2(\gamma^r_3)=0 \nonumber
	\end{align}
	That is $a_r=\pm 2^r$, we get $\gamma^r_3=\pm 2^r [ j^k_1,  j^{2k}_2 ]$.
\end{proof}

At the end of this section, we generalize the map $\bar H_2: \Omega P^{n+1}(2)\rightarrow \Omega P^{2n+1}(2)$ in  \cite[section 4.4]{WJ Proj plane} to the mod $2^r$ Moore space as follows

$\bar H_2: \Omega P^{n+1}(2^r)\simeq J(P^{n}(2^r))\xrightarrow{H_2}  J(P^{n}(2^r)\wedge P^{n}(2^r))\xrightarrow{J(q)}  J(P^{2n}(2^r)) \simeq \Omega  P^{2n+1}(2^r)$.

Thus we have the following homotopy commutative diagram of fibration sequence 
\begin{align}
	\small{\xymatrix{
			\Omega^2S^{n+1}\ar[d]^{\Omega H_2}\ar[r]^{\Omega\partial^{n+1,r}}&	\Omega F_{n+1}\ar[d]^{\bar H_2}\ar[r]^{\Omega\tau_{n+1}}& 	 \Omega P^{n+1}(2^r) \ar[d]^{\bar H_2}\ar[r]^-{\Omega p_{n+1}} &\Omega S^{n+1} \ar[d]_{H_2}\\
			\Omega^2S^{n+1}\ar[r]^{\Omega\partial^{2n+1,r}}&	\Omega F_{2n+1}\ar[r]^{\Omega\tau_{2n+1}}& \Omega  P^{2n+1}(2^r)\ar[r]^-{\Omega p_{2n+1}}& \Omega S^{2n+1}
	} }. \label{diam: H2 Moore}
\end{align}
where the right vertical map $H_2$ is the second James Hopf invariant defined in (\ref{def:James H2}); the homotopy commutativity of the right square is obtained by the naturality of James construction and the left vertical map $\bar H_2$ is induced by the right homotopy commutative square.

\section{Homotopy groups of $4,5$ and $6$ dimensional mod $2^r$ Moore spaces}
\label{sec:htyp Moore 456}

In this section, we compute homotopy groups of mod $2^r$ Moore spaces in Theorem \ref{Thm:Main thm}.  All the spaces and homotopy groups in the following are under $2$-localization.

Firstly, we give some notations. For a homomorphism $f:G\rightarrow H$ of abelian groups, where $G_1$ is subgroups of $G$, then we denote $f|_{res}:G_1\rightarrow H$ as the restriction of $f$ on $G_1$.
 For a monomorphism  $f:G\hookrightarrow H$ (resp. epimorphism $g: H\twoheadrightarrow K$),  if there is a map $f':H\rightarrow G$ (resp. $g':K\rightarrow H$) such that $f'f=id_{G}$ (resp. $gg'=id_K$), then $f$ (resp. $g$) is called split into (resp. split onto). In this case, $f(G)$ (resp. $g'(K)$) is a direct summand of $H$ and we say $G$ (resp. $K$) splits out of $H$.    $o(\alpha)$ denotes the order of $\alpha$, where $\alpha$ is an element of an abelian group and the sequence $0\rightarrow A\rightarrow B \rightarrow C\rightarrow 0$ always means  a short exact sequence. All the following ladders, squares and triangles of abelian groups (CW-complexes) are  commutative (homotopy commutative). In the following,  $i$ and $p$ denote the $\Sigma^{\infty} i_k\in \pi^s_{k}(P^{k+1}(2^r))$ and $\Sigma^{\infty} p_{k+1}\in \Sigma^{\infty}[P^{k+1}(2^r), S^{k+1}]$.

\subsection{Homotopy groups of $P^{4}(2^r)$}
\label{subsec: P4(2^r)}

$S^3\xrightarrow{2^{r}\iota_3} S^{3}\xrightarrow{i_{3}} P^{4}(2^r)\xrightarrow{p_{4}}  S^{4};
~~\Omega S^4 \xrightarrow{\partial^{4,r}}F_{4}\xrightarrow{\tau_4} P^{4}(2^r)\xrightarrow{p_{4}}  S^{4}$

By (iii) of Lemma \ref{lem:gamma3 k odd}

$Sk_{8}(F_{4})\simeq J_{2,3}^{r}\simeq S^3\cup_{\gamma^r_2}CS^5\simeq S^3\vee S^6 ;~~Sk_{11}(F_{5})\simeq J_{3,3}^{r}\simeq J_{2,3}^{r}\cup_{\gamma^r_3}CS^{8}$
where $\gamma^r_3=\pm 2^r[j_1^3, j_{2}^{6}]$.

\begin{lemma}\label{lem pi(J3(2rl3))}
	For $r\geq 1$, $\pi_{i}(F_4)\cong \pi_{i}(J_{3,3}^r)$ $(i\leq 11)$ 	and \begin{align*}
		(1)~&\pi_{8}(J_{3,3}^{r})=\Z_2\{\check{\beta}_3\nu'\eta_6^2\}\oplus\Z_2\{\check{\beta}_6\eta_6^2\}\oplus \Z_{2^r}\{[\check{\beta}_3,\check{\beta}_6]\} .\\
		(2)~&\pi_{9}(J_{3,3}^{r})=\Z_{8}\{\check\beta_{6}\nu_6\}\oplus\Z_{2}\{[\check{\beta}_3,\check{\beta}_6]\eta_8\}.\\
		(3)~&\pi_{10}(J_{3,3}^{1})=	\Z_{2}\{[\check{\beta_3}, [\check{\beta}_3,\check{\beta}_6]]\}\oplus \Z_{4}\{\widearc{\eta_9}\} ~\text{with}~ [\check{\beta}_3,\check{\beta}_6]\eta_8=2\widearc{\eta_9};\\
		&\pi_{10}(J_{3,3}^{r})=\Z_{2}\{[\check{\beta}_3,\check{\beta}_6]\eta_8^2\}\oplus	\Z_{2^r}\{[\check{\beta_3}, [\check{\beta}_3,\check{\beta}_6]]\}\oplus \Z_{2}\{\widearc{\eta_9}\} (r\geq 2).
	\end{align*}
\end{lemma}

\begin{proof}
	From Lemma \ref{lem: exact seq pi_m(J3(2rl4))} for $k=3,m=8,9$, we get the following isomorphism
	\begin{align*}
		\pi_{m}(S^3\vee S^6)/\gamma^r_{3\ast}(\pi_{m}(S^{3k-1}))\xrightarrow[\cong ]{I_{2\ast}^r} \pi_{m}(J_{3,3}^{r}) ~(m=8,9)
	\end{align*}
	which implies (1) and (2) of Lemma \ref{lem pi(J3(2rl3))}. 
	
	For (3) of Lemma \ref{lem pi(J3(2rl3))}, there is an exact sequence 
	\begin{align*}
		&0\rightarrow Coker(\check{\partial}^{9,r}_{10\ast})\xrightarrow{I_{2\ast}^r} \pi_{10}(J_{3,3}^r) \xrightarrow{\check p_{9\ast}} \pi_{10}(S^9)=\Z_2\{\eta_9\}\rightarrow 0,\\
	& Coker(\check{\partial}^{9,r}_{10\ast})\!=\!\pi_{10}(S^3\!\vee\! S^6)/(\gamma^r_{3}\fhe\pi_{10}(S^8)\!+\!\lr{[j_1^3,\! \gamma^r_3]})\!=\! \Z_{2}\{[j_1^3,\! j_2^6]\eta_8^2\}\!\oplus\! \Z_{2^r}\{[j_1^3,\![j_1^3,\! j_2^6]]\}.
	\end{align*}
	For $r\geq 1$, by Lemma \ref{lem split}, there is an $-\mathbf{x}\in \{I_2^r,\gamma^r_3,\eta_8\}\subset \pi_{10}(J_{3,3}^{r})$, such that $\check p_{9\ast}(\mathbf{x})=\eta_9$. By Proposition 1.2 (IV) and Proposition 1.4 of \cite{Toda},  
	\begin{align*}
		&2\mathbf{x}\in -\{I_2^r,\gamma^r_3,\eta_8\}(2\iota_{10})=I_2^r\{\gamma^r_3,\eta_8,2\iota_9\}.\\
	\Rightarrow ~	&\{\gamma^r_3,\eta_8,2\iota_9\}\ni  \pm 2^{r-1}[j_1^3,j_2^6]\eta_{8}^2 ~\text{mod}~\lr{2[j_1^3,[j_1^3,j_2^6]]}
		~(\text{Lemma \ref{lemma for [a,b,c]}}).
	\end{align*}
	Thus for $r=1$, $2\mathbf{x}=I_2^1[j_1^3,j_2^6]\eta_8=[\check{\beta}_3,\check{\beta}_6]\eta_8$, which implies that $o(\mathbf{x})=4$ and $\pi_{10}(J_{3,3}^{1})= \Z_{2}\{[\check{\beta_3}, [\check{\beta}_3,\check{\beta}_6]]\}\oplus \Z_{4}\{\widearc{\eta_9}\}$ with $[\check{\beta}_3,\check{\beta}_6]\eta_8=2\widearc{\eta_9}$.
	
	For $r\geq 2$, $2\mathbf{x}=I_{2}^r (2t[j_1^3,[j_1^3,j_2^6]]), t\in \Z$. Take 
	$\mathbf{x}'=\mathbf{x}-I_{2}^r (t[j_1^3,[j_1^3,j_2^6]])$. We get
	$2\mathbf{x}'=0$ and $ \check p_{9\ast}(\mathbf{x}')=\eta_9$. So $\pi_{10}(J_{3,3}^r)\stackrel{ \check p_{9\ast}}\twoheadrightarrow \pi_{10}(S^9)$ is split onto. which implies the (3) of Lemma \ref{lem pi(J3(2rl3))} for $r\geq 2$.
\end{proof}

$\bullet~~\pi_{8}(P^{4}(2^r))$. 

Consider diagram (\ref{exact:hgps piPk+1(2r)}) for $k=3, m=8$, we get
\newline
$\xymatrix{
	\pi_{9}(S^{4})\ar[r]^-{ \partial^{4,r}_{8\ast}}&\pi_{8}(F_{4})\ar[r]^-{\tau_{4\ast}}&\pi_{8}(P^{4}(2^r))\ar[r]^-{p_{4\ast}}&\pi_{8}(S^{4})\ar[r]^-{ \partial^{4,r}_{7\ast}}& \pi_{7}(F_{5}) }$
\newline
we have $\partial^{4,r}_{7\ast}(\nu_4\eta_7)=(1-\vartheta_r)\beta_3\nu'\eta_6; \partial^{4,r}_{7\ast}(\Sigma\nu'\eta_7)=0$~(\cite[(3.12),(3.13)]{ZP23}).
\newline
$\Rightarrow~ Ker(\partial^{4,r}_{7\ast})=\Z_{2}\{\Sigma\nu'\eta_7\}\oplus \vartheta_r \Z_{2}\{\nu_4\eta_7\}.$
\begin{align} 
	& \partial^{4,r}_{8\ast}(\Sigma\nu'\eta_{7}^2)=	\partial^{4,r}_{7\ast}(\Sigma\nu'\eta_{7})\eta_7=0; \partial^{4,r}_{8\ast}(\nu_4\eta_{7}^2)=\partial^{4,r}_{7\ast}(\nu_4\eta_{7})\eta_7=(1-\vartheta_r)\beta_3\nu'\eta_6^2;\label{equ:partial8^{4,r}(nu4eta7^2)}\\
	&\Rightarrow~	 Coker(\partial^{4,r}_{8\ast})=\vartheta_r\Z_2\{\beta_3\nu'\eta_6^2\}\oplus\Z_2\{\beta_6\eta_6^2\}\oplus \Z_{2^r}\{[\beta_3,\beta_6]\}.\label{equ:Cokerpartial_8^{4,r}}
\end{align}
Thus we have the following short exact sequence 
\begin{align}
	\!\!\!\!\small{0\rightarrow\vartheta_r\Z_2\!\oplus\!\Z_2\!\oplus\! \Z_{2^r}\!\rightarrow \!\pi_{8}(P^4(2^r))\!\rightarrow\!\vartheta_r\Z_2\{\nu_4\eta_7\}\!\oplus\! \Z_2\{\Sigma\nu'\eta_7\}\!\!\rightarrow\! 0} \label{exact:pi8(P4(2r))}
\end{align}
By Lemma \ref{lem split}, there is an $-\mathbf{x}\in \{i_3,2^r\iota_3,\nu'\eta_6\}\subset \pi_{8}(P^4(2^r))$, such that $ p_{4\ast}(\mathbf{x})=\Sigma\nu'\eta_7$. By Lemma \ref{lemA: Todabraket } (i) and (\ref{equ:partial8^{4,r}(nu4eta7^2)})
\begin{align*}
	&2\mathbf{x}\in -\{i_3,2^r\iota_3,\nu'\eta_6\}(2\iota_8)=i_3\{2^r\iota_3,\nu'\eta_6,2\iota_7\}=2^{r-1}i_3\nu'\eta_6^2=0~\Rightarrow~o(\mathbf{x})=2 (r\geq 1).
\end{align*}
Hence (\ref{exact:pi8(P4(2r))}) splits for $r=1$ and $\pi_{8}(P^4(2)\cong \Z_2^{\oplus 3}$.

For $r\geq 2$, firstly we prove the following lemma
\begin{lemma}\label{lem:P7(2r-1),P4(2r)} For $r\geq 2$, there are $\widetilde{\eta_{7}}\in \pi_{7}(P^7(2^{r-1}))$ and $\tilde{\nu}_{4}\in [P^7(2^{r-1}),P^4(2^{r}) ]$ satisfying the following homotopy commutative diagram
	$$	\footnotesize{\xymatrix{
			S^8\ar[rd]_{\eta_7}\ar[r]^{\widetilde{\eta_{7}}}	&P^7(2^{r-1})\ar[d]^{p_{7}^{r-1} } \ar[r]^{\tilde{\nu}_4} & P^{4}(2^r) \ar[d]^{p_{4}^r}\\
			&	S^7\ar[r]^{\nu_4} & S^4 ~~.} }
	$$
	Moreover, the order of $\tilde{\nu}_4\widetilde{\eta_{7}}$ is $4$ and $2$ respectively for $r=2$ and $r\geq 3$.
\end{lemma}
\begin{proof}
	
	$\widetilde{\eta_{7}}$ comes from $\pi_{8}(P^7(2^{r-1}))=\left\{
	\begin{array}{ll}
		\Z_4\{\widetilde{\eta_{7}}\}, & \hbox{$r=2$;} \\
		\Z_2\{i_6\eta_6^2\}\oplus	\Z_2\{\widetilde{\eta_{7}}\} ,& \hbox{$r\geq 3 $}\\
	\end{array}
	\right.$ by \cite{Bau n-1n+3}.
	
	We have the following exact sequence 
	$$	[P^7(2^{r-1}),P^4(2^{r}) ]\xrightarrow{p^r_{4\ast}} [P^7(2^{r-1}),S^4 ]\xrightarrow{\partial_{\ast}^{4,r}}  [P^6(2^{r-1}),F_{4}^r].
	$$
	By noting $o(p_{6}^{r-1})=2^{r-1}$, from Lemma \ref{lem: partial(asigamb)} and (3.5), (3.8) of \cite{ZP23},
	
	$\partial_{\ast}^{4,r}(\nu_4p_{7}^{r-1})=\partial_{6\ast}^{4,r}(\nu_4)p_{6}^{r-1}=(\pm 2^{r-1}\beta_{3}\nu'+2^r\beta_6)p_{6}^{r-1}=0$.
	
	Thus $\tilde\nu_4$ exists.  $\tilde{\nu}_4\widetilde{\eta_{7}}$ is a lift of $\nu_4\eta_7$ by $p_{4}^r$ implies $o(\tilde{\nu}_4\widetilde{\eta_{7}})\geq 2$. So
	$o(\tilde{\nu}_4\widetilde{\eta_{7}})= 2$ for $r\geq 3$, since $o(\widetilde{\eta_{7}})=2$ for $r\geq 3$. 
	
	For $r=2$, $o(\widetilde{\eta_{7}})=4$ and $2\widetilde{\eta_{7}}=i_6\eta_6^2$. Hence, $2\tilde{\nu}_4\widetilde{\eta_{7}}=\tilde{\nu}_4i_6\eta_6^2$.

	Since $(4\iota_{4})\circ \nu_{4}=16\nu_{4}+2\Sigma\nu'=(8\nu_{4}+\Sigma\nu') (2\iota_{7})$ by Lemma A.1 of \cite{ZP23}, the following diagram is homotopy commutative
	\begin{align}
		\small{\xymatrix{
				P^7(2)\ar[d]^-{\tilde{\nu}_4}\ar[r]^-{p^1_7}&S^7\ar[d]^-{\nu_{4}}\ar[r]^{2\iota_{7}} & S^7 \ar[d]^-{8\nu_{4}+\Sigma\nu'}\ar[r]^{i_{7}}& 	P^8(2)\ar[d]^-{\Sigma\tilde{\nu}_4}\\
				P^4(4)\ar[r]^-{p^2_4}&S^4\ar[r]^{4\iota_{4}} & S^4 \ar[r]^{i_{4}}& P^5(4)
		} }\nonumber
	\end{align}
	implies $\Sigma\tilde{\nu}_4i_7=i_4 (8\nu_{4}+\Sigma\nu')$. 
	\begin{align}
		\Sigma(\tilde{\nu}_4i_6\eta^2_{6})=\Sigma\tilde{\nu}_4i_7\eta^2_{7}=i_4 (8\nu_{4}+\Sigma\nu')\eta^2_{7}=i_4 \Sigma\nu'\eta^2_{7}\neq 0~~\text{by  (\ref{equ2: pi9(P5) Coker})}.
		\nonumber
	\end{align}
	Hence  $\Sigma(\tilde{\nu}_4i_6\eta^2_{6})=\Sigma(2\tilde{\nu}_4\widetilde{\eta_7})\neq 0$ $\Rightarrow ~ 2\tilde{\nu}_4\widetilde{\eta_7}\neq 0~\Rightarrow ~ o(\tilde{\nu}_4\widetilde{\eta_7})=4$. 
\end{proof}
From above lemma and (\ref{exact:pi8(P4(2r))}),  we  get 
\begin{align*}
	\pi_{8}(P^4(2^r)) &	\cong \Z_{2}^{\oplus 4}\oplus \Z_{2^r} ~(r\geq 3);\\
	\pi_{8}(P^4(4)) & \cong \Z_{2}^{\oplus 4}\oplus \Z_{4}~\text{or}~\Z_{2}^{\oplus 3}\oplus \Z_{8}  ~\text{or}~\Z_{2}^{\oplus 2}\oplus \Z_{4}^{\oplus 2}.
\end{align*}
From Lemma \ref{lem L3(X)split}, $\Z_4\cong \pi_{8}(P^9(4))$ is a summand of $ \pi_{8}(P^4(4))$, implies that $\pi_{8}(P^9(4))\ncong \Z_{2}^{\oplus 3}\oplus \Z_{8}$. Assume that there is  an order $2$ lift $\mathbf{y}\in \pi_{8}(P^4(4))$ of $\nu_4\eta_7$, then $\tilde{\nu}_4\widetilde{\eta_7}-\mathbf{y}\in Ker(p^2_{4\ast})$ $\Rightarrow~ \tilde{\nu}_4\widetilde{\eta_7}-\mathbf{y}=\tau_{4}^{2}[\beta_3,\beta_6]+\alpha$, with $o(\alpha)\leq 2$ (\ref{equ:Cokerpartial_8^{4,r}}). Then $2\tilde{\nu}_4\widetilde{\eta_7}=2(\tilde{\nu}_4\widetilde{\eta_7}-\mathbf{y})=2\tau_{4}^{2}[\beta_3,\beta_6]~\Rightarrow~ \Sigma(2\tilde{\nu}_4\widetilde{\eta_7})=0$. This is a contradiction since we  previously obtained $\Sigma(2\tilde{\nu}_4\widetilde{\eta_7})\neq 0$. Thus the short exact sequence (\ref{exact:pi8(P4(2r))}) does not split for $r=2$. This implies $\pi_{8}(P^4(4))\ncong  \Z_{2}^{\oplus 4}\oplus \Z_{4}$. So we get $	\pi_{8}(P^4(4)) \cong \Z_{2}^{\oplus 2}\oplus \Z_{4}^{\oplus 2}$. Thus 
\begin{align*}
	\pi_{8}(P^4(2^r))\cong \left\{
	\begin{array}{ll}
		\Z_2^{\oplus 3}, & \hbox{$r=1$;} \\
		\Z_{2}^{\oplus 2}\oplus \Z_{4}^{\oplus 2}, & \hbox{$r=2$;} \\
		\Z_{2}^{\oplus 4}\oplus \Z_{2^r},& \hbox{$r\geq 3$.}\\
	\end{array}
	\right.
\end{align*}

$\bullet~~\pi_{9}(P^{4}(2^r))$. 
\begin{align} 
	&\xymatrix{
		\Z_{8}\{\nu_4^2\}=\pi_{10}(S^{4})\ar[r]^-{ \partial^{4,r}_{9\ast}}&\pi_{9}(F_{4})\ar[r]^-{\tau_{4\ast}}&\pi_{9}(P^{4}(2^r))\ar[r]^-{p_{4\ast}}&\pi_{9}(S^{4})\ar[r]^-{ \partial^{4,r}_{8\ast}}& \pi_{8}(F_{5}) } \nonumber\\
	& Ker(\partial^{4,r}_{8\ast})=\vartheta_r\Z_2\{\nu_4\eta_7^2\}\oplus\Z_2\{\Sigma \nu'\eta_7^2\}, ~\text{by (\ref{equ:partial8^{4,r}(nu4eta7^2)}). }\nonumber\\ 
	& \partial^{4,r}_{9\ast}(\nu_4^2)=\partial^{4,r}_{6\ast}(\nu_4)\nu_6=(\pm 2^{r-1}\beta_{3}\nu'+2^r\beta_6)\nu_6=2^r\beta_6\nu_6. \nonumber\\
	&\Rightarrow~	 Coker(\partial^{4,r}_{9\ast})=\Z_{2^{m_{r}^3}}\{\beta_{6}\nu_6\}\oplus\Z_{2}\{[\beta_3,\beta_6]\eta_8\}.\nonumber\\
	&\Rightarrow~~0\rightarrow \Z_{2}\oplus\Z_{2^{m_{r}^3}}\rightarrow \pi_{9}(P^{4}(2^r)) \xrightarrow{p_{4\ast}} \vartheta_r\Z_2\{\nu_4\eta_7^2\}\oplus\Z_2\{\Sigma \nu'\eta_7^2\}\rightarrow 0  \nonumber
\end{align}
Take $\mathbf{x}\eta_8\in \pi_{9}(P^{4}(2^r))$ and $\tilde{\nu}_4\widetilde{\eta_{7}}\eta_8 \in\pi_{9}(P^{4}(2^r)) (r\geq 2)$, where $\mathbf{x}$ and $\tilde{\nu}_4\widetilde{\eta_{7}}\in \pi_{8}(P^{4}(2^r))$ are lifts of $\Sigma\nu'\eta_7$ and  $\nu_4\eta_7$ respectively. Thus 
$o(\mathbf{x}\eta_8)=o(\tilde{\nu}_4\widetilde{\eta_{7}}\eta_8)=2$ and $p_{4\ast}(\mathbf{x}\eta_8)=\Sigma \nu'\eta_7^2$,   $p_{4\ast}(\tilde{\nu}_4\widetilde{\eta_{7}}\eta_8)=\nu_4\eta_7^2(r\geq 2)$. Hence 
$\pi_{9}(P^{4}(2^r)) \stackrel{p_{4\ast}}\twoheadrightarrow \vartheta_r\Z_2\{\nu_4\eta_7^2\}\oplus\Z_2\{\Sigma \nu'\eta_7^2\}$ is split onto, which implies 
\begin{align*}
	\pi_{9}(P^{4}(2^r)) \cong \vartheta_r\Z_2\oplus \Z_{2}^{\oplus 2}\oplus \Z_{2^{m_{r}^3}}, r\geq 1.
\end{align*}

$\bullet~~\pi_{10}(P^{4}(2^r))$. 

By $\pi_{11}(S^4)=0$ and $ \partial^{4,r}_{9\ast}(\nu_4^2)=2^r\beta_6\nu_6$   implies the short exact sequence 
\begin{align} 
	&\xymatrix{
		0\ar[r]&\pi_{10}(F_{4})\ar[r]^-{\tau_{4\ast}}&\pi_{10}(P^{4}(2^r))\ar[r]^-{p_{4\ast}}&\Z_{2^{m_r^3}}\{\delta_r\nu_4^2\}\ar[r]& 0 } \nonumber\\
	&\pi_{10}(F^1_{4})\cong \Z_2\oplus \Z_4;~	\pi_{10}(F^r_{4})\cong \Z_2^{\oplus 2}\oplus \Z_{2^r} (r\geq 2)\nonumber
\end{align}
where  $\delta_r=4,2$ for $r=1,2$ and $\delta_r=1$ for  $r\geq 3$.

\begin{lemma}\label{lem:not divis 2 P4}
	For any integer $s$, such that $0\leq s<r$, 
	$2^s\tau_{4}\check I_3^r[\check{\beta}_3,[\check{\beta}_3,\check{\beta}_6]]$ is not divisible by $2^{s+1}$ in $\pi_{10}(P^{4}(2^r))$.
\end{lemma}
\begin{proof}
	Let $R=\Z_{2^r}$.  
	Note that $H_{\ast}(\Omega P^{4}(2^r), R)\cong T(u,v)$, which is the tensor algebra over $R$ generated by $u,v$ with degrees $2$ and $3$ respectively.
	
	$Sk_{8}F_{4}\simeq  S^{3}\vee S^{6}$ implies that $Sk_{6}(\Omega F_4)\simeq S^{2}\vee S^{4}\vee S^{5}\vee S^{6}$. 
	\newline
	Consider the $R$-homology  the Serre Spectral Sequence for the fibration sequence
	$$\Omega F_{4}\xrightarrow{\Omega \tau_{4}} \Omega  P^{4}(2^r)\xrightarrow{\Omega p_{4} } \Omega S^{4}$$
	$R\cong H_{5}(\Omega F_{4}, R)\cong E_{0,5}^{\infty}=F_{5}^0\subset  F_{5}^{5}=H_{5}( \Omega  P^{4}(2^r), R)\cong R^{\oplus 2}$.
	\newline
	where the above inclusion is just the induced homomorphism $$(\Omega \tau_{4})_{\ast}: H_{5}(\Omega F_{4}, R)\rightarrow H_{5}(\Omega P^{4}(2^r), R )$$
	by the first commutative diagram (the right commutative square) in page 22 of \cite{HatcherSpectr}. 
	
	By Proposition 21.1 of \cite{Cohen1987},  $[u,v]:=uv-vu\in  H_{5}(\Omega P^{4}(2^r), R )$ is  spherical, i.e., it is the mod $2^r$-reduction of  Hurewicz image $h_{w}$ of some element in $\pi_{5}(\Omega P^{4}(2^r))$,
	where $\Omega_0: \pi_{6}( P^{4}(2^r))	\xrightarrow{\cong }\pi_{5}(\Omega P^{4}(2^r))$. By the calculation of $\pi_{6}(P^{4}(2^r))$ in \cite[Section 3.1]{ZP23}, we get 
	\newline
	$\pi_{6}(P^4(2^r))=\left\{
	\begin{array}{ll}
		\Z_4\{\tau_4\beta_6\}\oplus \Z_2\{\widetilde{\eta_4}\eta_5\}, & \hbox{$r=1$;} \\
		\Z_8\{\tau_4\beta_6\}\oplus \Z_2\{2\tau_4\beta_6+i_3\nu'\}\oplus \Z_2\{\widetilde{\eta_4}\eta_5\}, & \hbox{$r=2$;}\\
		\Z_{2^r}\{\tau_4\beta_6\}\oplus \Z_4\{i_3\nu'\}\oplus \Z_2\{\widetilde{\eta_4}\eta_5\}, & \hbox{$ r\geq 3$.}
	\end{array}
	\right.$
	\newline
	where $\widetilde{\eta_4}\in \pi_{5}(P^4(2^r))$ is a lift of $\eta_4$ by $p_{4\ast}$.
	
	It is easy to see  $h_{w}(\Omega_0(i_3\nu'))=0$ and  $h_{w}(\Omega_0(\widetilde{\eta_4}\eta_5))=0$.
	
	Hence $\bar h_{w}(\Omega_0(\tau_4\beta_6))=l[u,v]\in  H_{5}(\Omega P^{4}(2^r), R )$ where $l$ is odd and  $\bar h_w$ is the   mod $2^r$-reduction of $h_w$.
	
	Now we compute $\bar h_w(\Omega_0([i_3,[i_3, \tau_4\beta_6]]))$ by 
	$(\Omega_0i_{3})_{\ast}(\iota_2)=u$, $(\Omega_0(\tau_4\beta_6))_{\ast}(\iota_5)=l[u,v]$
	\begin{align*}
		&(\Omega_0[i_3, \tau_4\beta_6])_{\ast}(\iota_7)=(-1)^{a_1}([\Omega_0i_3, \Omega_0(\tau_4\beta_6)]^S)_{\ast}(\iota_2\wedge \iota_5)\\
		&=(-1)^{a_1}((\Omega_0i_3)_{\ast}(\iota_2)\otimes (\Omega_0(\tau_4\beta_6))_{\ast}(\iota_5)-(\Omega_0(\tau_4\beta_6))_{\ast}(\iota_5)\otimes(\Omega_0i_3)_{\ast}(\iota_2))\\
		&=(-1)^{a_1}l(u\otimes [u,v]-[u,v]\otimes u )=(-1)^{a_1}l[u,[u,v]].\\
		&\bar h_w(\Omega_0([i_3,[i_3, \tau_4\beta_6]))=(-1)^{a_2}([\Omega_0i_3,\Omega_0[i_3, \tau_4\beta_6]]^S)_{\ast}(\iota_2\wedge \iota_7)\\
		&=(-1)^{a_2}(u\otimes (\Omega_0[i_3, \tau_4\beta_6])_{\ast}(\iota_7)- (\Omega_0[i_3, \tau_4\beta_6])_{\ast}(\iota_7)\otimes u)\\
		&=(-1)^{a_3}l [u,[u,[u,v]]], ~~ a_1,a_2\in \Z, a_3=a_1+a_2.
	\end{align*}
	This implies that $2^s\Omega_0([i_3,[i_3, \tau_4\beta_6]])$  is not divisible by $2^{s+1}$  in $\pi_{9}(\Omega P^{4}(2^r))$ for  $0\leq s<r$, so is $\tau_{4}\check I_3^r[\check{\beta}_3,[\check{\beta}_3,\check{\beta}_6]]=[i_3,[i_3, \tau_4\beta_6]]$  in $\pi_{10}( P^{4}(2^r))$.
\end{proof}
By Lemma \ref{lem L3(X)split}, $\pi_{10}(P^{9}(2^r))\cong \left\{
\begin{array}{ll}
	\Z_4, & \hbox{$r=1$;} \\
	\Z_2^{\oplus 2}, & \hbox{$ r\geq 2$}
\end{array}
\right.$ is a direct summand of $\pi_{10}(P^{4}(2^r))$ and from  Lemma \ref{lem:not divis 2 P4},  $\Z_{2^r}\{[\check{\beta}_3,[\check{\beta}_3,\check{\beta}_6]]\}$ in Lemma \ref{lem pi(J3(2rl3))} splits out of $\pi_{10}(P^{4}(2^r))$.  Hence 
\begin{align*}
	\pi_{10}(P^{4}(2^r))\cong \left\{
	\begin{array}{ll}
		\Z_2^{\oplus 2}\oplus \Z_4, & \hbox{$r=1$;} \\
		\Z_2^{\oplus 2}\oplus \Z_{2^r}\oplus \Z_{2^{m_r^3}}, & \hbox{$ r\geq 2$.}
	\end{array}
	\right.
\end{align*}

\subsection{Homotopy groups of $P^{5}(2^r)$}
\label{subsec: P5(2^r)}
There are cofibration and fibration sequences 
\begin{align*}
	&	S^4\xrightarrow{2^{r}\iota_4} S^{4}\xrightarrow{i_{4}} P^{5}(2^r)\xrightarrow{p_{5}}  S^{5};
	~~\Omega S^5 \xrightarrow{\partial^{5}}F_{5}\xrightarrow{\tau_5} P^{5}(2^r)\xrightarrow{p_{5}}  S^{5}\\
	& Sk_{11}(F_{5})\simeq J_{2,4}^{r}\simeq S^4\cup_{\gamma^r_2}CS^7;~~Sk_{15}(F_{5})\simeq J_{3,4}^{r}\simeq J_{2,4}^{r}\cup_{\gamma^r_3}CS^{11}. 
\end{align*}
where $\gamma_2^r=2^r[\iota_4,\iota_4]=2^{r+1}\nu_4-2^r\Sigma \nu'$.

Since $\pi_{11}(S^4)=0$,  from diagrams of (\ref{diam 2: J2(2slk) to J2(2tlk)}), we get
\begin{align}
\bar g^r_{0}\bar\tau_{8}^r=\bar\tau_{8}^0g^r_{0};~	&	g^r_0|_{J_2}\simeq \iota_4\vee 2^r\iota_{11}~\text{which implies}~~ g_{0}^{r}\bar{\beta}_{4}^r=\bar{\beta}_{4}^0 ~~\text{and}~~  g_{0}^{r}\bar{\beta}_{11}^r=2^r\bar{\beta}_{11}^0.  \label{equ: g_r^0beta_i^r}
\end{align}

\begin{lemma}\label{lem pi(J2(2rl4))} For $r\geq 0$, \newline
	$\pi_{9}(F^r_5)\cong\pi_{9}(J_{2,4}^{r})=\left\{
	\begin{array}{ll}
		\Z_{2}\{\beta_4\nu_4\eta_7^2\}\oplus \vartheta^0_r\Z_{4}\{\widehat{\eta_8}\}, & \hbox{$r=0,1$;}\\
		\Z_{2}\{\beta_4\nu_4\eta_7^2\}\oplus\Z_{2}\{\beta_4\Sigma\nu'\eta_7^2\}\oplus \Z_{2}\{\widehat{\eta_8}\}, & \hbox{$r\geq 2$,}
	\end{array}
	\right.$
	
	with $\beta_4^1\Sigma\nu'\eta_7^2=2\widehat{\eta_8}$ for $r=1$.
	\newline
	$\pi_{10}(F^r_5)\cong\pi_{10}(J_{2,4}^{r})=	\Z_{2^{m^3_{r+1}}}\{\beta_4\nu_4^2\}\oplus \vartheta^0_r\Z_2\{\widehat{\eta_8^2}\}$.
	\newline
	$\pi_{i}(J_{2,4}^{r})=\left\{
	\begin{array}{ll}
		\Z_{(2)}\{\beta_{11}\}\oplus \Z_{2^{m_{r+1}^3}}\{\widehat{\delta_{r+1}\nu_8}\}, & \hbox{$i=11$;}\\
		\Z_{2}\{\beta_4\varepsilon_4\}\oplus\Z_2\{\beta_{11}\eta_{11}\}, & \hbox{$i=12$;}\\
		\Z_{2}\{\beta_4\nu^3_4\}\oplus\Z_2\{\beta_{4}\mu_{4}\}\oplus\Z_2\{\beta_{4}\eta_{4}\varepsilon_5\}\oplus \Z_2\{\beta_{11}\eta^2_{11}\}, & \hbox{$i=13$.}
	\end{array}
	\right.$
	\newline	
	where  $\vartheta^0_r=0$ for $r=0$ and $\vartheta^0_r=1$ for $r\geq 1$;  $\widehat{\delta_{r+1}\nu_8}$ is a lift of $\delta_{r+1}\nu_8\in \pi_{11}(S^8)$ by $\bar p_{8}$; the $\varepsilon_4$, $\mu_{4}$, $\eta_{4}\varepsilon_5$ comes from the generators of $\pi_{12}(S^4)=\Z_2\{\varepsilon_4\}$ and 
	$\pi_{13}(S^4)=\Z_2\{\nu^3_4\}\oplus \Z_2\{\mu_{4}\}\oplus \Z_2\{\beta_{11}\eta^2_{11}\} $.
\end{lemma}

\begin{proof}\qquad
	
	$\bullet~~\pi_{9}(J_{2,4}^{r})$. Lemma \ref{lem: exact seq pi_m(J2(2rl4))} for $k=4$, $m=9$ implies the following short exact sequence 
	\begin{align*}
		0\!\rightarrow\! \Z_{2}\{\bar\beta_4\nu_4\eta_7^2\}\!\oplus\!\vartheta^0_r\Z_{2}\{\bar\beta_4\Sigma\nu'\eta_7^2\}\!\xrightarrow{\bar{\tau}_{8\ast}}\!\pi_{9}(J_{2,4}^{r})\!\xrightarrow{\bar p_{8\ast}}\!\!\vartheta^0_r\Z_{2}\{\eta_8\} \!\rightarrow\! 0. 
	\end{align*}	
	For $r\geq 1$, by Lemma \ref{lem split},  $\exists~-\mathbf{x}\in \{\beta_4,\gamma^r_2,\eta_7\}\subset \pi_{9}(J_{2,4}^{r})$, such that $\bar p_{8\ast}(\mathbf{x})=\eta_8$.
	\begin{align*}
		&2\mathbf{x}\in -\{\beta_4,\gamma^r_2,\eta_7\}(2\iota_9)=\beta_4\{\gamma^r_2,\eta_7,2\iota_8\}.\\
		&\{\gamma^r_2,\eta_7,2\iota_8\}\ni (2^{r-1}[\iota_4,\iota_4])\eta_8~\text{mod}~\pi_{9}(S^4)(2\iota_9)+(\gamma^r_2)\pi_{9}(S^7)=\{0\}.
		(\text{Lemma \ref{lem:Toda bracket}})
	\end{align*}
	Thus $2\mathbf{x}=0$ for $r\geq 2$, implies that
	\begin{align*}
		\pi_{9}(J_{2,4}^{r})=\Z_{2}\{\beta_4\nu_4\eta_7^2\}\oplus\Z_{2}\{\beta_4\Sigma\nu'\eta_7^2\}\oplus \Z_{2}\{\widehat{\eta_8}\}~(r\geq 2).
	\end{align*}
	$2\mathbf{x}=\beta^1_4\Sigma\nu'\eta_7^2$ for $r=1$, implies that $o(\mathbf{x})=4$.
	Hence 
	\begin{align*}
		\pi_{9}(J_{2}(2\iota_4))=\Z_{2}\{\beta^1_4\nu_4\eta_7^2\}\oplus \Z_{4}\{\widehat{\eta_8}\}~\text{with}~\beta_4^1\Sigma\nu'\eta_7^2=2\widehat{\eta_8}.
	\end{align*}

	$\bullet~~\pi_{10}(J_{2,4}^{r})$.   Lemma \ref{lem: exact seq pi_m(J2(2rl4))} for $k=4$, $m=10$ implies
	\begin{align*}
		0\rightarrow \Z_{2^{m^3_{r+1}}}\{\bar\beta_4\nu_4^2\}\xrightarrow{\bar{\tau}_{8\ast}}\pi_{10}(J_{2,4}^{r})\xrightarrow{\bar p_{8\ast}}\vartheta^0_r\Z_{2}\{\eta_8^2\} \rightarrow 0. 
	\end{align*}	
	For $r\geq 2$, by Lemma \ref{lem split}, there is $-\mathbf{x}\in \{\beta_4,\gamma^r_2,\eta_7\}\subset \pi_{9}(J_{2,4}^{r})$, such that $\bar p_{8\ast}(\mathbf{x})=\eta_8$. Then $\bar p_{8\ast}(\mathbf{x}\eta_9)=\eta_8^2$. Moreover,  $2(\mathbf{x}\eta_9)=\mathbf{x}(2\eta_9)=0$ implies $o(\mathbf{x}\eta_9)=2$. Hence $\pi_{10}(J_{2,4}^{r})\stackrel{\bar p_{8\ast}}\twoheadrightarrow\vartheta^0_r\Z_{2}\{\eta_8^2\}$ is split onto.

	$\bullet~~\pi_{11}(J_{2,4}^{r})$. 	Since $\pi_{11}(S^4)=0$,  Remark \ref{remark: pi3k-1(J2(2rl4)) } implies
	\begin{align*}
	0\rightarrow \Z_{(2)}\{\bar\beta_{11}\}\xrightarrow{\bar\tau_{8\ast}}\pi_{11}(J_{2,4}^{r})\xrightarrow{\bar p_{8\ast}}\Z_{2^{m_{r+1}^3}}\{\delta_{r+1}\nu_8\} \rightarrow 0. 
	\end{align*}
	The rest only requires proof of the split of above short exact sequence, which we only prove for $r=1$ since the proof of other cases are just similar.
	
	By Lemma \ref{lem split}, there is  $-\mathbf{x}\in \{\beta^1_4,\gamma_2^1, 2\nu_7\}\subset \pi_{11}(J_{2}(2\iota_4))$, such that $\bar p_{8\ast}(\mathbf{x})=2\nu_8$.  $o(2\nu_8)=4$ implies that   $o(\mathbf{x})\geq 4$. On the other hand, $4\mathbf{x}\in -\{\beta^1_4,\gamma_2^1, 2\nu_7\}(4\iota_{11})=\beta^1_4\{\gamma_2^1, 2\nu_7,4\iota_{10}\}\subset \beta^1_{4\ast}(\pi_{11}(S^4))=\{0\}$, which implies that $o(\mathbf{x})\leq 4$. Hence $o(\mathbf{x})= 4$. So $\pi_{11}(J_{2,4}^{r})\stackrel{\bar p_{8\ast}}\twoheadrightarrow \Z_{2^{m_{r+1}^3}}\{\delta_{r+1}\nu_8\}$ is split onto.

	$\bullet~~\pi_{i}(J_{2,4}^{r})~(i=12,13)$. These are easily obtained by the following isomorphisms for  $i=12,13$ from Lemma \ref{lem: exact seq pi_m(J2(2rl4))}
	\begin{align*}
		\frac{\bar\beta_{4\ast}(\pi_{i}(S^{4}))}{\bar\beta_{4\ast}\gamma^r_{2\ast}(\pi_{i}(S^{7}))}\oplus\bar\beta_{11\ast}(\pi_{i}(S^{11}))\xrightarrow{\bar{\tau}_{8\ast}\cong}\pi_{i}(J_{2,4}^r).
	\end{align*}
	
\end{proof}

\begin{lemma}\label{lem gamma3 P5}
	For $r\geq 0$,	$\gamma^r_3= 2^ra_0\beta_{11}\in\pi_{11}(J_{2,4}^{r})$, where   $a_0$ is an odd integer in $\Z_{(2)}$.
\end{lemma}
\begin{proof}
	From Lemma \ref{lem pi(J2(2rl4))}, we assume 
	\begin{align*}
		\gamma_3^r=a_r\beta_{11}^r+b_r\widehat{\delta_{r+1}\nu_8}=a_r\bar\tau_8^r\bar\beta_{11}^r+b_r\widehat{\delta_{r+1}\nu_8}, a_r\in \Z_{(2)}, b_r\in \Z_{2^{m_{r+1}^3}}.
	\end{align*}
	By Lemma \ref{lem:pinch gamma3} and $\bar p^r_8 \bar\tau_8^r\simeq 0$, we get $0=\bar p^r_8\gamma_3^r=\bar p^r_8(b_r\widehat{\delta_{r+1}\nu_8})=b_r\delta_{r+1}\nu_8$. Hence $b_r=0$, which implies $\gamma_3^r=a_r\beta_{11}^r$.
	
	By exact sequence (\ref{exact:hgps piPk+1(2r)}) for $k=4,r=0,m=11$, we get
	\begin{align*}
		&\Z_2\cong \pi_{12}(S^5)\cong \pi_{11}(F_{5}^0)\cong \pi_{11}(J_{3,4}^0)\cong \pi_{11}(J_{2,4}^0)/Im(\gamma_{3\ast}^0)\\
		= &\frac{\Z_{(2)}\{\beta^0_{11}\}\oplus \Z_2\{\widehat{4\nu_8}\}}{\lr{a_0\beta^0_{11}}} ~~~~~~\Rightarrow~~~~a_0~\text{is an odd integer in}~\Z_{(2)}.\\
		2^{2r}\gamma_3^0&=\bar g_{0}^r \gamma_3^r=a_r\bar g_{0}^r(\bar\tau_8^r\bar\beta_{11}^r)    ~~~~(\text{by (\ref{Equ:gamma30,3r})})\\
		=&a_r\bar\tau_8^0g_{0}^r\bar\beta_{11}^r=a_r\bar\tau_8^0g_{0}^rI_2^{\gamma,r}j_2^{11}=a_r\bar\tau_8^0I_2^{\gamma,0}(g_{0}^r)|_{J_2}j_2^{11}~~~~(\text{by (\ref{diam big}) and (\ref{equ1 g^s_t}) })\\
		=&a_r\bar\tau_8^0I_2^{\gamma,0}(j_1^{4}q_1^4+2^rj_2^{11}q_2^{11})j_2^{11}=2^ra_r\bar\tau_8^0\bar\beta_{11}^0~~~~~(\text{by  (\ref{equ2  g^s_t}) and } \pi_{11}(S^4)=0)\\
		\text{i.e.},~~&	2^{2r}a_0\bar\tau_8^0\bar\beta_{11}^0= 2^ra_r\bar\tau_8^0\bar\beta_{11}^0 ~~\Rightarrow~~ a_r=2^ra_0.
	\end{align*}
\end{proof}

\begin{lemma}\label{lem pi(J3(2rl4))} For $r\geq 0$,  $\pi_{i}(F_{5})\cong \pi_{i}(J_{3,4}^{r}) (i\leq 14)$ and 
	\begin{align*}
		(1)~&\pi_{11}(J_{3,4}^{r})=\Z_{2^r}\{\check\beta_{11}\}\oplus \Z_{2^{m_{r+1}^3}}\{I_{2}^{r}\widehat{\delta_{r+1}\nu_8}\}.\\
		(2)~&\pi_{12}(J_{3,4}^{r})=	\Z_{2}\{\check\beta_4\varepsilon_4\}\oplus\vartheta^0_r\Z_2\{\check\beta_{11}\eta_{11}\}.\\
		(3)~&\pi_{13}(J_{3,4}^{0})=	
		\Z_{2}\{\check\beta_4\nu^3_4\}\oplus\Z_2\{\check\beta_{4}\mu_{4}\}\oplus\Z_2\{\check\beta_{4}\eta_{4}\varepsilon_5\}.\\
		&\pi_{13}(J_{3,4}^{1})\!=\!	
		\Z_{2}\{\check\beta_4\nu^3_4\}\!\oplus\!\Z_2\{\check\beta_{4}\mu_{4}\}\!\oplus\!\Z_2\{\check\beta_{4}\eta_{4}\varepsilon_5\}\!\oplus\! \Z_4\{\widearc{\eta_{12}}\}~\text{with}~\check\beta_{11}\eta^2_{11}\!=\!2\widearc{\eta_{12}}.\\
		&\pi_{13}(J_{3,4}^{r})\!=	\!
		\Z_{2}\{\check\beta_4\nu^3_4\}\!\oplus\!\Z_2\{\check\beta_{4}\mu_{4}\}\!\oplus\!\Z_2\{\check\beta_{4}\eta_{4}\varepsilon_5\}\!\oplus\! \Z_2\{\check\beta_{11}\eta^2_{11}\}\!\oplus\! \Z_2\{\widearc{\eta_{12}}\} (r\!\geq \!2).
	\end{align*}
\end{lemma}
\begin{proof}
	From Lemmas \ref{lem pi(J2(2rl4))} and \ref{lem gamma3 P5},	(1) and (2) of this lemma  are easily get by the following isomorphism 
	\begin{align*}
		\frac{\pi_{m}(J_{2,4}^{r})}{ 2^ra_0\beta_{11\ast} (\pi_{m}(S^{11}))}\xrightarrow[\cong]{I_{2\ast}^r}	\pi_{m}(J_{3,4}^{r}).
	\end{align*}
	For (3), the case $r=0$ is easily obtained. For $r\geq 1$,  we have the following short exact sequence by Lemma \ref{lem pi(J2(2rl4))}
	\begin{align*}
		0\rightarrow \pi_{13}(J_{2,4}^{r})\xrightarrow{I_{2\ast}^r} \pi_{13}(J_{3,4}^{r}) \xrightarrow {\check{p}_{12\ast}} \Z_{2}\{\eta_{12}\}\rightarrow 0
	\end{align*}
	
	For $r\geq 1$, by Lemma \ref{lem split}, there is $-\mathbf{x}\in \{I^r_2,\gamma_3^r,\eta_{11}\}\subset \pi_{13}(J_{3,4}^{r})$, such that $\check p_{12\ast}(\mathbf{x})=\eta_{12}$ and $2\mathbf{x}\in -\{I^r_2,\gamma_3^r,\eta_{11}\}(2\iota_{13})=I^r_2\{\gamma_3^r,\eta_{11},2\iota_{12}\}=I^r_2\{2^ra_0\beta_{11},\eta_{11},2\iota_{12}\}$.
\begin{align*}
	\{2^ra_0\beta_{11},\eta_{11},2\iota_{12}\}\ni 2^{r-1}a_0\beta_{11}\eta^2_{11}~\text{mod}~\pi_{13}(J_{2,4}^{r})\fhe (2\iota_{13})+(2^ra_0\beta_{11})\fhe\pi_{13}(S^{11})=\{0\}
\end{align*}
	Thus $2\mathbf{x}=0$ for $r\geq 2$, implies that
	\begin{align*}
		\pi_{13}(J_{3,4}^{r})=\Z_{2}\{\check\beta_4\nu^3_4\}\oplus\Z_2\{\check\beta_{4}\mu_{4}\}\oplus\Z_2\{\check\beta_{4}\eta_{4}\varepsilon_5\}\oplus \Z_2\{\check\beta_{11}\eta^2_{11}\}\oplus \Z_2\{\widearc{\eta_{12}}\}.
	\end{align*}
	$2\mathbf{x}=I_2^r\beta_{11}\eta_{11}^2=\check\beta_{11}\eta_{11}^2$ for $r=1$, we get $o(\mathbf{x})=4$ and 
	this implies the result of $\pi_{13}(J_{3,4}^1)$ in (3) of Lemma \ref{lem pi(J3(2rl4))}.

\end{proof}

Next we compute $\pi_{14}(J_{2,4}^{r})$ and $\pi_{14}(J_{3,4}^{r})$. We have the following exact sequence with commutative  diagrams.
\begin{align}
	& \small{\xymatrix{
			\pi_{15}( S^{8})\ar[r]^-{ \bar{\partial}^{8}_{14\ast}}&\pi_{14}(F_{\bar{p}_{8}})\ar[r]^-{\bar{\tau}_{8\ast}}&\pi_{14}(J_{2,4}^{r})\ar[r]^-{\bar{p}_{8\ast}}&\pi_{14}(S^{8})\ar[r]^-{ \bar{\partial}^{8}_{13\ast}}& \pi_{13}(F_{\bar{p}_{8}})\\
			\pi_{14}(S^{7})\ar[r]^-{\gamma^r_{2\ast}}\ar[u]_{\Sigma } &\pi_{14}(S^{4})\ar[ur]_-{\beta_{4\ast}}\ar@{^{(}->}[u]_{ \bar{\beta}_{4\ast}}&&\pi_{13}(S^{7})\ar[r]^-{\gamma^r_{2\ast}}\ar[u]_{\cong\Sigma } &	\pi_{13}(S^{4})\ar@{^{(}->}[u]_{\bar{\beta}_{4\ast}} }}  \label{diam:P5, pi14(J2(2r))}\\
	&\text{where}~~	\pi_{15}( S^{8})=\Z_{(2)}\{\sigma_8\}\oplus \Z_8\{\Sigma \sigma'\},~~~~ 	\pi_{14}(S^{7})=\Z_8\{ \sigma'\}, \nonumber\\
	&\pi_{14}(F_{\bar{p}_{8}})=\Z_8\{\bar{\beta}_4\nu_4\sigma'\}\oplus \Z_4\{\bar{\beta}_4\Sigma\varepsilon'\}\oplus \Z_2\{\bar{\beta}_4\eta_4\mu_5\}\oplus \Z_8\{\bar{\beta}_{11}\nu_{11}\}\oplus \Z_{(2)}\{[\bar{\beta}_4,\bar{\beta}_{11}]\}.\nonumber
\end{align}
\begin{lemma}\label{lem: pi_{14}(J2(2r))}
	For $r\geq 1$, we have the following split short exact sequence 
	
	$0\rightarrow Coker(\bar{\partial}^{8,r}_{14\ast})\xrightarrow{\bar{\tau}_{8\ast}} \pi_{14}(J_{2,4}^{r})\xrightarrow{\bar{p}_{8\ast}}\pi_{14}(S^8)=\Z_2\{\nu_8^2\}\rightarrow 0$,  i.e., 
	\begin{align*}
		&	\pi_{14}(J_{2,4}^{r})=\bar{\tau}_{8\ast}(Coker(\bar{\partial}^{8,r}_{14\ast}))\oplus \Z_{2}\{\widehat{\nu_8^2}\}, ~~\text{where}~\\
		&\bar{\tau}_{8\ast}(Coker(\bar{\partial}^{8,r}_{14\ast}))\!=\!\frac{\Z_8\{\beta_4\!\nu_4\sigma'\}\!\!\oplus\! \Z_4\{\beta_4\Sigma\varepsilon'\}\!\!\oplus\! \Z_2\{\beta_4\eta_4\mu_5\}\!\!\oplus\! \Z_8\{\beta_{11}\!\nu_{11}\}\!\!\oplus\! \Z_{(2)}\{[\beta_4,\beta_{11}]\}}{\lr{2^{r+1}\beta_{4}\nu_4\sigma',\pm2^r\beta_4\nu_4\sigma'\!+\!2^r\beta_4\Sigma\varepsilon'\!+\!2^r(1\!-\!\vartheta_r)c_0\beta_{11}\nu_{11}\!+\!2^rd_0[\beta_{4},\beta_{11}]}}
	\end{align*}
	$c_0$ and $d_0$ are some integers. 
\end{lemma}
\begin{proof}
	\begin{align}
		&\gamma^r_{2\ast}(\nu_7^2)=(2^{r+1}\nu_4-2^r\Sigma\nu')(\nu_7^2)=0~\text{by}~\nu'\nu_6=0 ~\text{and}~2\nu_4^3=0, \nonumber\\
		\Rightarrow~~& Ker(\bar{\partial}^{8,r}_{13\ast})=\pi_{14}(S^8)=\Z_2\{\nu_8^2\} (r\geq 0). \label{equ: Kerpartial^8_13}\\
		&(t\alpha)\sigma'=t\alpha\sigma', t\in \Z, \alpha\in \pi_{7}(X),~~\text{from Lemma 4.5 of \cite{Toda}}. \label{equ: talphasigma'}\\
		&\gamma^r_{2\ast}(\sigma')=(2^{r+1}\nu_4-2^r\Sigma\nu')(\sigma') \nonumber\\
		=&(2^{r+1}\nu_4)\sigma'-(2^r\Sigma\nu')\sigma'\pm[2^{r+1}\nu_4, 2^r\Sigma\nu']H_2(\sigma') ~\text{( Proposition 2.10 of \cite{N.Oda})}~ \nonumber\\
		&[2^{r+1}\nu_4, 2^r\Sigma\nu']H_2(\sigma')= [\nu_4, \Sigma\nu'](2^{2r+1}H_2(\sigma'))= [\nu_4, \Sigma\nu'](2^{2r+1}\eta_{13})=0 \nonumber\\
		\Rightarrow~&\gamma^r_{2\ast}(\sigma')=2^{r+1}\nu_4\sigma'-2^r\Sigma\nu'\sigma'
		=\left\{
		\begin{array}{ll}
			2^{r+1}\nu_4\sigma', &\! \!\!\hbox{$r\geq 1$;} \\
			2\nu_4\sigma'-2\Sigma \varepsilon' , & \!\!\!\hbox{$r=0$}
		\end{array}
		\right.(\text{by Lemma \ref{lemA: Sigma nu'sigam'} (i)}) \nonumber\\
		\Rightarrow~& \bar{\partial}^{8,r}_{14\ast}(\Sigma\sigma')=\left\{
		\begin{array}{ll}
			2^{r+1}\bar{\beta}^r_{4}\nu_4\sigma', & \hbox{$r\geq 1$;} \\
			2\bar{\beta}^0_{4}\nu_4\sigma'-2\bar{\beta}^0_{4}\Sigma \varepsilon' , & \hbox{$r=0$}
		\end{array}
		\right.
		\label{equ: partial Sigmasigma'}
	\end{align}
	
	\begin{align*}
		&\text{Assume}~~\bar{\partial}_{14\ast}^{8,r}(\sigma_8)=x_r\bar{\beta}_4^r\nu_4\sigma'+y_r\bar{\beta}_4^r\Sigma\varepsilon'+z_r\bar{\beta}_4^r\eta_4\mu_5+c_r\bar{\beta}_{11}^{r}\nu_{11}+d_r[\bar{\beta}^r_{4},\bar{\beta}^r_{11}],\\
		&\text{where}~~x_r\in \Z_8, y_r\in \Z_4, z_r\in \Z_2, c_r\in \Z_8, d_r\in \Z_{(2)}, r\geq 0.
	\end{align*} 
	Then by (\ref{equ: g_r^0beta_i^r}),
	\begin{align*}
		g_{0\ast}^{r}\bar{\partial}_{14\ast}^{8,r}(\sigma_8)=x_r\bar{\beta}_4^0\nu_4\sigma'+y_r\bar{\beta}_4^0\Sigma\varepsilon'+z_r\bar{\beta}_4^0\eta_4\mu_5+2^rc_r\bar{\beta}_{11}^0\nu_{11}+2^rd_r[\bar{\beta}^0_{4},\bar{\beta}^0_{11}]. 
	\end{align*}
	On the other hand, from left commutative diagram  (\ref{diam 2: J2(2slk) to J2(2tlk)}) for $s=r, t=0, k=4$,
	\begin{align*}
		g_{0\ast}^{r}\bar{\partial}_{14\ast}^{8,r}(\sigma_8)=&\bar{\partial}_{14\ast}^{8,0}(2^r\iota_8)_\ast(\sigma_8)=2^{2r}\bar{\partial}_{14\ast}^{8,0}(\sigma_8)-2^{r-1}(2^r-1)\bar{\partial}_{14\ast}^{8,0}(\Sigma \sigma')\\
		=&	2^{2r}(x_0\bar{\beta}_4^0\nu_4\sigma'+y_0\bar{\beta}_4^0\Sigma\varepsilon'+z_0\bar{\beta}_4^0\eta_4\mu_5+c_0\bar{\beta}_{11}^{0}\nu_{11}+d_0[\bar{\beta}^0_{4},\bar{\beta}^0_{11}])\\
		-&2^{r-1}(2^r-1)(2\bar{\beta}^0_{4}\nu_4\sigma'-2\bar{\beta}^0_{4}\Sigma \varepsilon')\\
		=&\left\{
		\begin{array}{ll}
			(4x_0-2)\bar{\beta}^0_{4}\nu_4\sigma'+2\bar{\beta}^0_{4}\Sigma \varepsilon'+4c_0\bar{\beta}_{11}^0\nu_{11}+4d_0[\bar{\beta}^0_{4},\bar{\beta}^0_{11}], & \hbox{$r=1$;} \\
			4\bar{\beta}^0_{4}\nu_4\sigma'+2^4d_0[\bar{\beta}^r_{4},\bar{\beta}^r_{11}], & \hbox{$r=2$;} \\
			2^rd_0[\bar{\beta}^r_{4},\bar{\beta}^r_{11}], & \hbox{$r\geq 3$.}
		\end{array}
		\right.
	\end{align*}
	By comparing above two equations of $g_{0\ast}^{r}\bar{\partial}_{14\ast}^{8,r}(\sigma_8)$, we get
	\begin{align}
		x_r=\pm 2^r,~~ y_r=2^r,~~ z_r=0,~~ c_r=2^r(1\!-\!\vartheta_r)c_0,~~ d_r=2^rd_0~~(r\geq 1).\nonumber\\
		\bar{\partial}_{14\ast}^{8,r}(\sigma_8)=\pm2^r\bar{\beta}_4^r\nu_4\sigma'+2^r\bar{\beta}_4^r\Sigma\varepsilon'+2^r(1\!-\!\vartheta_r)c_0\bar{\beta}_{11}^{r}\nu_{11}+2^rd_0[\bar{\beta}^r_{4},\bar{\beta}^r_{11}].\nonumber
	\end{align}
	
	Together with  (\ref{equ: partial Sigmasigma'}), we have
	\begin{align*} Coker(\bar{\partial}_{14\ast}^{8,r})\!=\!\frac{\Z_8\{\bar{\beta}^{r}_4\nu_4\sigma'\}\!\oplus\! \Z_4\{\bar{\beta}^{r}_4\Sigma\varepsilon'\}\!\oplus\! \Z_2\{\bar{\beta}^{r}_4\eta_4\mu_5\}\!\oplus\! \Z_8\{\bar{\beta}^{r}_{11}\nu_{11}\}\!\oplus\! \Z_{(2)}\{[\bar{\beta}^{r}_4,\bar{\beta}^{r}_{11}]\}}{\lr{2^{r+1}\bar{\beta}^r_{4}\nu_4\sigma',~~\pm\!2^r\bar{\beta}_4^r\nu_4\sigma'\!+\!2^r\bar{\beta}_4^r\Sigma\varepsilon'\!+\!2^r(1\!-\!\vartheta_r)c_0\bar{\beta}_{11}^{r}\nu_{11}\!+\!2^rd_0[\bar{\beta}^r_{4},\bar{\beta}^r_{11}]}}
	\end{align*}
	Now we get the short exact sequence of Lemma \ref{lem: pi_{14}(J2(2r))} by above equation and  (\ref{equ: Kerpartial^8_13}). It remains to show the split.
	
	By Lemma \ref{lem split}, there is  $-\mathbf{x}\in \{\beta^r_4,\gamma_2^r, \nu_7^2\}\subset \pi_{14}(J_{2,4}^{r})$, such that $\bar p_{8\ast}(\mathbf{x})=\nu_8^2$.  $2\mathbf{x}\in -\{\beta^r_4,\gamma_2^r, \nu^2_7\}(2\iota_{14})=\beta^r_4\{\gamma_2^r, \nu^2_7,2\iota_{13}\}$.
	\begin{align*}
		&\{\gamma_2^r, \nu^2_7,2\iota_{13}\}\ni \gamma_2^{r-1}\nu^2_7\eta_{13}=0~\text{mod}~ \pi_{14}(S^4)(2\iota_{14})+(2^{r}[\iota_4,\iota_4])\pi_{14}(S^7)\\
			&(2^{r}[\iota_4,\iota_4])\sigma'=(2^{r+1}\nu_4-2^{r}\Sigma\nu')\sigma'=2^{r+1}\nu_4\sigma'-2^{r+1}\Sigma\varepsilon',\\
		\text{we get}~~&\{\gamma_2^r, \nu^2_7,2\iota_{13}\}=2\pi_{14}(S^4), \Rightarrow 2\mathbf{x}=\beta_4^r(2\theta)=2\beta_4^r\theta,~\theta\in \pi_{14}(S^4).\\
		\text{Let}~~& \mathbf{x}'=\mathbf{x}-\beta_4^r\theta,~~\text{we have}~~\bar p_{8\ast}(\mathbf{x}')=\nu_8^2,~~o(\mathbf{x}')=2.
	\end{align*}
	So the split of the short exact sequence in Lemma \ref{lem: pi_{14}(J2(2r))} is obtained.
\end{proof}

There is the following exact sequence with commutative squares
\begin{align*} 
	&\small{\xymatrix{
			\Z_{8}\{\nu_{12}\}=\pi_{15}( S^{12})\ar[r]^-{ \check{\partial}^{12,r}_{14\ast}}&\pi_{14}(F_{\check p_{12}^{r}})\ar[r]^-{\check \tau_{3k\ast}}&\pi_{14}(J_{3,4}^{r})\ar[r]^-{\check p_{12\ast}^{r}}&\pi_{14}(S^{12})\ar[r]^-{ \check{\partial}^{12,r}_{13\ast}}& \pi_{13}(F_{\check p_{12}^{r}})\\
			\Z_{8}\{\nu_{11}\}=\pi_{14}(S^{11})\ar[r]^-{\gamma^r_{3\ast}}\ar[u]_{\Sigma~\cong } &\pi_{14}(J_{2,4}^{r})\ar[ur]_-{I_{4\ast}^{r}}\ar@{->>}[u]^-{empi.}&&\pi_{13}(S^{11})\ar[r]^-{\gamma^r_{3\ast}}\ar[u]_{\cong\Sigma} &	\pi_{13}(J_{2,4}^{r})\ar[u]^{\cong} }}
\end{align*}
where   $\pi_{14}(F_{\check p_{12}^{r}})\cong \frac{\pi_{14}(J_{2,4}^{r}) }{\lr{2^ra_0[\beta_{4}^r, \beta_{11}^r]}} $   by (\ref{equ: Sk4k-1F_3k}).
\begin{lemma}\label{lem: pi_{14}(J3(2r))} For $r\geq 1$,
	\begin{align*}
		& \pi_{14}(J_{3,4}^{r})	=\!\left\{
		\begin{array}{ll}
			\Z_2\{\check\beta_4\nu_4\sigma'\!+\check\beta_4\Sigma\varepsilon'\}\oplus  H_{14}^1, & \hbox{$r=1$;} \\
			\Z_{2^{m_r^3}}\{\check\beta_4\nu_4\sigma'\}\oplus  H_{14}^r, &\hbox{$r\geq 2$.}
		\end{array}
		\right.\\
		&H_{14}^r\!=\!\Z_4\{\check\beta_4\Sigma\varepsilon'\}\!\oplus\! \Z_2\{\check\beta_4\eta_4\mu_5\}\!\oplus\! \Z_{2^{m_r^3}}\{\check\beta_{11}\nu_{11}\}\!\oplus\! \Z_{2^r}\{[\check\beta_4,\check\beta_{11}]\}\! \oplus\! \Z_{2}\{I_{2}^{r}\widehat{\nu_8^2}\}\!\oplus\! \Z_2\{\widearc{\eta_{12}}\eta_{13}\}.
	\end{align*}
	
\end{lemma}
\begin{proof}
	From Lemma \ref{lem gamma3 P5}, we have the following short exact sequence 
	\begin{align}
		&	0\rightarrow Coker(\check{\partial}^{12,r}_{14\ast})\rightarrow \pi_{14}(J_{3,4}^{r})\xrightarrow{\check p_{12\ast}^{r}}\pi_{14}(S^{12})\rightarrow 0, \nonumber\\
		\text{where}~~&	Coker(\check{\partial}^{12,r}_{14\ast})\cong \frac{\pi_{14}(J_{2,4}^{r}) }{\lr{2^ra_0[\beta_{4}, \beta_{11}], 2^ra_0\beta_{11}\nu_{11}}}. \nonumber
	\end{align}
		Moreover, the order 2 element  $\widearc{\eta_{12}}\eta_{13}$ ($\widearc{\eta_{12}}$ is get in Lemma \ref{lem pi(J3(2rl4))}) satisfies 
	$\check{p}_{12\ast}^r(\widearc{\eta_{12}}\eta_{13})=\eta_{12}^2$, which implies  $\pi_{14}(J_{3,4}^{r})\stackrel{\check p_{12\ast}^{r}}\twoheadrightarrow\pi_{14}(S^{12})$ is split onto. Then 
	this  lemma is obtained by the following computation of $	Coker(\check{\partial}^{12,r}_{14\ast})$ by Lemma \ref{lem: pi_{14}(J2(2r))}, i.e.,  $Coker(\check{\partial}^{12,r}_{14\ast})$ is isomorphic to
	 \small{\begin{align*}
			&\!\frac{\Z_8\{\beta_4\nu_4\sigma'\}\!\oplus\! \Z_4\{\beta_4\Sigma\varepsilon'\}\!\oplus\! \Z_2\{\beta_4\eta_4\mu_5\}\!\oplus\! \Z_8\{\beta_{11}\nu_{11}\}\!\oplus\! \Z_{(2)}\{[\beta_4,\beta_{11}]\} \!\oplus\! \Z_{2}\{\widehat{\nu_8^2}\}}{\lr{2^{r\!+\!1}\beta_{4}\nu_4\sigma',\pm2^r\!\beta_4\nu_4\sigma'\!+\!2^r\!\beta_4\Sigma\varepsilon'\!+\!2^r\!(1\!-\!\vartheta_r)c_0\beta_{11}\nu_{11}\!+\!2^rd_0[\beta_{4},\beta_{11}],2^r\!a_0[\beta_{4}, \beta_{11}], 2^r\!a_0\beta_{11}\nu_{11} }}\\
			&=	\!\frac{\Z_8\{\beta_4\nu_4\sigma'\}\!\oplus\! \Z_4\{\beta_4\Sigma\varepsilon'\}\!\oplus\!  \Z_8\{\beta_{11}\nu_{11}\}\!\oplus\! \Z_{(2)}\{[\beta_4,\beta_{11}]\}}{\lr{2^{r\!+\!1}\beta_{4}\nu_4\sigma',\pm2^r\!\beta_4\nu_4\sigma'\!+\!2^r\!\beta_4\Sigma\varepsilon',2^r\!a_0[\beta_{4}, \beta_{11}], 2^r\!a_0\beta_{11}\nu_{11} }}\!\oplus\! \Z_2\{\beta_4\eta_4\mu_5\}\!\oplus\! \Z_{2}\{\widehat{\nu_8^2}\}\\
			&=	\!\frac{\Z_8\{\beta_4\nu_4\sigma'\}\!\oplus\! \Z_4\{\beta_4\Sigma\varepsilon'\}}{\lr{2^{r\!+\!1}\beta_{4}\nu_4\sigma',\pm2^r\!\beta_4\nu_4\sigma'\!+\!2^r\!\beta_4\Sigma\varepsilon' }}\!\oplus\! \Z_2\{\beta_4\eta_4\mu_5\}\!\oplus\! \Z_{2^{m_r^3}}\{\beta_{11}\nu_{11}\}\!\oplus\! \Z_{2^r}\{[\beta_4,\beta_{11}]\}\! \oplus\! \Z_{2}\{\widehat{\nu_8^2}\}\\
			&=\!\left\{
			\begin{array}{ll}
				\!\!\!\Z_2\{\beta_4\nu_4\sigma'\!\!+\!\beta_4\Sigma\varepsilon'\}\!\oplus\! \Z_4\{\beta_4\Sigma\varepsilon'\}\!\oplus\! \Z_2\{\beta_4\eta_4\mu_5\}\!\oplus\! \Z_{2}\{\beta_{11}\nu_{11}\}\!\oplus\! \Z_{2}\{[\beta_4,\beta_{11}]\}\! \oplus\! \Z_{2}\{\widehat{\nu_8^2}\}, &\!\!\!\! \hbox{$r=1$;} \\
				\!\!\!\Z_{2^{m_r^3}}\{\beta_4\nu_4\sigma'\}\!\oplus\! \Z_4\{\beta_4\Sigma\varepsilon'\}\!\oplus\! \Z_2\{\beta_4\eta_4\mu_5\}\!\oplus\! \Z_{2^{m_r^3}}\{\beta_{11}\nu_{11}\}\!\oplus\! \Z_{2^r}\{[\beta_4,\beta_{11}]\}\! \oplus\! \Z_{2}\{\widehat{\nu_8^2}\}
				, &\!\!\!\! \hbox{$r\geq 2$.}
			\end{array}
			\right.
	\end{align*}}

\end{proof}

$\bullet~~\pi_{9}(P^{5}(2^r))$. 

Consider diagram (\ref{exact:hgps piPk+1(2r)}) for $k=4, m=9$.
\newline
$\partial_{8\ast}^{5,r}(\nu_5\eta_8)\!=\beta_{4\ast}(2^r\iota_4)_{\ast}(\nu_4\eta_7)=(2^{2r}\nu_4-2^{r\!-\!1}(2^r\!-\!1)\Sigma\nu')\eta_7=\!\left\{
\begin{array}{ll}
	\!\!\!	\beta_4\Sigma\nu'\eta_7, & \!\!\!\hbox{$r=1$;} \\
	\!\!\!	0, &\!\!\! \hbox{$r\geq 2$.}
\end{array}
\right.$
\newline
$\pi_{8}(F_5)\cong	\pi_{8}(J_{2,4}^{r})=\Z_2\{\beta_4\nu_4\eta_7\}\oplus \Z_2\{\beta_4\Sigma\nu'\eta_7\}$, by (9) of  \cite{JZhtpygps}
\begin{align}
	\Rightarrow~~& Ker(	\partial_{8\ast}^{5,r})=\vartheta_{r}\Z_2\{\nu_5\eta_8\}. \nonumber\\
	\partial_{9\ast}^{5,r}(\nu_5\eta_8^2)&=\beta_{4\ast}(2^r\iota_4)_{\ast}(\nu_4\eta_7^2)=\left\{
	\begin{array}{ll}
		\beta_4\Sigma\nu'\eta^2_7, & \hbox{$r=1$;} \\
		0, & \hbox{$r\geq 2$.}
	\end{array}
	\right. \label{equ: pi9(P5) Coker}\\
	\text{Lemma \ref{lem pi(J2(2rl4))}}\Rightarrow~& Coker(	\partial_{9\ast}^{5,r})\!=\!	\Z_{2}\{\beta_4\nu_4\eta_7^2\}\!\oplus\!\vartheta_{r}\Z_{2}\{\beta_4\Sigma\nu'\eta_7^2\}\!\oplus\! \Z_{2}\{\widehat{\eta_8}\}.  \label{equ2: pi9(P5) Coker}
\end{align}
Thus we have $ \pi_{9}(P^5(2))\cong \Z_2\oplus \Z_2 $ and the short exact sequence 
\begin{align}
	0\rightarrow \Z_{2}^{\oplus 3}\cong Coker(	\partial_{9\ast}^{5,r})\rightarrow \pi_{9}(P^{5}(2^r)) \xrightarrow{p_{5\ast}} \Z_2\{\nu_5\eta_8\}\rightarrow 0 ~(r\geq 2). \nonumber
\end{align}
By Lemma \ref{lem split}, there is  $-\mathbf{x}_r\in \{i_4,2^r\iota_4, \nu_4\eta_7\}$, such that $ p_{5\ast}(\mathbf{x}_r)=\nu_5\eta_8$. \newline $2\mathbf{x}_r\in - \{i_4,2^r\iota_4, \nu_4\eta_7\}(2\iota_{9})=i_4\{2^r\iota_4, \nu_4\eta_7,2\iota_{8}\}$.
Hence  $2\mathbf{x}_r=i_4\Sigma\nu'\eta_7^2$ and $0$ for $r=2$ and $r\geq 3$ respectively by (ii) of Lemma \ref{lemA: Todabraket }, which implies $o(\mathbf{x}_2)=4$ and $o(\mathbf{x}_r)=2$,$r\geq 3$. So 
\begin{align*}
	\pi_{9}(P^{5}(2^r))\cong \left\{
	\begin{array}{ll}
		\Z_2^{\oplus 2}\oplus \Z_4, & \hbox{$r=2$;} \\
		\Z_{2}^{\oplus 4}, & \hbox{$r\geq 3$.}
	\end{array}
	\right.
\end{align*}

$\bullet~~\pi_{10}(P^{5}(2^r))$. 

Consider diagram (\ref{exact:hgps piPk+1(2r)}) for $k=4, m=10$.
\newline
Computation of $\pi_{9}(P^{5}(2^r))$ implies  $Ker(	\partial_{9\ast}^{5,r})=\vartheta_r\Z_2\{\nu_5\eta_8^2\}$.
\begin{align}
	\partial_{10\ast}^{5,r}(\nu_5^2)=&\beta_{4\ast}(2^r\iota_4)_{\ast}(\nu^2_4)=\beta_{4}(2^{2r}\nu_4-2^{r-1}(2^r-1)\Sigma\nu')\nu_7\nonumber\\
	=&2^{2r}\beta_{4}\nu_4^2=0, ~\text{by Lemma \ref{lem pi(J2(2rl4))}}. \label{equ: pi10(P5) Coker}\\
	\Rightarrow~~& Coker(	\partial_{10\ast}^{5,r})=	\Z_{2^{m^3_{r+1}}}\{\beta_4\nu_4^2\}\oplus \Z_2\{\widehat{\eta_8^2}\}.  \nonumber
\end{align}
Thus we have $ \pi_{10}(P^5(2))\cong \Z_4\oplus \Z_2 $ and the short exact sequence 
\begin{align*}
	0\!\rightarrow\! \Z_{2^{m^3_{r+1}}}\!\oplus\! \Z_2\!\cong\! Coker(	\partial_{10\ast}^{5,r})\!\rightarrow\! \pi_{10}(P^{5}(2^r)) \xrightarrow{p_{5\ast}} \Z_2\{\nu_5\eta_8^2\}\!\rightarrow\! 0 (r\geq 2) 
\end{align*}
Since  there is an $\mathbf{x}\in \pi_{9}(P^5(2^r)) (r\geq 2)$ such that $p_{5\ast}(\mathbf{x})=\nu_5\eta_8$,  the order 2 element $\mathbf{x}\eta_9\in  \pi_{10}(P^5(2^r))$ satisfies $p_{5\ast}(\mathbf{x}\eta_9)=\nu_5\eta^2_8$, i.e., $\pi_{10}(P^{5}(2^r))\stackrel{p_{5\ast}}\twoheadrightarrow  \Z_2\{\nu_5\eta_8^2\}$ is split onto.  We get
\begin{align*}
	\pi_{10}(P^{5}(2^r))\cong  \Z_{2^{m^3_{r+1}}}\oplus \Z_2^{\oplus 2} (r\geq 2).
\end{align*}

$\bullet~~\pi_{11}(P^{5}(2^r))$. 

We have the  exact sequence with $\pi_{11}(F_{5})$ obtained by Lemmas \ref{lem gamma3 P5}, \ref{lem pi(J3(2rl4))}
\begin{align} 
	&\xymatrix{
		\pi_{12}(S^{5})\ar[r]^-{ \partial^{5,r}_{11\ast}}&\pi_{11}(F_{5})\ar[r]^-{\tau_{5\ast}}&\pi_{11}(P^{5}(2^r))\ar[r]^-{p_{5\ast}}&\pi_{11}(S^{5})\ar[r]^-{ \partial^{5,r}_{10\ast}=0}& \pi_{10}(F_{5}) } \nonumber\\
	&  \pi_{11}(F_{5})\cong \pi_{11}(J_{3,4}^{r})\cong \Z_{2^r}\{\check\beta^r_{11}\}\oplus \Z_{2^{m_{r+1}^3}}\{\widehat{\delta_{r+1}\nu_8}\}. \nonumber
\end{align}
The right commutative diagram of (\ref{diam Moore s to t}) for $k=4, s=1, t=r
\geq 2$ induces 
\begin{align*}
	\small{\xymatrix{
			\Z_2\{\sigma'''\}=\pi_{12}(S^{5})\ar@{=}[d] \ar[r]^-{\partial^{5,1}_{11\ast}}&\pi_{11}(F^1_{5})\ar[d]^{\chi_{r\ast}^{1}}\ar[r]^-{\tau_{5\ast}^{1}} &\pi_{11}(P^{5}(2)) \ar[d]^{\bar{\chi}_{r\ast}^{1}}\ar[r]^-{p^1_{5\ast}}&\pi_{11}(S^{5})\ar[r]\ar@{=}[d]&0\\
			\pi_{12}(S^{5}) \ar[r]^-{\partial^{5,r}_{11\ast} }&\pi_{11}(F^r_{5}) \ar[r]^-{\tau^{r}_{5\ast}}& \pi_{11}(P^{5}(2^{r}))\ar[r]^-{p^r_{5\ast}}&\pi_{11}(S^{5})\ar[r]&0
	} }
\end{align*}
By Theorem 5.10 of \cite{WJ Proj plane}, $\pi_{11}(P^{5}(2))\cong \Z_2\oplus\Z_2\oplus \Z_4$. It implies $\partial^{5,1}_{11\ast}(\sigma''')=0$, hence $\partial^{5,r}_{11\ast}(\sigma''')=\chi_{r\ast}^{1}\partial^{5,1}_{11\ast}(\sigma''')=0$. So we get the following short exact sequence which is split for $r=1$.
\begin{align}
	\small{\xymatrix{
			0\ar[r]&\Z_{2^r}\!\oplus \! \Z_{2^{m_{r+1}^3}}\!\cong Coker(\partial^{5,r}_{11\ast}) \ar[r]& \pi_{11}(P^{5}(2^{r}))\ar[r]^-{p^r_{5\ast}}&\Z_2\{\nu_5^2\}\ar[r]&0
	} }\label{exact:short pi11(P5)}
\end{align}
There is an order $2$ element $\mathbf{x}_1\in \pi_{11}(P^5(2))$ such that $p_{5\ast}^1(\mathbf{x}_1)=\nu_5^2$. Then the order $2$ element $\mathbf{x}_r:=\bar{\chi}_{r}^{1}\mathbf{x}_1\in \pi_{11}(P^5(2^r)) (r\geq 2)$ makes $p_{5\ast}^r(\mathbf{x}_r)=\nu_5^2$. So above exact sequence is also split for $r\geq 2$. Hence 
\begin{align*}
	\pi_{11}(P^{5}(2^{r}))\cong \Z_{2^r}\oplus  \Z_{2^{m_{r+1}^3}}\oplus \Z_2, r\geq 1.
\end{align*}

$\bullet~~\pi_{12}(P^{5}(2^r))$. 

Consider diagram (\ref{exact:hgps piPk+1(2r)}) for $k=4, m=12$.
\begin{align}
	&\xymatrix{
		\Z_{2}\{\varepsilon_5\}=\pi_{13}(S^{5})\ar[r]^-{ \partial^{5,r}_{12\ast}}&\pi_{12}(F_{5})\ar[r]^-{\tau_{5\ast}}&\pi_{12}(P^{5}(2^r))\ar[r]^-{p_{5\ast}}&\pi_{12}(S^{5})\ar[r]^-{ \partial^{5,r}_{11\ast}=0}& \pi_{11}(F_{5}) } \nonumber\\
	&\partial^{5,r}_{12\ast}(\varepsilon_5)=\beta_{4\ast}(2^r\iota_4)_{\ast}(\varepsilon_4)=0~~\Rightarrow~~Coker(\partial^{5,r}_{12\ast}(\varepsilon_5))=\pi_{12}(F_{5})\cong \Z_2\oplus\Z_2.\nonumber\\
	&	\xymatrix{
		0\ar[r]&\Z_{2}\oplus  \Z_{2}\ar[r]& \pi_{12}(P^{5}(2^{r}))\ar[r]^-{p_{5\ast}}&\Z_2\{\sigma'''\}\ar[r]&0.
	}  \label{exact:short pi12(P5)}
\end{align}
Since $\pi_{12}(P^{5}(2))\cong \Z_2^{\oplus 3}$  by Theorem 5.10 of \cite{WJ Proj plane}, the exact sequence (\ref{exact:short pi12(P5)}) splits for $r=1$. By the same proof as that of  the split of short exact sequence (\ref{exact:short pi11(P5)}), we get (\ref{exact:short pi12(P5)}) splits for any $r\geq 1$. Hence 
\begin{align*}
	\pi_{12}(P^{5}(2^{r}))\cong \Z_2^{\oplus 3}.
\end{align*}

$\bullet~~\pi_{13}(P^{5}(2^r))$. 

Consider diagram (\ref{exact:hgps piPk+1(2r)}) for $k=4, m=13$.
\begin{align}
	&\pi_{14}(S^{5})=\Z_2\{\nu_5^3\}\oplus \Z_2\{\mu_5\}\oplus \Z_2\{\eta_5\varepsilon_6\}.\nonumber\\
	&\partial^{5,r}_{13\ast}(\nu_5^3)=\beta_{4\ast}(2^r\iota_4)_{\ast}(\nu_4^3)=\beta_{4}(2^{2r}\nu_4-2^{r-1}(2^r-1)\Sigma \nu')\nu_7^2=0.
	\nonumber\\
	&\text{Similarly},~	\partial^{5,r}_{13\ast}(\mu_5)=\partial^{5,r}_{13\ast}(\eta_5\varepsilon_6)=0.  ~~\text{We have}~
	\nonumber\\
	&	\xymatrix{
		0\!\ar[r]&\!\pi_{13}(F_5)\!\cong\!\pi_{13}(J_{3,4}^{r}) \!\ar[r]&\! \pi_{13}(P^{5}(2^{r}))\!\ar[r]^-{p_{5\ast}}&\Z_2\{\varepsilon_5\}\ar[r]\!&0.
	}  \nonumber
\end{align}
By Lemma \ref{lem split}, $\exists -\mathbf{x_r}\in \{i_4,2^r\iota_4, \varepsilon_4 \}$ such that $ p_{5\ast}(\mathbf{x})=\varepsilon_5$ and $2\mathbf{x}\in i_4\{2^r\iota_4, \varepsilon_4, 2\iota_{12}\}$.
By Lemma \ref{lemA: Todabraket } (iii), $2\mathbf{x}_1=i_4\eta_4\varepsilon_5$ and 
$2\mathbf{x}_r=0$ for $r\geq 2$. Hence 
\begin{align*}
		\pi_{13}(P^{5}(2^r))\cong \left\{
		\begin{array}{ll}
		\Z_{2}^{ \oplus 2}\oplus \Z_{4}^{ \oplus 2}, & \hbox{$r=2$;} \\
			\Z_{2}^{\oplus 6} , & \hbox{$r\geq 3$.}
		\end{array}
		\right.
\end{align*}

$\bullet~~\pi_{14}(P^{5}(2^r))$.

Consider diagram (\ref{exact:hgps piPk+1(2r)}) for $k=4, m=14$. $\pi_{14}(S^5)=\Z_2\{\nu_5^3\}\oplus\Z_2\{\mu_5\}\oplus \Z_2\{\eta_5\varepsilon_6\} $.
\begin{align*}
	&\xymatrix{
		\pi_{15}(S^5)\!\ar[r]^-{\partial^{5,r}_{14\ast}}&\!\pi_{14}(F_5)\!\cong\!\pi_{14}(J_{3,4}^{r}) \!\ar[r]&\! \pi_{14}(P^{5}(2^{r}))\!\ar[r]^-{p_{5\ast}}&\pi_{14}(S^5)\ar[r]\!&0
	}
\end{align*}
\begin{lemma}\label{lem:split exact pi14(P5)}
	The short exact sequence $0\rightarrow  Coker(\partial^{5,r}_{14\ast})\rightarrow \pi_{14}(P^{5}(2^r))\rightarrow \pi_{14}(S^5)\rightarrow 0$ dose not split for $r=1$ and splits for $r\geq 2$.
\end{lemma}
\begin{proof}
	We have the order $2$ element $\mathbf{x}_r\in\pi_{11}(P^5(2^r))$ which is obtained in the proof of the split of the exact sequence (\ref{exact:short pi11(P5)}) satisfies $p_{5\ast}(\mathbf{x}_r)=\nu_5^2$. This implies  the order $2$ element $\mathbf{x}_r\nu_{11}\in\pi_{14}(P^5(2^r))$ satisfies $p_{5\ast}(\mathbf{x}_r\nu_{11})=\nu_5^3$. Since $\pi_{6}(P^5(2^r))\xrightarrow{p_{5\ast}} \pi_{6}(S^5)$ is epimorphism, let $\mathbf{y}_r\in\pi_{6}(P^5(2^r))$ such that $p_{5\ast}(\mathbf{y}_r)=\eta_5$. Then the order 2 element  $\mathbf{y}_r\varepsilon_6\in\pi_{13}(P^5(2^r))$ satisfies  $p_{5\ast}(\mathbf{y}_r\varepsilon_6)=\eta_5\varepsilon_6$. For $\mu_5\in \pi_{14}(S^5)$, there is a $-\mathbf{z}_r\in \{i_4, 2^r\iota_4, \mu_4\}$ such that $p_{5\ast}(\mathbf{z}_r)=\mu_5$ with $2 \mathbf{z}_r\in \{i_4, 2^r\iota_4, \mu_4\}(2\iota_{14})=i_4 \{2^r\iota_4, \mu_4, 2\iota_{13}\}$. By Lemma \ref{lemA: Todabraket } (iv), $o(\mathbf{z}_1)\geq 4$ and $2\mathbf{z}_r=i_{4}(2\varrho), \varrho\in \pi_{14}(S^4)$ for $r\geq 2$. 
	Thus the short exact sequence in the lemma dose not split for $r=1$.  For $r\geq 2$, take $\mathbf{z}'_r=\mathbf{z}_r-i_{4}(\varrho)$.  Then we get $p_{5\ast}(\mathbf{z}'_r)=\mu_5$ with $o(\mathbf{z}'_r)=2$, which implies the short exact sequence in the lemma splits for $r\geq 2$.
\end{proof}

$\pi_{15}(S^5)=\Z_8\{\nu_5\sigma_8\}\oplus \Z_2\{\eta_5\mu_6\}~\text{and}~\partial^{5,r}_{14\ast}(\eta_5\mu_6)=\check\beta_{4\ast}(2^r\iota_4)_{\ast}(\eta_4\mu_5)=0$
(Here we calculate the image of $\partial^{5,r}_{14\ast}$ in $\pi_{14}(J_{3,4}^{r}) $ under isomorphism $\pi_{14}(F_5)\cong\pi_{14}(J_{3,4}^{r})$).
\qquad

$\mathbf{Case: r=1}$~~Since $2\nu_5\sigma_8=\pm \Sigma^2\varepsilon'$  \cite[(7.10)]{Toda},
\begin{align*}
	&2\partial^{5,1}_{14\ast}(\nu_5\sigma_8)=\partial^{5,1}_{14\ast}(2\nu_5\sigma_8)=\pm\check\beta_{4\ast}(2\iota_4)_{\ast}(\Sigma\varepsilon')=2\check\beta_{4}\Sigma\varepsilon'\neq 0,\\
	\text{implies}~~~ & \partial^{5,1}_{14\ast}(\nu_5\sigma_8)=\pm \check\beta_{4}\Sigma\varepsilon'+x_1(\check\beta_4\nu_4\sigma'\!+\check\beta^{1}_4\Sigma\varepsilon')+y_1\check\beta_4\eta_4\mu_5+z_1\check\beta_{11}\nu_{11}\\
	+&w_1[\check\beta_4,\check\beta_{11}]+u_1I_{2}^{1}\widehat{\nu_8^2}+v_1\widearc{\eta_{12}}\eta_{13}, ~x_1,y_1,z_1,w_1,u_1,v_1\in \Z_2~\text{(Lemma \ref{lem: pi_{14}(J3(2r))})}.\\
	\Rightarrow&  Coker(\partial^{5,1}_{14\ast})= (\Z_2\{\check\beta_4\nu_4\sigma'\!+\check\beta_4\Sigma\varepsilon'\}\oplus  H_{14}^1)/\lr{\partial^{5,1}_{14\ast}(\nu_5\sigma_8)}\\
	=& \Z_2\{\check\beta_4\nu_4\sigma'\!+\check\beta_4\Sigma\varepsilon'\}\oplus(H_{14}^1/\Z_4\{\check\beta_4\Sigma\varepsilon'\})\cong\Z_{2}^{\oplus 6} .\\
	\Rightarrow& \pi_{14}(P^{5}(2))\cong \Z_{2}^{\oplus 7}\oplus \Z_{4}~~\text{by Lemma \ref{lem:split exact pi14(P5)}}.
\end{align*}

$\mathbf{Case: r\geq 2}$~~We have the following left and right homotopy commutative diagrams of cofibration sequences respectively for $r=2$ and $r\geq 3$
\begin{align}
	\small{\xymatrix{
			S^7\ar[d]^{\nu_4} \ar[r]^-{2\iota_7}&S^7 \ar[d]^{8\nu_4+\Sigma\nu'}\ar[r] & P^{8}(2) \ar[d]_{\bar{\phi}^{1}_{2}}\ar[r]^{p_{8}^1}& S^{7}\ar[d]_{\nu_5} \\
			S^4 \ar[r]^-{4\iota_4}&S^4 \ar[r]& P^{5}(4)\ar[r]^-{p^{2}_{5}}& S^{5}
	} }
	, \small{\xymatrix{
			S^7\ar[d]^{\nu_4} \ar[r]^-{2^{2r}\iota_7}&S^7 \ar[d]^{\nu_4}\ar[r] & P^{8}(2^{2r}) \ar[d]_{\bar{\phi}^{2r}_{r}}\ar[r]^{p_{8}^{2r}}& S^{7}\ar[d]_{\nu_5} \\
			S^4 \ar[r]^-{2^r\iota_4}&S^4 \ar[r]& P^{5}(2^r)\ar[r]^-{p^{r}_{5}}& S^{5}
	} }
\end{align}
There are maps $\phi^{s}_{r}:F_{8}^{s}\rightarrow F_{5}^{r}$  making the following diagrams of exact sequences commutative where $s=2$ and $2r$ repectively for $r=2$ and $r\geq 3$
\begin{align}
	\small{\xymatrix{
			\pi_{15}(S^{8})\ar[d]^{\nu_{5\ast}} \ar[r]^-{\partial^{8,s}_{14\ast}}&\pi_{14}(F^s_{8})\ar[d]^{\phi^{s}_{r\ast}}\ar[r]^-{\tau_{8\ast}^{s}} &\pi_{14}(P^{8}(2^{s})) \ar[d]^{\bar\phi^{s}_{r\ast}}\ar[r]^-{p^s_{8\ast}}&\pi_{14}(S^{8})\ar[r]^-{\partial^{8,s}_{13\ast}}\ar[d]^{\nu_{5\ast}}&0\\
			\pi_{15}(S^{5}) \ar[r]^-{\partial^{5,r}_{14\ast} }&\pi_{14}(F^r_{5}) \ar[r]^-{\tau^{r}_{5\ast}}& \pi_{14}(P^{5}(2^r))\ar[r]^-{p^r_{5\ast}}&\pi_{14}(S^{5})\ar[r]&0
	} } \label{diam:pi14(P8(2^s))to pi14(P5(2^r))}
\end{align}
where $\partial^{8,s}_{13\ast}=0$ in the top row since $\partial^{8,s}_{13\ast}(\nu_8^2)=\beta_{7\ast}(2^s\iota_7)_{\ast}(\nu_7^2)=0$; $\pi_{14}(F^s_{8})=\Z_{(2)}\{\beta^s_{14}\}\oplus \Z_{8}\{\beta^s_{7}\sigma'\}$; $\phi_{2}^1\beta^1_7=\beta^2_4(8\nu_4+\Sigma\nu')$ and $\phi_{r}^{2r}\beta^{2r}_7=\beta^r_4\nu_4$ by Lemma \ref{Lem: natural J(Mf,X)}.

\begin{lemma}\label{lem:partial^{8,s}_{14}(sigma8)}
	$\partial^{8,s}_{14\ast}(\sigma_8)=2^s\beta_{14}+y\beta_{14}\sigma'$, where 
	\begin{align*}
		\!\!\!\!\!\!\footnotesize{\begin{tabular}{r|c|c|c|c|}
				\cline{2-5}
				& $ s=1$&$s=2$&$s=3$& $s\geq 4$ \\
				\cline{2-5}
				$\Z_8\ni y=$  &  odd & $\pm 2$ &$4$ & $0$\\
				\cline{2-5}
		\end{tabular}}~ 
	\end{align*}
\end{lemma}
\begin{proof}
	Lemma \ref{Lem: compute H2} for map $S^{7}\xrightarrow{2^s\iota_{7}}S^{7}$ gives the following commutative diagram
	\begin{align*}
		\footnotesize{\xymatrix{
				\pi_{15}( S^{8})\ar[d]_{H_2} \ar[r]^-{\partial^{8,s}_{14\ast}} & \pi_{14}(F^s_{8})\ar[d]_{H'_2 }&\pi_{14}(S^{7}\vee S^{14})\!\ar@{_{(}->}[l]^-{\cong}\ar[r]^-{Proj}&\!\!\pi_{14}(S^{14})\ar[d]_{\Sigma_{\ast} \cong}\\
				\pi_{15}(S^{15}) ~~~\ar[r]^-{(2^s\iota_{15})_{\ast}} & \pi_{15}(S^{15})\ar@{=}[rr]&&\! \pi_{15}(S^{15}) } }  
	\end{align*}
	By the right commutative square, $H'_{2}(\beta_{14})=\iota_{15}$. By the left commutative square, we get $H_{2}'\partial^{8,s}_{14\ast}(\sigma_8)=2^{s} H_2(\sigma_8)=2^s\iota_{15}$. This implies that 
	\begin{align*}
		\partial^{8,s}_{14\ast}(\sigma_8)=2^s\beta_{14}+y\beta_{14}\sigma', y\in \Z_8.
	\end{align*}
	From Lemma 2.4 of \cite{ZP23}, there is a map $\phi: F^s_{8}\rightarrow \Omega S^8$ such that the following left diagram is homotopy commutative and it induces the right commutative diagram

	\begin{align} \footnotesize{\xymatrix{
				& S^7 \ar@{_{(}->}[d]_{\beta_{7}}\ar@/^1pc/[dd]\\
				\Omega S^8\ar@{=}[d] \ar[r]^-{\partial^8}&F_{8}\ar[r]\ar[d]_{\phi}& P^{8}(2^s)\ar@{_{(}->}[d]\\
				\Omega S^8 \ar[r]^-{\Omega(\!-\!2^s\!\iota_8)}&\Omega S^8\ar[r] & J(P^{8}(2^s),S^7),
		} }~~ \footnotesize{\xymatrix{
				&\Z_{(2)}\{\beta_{14}\}\oplus\Z_8\{\beta_7\sigma'\}\ar@{=}[d]\\
				\pi_{15}( S^8)\ar[rd]_-{P_1(-2^s\iota_8)_\ast} \ar[r]^-{{\partial^{8}_{14\ast}}} & \!\pi_{14}(F_{8})\ar[d]_{P_1\phi_{\ast} }\\
				& \Z_8\{\Sigma \sigma'\} } }   \label{Diagram phi j} 
	\end{align}
	$P_1:\pi_{15}( S^8)=\Z_{(2)}\{\sigma_{8}\}\oplus \Z_8\{\Sigma \sigma' \}\rightarrow \Z_8\{\Sigma \sigma' \}$ is the canonical projection.

	By comparing the integral Homology $H_{7}(-)$, we get
	\begin{align}
		\phi\beta_7\simeq h\Omega\Sigma: S^7\rightarrow \Omega S^8 =\Omega\Sigma S^7,  ~ h ~\text{is odd integer}.\nonumber
	\end{align}
	By the same proof as that of Lemma A.1 of \cite{ZP23}, 
	$(-2^s\iota_8)_{\ast}(\sigma_8)=2^{2s}\sigma_8-2^{s-1}(2^s+1)\Sigma \sigma'$.
	\begin{align}
		&P_1\phi_{\ast}\partial^8_{14}( \sigma_8)=P_1(y\phi_{\ast}(\beta_{7}\sigma')+ 2^s\phi_{\ast}(\beta_{14}))=hy\Sigma \sigma'+2^sP_1\phi_{\ast}(\beta_{14})  ~~~~\nonumber\\
		&=(hy+2^sl)\Sigma \sigma ' ~~~~~~(\text{for some  integer } l). \nonumber\\
		&P_1(-2^s\iota_8)_\ast(\sigma_4)=P_1(2^{2s}\sigma_8-2^{s-1}(2^s+1)\Sigma \sigma')=-2^{s-1}(2^s+1)\Sigma \sigma'.\nonumber
	\end{align}
	From  (\ref{Diagram phi j}), we get $(hy+2^sl)\Sigma \sigma '=-2^{s-1}(2^s+1)\Sigma\sigma'$, thus we get the value of $y$ in Lemma \ref{lem:partial^{8,s}_{14}(sigma8)}. 
\end{proof}
Now by commutative diagram (\ref{diam:pi14(P8(2^s))to pi14(P5(2^r))}) and Lemma \ref{lem: pi_{14}(J3(2r))}, Lemma \ref{lem:partial^{8,s}_{14}(sigma8)}
\begin{align*}
	&\partial_{14\ast}^{5,r}(\nu_5\sigma_8)=\phi_{r\ast}^s\partial_{14\ast}^{8,s}(\sigma_8)=\phi_{r\ast}^s(2^s\beta_{14}+y\beta_{7}\sigma')=2^s\phi_{r\ast}^s(\beta_{14})+y\phi_{r\ast}^s\beta_{7}\sigma'\\
	=&\left\{
	\begin{array}{ll}
		2\phi_{2\ast}^1(\beta_{14})+y\beta_4(8\nu_4+\Sigma\nu')\sigma'=2(\phi_{2\ast}^1(\beta_{14})+\beta_4\Sigma \varepsilon'), & \hbox{$r=2$;} \\
		0, &\hbox{$r\geq 3$.}
	\end{array}
	\right.
\end{align*}
Thus for $r\geq 3$, $Coker(\partial_{14\ast}^{5,r})\cong \pi_{14}(J_{3,4}^{r})\cong \Z_{2^r}\oplus \Z_{8}^{\oplus 2}\oplus \Z_4\oplus \Z_{2}^{\oplus 3}$.

For $r=2$,  we have $x_2, y_2,z_2,w_2\in\Z_4 $ such that 
\begin{align}
	\partial_{14\ast}^{5,2}(\nu_5\sigma_8)=2x_2 \check\beta^{2}_4\nu_4\sigma'+2y_2\check\beta^{2}_4\Sigma\varepsilon'+2z_2\check\beta^{2}_{11}\nu_{11}+2w_2[\check\beta^{2}_4,\check\beta^{2}_{11}],  \label{equ:ptial14^{5,2}(nu_5sigma_8)}
\end{align}
Consider commutative diagrams in (\ref{diam Moore s to t}) for $s=2$, $t=1$, 
\newline
$\chi_{1\ast}^{2}(\partial_{14\ast}^{5,2}(\nu_5\sigma_8))=\partial_{14\ast}^{5,1}(2\iota_5)_{\ast}(\nu_5\sigma_8)=\partial_{14\ast}^{5,1}(2\nu_5\sigma_8)\neq 0$, where the second equation holds since $\Sigma ((2\iota_5)(\nu_5\sigma_8))=\Sigma(2\nu_5\sigma_8)$ and $\Sigma:\pi_{15}(S^8)\rightarrow \pi_{16}(S^9)$ is a monomorphism (\cite[page 69]{Toda}). Hence 
$\partial_{14\ast}^{5,2}(\nu_5\sigma_8)\neq 0$, i.e., one of $ x_2, y_2,z_2,w_2$ must be nonzero in (\ref{equ:ptial14^{5,2}(nu_5sigma_8)}). So
$$Coker(\partial_{14\ast}^{5,2})\cong  \Z_{4}^{\oplus 3}\oplus \Z_{2}^{\oplus 4}$$ which is obtained by reducing one of $\Z_4$-summands  of  $\pi_{14}(F_{5}^2)\cong \pi_{14}(J_{3,4}^2)$ to $\Z_2$. Now from Lemma \ref{lem:split exact pi14(P5)}, 
\begin{align*}
	& \pi_{14}(P^5(2^r))\cong \left\{
	\begin{array}{ll}
		\Z_{4}^{\oplus 3}\oplus \Z_{2}^{\oplus 7}	, & \hbox{$r=2$;} \\
		\Z_{2^r}\oplus \Z_{8}^{\oplus 2}\oplus \Z_4\oplus \Z_{2}^{\oplus 6}	, &\hbox{$r\geq3$.}
	\end{array}
	\right.
\end{align*}

\subsection{Homotopy groups of $P^{6}(2^r)$}
\label{subsec: P6(2^r)}

There are cofibration and fibration sequences 
\begin{align*}
	&	S^5\xrightarrow{2^{r}\iota_5} S^{5}\xrightarrow{i_{5}} P^{6}(2^r)\xrightarrow{p_{6}}  S^{6};
	~~\Omega S^6 \xrightarrow{\partial^{6}}F_{6}\xrightarrow{\tau_6} P^{6}(2^r)\xrightarrow{p_{6}}  S^{6}\\
	& Sk_{14}(F_{6})\simeq J_{2,5}^{r}\simeq S^5\vee S^{10}(r\geq 1);~~Sk_{19}(F_{6})\simeq J_{3,5}^{r}\simeq J_{2,5}^{r}\cup_{\gamma^r_3}CS^{14}. 
\end{align*}
Here $\beta_5=j_1^5$, $\beta_{10}=j_2^{10}$.

\begin{lemma}\label{lem: gamma_3^r P6}
	For $r\geq 1$,	$\gamma^r_{3}= 2^ra_0[j_1^5, j_{2}^{10}]\in \pi_{14}(J_{2,5}^{r})$,   where $a_0$ is an odd integer in $\Z_{(2)}$.
\end{lemma}
\begin{proof}
	$J_{2,5}^{0}\simeq S^5\cup_{[\iota_5,\iota_5]}e^{10}\simeq S^5\cup_{\nu_5\eta_8}e^{10}$. It is easy to see from (\ref{equ: Sk_5k-3 barFp2k}) that 
	\begin{align*}
		&\pi_{14}(J_{2,5}^{0})\cong 	\pi_{14}(F_{\bar p^0_{10}})\cong \pi_{14}(S^5\vee S^{14})\\
		\Rightarrow~& \pi_{14}(J_{2,5}^{0})=\Z_{2}\{\beta^0_{5}\nu_5^3\}\oplus \Z_{2}\{\beta^0_{5}\mu_5\}\oplus\Z_2\{\beta^0_{5}\eta_5\varepsilon_6\}\oplus \Z_{(2)}\{\beta^0_{14}\}.\\
		\text{Assume}~&  \gamma_{3}^0=\beta^0_{5}\theta_0+a_0\beta^0_{14}, \theta_0\in \pi_{14}(S^5)  , a_0\in \Z_{(2)}.\\
		\text{By}~& \Z_{2}^{\oplus 3}\cong \pi_{15}(S^6)\xrightarrow[\cong]{\partial^{6,0}_{14\ast}}\pi_{14}(F_6^0)\cong\pi_{14}(J_{3,5}^{0})\cong \pi_{14}(J_{2,5}^{0})/\lr{\gamma_{3}^0}\\
		\Rightarrow~ & a_0~\text{is an odd integer in}~\Z_{(2)}.
	\end{align*}
	From lemma \ref{lem:gamma3 k odd}, it is easy to get
	$$\gamma_{3}^r=a_r[j_1^5, j_2^{10}], a_r\in \Z$$
	
	Since  the composition of the maps   $S^5\vee S^{14}\simeq J_{2,5}^{\gamma, r}\hookrightarrow F_{\bar p_{10}^r}\xrightarrow{\bar \tau_{10}^r} J_{2,5}^r\simeq S^5\vee S^{10}$   induces the isomorphism $\pi_{14}(S^5\vee S^{14})\xrightarrow{\cong } \pi_{14}(S^5\vee S^{10})$. We get 
	$$\bar \tau_{10}^r\bar \beta_{14}^r=\pm [j_1^5, j_{2}^{10}]+j_1^5\theta_1, \theta_1\in \pi_{14}(S^5).$$
	\begin{align*}
		2^{2r}\gamma_3^0&=\bar g_{0}^r \gamma_3^r=\pm a_r\bar g_{0}^r(\bar\tau_{10}^r\bar\beta_{14}^r-j_1^5\theta_1)=\pm a_r\bar g_{0}^r(\bar\tau_{10}^r\bar\beta_{14}^r-\bar\tau_{10}^r\bar\beta_{5}^r\theta_1)     ~~~~(\text{by (\ref{Equ:gamma30,3r})})\\
		=&\pm a_r\bar\tau_{10}^0g_{0}^r(\bar\beta_{14}^r-\bar\beta^r_{5}\theta_1)=\pm 2^ra_r\bar\tau_{10}^0\bar\beta_{14}^0+\bar\tau_{10}^0\bar\beta_{5}^0\theta', ~~(\text{by (\ref{equ1 g^s_t}) and (\ref{equ2  g^s_t})})\\
		&=\pm 2^ra_r\beta_{14}^0+\beta_{5}^0\theta', ~~\theta'\in \pi_{14}(S^5) \\
		\text{i.e.},~~&	2^{2r}(a_0\beta^0_{14}+\beta^0_{5}\theta_0)=\pm 2^ra_r\beta_{14}^0+\beta_{5}^0\theta'~~\Rightarrow~~ a_r=\pm 2^ra_0.	
	\end{align*}
	Thus we obtain this lemma by replacing  $\pm a_0$ by $a_0$ in equation $ a_r=\pm 2^ra_0$.
\end{proof}
Lemma \ref{Lem: compute H2} gives the following commutative diagram
\begin{align}
	&\small{\xymatrix{
			\pi_{m+1}( S^{6})\ar[d]_{H_2} \ar[r]^-{\partial^{6,r}_{m\ast}} & \pi_{m}(F^r_{6})\ar[d]_{H'_2 }&\pi_{m}(S^{5}\vee S^{10})\!\ar[l]_-{\check I^r_{2\ast}}\ar[r]^-{Proj}&\!\!\pi_{m}(S^{10})\ar[d]_{\Sigma_{\ast}}\\
			\pi_{m+1}(S^{11}) ~~~\ar[r]^-{(2^r\iota_{11})_{\ast}} & \pi_{m+1}(S^{11})\ar@{=}[rr]&&\! \pi_{m+1}(S^{11}) } }  \label{diam: H2H2' P6}
\end{align}
where $\check I^r_{2\ast}$ is isomorphic for $ m\leq 13$.
\begin{lemma}
	$\pi_{i}(F_6)\cong \pi_{i}(J_{3,5}^r), (i\leq 18)$ and 
	\begin{align*} 
		(1)~& \pi_{14}(J_{3,5}^{r})=\Z_{2}\{\check\beta_{5}\nu_5^3\}\oplus \Z_{2}\{\check\beta_{5}\mu_5\}\oplus\Z_2\{\check\beta_{5}\eta_5\varepsilon_6\}\oplus \Z_{2^r}\{[\check\beta_{5}, \check\beta_{10}]\}.\\
		(2)~& \pi_{15}(J_{3,5}^{r})=\Z_{8}\{\check\beta_{5}\nu_5\sigma_{8}\}\oplus \Z_{2}\{\check\beta_{5}\eta_5\mu_6\}\oplus \Z_{2}\{[\check\beta_{5}, \check\beta_{10}]\eta_{14}\}.\\
		(3)~&\pi_{16}(J_{3,5}^{1})=\Z_{8}\{\check\beta_{5}\zeta_5\}\oplus \Z_{2}\{\check\beta_{5}\nu_5\bar\nu_8\}\oplus \Z_{2}\{\check\beta_{5}\nu_5\varepsilon_8\}\oplus \Z_{2}\{\check\beta_{10}\nu_{10}^2\}\oplus\Z_4\{\widearc{\eta_{15}}\},\\
		&\qquad\text{with}~[\check\beta_{5}, \check\beta_{10}]\eta_{14}^2=2\widearc{\eta_{15}}.\\
		&\pi_{16}(J_{3,5}^{r})=\Z_{8}\{\check\beta_{5}\zeta_5\}\oplus \Z_{2}\{\check\beta_{5}\nu_5\bar\nu_8\}\oplus \Z_{2}\{\check\beta_{5}\nu_5\varepsilon_8\}\oplus \Z_{2}\{\check\beta_{10}\nu_{10}^2\}\\
		&\qquad \qquad \oplus  \Z_2\{[\check\beta_{5},\check\beta_{10}]\eta_{14}^2\}\oplus\Z_2\{\widearc{\eta_{15}}\}, (r\geq 2).\\
		(4)~ &\pi_{17}(J_{3,5}^{r})=\Z_{2}\{\check\beta_{5}\nu^4_5\}\oplus \Z_{2}\{\check\beta_{5}\nu_5\bar\mu_8\}\oplus \Z_{2}\{\check\beta_{5}\nu_5\eta_8\varepsilon_9\}\oplus \Z_{16}\{\check\beta_{10}\sigma_{10}\} \\
		&\qquad \qquad \oplus \Z_{2^{m_r^3}}\{[\check\beta_{5}, \check\beta_{10}]\nu_{14}\}\oplus \Z_{2}\{\widearc{\eta_{15}}\eta_{16}\}.\\
		(5)~ &\pi_{18}(J_{3,5}^{r})=\Z_{2}\{\check\beta_{5}\nu_5\sigma_8\nu_{15}\}\oplus \Z_{2}\{\check\beta_{5}\nu_5\eta_8\mu_{9}\}\oplus  \Z_{2}\{\check\beta_{10}\bar\nu_{10}\}\oplus  \Z_{2}\{\check\beta_{10}\varepsilon_{10}\} \\
		& \qquad \qquad \oplus \Z_{2^r}\{[\check\beta_{5}, [\check\beta_{5},\check\beta_{10}]]\}\oplus \Z_{2^{m_r^3}}\{\widearc{\delta_r\nu_{15}}\}.
	\end{align*}
\end{lemma}
\begin{proof} ~$(1)$ of this lemma is easily obtained by  Lemma \ref{lem: gamma_3^r P6}.
	
	Consider the Lemma \ref{lem: exact seq pi_m(J3(2rl4))} for $k=5$. 
	
	$(2)$ is obtained by the isomorphism $\pi_{15}(S^5\vee S^{10})\xrightarrow[\cong]{I_{2\ast}^r} \pi_{15}(J_{3,5}^r)$. 
	
	For $(3)$, we have the short exact sequence
	\begin{align*}
		&	0\rightarrow \pi_{16}(S^5\vee S^{10})\xrightarrow{I_{2\ast}^r} \pi_{16}(J_{3,5}^r)\xrightarrow{\check p_{15\ast}}\pi_{16}(S^{15})=\Z_{2}\{\eta_{15}\}\rightarrow 0,\\
		&	 \pi_{16}(S^5\vee S^{10})=\Z_{8}\{j_1^5\zeta_5\}\oplus \Z_{2}\{j_1^5\nu_5\bar\nu_8\}\oplus \Z_{2}\{j_1^5\nu_5\varepsilon_8\}\oplus \Z_{2}\{j_2^{10}\nu_{10}^2\}\oplus \Z_2\{[j_1^5,j_2^{10}]\eta_{14}^2\}.
	\end{align*}
	By Lemma \ref{lem split}, there is an $-\mathbf{x}_r\in \{I^r_2,\gamma_3^r,\eta_{14}\}\subset \pi_{16}(J_{3,5}^{r})$, such that $\check p_{15\ast}(\mathbf{x}_r)=\eta_{15}$  and  $2\mathbf{x}_r\in I^r_2\{2^ra_0[j_1^5,j_2^{10}],\eta_{14},2\iota_{15}\}$.  By Lemma \ref{lem:Toda bracket}
	\begin{align*}
	&\{2^ra_0[j_1^5,j_2^{10}],\eta_{14},2\iota_{15}\}\ni 2^{r-1}[j_1^5,j_2^{10}]\eta_{14}^2 ~\\
	&\text{mod}~	(2^ra_0[j_1^5,j_2^{10}])\fhe\pi_{16}(S^{14})+\pi_{16}(S^5\vee S^{10})\fhe(2\iota_{16})=\lr{2j_1^5\zeta_5}.
	\end{align*}
	
	For $r=1$, $2\mathbf{x}_1=[\check\beta_5,\check\beta_{10}]\eta_{14}^2+2u_1\check\beta_5\zeta_5, u_1\in\Z_4$.  Let $\mathbf{x}'_1=\mathbf{x}_1-u_1\check\beta_5\zeta_5$, then $o(\mathbf{x}'_1)=4$ and $\check{p}_{15\ast}(\mathbf{x}'_1)=\eta_{15}$ with $2\mathbf{x}'_1=[\check\beta_5,\check\beta_{10}]\eta_{14}^2$. Thus we get $\pi_{16}(J_{3,5}^1)$. 
	
	For $r\geq 2$, $2\mathbf{x}_r=2u_2\check\beta_5\zeta_5,  u_2\in\Z_4$. Let $\mathbf{x}'_r=\mathbf{x}_r-u_2\check\beta_5\zeta_5$, then $o(\mathbf{x}'_r)=2$ and $\check{p}_{15\ast}(\mathbf{x}'_r)=\eta_{15}$.  Thus $\pi_{16}(J_{3,5}^r)\stackrel{\check p_{15\ast}}\twoheadrightarrow\Z_{2}\{\eta_{15}\}$ is split onto. Hence we get $\pi_{16}(J_{3,5}^r), r\geq 2$ in (3) of this lemma. 
	
	For $(4)$, we have the  exact sequence
	\newline
	$0\rightarrow\frac{	\pi_{17}(S^5\vee S^{10})}{\lr{\gamma_{3}^r\nu_{14}}}\xrightarrow{I_{2\ast}^r} \pi_{17}(J_{3,5}^r)\xrightarrow{\check p_{15\ast}}\pi_{17}(S^{15})=\Z_{2}\{\eta^2_{15}\}\rightarrow 0$,
	\newline
	$ \frac{\pi_{17}(S^5\vee S^{10})}{\lr{\gamma_{3}^r\nu_{14}}}=\Z_{2}\{j_1^5\nu^4_5\}\oplus \Z_{2}\{j_1^5\nu_5\mu_8\}\oplus \Z_{2}\{j_1^5\nu_5\eta_8\varepsilon_9\}\oplus \Z_{16}\{j_1^{10}\sigma_{10}\}
	\oplus \Z_{2^{m_r^3}}\{[j_1^5, j_2^{10}]\nu_{14}\}.$
	\newline
	Note that the order $2$ element $\widearc{\eta_{15}}\eta_{16}$ satisfies $\check p_{15\ast}(\widearc{\eta_{15}}\eta_{16})=\eta_{15}^2$. Hence 
	$\pi_{17}(J_{3,5}^r)\stackrel{\check p_{15\ast}}\twoheadrightarrow\Z_{2}\{\eta_{15}^2\}$ is split onto. We get the $\pi_{17}(J_{3,5}^r)$ of this lemma. 
	
	For $(5)$,  we have the  exact sequence
	\newline
	$0\rightarrow\frac{	\pi_{18}(S^5\vee S^{10})}{\lr{2^ra_0[j_1^5,[j_1^5,j_2^{10}]]}}\xrightarrow{I_{2\ast}^r} \pi_{18}(J_{3,5}^r)\xrightarrow{\check p_{15\ast}}\Z_{2^{m_r^3}}\{\delta_r\nu_{15}\}\rightarrow 0$,  ~$ \frac{\pi_{18}(S^5\vee S^{10})}{\lr{2^ra_0[j_1^5,[j_1^5,j_2^{10}]]}}=$
	\newline
	$ \Z_{2}\{j_1^5\nu_5\sigma_8\nu_{15}\}\oplus \Z_{2}\{j_1^5\nu_5\eta_8\mu_{9}\}\oplus  \Z_{2}\{j_2^{10}\bar\nu_{10}\}\oplus  \Z_{2}\{j_2^{10}\varepsilon_{10}\}\oplus \Z_{2^r}\{[j_1^5, [j_1^5,j_2^{10}]]\}$.

	There is $-\mathbf{x}_r\in \{I^r_2,\gamma_3^r,\delta_r\nu_{14}\}\subset \pi_{18}(J_{3,5}^{r})$, such that $\check p_{15\ast}(\mathbf{x}_r)=\delta_r\nu_{15}$  and  $2^{m_r^3}\mathbf{x}_r\in I^r_2\{\gamma_3^r,\delta_r\nu_{14},2^{m_r^3}\iota_{17}\}$.
	$\{\gamma_3^r,\delta_r\nu_{14},2^{m_r^3}\iota_{17}\}$ is a coset of subgroup $\gamma_3^r\fhe\pi_{18}(S^{14})+\pi_{18}(S^5\vee S^{10})\fhe(2^{m_r^3}\iota_{18})=2^{m_r^3}\pi_{18}(S^5\vee S^{10})$. 
	Moreover, by $\pi_{18}(S^{14})=0$, 
	$\{\gamma_3^r,\delta_r\nu_{14},2^{m_r^3}\iota_{17}\}$ $\supset (a_0[j_1^5,j_2^{10}])\fhe\{2^{r}\iota_{14},\delta_r\nu_{14},2^{m_r^3}\iota_{17}\})=\{0\}$.
	Thus $\{\gamma_3^r,\delta_r\nu_{14},2^{m_r^3}\iota_{17}\}=2^{m_r^3}\pi_{18}(S^5\vee S^{10})$ which implies $2^{m_r^3}\mathbf{x}_r=2^{m_r^3}I_{2}^r\theta$ for some $\theta\in \pi_{18}(S^5\vee S^{10})$. Take $\mathbf{x}'_r=\mathbf{x}_r-I_{2}^r\theta$, then $o(\mathbf{x}'_r)=2^{m_r^3}$ and $\check p_{15\ast}(\mathbf{x}'_r)=\delta_r\nu_{15}$. Hence we finish the proof of $(5)$ of this lemma.
\end{proof}

$\bullet~~\pi_{10}(P^{6}(2^r))$.

There is the exact sequence 
\begin{align*}
	&\Z_{(2)}\{\Delta(\iota_{13})\}=\pi_{11}(S^6)\xrightarrow{\partial^{6,r}_{10\ast}}\pi_{10}(F_6)=\Z_{2}\{\beta_5\nu_5\eta_8^2\}\oplus \Z_{(2)}\{\beta_{10}\}\xrightarrow{\tau_{6\ast}} \pi_{10}(P^{6}(2^r))\rightarrow 0.\\
	&	\text{Assume} ~~\partial^{6,r}_{10\ast}(\Delta(\iota_{13}))=x_r\beta_5\nu_5\eta_8^2+y_r\beta_{10}, x_r\in \Z_2, y_r\in \Z.
\end{align*}
By the right commutative square of (\ref{diam: H2H2' P6}) for $m=10$, $H'_{2}(\beta_{10})=\iota_{11}$. Since $H_2(\Delta(\iota_{13}))=\pm \iota_{11}$ (\cite[Proposition 2.7]{Toda}),
by the left commutative square of (\ref{diam: H2H2' P6}) for $m=10$, we get $H_{2}'\partial^{6,r}_{10\ast}(\Delta(\iota_{13}))=2^{r} H_2(\Delta(\iota_{13}))=2^{r+1}\iota_{11}$. This implies that $y_r=\pm 2^{r+1}$.

For $r=1$, by \cite[Theorem 2.2]{CohenWu95}, there is an element in
$\pi_{10}(P^6(2))\cong \frac{\Z_{2}\{\beta_5\nu_5\eta_8^2\}\oplus \Z_{(2)}\{\beta_{10}\}}{\lr{\pm 4\beta_{10}+x_1\beta_5\nu_5\eta_8^2}}$ with order 8. So $\pi_{10}(P^6(2))\cong \Z_8$ and $x_1=1$. 

For $r\geq 2$,  consider the diagram (\ref{diam Moore s to t}) for $s=1, t=r$, we get
\begin{align*}
	& \partial^{6,r}_{10\ast}(\Delta(\iota_{13}))=\chi_{r\ast}^1\partial^{6,1}_{10\ast}(\Delta(\iota_{13}))=\chi_{r\ast}^1(4\beta_{10}^1+\beta_5^1\nu_5\eta_8^2)\\
	&=4\chi_{r}^1\beta_{10}^1+2^{r-1}\beta_5^r\nu_5\eta_8^2=4\chi_{r}^1\beta_{10}^1,~\text{by first equation of (\ref{Equ:gamma30,3r})}.
\end{align*}
This implies that $x_r=0$ for $r\geq 2$. Hence 
\begin{align*}
	\pi_{10}(P^6(2^r))\cong \frac{\Z_{2}\{\beta_5\nu_5\eta_8^2\}\oplus \Z_{(2)}\{\beta_{10}\}}{\lr{\pm 2^{r+1}\beta_{10}}}\cong \Z_2\oplus \Z_{2^{r+1}}, r\geq 2.
\end{align*}

$\bullet~~\pi_{11}(P^{6}(2^r))$. 

Consider diagram (\ref{exact:hgps piPk+1(2r)}) for $k=5, m=11, r\geq 1$, we get 
$\partial^{6,r}_{10\ast}(\nu_6^2)=0$.
Since $\partial^{6,r}_{10\ast}$ is a monomorphism, $Ker(\partial^{6,r}_{10\ast})=0$. So we have 
\begin{align*}
	\pi_{11}(P^{6}(2^r))\cong \pi_{11}(F_{6})\cong  \pi_{11}(S^5\vee S^{10})\cong \Z_2^{\oplus 2}.
\end{align*}

$\bullet~~\pi_{12}(P^{6}(2^r))$. 

There is the exact sequence  
\begin{align*}
	&\Z_{4}\{\sigma''\}=\pi_{13}(S^6)\xrightarrow{\partial^{6,r}_{12\ast}}\pi_{12}(F_6)\xrightarrow{\tau_{6\ast}} \pi_{12}(P^{6}(2^r))\xrightarrow{p_{6\ast}} \Z_2\{\nu_6^2\}\rightarrow 0.\\
	\text{Assume}~& \partial^{6,r}_{12\ast}(\sigma'')=  x_r\beta_5\sigma''' +y_r\beta_{10}\eta_{10}^2\in      \pi_{12}(F_6)=\Z_{2}\{\beta_5\sigma'''\}\oplus \Z_{2}\{\beta_{10}\eta_{10}^2\}.
\end{align*}
Consider the commutative diagram (\ref{diam: H2H2' P6}) for $m=12, r\geq 1$.
We get 
$$H_{2}'\partial^{6,r}_{12\ast}(\sigma'')=H_{2}'( x_r\beta_5\sigma''' +y_r\beta_{10}\eta_{10}^2)=y_r\eta_{11}^2=0.$$
This implies that $y_r=0$.

For $r=1$,
\begin{align*}
	\chi_{0\ast}^1\partial_{12\ast}^{6,1}(\sigma'')=\partial_{12\ast}^{6,0}(2\iota_6)_{\ast}(\sigma'')=\partial_{12\ast}^{6,0}(2\sigma'')=2\partial_{12\ast}^{6,0}(\sigma'')\neq 0
\end{align*}
where the first equation is from  the homotopy commutative diagrams (\ref{diam Moore s to t}) for $s=1,r=0,k=5$; the second equation holds since $(2\iota_6)\sigma''=2\sigma''\pm [\iota_6,\iota_6]H_2(\sigma'')=2\sigma''$ ($[\iota_6,\iota_6]H_2(\sigma'')=0$ because  $\Sigma([\iota_6,\iota_6]H_2(\sigma''))=0$ and $\pi_{13}(S^6)\xrightarrow{\Sigma} \pi_{14}(S^7)$ is injective); the last  inequality holds since $ \pi_{13}(S^6)\xrightarrow{\partial_{12\ast}^{6,0}}\pi_{12}(F_{6}^0)\cong \Z_4$ is isomorphic. This implies $x_1=0$, i.e., 
$$\partial^{6,1}_{12\ast}(\sigma'')=\beta^1_5\sigma'''.$$
For $r\geq 2$, 
\begin{align*}
	\partial_{12\ast}^{6,r}(\sigma'')\!=\!\chi_{r\ast}^1	\partial_{12\ast}^{6,1}(\sigma'')\!=\!	\chi_{r}^1\beta^1_5\sigma'''\!=\!(2^{r\!-\!1}\beta^r_5)\sigma'''\!=\!\beta^r_5(2^{r\!-\!1}\iota_5)\sigma'''\!=\!\beta^r_5(2^{r\!-\!1}\sigma''')\!=\!0
\end{align*} 
where the first equation is induced by (\ref{diam Moore s to t}); the third equation is from the first equation of (\ref{Equ:gamma30,3r}); the last equation holds since 
$(2^{r-1}\iota_5)\sigma'''=2^{r-1}\sigma'''\pm \binom{2^{r-1}}{2}[\iota_5,\iota_5]H_2(\sigma''')=2^{r-1}\sigma'''\pm \binom{2^{r-1}}{2}(\nu_5\eta_8)(4\nu_4)=2^{r-1}\sigma'''$ by \cite[Proposition 2.10]{N.Oda} and  \cite[Lemma 5.13]{Toda}.

Now we have the following short exact sequence 
\begin{align*}
	0\rightarrow \vartheta_r\Z_2\oplus \Z_2\rightarrow \pi_{12}(P^6(2^r))\xrightarrow{p_{6\ast}} \Z_2\{\nu_6^2\}\rightarrow 0.
\end{align*}
since short exact sequence (\ref{exact:short pi11(P5)}) splits, so is the above sequence by Lemma \ref{lem: split Suspen Cf}. Hence 
\begin{align*}
	\pi_{12}(P^6(2^r))\cong  \vartheta_r\Z_2\oplus \Z_2^{\oplus 2}.
\end{align*}

$\bullet~~\pi_{13}(P^{6}(2^r))$. 

We have the following exact sequence 
\begin{align*}
	&\Z_{8}\{\bar\nu_6\}\oplus \Z_{2}\{\varepsilon_6\} =\pi_{14}(S^6)\xrightarrow{\partial^{6,r}_{13\ast}}\pi_{13}(F_6)\xrightarrow{\tau_{6\ast}} \pi_{13}(P^{6}(2^r))\xrightarrow{p_{6\ast}} Ker(\partial^{6,r}_{12\ast})\rightarrow 0.\\
	&  Ker(\partial^{6,1}_{12\ast})=\Z_2\{2\sigma''\};  ~~ Ker(\partial^{6,r}_{12\ast})=\Z_4\{\sigma''\} r\geq 2.\\
	&\pi_{13}(F_6)=\Z_2\{\beta_5\varepsilon_5\}\oplus \Z_8\{\beta_{10}\nu_{10}\};~~
	\partial^{6,r}_{13\ast}(\varepsilon_6)=\beta_{5\ast}(2^r\iota_5)_{\ast}(\varepsilon_5)=0.\\
	&\text{Assume}~\partial^{6,r}_{13\ast}(\bar\nu_6)=x_r\beta_5\varepsilon_5+y_r\beta_{10}\nu_{10}, x_r\in \Z_2, y_r\in \Z_8. 
\end{align*}
Consider the commutative diagram (\ref{diam: H2H2' P6}) for $m=13, r\geq 1$. By $H_2(\bar\nu_6)=l\nu_{11}, l$ is odd \cite[Lemma 6.2]{Toda},
we get 
$$H_{2}'\partial^{6,r}_{13\ast}(\bar\nu_6)=H_{2}'( x_r\beta_5\varepsilon_5+y_r\beta_{10}\nu_{10})=y_r\nu_{11}=2^rl\nu_{11}.$$
This implies that $y_r=2^rl$.

From left commutative square of  Corollary \ref{cor: suspen Fp Moore} for  $m=n=5, u=13$, 
\begin{align*}
	&	\xymatrix{
		\Z_{8}\{\bar\nu_6\}\oplus \Z_{2}\{\varepsilon_6\}=	\pi_{14}(S^{6}) \ar[r]^-{\partial^{6,r}_{13\ast}}\ar[d]_{\Sigma^{\infty} } & \pi_{13}(F_5) \ar[d]^{E^{\infty}}\\
		\Z_{2}\{\bar\nu\}\oplus \Z_{2}\{\varepsilon\}=	\pi^s_{13}(S^{5})\ar[r]^-{(2^r\iota)_{\ast}=0} & 	\pi^s_{13}(S^{5})} ~~\\
	\Rightarrow~& 0=E^{\infty}\partial^{6,r}_{13\ast}(\bar\nu_6)=E^{\infty}(x_r\beta_5\varepsilon_5+2^rl\beta_{10}\nu_{10})=x_rE^{\infty}(\beta_5\varepsilon_5)+2^rlE^{\infty}(\beta_{10}\nu_{10})\\
	&=  x_r\varepsilon, ~\text{by right commutative triangle of Corollary \ref{cor: suspen Fp Moore}}.
\end{align*}
So $x_r=0$ $\Rightarrow$ $Coker (\partial^{6,r}_{13\ast})=\Z_2\{\beta_5\varepsilon_5\}\oplus \Z_{2^{m_r^3}}\{\beta_{10}\nu_{10}\}$.

From (\ref{diam: H2 Moore}) we have the following commutative diagram
\begin{align}
	\small{\xymatrix{
			\Z_2\{\beta_5\varepsilon_5\}\oplus \Z_{2^{m_r^3}}\{\beta_{10}\nu_{10}\}=Coker \partial^{6,r}_{13\ast}\ar[d]^{\bar H_{2\ast}}\ar[r]^-{\tau_{6\ast}}&\pi_{13}(P^{6}(2^r)) \ar[d]^{\bar H_{2\ast}}\\
			Coker(\partial^{11,r}_{13\ast})\cong  \Z_{2^{m_r^3}}\{\nu_{10}\}\ar@{^{(}->}[r]^-{i_{10\ast}}& \pi_{13}(P^{11}(2^r))
	} }. \label{diam: H2 P13(P6)}
\end{align}
\begin{lemma}\label{lem: H2(beta10nu10)}
	In the diagram $(\ref {diam: H2 P13(P6)})$, $\bar H_{2}(\beta_{10}\nu_{10})=t\nu_{10}, t$ is odd.
\end{lemma}
\begin{proof}
	By the same proof as that of Lemma \ref{lem:not divis 2 P4}, we get
	\begin{align}
		\bar h_{w}(\Omega_0(\tau_{6}\beta_{10}))=[u,v]\in  H_{9}(\Omega P^{6}(2^r), \Z_2 )   \label{equ:hwOmega0(tau6beta10)}
	\end{align}
	where $\bar h_{w}$ is the mod $2$-reduction of Hurewicz  homomorphism, i.e., the composition  $\pi_{9}(\Omega P^{6}(2^r))\xrightarrow{h_w} H_{9}(\Omega P^{6}(2^r))\xrightarrow{reduce} H_{9}(\Omega P^{6}(2^r); \Z_2)\subset T(u,v)$, where $T(u,v)$ the tensor algebra over $\Z_2$ generated by $u,v$ with degrees $4$ and $5$ respectively.
	\newline
	$\bar{H}_{2\ast}: H_{\ast}(\Omega P^{6}(2^r),\Z_2)\xrightarrow{H_{2\ast}}H_{\ast}(\Omega \Sigma P^{5}(2^r)\wedge P^{5}(2^r),\Z_2)\xrightarrow{(\Omega\Sigma q)_{\ast}}H_{\ast}(\Omega P^{11}(2^r),\Z_2)$.
	where  $H_{\ast}(\Omega P^{11}(2^r),\Z_2)=T(u',v')$ and the degrees of $u'$ and $v'$ are 9 and 10 respectively; $q$ is the pinch map $P^{5}(2^r)\wedge P^{5}(2^r)\rightarrow P^{5}(2^r)\wedge P^{5}(2^r)/(S^4\wedge P^{5}(2^r))=P^{10}(2^r)$.
	By \cite[Proposition 5.3]{CohenTalor} we can get $\bar{H}_{2\ast}([u,v])=u'$ and $\bar{H}_{2\ast}(v\otimes v)=v'$.  
	Hence  $\bar H_{2\ast}\bar h_{w}(\Omega_0(\tau_{6}\beta_{10}))=u'$. This implies the following composition is homotopy equivalent to $t\Omega_{0}(i_{10}):S^9\rightarrow \Omega P^{10}(2^r)$ with odd $t$
	\begin{align*}
		&	\xymatrix{
			S^9\ar@/_0.5pc/[rrr]_{\Omega_0(\tau_{6}\beta_{10})}	\ar@{^{(}->}[r]^-{\Omega\Sigma}& \Omega S^{10}\ar[r]^-{\Omega\beta_{10}}&\Omega F_6\ar[r]^-{\Omega\tau_{6}}
			&	\Omega P^{6}(2^r)\ar[r]^-{\bar H_2}	&\Omega P^{10}(2^r)}. \\
		& i_{10\ast}\bar H_{2\ast}(\beta_{10}\nu_{10})=\bar H_{2\ast} (\tau_6\beta_{10}\nu_{10})~~\text{(by (\ref{diam: H2H2' P6}))}\\
		&=\Omega_{0}^{-1}(\bar H_2\Omega_{0}(\tau_6\beta_{10}\nu_{10}) ) ~~\text{(by definition of right  map $\bar H_{2\ast}$ in (\ref{diam: H2H2' P6}))}\\
		&=\Omega_{0}^{-1}(\bar H_2\Omega(\tau_6\beta_{10}\nu_{10})E_{S^{12}} )=\Omega_{0}^{-1}(\bar H_2\Omega(\tau_6\beta_{10})(\Omega\nu_{10})E_{S^{12}} )\\
		&=\Omega_{0}^{-1}(\bar H_2\Omega(\tau_6\beta_{10})E_{S^9}\nu_{9})=\Omega_{0}^{-1}(\bar H_2\Omega_0(\tau_6\beta_{10})\nu_{9} )\\
		&=\Omega_{0}^{-1}(t\Omega_{0}(i_{10})\nu_{9})=\Omega_{0}^{-1}(t\Omega(i_{10})E_{S^9}\nu_9)=\Omega_{0}^{-1}(t\Omega(i_{10})\Omega(\nu_{10})E_{S^{12}})\\
		&=\Omega_{0}^{-1}(t\Omega(i_{10}\nu_{10})E_{S^{12}})=\Omega_{0}^{-1}(t\Omega_{0}(i_{10}\nu_{10}))=i_{10\ast}(t\nu_{10}).\\
		&\Rightarrow ~~\bar H_{2\ast}(\beta_{10}\nu_{10})=t\nu_{10}, \text{~since $i_{10\ast}$ is an injection}.
	\end{align*}
\end{proof}
From above lemma, the composition  $\Z_{2^{m_r^3}}\{\beta_{10}\nu_{10}\}\stackrel{j_2}\hookrightarrow Coker \partial^{6,r}_{13\ast}\xrightarrow{\bar H_{2\ast}} \Z_{2^{m_r^3}}\{\nu_{10}\} $ is isomorphic.
Since the bottom injection $i_{10\ast}:\Z_{2^{m_r^3}}\{\nu_{10}\}\rightarrow \pi_{13}(P^{11}(2^r))$ of (\ref{diam: H2H2' P6}) is split into.
So is the map $\Z_{2^{m_r^3}}\{\beta_{10}\nu_{10}\}\stackrel{j_2}\hookrightarrow Coker \partial^{6,r}_{13\ast}\xrightarrow{\tau_{6\ast}}\pi_{13}(P^{6}(2^r))$. 

For $r=1$, $Ker(\partial^{6,1}_{12\ast})=\Z_2\{2\sigma''\}$. By $2\sigma''=\Sigma \sigma'''$ \cite[Lemma 5.14]{Toda},
there is  $-\mathbf{x}_1\in \{i_5,2\iota_5, \sigma'''\}$, such that $ p_{6\ast}(\mathbf{x}_1)=2\sigma''$.  $2\mathbf{x_1}\in - \{i_5,2\iota_5, \sigma'''\}(2\iota_{13})=i_5\{2\iota_5, \sigma''',2\iota_{12}\}=\{0\}$ by Lemma \ref{lemA: Todabraket } (v). 
Hence $2\mathbf{x_1}=0$  implies that  $\pi_{13}(P^{6}(2))\stackrel{p_{6\ast}
}\twoheadrightarrow \Z_2\{2\sigma''\}$ is split onto. We get $$\pi_{13}(P^{6}(2))\cong \Z_2^{\oplus 3}.$$

By the diagram (\ref{diam Moore s to t}) for $s=r\geq 2,t=1, k=5$,  we have the commutative diagram
\begin{align*}
	\small{\xymatrix{
			\Z_2\{\beta_5^r\varepsilon_5\}\ar@{^{(}->}[r]^{j_1}\ar@{=}[d]& 	\pi_{13}(	F_{6}^r)\ar[d]^{\chi_{1\ast}^{r}}\ar[r]^-{\tau_{6\ast}^r} &\pi_{13}(P^{6}(2^r)) \ar[d]_{\bar \chi^r_{1\ast}}\\
			\Z_2\{\beta_5^1\varepsilon_5\}\ar@{^{(}->}[r]^{j_1}& \pi_{13}(F_{6}^1)\ar[r]^-{\tau_{6\ast}^1}&  \pi_{13}(P^{6}(2))
	} }.
\end{align*}
The composition of  the bottom maps $\tau_{6\ast}^1j_1$ is split into, so is the composition of the top maps $\tau_{6\ast}^rj_1$. Hence 
\begin{align*}
	\pi_{13}(P^6(2^r))\cong \Z_2\oplus \Z_{2^{m_r^3}}\oplus \Z_4 (r\geq 2). 
\end{align*}

$\bullet~~\pi_{14}(P^{6}(2^r))$. 

Commutative diagram (\ref{exact:hgps piPk+1(2r)}) for $k=5, m=14, r\geq 1$ gives the following short exact sequence 
\begin{align*}
	&0\rightarrow \pi_{14}(J_{3,5}^r)\cong \pi_{14}(F_6) \xrightarrow{\tau_{6\ast}} \pi_{14}(P^{6}(2^r))\xrightarrow{p_{6\ast}}    \Z_{2^{m_r^3}}\{\delta_r\bar\nu_6\}\oplus \Z_{2}\{\varepsilon_6\}\rightarrow 0.
\end{align*}

Consider the commutative diagram of exact sequences
\begin{align*}
	\small{\xymatrix{
			\pi_{15}(S^{6}) \ar[r]^-{\partial^{6,r}_{14\ast}}\ar[d]_{\Sigma^{\infty} }& 	\pi_{14}(	F_{6}^r)\ar[d]^{E^{\infty}}\ar[r]^-{\tau_{6\ast}} &\pi_{14}(P^{6}(2^r))\ar[r]^-{p_{6\ast}} \ar[d]_{\Sigma^{\infty}} &   \Z_{2^{m_r^3}}\{\delta_r\bar\nu_6\}\oplus\Z_2\{\varepsilon_6\}\ar[d]^{\Sigma^{\infty}}\ar[r]&0 \\
			\pi^s_{14}(S^{5})\ar[r]^-{(2^r\iota)_{\ast}=0} & 	\pi^s_{14}(S^{5})\ar[r]^-{i_{\ast}}&  \pi_{14}^s(P^{6}(2^r))\ar[r]^-{p_{\ast}}&\Z_{2}\{\bar\nu\}\oplus\Z_2\{\varepsilon\}\ar[r]&0 \\
	} }.
\end{align*}
There is an $\alpha_r\in \pi_{14}(P^{6}(2^r))$, such that $p_{6\ast}(\alpha_r)=\delta_r\bar\nu_6$. 

For $r= 1$, $p_{\ast}\Sigma^{\infty}(\alpha_1)=\Sigma^{\infty}p_{6\ast}(\alpha_1)=\Sigma^{\infty}(4\bar\nu_6)=4\Sigma^{\infty}(\bar\nu_6)=4\bar\nu=0$.
Hence $\Sigma^{\infty}(\alpha_1)\in i_{\ast}(\pi^s_{14}(S^{5}))$, which implies that  $\Sigma^{\infty}(2\alpha_1)=0$.  For $r\geq 2$,  $\Sigma^{\infty}(2^{m_r^3}\alpha_r)=2^{m_r^3}\Sigma^{\infty}(\alpha_r)=0$ since  $4\pi_{14}^s(P^{6}(2^r))=0$.

From  Corollary \ref{cor: suspen Fp Moore},  we get 
\begin{align*}
	2^{m_r^3}\alpha_r=b_r\tau_6 \check I_{2}[\check\beta_5, \check\beta_{10}]=b_r[i_5, \tau_6\beta_{10}], b_r\in \Z_{2^r}.
\end{align*}
By the same proof as that of Lemma \ref{lem:not divis 2 P4}, we get $2^s[i_5, \tau_6\beta_{10}]$ is not divisible by $2^{s+1}$ in $\pi_{14}(P^{6}(2^r))$ for any $0\leq s\leq r-1$.  Hence $b_r=0\in \Z_{2^r}$. This implies that the order of  $\alpha_r$ is $2^{m_r^3}$. 

There is also an element $-\mathbf{x}_r^{\varepsilon}\in \{i_5, 2^r\iota_5, \varepsilon_5\}\subset \pi_{14}(P^{6}(2^r))$, such that $p_{6\ast}(\mathbf{x}_r^{\varepsilon})=\varepsilon_6$. $2\mathbf{x}_r^{\varepsilon}\in i_5\fhe\{2^r\iota_5, \varepsilon_5, 2\iota_{13}\}$. By Lemma \ref{lemA: Todabraket } (vi), 
\begin{align*}
	&\text{for}~r=1,~~~ 2\mathbf{x}_1^{\varepsilon}=i_{5}\eta_6\varepsilon_5~~\Rightarrow~~	\pi_{14}(P^{6}(2))\cong \Z_{2}^{\oplus 4}\oplus \Z_4;\\
	&\text{for}~ r\geq 2,~~~ 2\mathbf{x}_r^{\varepsilon}=0~~\Rightarrow~~	\pi_{14}(P^{6}(2^r))\cong \Z_{2}^{\oplus 4}\oplus \Z_{2^r}\oplus \Z_{2^{m_r^3}}. 
\end{align*}

$\bullet~~\pi_{15}(P^{6}(2^r))$. 

Consider diagram (\ref{exact:hgps piPk+1(2r)}) for $k=5, m=15, r\geq 1$. 
$\partial_{15\ast}^{6,r}(\nu_6\sigma_9)=\beta_{5}(2^r\iota_5)(\nu_5\sigma_8)=\beta_{5}\nu_5(2^r\iota_8)\sigma_8=\beta_{5}\nu_5(2^{2r}\sigma_8-2^{r-1}(2^r-1)\Sigma\sigma)=2^{r}\beta_{5}\nu_5\sigma_8$. $\partial_{15\ast}^{6,r}(\eta_6\mu_7)=\beta_{5}(2^r\iota_6)(\eta_5\mu_6)=0$. There is the short exact sequence 
\begin{align*}
	&	0\rightarrow Coker(\partial_{15\ast}^{6,r})\xrightarrow{\tau_{6\ast}}\pi_{15}(P^{6}(2^r))\xrightarrow{p_{6\ast}}\pi_{15}(S^6)=\Z_2\{\nu_6^3\}\oplus \Z_2\{\mu_6\}\oplus  \Z_2\{\eta_6\varepsilon_7\}\rightarrow 0,\\
	&Coker(\partial_{15\ast}^{6,r})=\Z_{2^{m_r^3}}\{\beta_{5}\nu_5\sigma_{8}\}\oplus \Z_{2}\{\beta_{5}\eta_5\mu_6\}\oplus \Z_{2}\{[\beta_{5}, \beta_{10}]\eta_{14}\}.
\end{align*}
Let $x_1=\nu_5^3$; $x_2=\mu_5$; $x_3=\eta_5\varepsilon_6$. For $i=1,2,3$, there is an $-\mathbf{x}_r^{i}\in \{i_5, 2^r\iota_5, x_i\}$ such that $p_{6\ast}(-\mathbf{x}_r^{i})=\Sigma x_i$ and $2\mathbf{x}_r^{i}\in  i_5\fhe\{ 2^r\iota_5, x_i, 2\iota_{14}\}$. By Lemma \ref{lem:Toda bracket},
\begin{align*}
	&\{ 2^r\iota_5, x_i, 2\iota_{14}\}\ni 2^{r-1}x_i\eta_{14}~\text{mod}~(2^r\iota_5)\fhe \pi_{15}(S^5)+ \pi_{15}(S^5)\fhe(2\iota_{15})=\lr{2\nu_5\sigma_{8}}.\\
	&x_1\eta_{14}=\nu_5^2(\nu_{11}\eta_{14})=0, ~\Rightarrow 2\mathbf{x}_r^{1}=2a_1i_5\nu_5\sigma_{8},\\
	&x_2\eta_{14}=\mu_5\eta_{14}=\eta_5\mu_6 ~\text{\cite[page 66]{Toda}},  ~\Rightarrow 2\mathbf{x}_r^{2}= 2^{r-1}i_5\eta_5\mu_6+2a_2i_5\nu_5\sigma_{8},\\
	&x_3\eta_{14}=\eta_5\varepsilon_6\eta_{14}=\eta_5^2\varepsilon_7=2\Sigma^2\varepsilon'=4\nu_5\sigma_8 ~\text{\cite[page 67, (7.5),(7.10)]{Toda}}, \\
	& \qquad\qquad \Rightarrow ~2\mathbf{x}_r^{3}=2a_3i_5\nu_5\sigma_{8}, ~\text{where}~ a_1, a_2, a_3\in \Z.
\end{align*}
Let $\mathbf{x}_r^{'i}=\mathbf{x}_r^i-a_ii_5\nu_5\sigma_8 (i=1,2,3)$.  Then $p_{6\ast}(\mathbf{x}_r^{'i})=\Sigma x_i$ with $o(\mathbf{x}_r^{'i})=2$ for $(r,i)\neq (1,2)$ and $o(\mathbf{x}_2^{'1})=4$. Thus 
\begin{align*}
	\pi_{15}(P^{6}(2^r))\cong \left\{
	\begin{array}{ll}
		\Z_2^{\oplus 4}\oplus \Z_4, & \hbox{$r=1$;} \\
		\Z_{2}^{\oplus 5}\oplus \Z_{2^{m_r^3}}, & \hbox{$r\geq 2$.}
	\end{array}
	\right.
\end{align*}

$\bullet~~\pi_{16}(P^{6}(2^r))$.

Consider diagram (\ref{exact:hgps piPk+1(2r)}) for $k=5, m=16, r\geq 1$. By \cite[Proposition 2.10]{N.Oda} and  $H_2(\zeta_5)=8\sigma_9$ \cite[Lemma 6.7]{Toda}
\begin{align*}
	&\partial_{16\ast}^{6,r}(\zeta_6)=\beta_{5}(2^r\iota_5)(\zeta_5)=\beta_{5}(2^r\zeta_5\pm \binom{2^{r}}{2}[\iota_5,\iota_5]H_2(\zeta_5))=2^{r}\beta_{5}\zeta_5;\\
	&~~\partial_{16\ast}^{6,r}(\bar\nu_6\nu_{14})=\partial_{13\ast}^{6,r}(\bar\nu_6)\nu_{13}=2^rl\nu_{11}\nu_{13}=0~ (l~\text{is odd}).\\
	&\Rightarrow~	0\rightarrow Coker(\partial_{16\ast}^{6,r})\xrightarrow{\tau_{6\ast}}\pi_{16}(P^{6}(2^r))\xrightarrow{p_{6\ast}}\Z_{2^{m_{r}^3}}\{\delta_r\nu_6\sigma_9\}\oplus  \Z_2\{\eta_6\mu_7\}\rightarrow 0, \nonumber \\
	&Coker(\partial_{16\ast}^{6,r})=\Z_{2^{m_r^3}}\{\beta_{5}\zeta_5\}\oplus \Z_{2}\{\beta_{5}\nu_5\bar\nu_8\}\oplus \Z_{2}\{\beta_{5}\nu_5\varepsilon_8\}\oplus\Z_{2}\{\beta_{10}\nu_{10}^2\}\oplus A_r \nonumber\\
	&\text{where} \qquad A_1=\Z_4\{\check I_3^1\widearc{\eta_{15}}\}, \qquad A_r=\Z_{2}\{[\beta_{5}, \beta_{10}]\eta_{14}^2\}\oplus\Z_2\{\check I_3^r\widearc{\eta_{15}}\}, r\geq 2. \nonumber
\end{align*}
By Lemma 3.2 of \cite{JZhtpygps}, $\pi_{9}(P^{6}(2^r))\stackrel{p_{6\ast}}\twoheadrightarrow \Z_{2^{m_r^3}}\{\delta_r\nu_6\}$ is split onto, i.e. there is an order $2^{m_r^3}$ element $\widetilde{\delta_r\nu_6}\in \pi_{9}(P^{6}(2^r))$ such that 
$p_{6\ast}(\widetilde{\delta_r\nu_6})=\delta_r\nu_6$.  Then $\widetilde{\delta_r\nu_6}\sigma_9$ is the  order $2^{m_r^3}$ element such that 
$p_{6\ast}(\widetilde{\delta_r\nu_6}\sigma_9)=\delta_r\nu_6\sigma_9$. There is an element $\widetilde{\eta_6}\in \pi_{7}(P^{6}(2^r))$ such that $ p_{6\ast}(\widetilde{\eta_6})=\eta_6$. Then the order 2 element  $\widetilde{\eta_6}\mu_7$ satisfies $p_{6\ast}(\widetilde{\eta_6}\mu_7)=\eta_6\mu_7$. So $\pi_{16}(P^{6}(2^r))\stackrel{p_{6\ast}}\twoheadrightarrow\Z_{2^{m_{r}^3}}\{\delta_r\nu_6\sigma_9\}\oplus  \Z_2\{\eta_6\mu_7\}$ is split onto. Hence 
\begin{align*}
	\pi_{16}(P^{6}(2^r))\cong  \left\{
	\begin{array}{ll}
		\Z_2^{\oplus 6}\oplus \Z_4, & \hbox{$r=1$;} \\
		\Z_{2}^{\oplus 6}\oplus \Z_{2^{m_r^3}}^{\oplus 2}, & \hbox{$r\geq 2$.}
	\end{array}
	\right.
\end{align*}

$\bullet~~\pi_{17}(P^{6}(2^r))$. 

By \cite[Proposition 2.5]{Toda}, $\Delta(\sigma_{13})=l\Delta(\iota_{13})\sigma_{11}$, $l$ is odd. Hence by the previous calculation of $\partial_{10\ast}^{6,r}$ and Lemma \ref{lemA: Sigma nu'sigam'}

$\partial_{17\ast}^{6,r}(\Delta(\sigma_{13}))=l\partial_{10\ast}^{6,r}(\Delta(\iota_{13}))\sigma_{10}= \left\{
\begin{array}{ll}
	\beta_5\nu_5^4+\beta_5\nu_5\eta_8\varepsilon_9\pm 4\beta_{10}\sigma_{10}, & \hbox{$r=1$;} \\
	\pm 2^{r+1}\beta_{10}\sigma_{10}, & \hbox{$r\geq 2$.}
\end{array}
\right.$
There is the short exact sequence 
\begin{align*}
	&	0\rightarrow Coker(\partial_{17\ast}^{6,r})\xrightarrow{\tau_{6\ast}}\pi_{17}(P^{6}(2^r))\xrightarrow{p_{6\ast}}\Z_{2^{m_{r}^3}}\{\delta_r\zeta_6\}\oplus  \Z_4\{\bar\nu_6\nu_{14}\}\rightarrow 0,\\
	&Coker(\partial_{17\ast}^{6,1})\!=\! \Z_{2}\{\beta_{5}\nu_5\bar\mu_8\} \!\oplus\!\Z_{2}\{\beta_{5}\nu_5\eta_8\varepsilon_9\} \!\oplus\! \Z_{8}\{\beta_{10}\sigma_{10}\}  \!\oplus\! \Z_{2}\{[\beta_{5}, \beta_{10}]\nu_{14}\}\!\oplus\! \Z_{2}\{\check I_3^1\widearc{\eta_{15}}\eta_{16}\}\\
	&Coker(\partial_{17\ast}^{6,r})=\Z_{2}\{\beta_{5}\nu^4_5\}\oplus \Z_{2}\{\beta_{5}\nu_5\bar\mu_8\}\oplus \Z_{2}\{\beta_{5}\nu_5\eta_8\varepsilon_9\}\oplus \Z_{2^{m_{r+1}^4}}\{\beta_{10}\sigma_{10}\} \\
	&\qquad\qquad \qquad \oplus \Z_{2^{m_r^3}}\{[\beta_{5}, \beta_{10}]\nu_{14}\}\oplus \Z_{2}\{\check I_3^r\widearc{\eta_{15}}\eta_{16}\}, r\geq 2.
\end{align*}

There is  $-\mathbf{x}_r\in\{i_5, 2^r\iota_5, \delta_r\zeta_5\}$
such that $2^{m_{r}^3}\mathbf{x}_r\in i_5\fhe\{ 2^r\iota_5, \delta_r\zeta_5,2^{m_{r}^3}\iota_{16} \}$ and \newline
$\{ 2^r\iota_5, \delta_r\zeta_5,2^{m_{r}^3}\iota_{16} \}$ is a coset of $\{0\}$. 
For $r=1$, by $4\zeta_5=\eta_5^2\mu_7=\Sigma(\eta_4^2\mu_6) $  \cite[(7.14)]{Toda},  $\{ 2\iota_5, 4\zeta_5,2\iota_{16}\}\ni \eta_5^2\mu_7\eta_{16}=\eta_5^3\mu_8=4\nu_5\mu_8=0$ (Lemma \ref{lemA: Sigma nu'sigam'}). For $r\geq 2$,  $\{ 2^r\iota_5, \delta_r\zeta_5,2^{m_{r}^3}\iota_{16} \}\subset\{ 2\iota_5, (2^{r-1}\iota_5)\delta_r\zeta_5(2^{m_{r}^3-1}),2\iota_{16} \}=
\{ 2\iota_5, 0,2\iota_{16} \}=\{0\}$. Thus $\{ 2^r\iota_5, \delta_r\zeta_5,2^{m_{r}^3}\iota_{16} \}=\{0\}$, $\Rightarrow$ $2^{m_{r}^3}\mathbf{x}_r=0$, which implies that $\pi_{17}(P^{6}(2^r))\stackrel{q_1p_{6\ast}}\twoheadrightarrow \Z_{2^{m_{r}^3}}\{\delta_r\zeta_6\}$ is split onto.  

There is $-\mathbf{y}_r\in \{i_5, 2^r\iota_5, \nu_5\varepsilon_8\}$ such that $p_{6\ast}(\mathbf{y}_r)=\nu_6\varepsilon_9=2\bar\nu_6\nu_{14}$ \cite[(7.18)]{Toda} and $2\mathbf{y}_r\in i_5\fhe\{2^r\iota_5, \nu_5\varepsilon_8, 2\iota_{16}\}$. By Lemma \ref{lem:Toda bracket}
\begin{align*}
\{2^r\iota_5, \nu_5\varepsilon_8, 2\iota_{16}\}\ni 2^{r-1}\nu_5\varepsilon_8\eta_9~\text{mod}~\{0\}.
\end{align*}
 Hence 
$2\mathbf{y}_r=(1-\vartheta_r)i_5\nu_5\varepsilon_8\eta_9$.

\textbf{Claim 1:~} There is an element $\mathbf{z}_r\in \pi_{17}(P^{6}(2^r)) $  with $o(\mathbf{z}_r)\leq 8$ such that $p_{6\ast}(\mathbf{z}_r)=\bar\nu_6\nu_{14}$.
\begin{proof}[Proof of the Claim 1]
	Consider the cofibration sequence $S^{k+3}\xrightarrow{2\nu_{k}} S^{k}\xrightarrow{i^{2\nu}_{k}} C_{2\nu_{k}}\xrightarrow{p^{2\nu}_{k+4}} S^{k+4}$ $(k\geq 5)$. For $k=10$, it induces exact sequence
	\begin{align*}
		& [C_{2\nu_{10}}, P^{6}(2^r)]\xrightarrow{p_{6\ast}} 	 [C_{2\nu_{10}}, S^6]\xrightarrow{\partial_{\ast}^{6,r}}[C_{2\nu_{10}}, F_{6}],\\
		\text{with}~& \partial_{\ast}^{6,r}(\bar\nu_6p^{2\nu}_{14})=\partial_{13\ast}^{6,r}(\bar\nu_6)p^{2\nu}_{13}=2^rl\nu_{10}p^{2\nu}_{13}=0 ~(r\geq 1, l~\text{is odd}).
	\end{align*}
	Thus there is a $f_r\in [C_{2\nu_{10}}, P^{6}(2^r)]$, such that $p_{6}f_r=\bar\nu_6p^{2\nu}_{14}$. This implies the following commutative diagram with the top row which is exact sequence
	\begin{align*}
		\small{\xymatrix{
				\pi_{17}(S^{10})\ar[r]	& 	\pi_{17}(	C_{2\nu_{10}})\ar[d]^{f_{r\ast}}\ar[r]^-{p^{2\nu}_{14\ast}} &\pi_{17}(S^{14})=\Z_{8}\{\nu_{14}\} \ar[d]_{\bar\nu_{6\ast}}\ar[r]& 0\\
				& \pi_{17}( P^{6}(2^r))\ar[r]^-{p_{6\ast}}&  \pi_{17}(S^6)
		} }.
	\end{align*}
	The  following left homotopy commutative diagram induces the  right commutative one
	\begin{align*}
		\small{\xymatrix{
				S^{13}\ar[r]^{2\nu_{10}}\ar[d]^{2\iota_{13}}	& 	S^{10}\ar@{=}[d]\ar[r]^-{i_{10}^{2\nu}} &C_{2\nu_{10}}\ar[d]_{\chi^{2\nu}_{\nu}}\\
				S^{13}\ar[r]^{\nu_{10}}	& S^{10}\ar[r]^-{i_{10}^{\nu}}&  C_{\nu_{10}}
		} }  ;~~	\small{\xymatrix{
				0\ar[r]&\pi_{17}(S^{10})\ar@{=}[d]\ar[r]^-{i_{10\ast}^{2\nu}} &\pi_{17}(C_{2\nu_{10}})\ar[d]_{\chi^{2\nu}_{\nu\ast}}\ar[r]& \Z_{8}\{\nu_{14}\}\ar[r]\ar[d]_{\times 2}&0\\
				0\ar[r]&\pi_{17}(S^{10})\ar[r]^-{i_{10\ast}^{\nu}}& \pi_{17}( C_{\nu_{10}})\ar[r]& \Z_{4}\{2\nu_{14}\}\ar[r]&0. 
		} }  
	\end{align*}
	From Theorem 1 of \cite{JXYang},  $\pi_{17}(C_{\nu_{10}})\cong \Z_{16}\oplus \Z_{4}$ implies that $\pi_{17}(S^{10})\stackrel{i_{10\ast}^{\nu}}\hookrightarrow \pi_{17}(C_{\nu_{10}})$ is split into, so is $\pi_{17}(S^{10})\stackrel{i_{10\ast}^{2\nu}}\hookrightarrow\pi_{17}(C_{2\nu_{10}})$. Hence $\pi_{17}(C_{2\nu_{10}})\stackrel{p^{2\nu}_{14\ast}}\twoheadrightarrow \Z_{8}\{\nu_{14}\} $ is split onto, i.e., 
	there exists an order $8$ element $\mathbf{z}^{2\nu}_r\in \pi_{17}(C_{2\nu_{10}})$, such that $p^{2\nu}_{14\ast}(\mathbf{z}^{2\nu}_r)=\nu_{14}$. 
	Let $\mathbf{z}_r:=f_r\mathbf{z}^{2\nu}_r\in \pi_{17}(P^{6}(2^r))$. We get $o(\mathbf{z}_r)\leq 8$ and $p_{6\ast}(\mathbf{z}_r)=\bar\nu_6\nu_{14}$.  We finish the proof of \textbf{Claim 1}.
\end{proof}
\textbf{Claim 2:~} The monomorphism  $\Z_{2^{m_{r}^3}}\{[\beta_5, \beta_{10}]\nu_{14}\}\stackrel{\tau_{6\ast}}\hookrightarrow 	\pi_{17}(P^6(2^r)) $
is  split into.
\begin{proof}[Proof of the Claim 2]
	From \cite{Wucombina}, $\Omega\Sigma L_{3}(P^5(2^r))\simeq \Omega\Sigma P^{14}(2^r)$ is the Cartesion product factor of $\Omega\Sigma P^5(2^r)$ with the  inclusion 
	$\Omega\Sigma L_{3}(P^5(2^r))\simeq \Omega\Sigma P^{14}(2^r)\stackrel{\mathcal{I}}\hookrightarrow \Omega\Sigma P^5(2^r)$, where  $\mathcal{I}$ satisfies following commutative diagram
	\begin{align*}
		\small{\xymatrix{
				S^4\wedge S^9\simeq  S^{13}\ar[rd]_{i_4\wedge \alpha}\ar[r]^-{i_{13}}	& P^{14}(2^r)\simeq  L_{3}(P^5(2^r))	\ar[r]^-{\Omega\Sigma} &\Omega\Sigma  L_{3}(P^5(2^r)) \ar[r]^-{\mathcal{I}}&\Omega\Sigma  P^5(2^r)\\
				& (P^5(2^r))^{\wedge 3}\ar[u]^-{proj.}\ar[rru]_{[E,[E,E]^S]^S} & 
		} } 
	\end{align*}
	where $[E,[E,E]^S]^S$ ($E=E_{P^{5}(2^r)}$) is the 3-fold Samelson product;  the map $\alpha$ and the left commutative triangle  come from Proposition 2.1 and  Proposition 3.1 of \cite{ChenWu}; the right  commutative triangle  comes from the proof of Theorem 1.6 of \cite{Wucombina}.
	
	By the result of $\pi_{10}(P^{6}(2^r))$, assume the composition $S^9\xrightarrow{\alpha} P^{5}(2^r)\wedge P^{5}(2^r)\xrightarrow{[E,E]^S} \Omega\Sigma  P^5(2^r)$ to be $\Omega_0(ai_5\nu_5\eta_{8}^2+b\tau_6\beta_{10})$. Since by Proposition 2.1 of \cite{ChenWu}
	\begin{align*}
	H_{9}(S^9,\Z_2)=\Z_{2}\{[\iota_9]\}\xrightarrow{([E,E]^S\alpha)_{\ast}} H_{\ast}(\Omega\Sigma  P^5(2^r);\Z_2)=T(u,v),  ([E,E]^S\alpha)_{\ast}([\iota_9])=[u,v]
	\end{align*}
where $[\iota_9]$ is the homology class represented by $\iota_9$,  we have $b$ is odd by (\ref{equ:hwOmega0(tau6beta10)}).

	Note that $\Z_{2^{m_{r}^3}}\{\nu_{14}\}\stackrel{i_{14\ast}}\hookrightarrow  	\pi_{17}(P^{15}(2^r))$ is split into.  Consider the following composition of maps 
	\newline
	$\Z_{2^{m_{r}^3}}\{\nu_{14}\}\stackrel{i_{14\ast}}\hookrightarrow  	\pi_{17}(P^{15}(2^r))\xrightarrow[\cong]{\Omega_0} 	\pi_{16}(\Omega P^{15}(2^r))\stackrel{\mathcal{I}_\ast}\hookrightarrow  \pi_{16}(\Omega P^{6}(2^r))\xrightarrow[\cong]{\Omega^{-1}_0} 	\pi_{17}( P^{6}(2^r))$.
	\begin{align*}
		&\mathcal{I}_\ast\Omega_0i_{14\ast}(\nu_{14})=\mathcal{I} \Omega_0(i_{14}\nu_{14}) =\mathcal{I}\Omega\Sigma(i_{13}\nu_{13})E_{S^{16}}=\mathcal{I}\Omega\Sigma(i_{13})E_{S^{13}}\nu_{13}\\
		=&[E,[E,E]^S]^S (i_4\wedge\alpha )\nu_{13}=[\Omega_0(i_5), \Omega_0(ai_5\nu_5\eta_{8}^2+b\tau_6\beta_{10})]^S\nu_{13}\\
		=&(\Omega_0[i_5, ai_5\nu_5\eta_{8}^2+b\tau_6\beta_{10}]) \nu_{13}=b( \Omega_0[i_5,\tau_6\beta_{10}]) \nu_{13}.
	\end{align*}
	where the last equation holds because $[i_5, ai_5\nu_5\eta_{8}^2]=ai_5[\iota_5,\iota_5]\nu_9\eta_{12}^2=0$. 
	\begin{align*}
		\Omega_0^{-1}\mathcal{I}_\ast\Omega_0i_{14\ast}(\nu_{14})=\Omega_0^{-1}(b( \Omega_0[i_5,\tau_6\beta_{10}]) \nu_{13})=b[i_5,\tau_6\beta_{10}] \nu_{14}.
	\end{align*}
	Since   $\Omega_0^{-1}\mathcal{I}_\ast\Omega_0i_{14\ast}: \Z_{2^{m_{r}^3}}\{\nu_{14}\}\hookrightarrow  	\pi_{17}( P^{6}(2^r))$  is  split into, so is the monomorphism in \textbf{Claim 2}.
\end{proof}

For $\mathbf{z}_r$ in \textbf{Claim 1},  $p_{6\ast}(2\mathbf{z}_r-\mathbf{y}_r)=0$, i.e. $2\mathbf{z}_r-\mathbf{y}_r\in \tau_{6\ast}(Coker(\partial_{17\ast}^{6,r}))$, which implies $4\mathbf{z}_r-\epsilon_ri_5\nu_5\varepsilon_8\eta_9=2a_r\tau_6\beta_{10}\sigma_{10}+2b_r\tau_6[\beta_{5}, \beta_{10}]\nu_{14}$, $a_1\in \Z_8$, $a_r\in \Z_{2^{m_{r+1}^4}}(r\geq 2)$, $b_r\in \Z_{2^{m_r^3}}$. 

For $r=1$, $4\mathbf{z}_1=i_5\nu_5\varepsilon_8\eta_9+2a_1\tau_6\beta_{10}\sigma_{10}$.
Since $o(\mathbf{z}_1)\leq 8$, $2a_1=4a'_1$ for some $a_1'\in \Z$. Let 
$\mathbf{z}'_1= \mathbf{z}'_1-  a'_1\tau_6\beta_{10}\sigma_{10}$, then $o(\mathbf{z}'_1)=8$ and $4\mathbf{z}'_1=i_5\nu_5\varepsilon_8\eta_9$. Hence 
\begin{align*}
	&\pi_{17}(P^6(2))=\Z_{2}\{i_{5}\nu^4_5\}\oplus \Z_{2}\{i_{5}\nu_5\bar\mu_8\}\oplus  \Z_{8}\{\tau_6\beta_{10}\sigma_{10}\}\oplus \Z_{2}\{[i_{5}, \tau_6\beta_{10}]\nu_{14}\}\\
	&\qquad\quad \oplus \Z_{2}\{\tau_6\check I_3^r\widearc{\eta_{15}}\eta_{16}\}\oplus \Z_2\{\widetilde{4\zeta_6}\}\oplus \Z_8\{\mathbf{z}'_1\},~4\mathbf{z}'_1=i_5\nu_5\varepsilon_8\eta_9.\\
	\Rightarrow~~ & \qquad\quad 	\pi_{17}(P^6(2))\cong \Z_{2}^{\oplus 5}\oplus  \Z_8^{\oplus 2}.
\end{align*}

For $r\geq 2$, $2\mathbf{y}_r=0$. $4\mathbf{z}_r=2a_r\tau_6\beta_{10}\sigma_{10}+2b_r\tau_6[\beta_{5}, \beta_{10}]\nu_{14}$.  $o(\mathbf{z}_r)\leq 8$ implies $a_r=2a_r'$, $a_r'\in \Z$. Then $2b_r\tau_6[\beta_{5}, \beta_{10}]\nu_{14}=4(\mathbf{z}_r-a_r'\tau_6\beta_{10}\sigma_{10})$ By \textbf{Claim 2}, $b_r=2b_r'$, $b_r'\in \Z$. Take $\mathbf{z}'_r=\mathbf{z}_r-a_r'\tau_6\beta_{10}\sigma_{10}-b_r'\tau_6[\beta_{5}, \beta_{10}]\nu_{14}$, which is an order $4$ element such that $p_{6\ast}(\mathbf{z}'_r)=\bar\nu_6\nu_{14}$. So
\begin{align*}
	\pi_{17}(P^6(2^r))\cong \Z_{2}^{\oplus 4}\oplus \Z_4\oplus \Z_{2^{m_r^3}}^{\oplus 2}\oplus \Z_{2^{m_{r+1}^4}}.
\end{align*}

$\bullet~~\pi_{18}(P^{6}(2^r))$. 

Consider the commutative diagram (\ref{exact:hgps piPk+1(2r)}) for $k=5, m=18, r\geq 1$.
\begin{align*}
	\partial_{18\ast}^{6,r}(\nu_6\sigma_{9}\nu_{16})=\beta_{5\ast}(2^r\iota_6)_{\ast}(\nu_5\sigma_{8}\nu_{15})=0.
\end{align*}
We have the short exact sequence 
\begin{align}
	&	0\rightarrow Coker(\partial_{18\ast}^{6,r})\xrightarrow{\tau_{6\ast}}\pi_{18}(P^{6}(2^r))\xrightarrow{p_{6\ast}}\Z_{2^{m_{r+1}^4}}\{\delta_{r}\Delta(\sigma_{13})\}\rightarrow 0, \label{exact:short pi18P6(2r)}\\
	&Coker(\partial_{18\ast}^{6,r})=\Z_{2}\{\beta_{5}\nu_5\sigma_8\nu_{15}\}\oplus \Z_{2}\{\beta_{5}\nu_5\eta_8\mu_{9}\}\oplus  \Z_{2}\{\beta_{10}\bar\nu_{10}\}\oplus  \Z_{2}\{\beta_{10}\varepsilon_{10}\} \nonumber\\
	& \quad\qquad \qquad \oplus \Z_{2^r}\{[\beta_{5}, [\beta_{5},\beta_{10}]]\}\oplus \Z_{2^{m_r^3}}\{\check I_3^r\widearc{\delta_r\nu_{15}}\}. \nonumber
\end{align}

For $r=1$,  by Lemma \ref{lem split},  there is $-\mathbf{x}_1\in \{i_5, 2\iota_5, \nu_5\mu_8 \}$ such that $p_{6\ast}(\mathbf{x}_1)=\nu_6\mu_9=8\Delta(\nu_{13})$ \cite[(7.25)]{Toda} and 
$2\mathbf{x}_1\in i_5\fhe\{2\iota_5, \nu_5\mu_8, 2\iota_{17}\}$.  $\{2\iota_5, \nu_5\mu_8, 2\iota_{17}\}\ni \nu_5\mu_8\eta_{17}=\nu_5\eta_8\mu_9$ mod $\{0\}$ by Lemma \ref{lemA: Sigma nu'sigam'} (ii) and \cite[Corollary 3.7]{Toda}.  Hence $2\mathbf{x}_1=i_5\nu_5\eta_8\mu_9\neq 0$, which implies that $o(\mathbf{x}_1)=4$. Assume that the exact sequence $(\ref{exact:short pi18P6(2r)})$ is split,  then $\pi_{18}(P^6(2))\cong  \Z_2^{\oplus 6}\oplus \Z_4$ has no order $4$ element taken to $8\Delta(\sigma_{13})\in \Z_{4}\{4\Delta(\sigma_{13})\}$ by $p_{6\ast}$. This contradicts to existence of $\mathbf{x}_1$. So $(\ref{exact:short pi18P6(2r)})$ is not split for $r=1$. We have 
\begin{align*}
	\pi_{18}(P^6(2))\cong  \Z_2^{\oplus 5}\oplus \Z_8.
\end{align*}

For $r\geq 2$, consider the following exact sequence 
\begin{align*}
	& [P^{11}(2^r), P^{6}(2^r)]\xrightarrow{p_{6\ast}} 	 [P^{11}(2^r), S^6]\xrightarrow{\partial_{\ast}^{6,r}}[P^{11}(2^r), F_{6}],\\
	\text{with}~& \partial_{\ast}^{6,r}(l\Delta(\iota_{13})p_{11})=l\partial_{10\ast}^{6,r}(\Delta(\iota_{13}))p_{10}=2^{r+1}l\beta_{10}p_{10}=l\beta_{10}(2^{r+1}p_{10})=0,
\end{align*}
where the odd integer $l$ is given by $\Delta(\sigma_{13})=l\Delta(\iota_{13})\sigma_{11}$.

Thus there is an $f_r\in [P^{11}(2^r), P^{6}(2^r)]$, such that $p_{6}f_r=l\Delta(\iota_{13})p_{11}$. This implies the following commutative diagram of exact sequences
\begin{align*}
	\small{\xymatrix{
			0\ar[r]&  \Z_2^{\oplus 2}	\cong \pi_{18}(S^{10})\ar[d]\ar[r]	& 	\pi_{18}(	P^{11}(2^r))\ar[d]^{f_{r\ast}}\ar[r]^-{p_{11\ast}} &\Z_{2^{m_r^{4}}}\{\delta_r^{16}\sigma_{11}\} \ar[d]_{(l\Delta(\iota_{13}))_{\ast}}\ar[r]& 0\\
			0\ar[r]&  Coker(\partial_{18\ast}^{6,r})\ar[r]	& \pi_{18}( P^{6}(2^r))\ar[r]^-{p_{6\ast}}& \Z_{2^{m_{r+1}^4}}\{\delta_{r}\Delta(\nu_{13})\}\ar[r]& 0
	} }.
\end{align*}
where $\delta_r^{16}=2\delta_r$ for $r\leq 3$ and $\delta_r^{16}=1$ for $r\geq 4$.

There is $-\mathbf{x}_r\in\{i_{10}, 2^r\iota_{10}, \delta_r^{16}\sigma_{10}\}$ such that $p_{11\ast}(\mathbf{x}_r)=\delta_r^{16}\sigma_{11}$ and $2^{m_r^{4}}\mathbf{x}_r\in i_{10}\{2^r\iota_{10}, \delta_r^{16}\sigma_{10},  2^{m_r^{4}}\iota_{17}\}$. 
$$\{2^r\iota_{10}, \delta_r^{16}\sigma_{10},  2^{m_r^{4}}\iota_{17}\}\subset \{2\iota_{10}, (2^{r-1}\iota_{10})(\delta_r^{16}\sigma_{10})(2^{m_r^{4}-1}\iota_{17}),  2\iota_{17}\}=\{2\iota_{10}, 0, 2\iota_{17}\}=\{0\}.$$
Thus $2^{m_r^{4}}\mathbf{x}_r=0$. We get that  the top short exact sequence in the above diagram is split, so is the bottom one for $2\leq r\leq 4$ by $2^{m_r^4}\cdot  Coker(\partial_{18\ast}^{6,r})=0$ with \cite[Lemma 2.9]{ZP23} and for $r\geq 4$ by the isomorphism of right map $(l\Delta(\iota_{13}))_{\ast}$.
Thus 
\begin{align*}
	\pi_{18}( P^{6}(2^r))\cong \Z_{2}^{4}\oplus \Z_{2^r}\oplus \Z_{2^{m_r^3}}\oplus \Z_{2^{m_{r+1}^4}}, r\geq 2. 
\end{align*}

\section{Homotopy groups of $7$ and $8$  dimensional mod $2^r$ Moore spaces}
\label{sec:htyp Moore78}

In this section we will give the unstable homotopy groups $\pi_{i}(P^{7}(2^r))$ for $i\leq 19$ and $\pi_{i}(P^{8}(2^r))$ for $i\leq 21$.

\subsection{Homotopy groups of $P^{7}(2^r)$}
\label{subsec: P7(2^r)}

There are cofibration and fibration sequences 
\begin{align*}
	&	S^6\xrightarrow{2^{r}\iota_6} S^{6}\xrightarrow{i_{6}} P^{7}(2^r)\xrightarrow{p_{7}}  S^{7};
	~~\Omega S^7\xrightarrow{\partial^{7}}F_{7}\xrightarrow{\tau_7} P^{7}(2^r)\xrightarrow{p_{7}}  S^{7}\\
	& Sk_{17}(F_{7})\simeq J_{2,6}^{r}\simeq S^6\cup_{\gamma^r_2=2^r\Delta(\iota_{13})} CS^{11}(r\geq 1);~~Sk_{23}(F_{7})\simeq J_{3,6}^{r}\simeq J_{2,6}^{r}\cup_{\gamma^r_3}CS^{17}. 
\end{align*}

\begin{lemma}\label{lem pi(J2(2rl6))} For $r\geq 1$ and $i\leq 16$, $\pi_{i}(F_{7})\cong \pi_{i}(J_{2,6}^{r})$ and   
	\begin{align*}
		(1)~&\pi_{11}(J_{2,6}^{r})=\Z_{2^r}\{\beta_6\Delta(\iota_{13})\}.\\
		(2)~&\pi_{12}(J_{2,6}^{r})=	\Z_2\{\beta_6\nu_6^2\}.\\
		(3)~&\pi_{13}(J_{2,6}^{r})=	\Z_4\{\beta_6\sigma''\}\oplus \Z_2\{\widehat{\eta_{12}}\}.\\
		(4)~&\pi_{14}(J_{2,6}^{r})=
		\Z_{2^{m_{r+1}^3}}\{\beta_6\bar\nu_6\}\oplus\Z_2\{\beta_{6}\varepsilon_{6}\}\oplus\Z_2\{\widehat{\eta_{12}}\eta_{13}\}.\\
		(5)~&\pi_{15}(J_{2,6}^{r})=
		\Z_{2}\{\beta_6\nu^3_6\}\!\oplus\!\Z_2\{\beta_{6}\mu_{6}\}\!\oplus\!\Z_2\{\beta_{6}\eta_{6}\varepsilon_7\}\!\oplus\! \Z_{2^{m_{r+1}^3}}\{\widehat{\delta_{r+1}\nu_{12}}\}.\\
		(6)~&\pi_{16}(J_{2,6}^{r})=\Z_8\{\beta_6\nu_6\sigma_9\}\oplus \Z_2\{\beta_6\eta_6\mu_7\}.\\
		(7)~&\pi_{17}(J_{2,6}^{r})=\Z_8\{\beta_6\zeta_6\}\oplus \Z_4\{\beta_6\bar\nu_6\nu_{14}\}\oplus \Z_{(2)}\{\beta_{17}\}\\
		  &\pi_{17}(J_{2,6}^0)=\Z_8\{\beta^0_6\zeta_6\}\oplus \Z_2\{\beta_6^0\bar\nu_6\nu_{14}\}\oplus \Z_{(2)}\{\beta^0_{17}\}.\\
		 (8)~&\pi_{18}(J_{2,6}^{r})=\Z_{2^{m_{r}^4}}\{\beta_6\Delta(\sigma_{13})\}\oplus \Z_{2}\{\beta_{17}\eta_{17}\}\oplus \Z_{2}\{\widehat{\nu_{12}^2}\}.\\
		 (9)~&\pi_{19}(J_{2,6}^{r})=\Z_{2}\{\beta_6\nu_6\sigma_9\nu_{16}\}\oplus \Z_{2}\{\beta_{17}\eta_{17}^2\}\oplus \Z_{2^{m_{r}^4}}\{\widehat{\delta_{r}^{16}\sigma_{12}}\}.
	\end{align*}
\end{lemma}

\begin{proof}
	
	There are fibration sequence $F_{\bar{p}_{12}}\xrightarrow{\bar{\tau}_{12}}J_{2,6}^{r}\xrightarrow{\bar{p}_{12}} S^{12}$, with $Sk_{27}(F_{\bar{p}_{12}})\simeq S^6\vee S^{17}$.
	Consider the Lemma \ref{lem: exact seq pi_m(J2(2rl4))} for $k=6, 11\leq m\leq 21$.
	
   For  $(1)$, $(2)$, $(6)$ and $(7)$ with $r\geq 1$ of Lemma \ref{lem pi(J2(2rl6))}, they  are easily obtained by the following isomorphism
	\begin{align*}
    \pi_{i}(F_{\bar{p}_{12}})\xrightarrow[\cong]{\bar{\tau}_{12\ast}}  \pi_{i}(J_{2,6}^{r}), i=11,12,16,17.
	\end{align*}
	
      For $\pi_{17}(J_{2,6}^0)$ of $(7)$, by \cite[Lemma 6.2]{Toda}, $\Delta(\iota_{13})\nu_{11}=\pm \Delta(\nu_{13})=\pm 2\bar\nu_6$,  it is get by
      \begin{align*}
      	\bar\partial^{12,0}_{17\ast}(\nu_{12}^2)=\bar\beta^0_{6\ast}(\Delta(\iota_{13}))_{\ast}(\nu_{11}^2)=\bar\beta^0_{6}(\pm \Delta(\nu_{13}))\nu_{14}=2\bar\beta^0_{6}\bar\nu_6\nu_{14}.
      \end{align*}

	 For $(3)$ of Lemma \ref{lem pi(J2(2rl6))}, we have the following exact sequence
	 \begin{align*}
	 	0\rightarrow   \Z_4\{\bar\beta_{6}\sigma''\}=\pi_{13}(F_{\bar{p}_{12}})\xrightarrow{\bar{\tau}_{12\ast}}  \pi_{13}(J_{2,6}^{r})\xrightarrow{\bar p_{12\ast}}\pi_{13}(S^{12})=\Z_2\{\eta_{12}\}\rightarrow 0. 
 	 \end{align*}
There is $-\mathbf{x}\in \{\beta_6, \gamma_2, \eta_{11}\}$ such that $\bar p_{12\ast}(\mathbf{x})=\eta_{12}$ and $2\mathbf{x}\in \beta_6\{2^r\Delta(\iota_{13}), \eta_{11}, 2\iota_{12}\}$. By \cite[Page 47]{Toda}, $\Delta(\iota_{13})\eta_{11}=\Delta(\eta_{13})=0$. By Lemma \ref{lem:Toda bracket}
\begin{align*}
	\{2^r\Delta(\iota_{13}), \eta_{11}, 2\iota_{12}\}\ni  (2^{r-1}\Delta(\iota_{13}))\eta_{11}^2=0~\text{mod}~\{0\}. 
\end{align*}
Hence $2\mathbf{x}=0$, which implies $\pi_{13}(J_{2,6}^{r})\stackrel{\bar p_{12\ast}}\twoheadrightarrow \Z_2\{\eta_{12}\}$ is split onto. 

 For $(4)$ of Lemma \ref{lem pi(J2(2rl6))}, since $\gamma_{2\ast}(\nu_{11})=2^r\Delta(\iota_{13})\nu_{11}=2^r\Delta(\nu_{13})=2^{r+1}\bar\nu_6$ by \cite[Lemma 6.2]{Toda},
  we have 
\begin{align*}
	0\rightarrow   \Z_{2^{m_{r+1}^3}}\{\bar\beta_6\bar\nu_6\}\oplus \Z_2\{\beta_6\varepsilon_6\}=Coker(\bar\partial_{14\ast}^{12,r})\xrightarrow{\bar{\tau}_{12\ast}}  \pi_{14}(J_{2,6}^{r})\xrightarrow{\bar p_{12\ast}}\Z_2\{\eta^2_{12}\}\rightarrow 0. 
\end{align*}
Since the order $2$ element $\widehat{\eta_{12}}\eta_{13}\in  \pi_{14}(J_{2,6}^{r})$ is the lift of $\eta_{12}^2$ by $\bar{p}_{12}$, where $\widehat{\eta_{12}}\in  \pi_{13}(J_{2,6}^{r})$ is the lift of $\eta_{12}$ by $\bar{p}_{12}$, we get  $\pi_{14}(J_{2,6}^{r})\stackrel{\bar p_{12\ast}}\twoheadrightarrow \Z_2\{\eta_{12}^2\}$ is split onto. 

For $(5)$ of  Lemma \ref{lem pi(J2(2rl6))},  we have the following exact sequence
\begin{align*}
	0\!\!\rightarrow \!\!   \Z_{2}\{\bar\beta_6\nu^3_6\}\!\oplus\!\Z_{2}\{\bar\beta_6\mu_6\}\!\oplus\! \Z_2\{\beta_6\eta_6\varepsilon_7\}\!=\!\pi_{15}(F_{\bar{p}_{12}})\!\xrightarrow{\bar{\tau}_{12\ast}} \! \pi_{15}(J_{2,6}^{r})\!\xrightarrow{\bar p_{12\ast}}\!\Z_{2^{m_{r+1}^3}}\!\{\delta_{r+1}\nu_{12}\}\!\!\rightarrow\!\! 0. 
\end{align*}
$\exists -\mathbf{x}_r\in \{\beta_6, \gamma_2, \delta_{r+1}\nu_{11}\}$ such that  $\bar p_{12}\mathbf{x}_r=\delta_{r+1}\nu_{12}$ and $2^{m_{r+1}^3}\mathbf{x}_r\in \beta_6\{2^r\Delta(\iota_{13}),\delta_{r+1}\nu_{11},  $ $2^{m_{r+1}^3}\iota_{14}\}$,  where 
$\{2^r\Delta(\iota_{13}),\delta_{r+1}\nu_{11},  $ $2^{m_{r+1}^3}\iota_{14}\}\subset \{\Delta(\iota_{13}),(2^r\iota_{11})\delta_{r+1}\nu_{11}(2^{m_{r+1}^3-1}\iota_{14}), 2\iota_{14}\}$ $=\{\Delta(\iota_{13}),0, 2\iota_{14}\}\ni 0$ mod $\{0\}$. Hence $\{2^r\Delta(\iota_{13}),\delta_{r+1}\nu_{11}, 2^{m_{r+1}^3}\iota_{14}\}=\{0\}$ implies that $2^{m_{r+1}^3}\mathbf{x}_r=0$. So  $\pi_{15}(J_{2,6}^{r})\stackrel{\bar p_{12\ast}}\twoheadrightarrow\Z_{2^{m_{r+1}^3}}\{\delta_{r+1}\nu_{12}\}$ is split onto. 
	
For $(8)$ of  Lemma \ref{lem pi(J2(2rl6))}, 	 since $\gamma_{2\ast}(\sigma_{11})=2^r\Delta(\iota_{13})\sigma_{11}=2^rl'\Delta(\sigma_{13})$ where $l'$ is the inverse of the odd integer $l$ in $\Z_{16}$
 which is given by $\Delta(\sigma_{13})=l\Delta(\iota_{13})\sigma_{11}$.
we have
\begin{align*}
	0\!\!\rightarrow \!\!   \Z_{2^{m_{r}^4}}\{\bar\beta_6\Delta(\sigma_{13})\}\!\oplus\!\Z_{2}\{\bar\beta_{17}\eta_{17}\}=\!\pi_{18}(F_{\bar{p}_{12}})\!\xrightarrow{\bar{\tau}_{12\ast}} \! \pi_{18}(J_{2,6}^{r})\!\xrightarrow{\bar p_{12\ast}}\!\Z_{2}\!\{\nu^2_{12}\}\!\!\rightarrow\!\! 0. 
\end{align*}	
$\exists -\mathbf{x}_r\in\{\beta_6,\gamma_2,\nu^2_{11}\}$ such that  $\bar p_{12}\mathbf{x}_r=\nu^2_{12}$ and $2\mathbf{x}_r\in \beta_6\{2^r\Delta(\iota_{13}),\nu^2_{11},2\iota_{17}\}$,  where 
$\{2^r\Delta(\iota_{13}),\nu^2_{11},2\iota_{17}\}\ni 2^{r-1}\Delta(\iota_{13})\nu_{11}^2\eta_{17}=0$ mod $\lr{2\Delta(\sigma_{13})}$. Hence $2\mathbf{x}_r=2t_r\beta_6\Delta(\sigma_{13}), t_r\in \Z$, which implies  $\pi_{18}(J_{2,6}^{r})\stackrel{\bar p_{12\ast}}\twoheadrightarrow\Z_{2}\{\nu^2_{12}\}$ is split onto. 
	
For $(9)$ of  Lemma \ref{lem pi(J2(2rl6))}, we have 
\begin{align*}
	0\rightarrow \Z_{2}\{\bar\beta_6\nu_6\sigma_9\nu_{16}\}\oplus \Z_{2}\{\bar\beta_{17}\eta_{17}^2\}=\!\pi_{19}(F_{\bar{p}_{12}})\!\xrightarrow{\bar{\tau}_{12\ast}} \! \pi_{19}(J_{2,6}^{r})\!\xrightarrow{\bar p_{12\ast}}\!\Z_{2^{m_{r}^4}}\{\delta_{r}^{16}\sigma_{12}\}\!\!\rightarrow\!\! 0. 
\end{align*}
	$\exists -\mathbf{x}_r\in\{\beta_6,\gamma_2, \delta_{r}^{16}\sigma_{11}\}$ such that  $\bar p_{12}\mathbf{x}_r=\delta_{r}^{16}\sigma_{12}$ and $2^{m_{r}^4}\mathbf{x}_r\in \beta_6\{2^r\Delta(\iota_{13}),\delta_{r}^{16}\sigma_{11},2^{m_{r}^4}\iota_{18}\}$,  where 
$\{2^r\Delta(\iota_{13}),\delta_{r}^{16}\sigma_{11},2^{m_{r}^4}\iota_{18}\}$  is a coset of subgroup $(2^r\Delta(\iota_{13}))\fhe\pi_{19}(S^{11})+\pi_{19}(S^{6})\fhe(2^{m_{r}^4}\iota_{19})$ $=\{0\}$. For $r=1$, $\{2\Delta(\iota_{13})\iota_{11},8\sigma_{11},2\iota_{18}\}\ni  \Delta(\iota_{13})(8\sigma_{11})\eta_{18}=0$;
\newline
 For $r\geq 2$,  $\{2^r\Delta(\iota_{13}),\delta_{r}^{16}\sigma_{11},2^{m_{r}^4}\iota_{18}\}\subset \{2\Delta(\iota_{13}),0,2\iota_{18}\}=\{0\}$. So 
 $2^{m_{r}^4}\mathbf{x}_r=0, r\geq 1$,
which implies  $\pi_{19}(J_{2,6}^{r})\stackrel{\bar p_{12\ast}}\twoheadrightarrow\Z_{2^{m_{r}^4}}\{\delta_{r}^{16}\sigma_{12}\}$ is split onto.

\end{proof}

\begin{lemma}\label{lem: gamma3 for P7}
	For $r\geq 0$, 
$ \gamma^r_3=2^ra_0\beta_{17}+2^rb_r\beta_6\zeta_6+2^rc_r\beta_6\bar\nu_6\nu_{14}$, where $a_0$ is an odd integer, $b_r\in \Z_8$, $c_r\in \Z_4$.
\end{lemma}
\begin{proof} 
	By (7) of Lemma \ref{lem pi(J2(2rl6))}, assume $\gamma_3^r=a_r\beta_{17}+b'_r\beta_6\zeta_6+c'_r\beta_6\bar\nu_6\nu_{14}\in \pi_{17}(J_{2,6}^{r})$, where $a_r\in \Z\subset \Z_{(2)}$, $b_r'\in\Z_8$, $c_0'\in \Z_2$, $c_r'\in \Z_4, r'\geq 1$. There is the following isomorphism
\begin{align*}
  &	\Z_8\oplus \Z_2\cong \pi_{18}(S^7)\xrightarrow[\cong ]{\bar\partial_{17\ast}^{12,0}} \pi_{17}(F_{p_{7}^0})\cong 	\pi_{17}(J_{3,6}^{0})=\pi_{17}(J_{2,6}^{0})/\lr{\gamma_3^0}\\
 = &\frac{\Z_8\{\beta^0_6\zeta_6\}\oplus \Z_2\{\beta_6^0\bar\nu_6\nu_{14}\}\oplus \Z_{(2)}\{\beta^0_{17}\}}{\lr{a_0\beta^0_{17}+b'_0\beta^0_6\zeta_6+c'_0\beta^0_6\bar\nu_6\nu_{14}}}.
\end{align*}
If $a_0=2^{t_0}a_0', a'_0$ is odd,$0\leq t_0\in \Z$  , then we get the short exact sequence 
\begin{align*}
	0\rightarrow \Z_8\{\beta^0_6\zeta_6\}\oplus \Z_2\{\beta_6^0\bar\nu_6\nu_{14}\}\xrightarrow{I}\pi_{17}(J_{2,6}^{0})/\lr{\gamma_3^0}\xrightarrow{P} \Z_{2^{t_0}}\{\beta^0_{17}\}\rightarrow 0     
\end{align*}
where the definition of inclusion $I$ and quotient $P$ are obvious. So $t_0=0$, i.e., $a_0$ is odd. By the equations (\ref{Equ:gamma30,3r}), (\ref{equ2  g^s_t}) and diagram  (\ref{diam big}) for $k=6, s=r\geq 1, t=0$, we get 
\begin{align*}
&	\bar g_0^r(a_r\beta^r_{17}+b'_r\beta^r_6\zeta_6+c'_r\beta^r_6\bar\nu_6\nu_{14})=2^{2r}(a_0\beta^0_{17}+b'_0\beta^0_6\zeta_6+c'_0\beta^0_6\bar\nu_6\nu_{14}), ~\text{with}~g_0^r\beta^r_6=\beta^0_6.\\
&( g^r_0)|_{J_2}=j_1^6q_1^6+2^rj_2^{17}q_2^{17}+j_1^6\theta q_2^{17}, ~\theta=b_{\theta}\zeta_6+ c_{\theta}\bar\nu_6\nu_{14}\in\pi_{17}(S^6)=\Z_8\{\zeta_6\}\oplus \Z_4\{\bar\nu_6\nu_{14}\}.\\
&
\xymatrix{
		S^{17}\ar@{^{(}->}[r]^-{j_2^{17}}&S^6\vee S^{17}\simeq J_{2,6}^{\gamma,r}\ar[r]^-{I_2^{\gamma,r}}\ar[d]^{( g^r_0)|_{J_2}}&F_{\bar p_{12}^r}\ar[d]^{ g^r_0}\ar[r]^-{\bar\tau_{12}^r}&J_{2,6}^r\ar[d]^{\bar g^r_0}\\
			S^{17}\ar@{^{(}->}[r]^-{j_2^{17}}& S^6\vee S^{17}\simeq	J_{2,6}^{\gamma,0}\ar[r]^-{I_2^{\gamma,0}} &F_{\bar p_{12}^0}\ar[r]^-{\bar\tau_{12}^0}&J_{2,6}^0.
	},    \\
& 	\bar g_0^r\beta^r_{17}\!=\!	\bar g_0^r(\bar\tau_{12}^rI_2^{\gamma,r}\!j_2^{17})\!=\!\bar\tau_{12}^0I_2^{\gamma,0}( g^r_0)|_{J_2}j_2^{17}\!=\!\bar\tau_{12}^0I_2^{\gamma,0}(2^rj_2^{17}\!+\!j_1^6\theta)\!=\!2^r\beta_{17}^0\!+\!b_{\theta}\beta^0_6\zeta_6\!+\! c_{\theta}\beta^0_6\bar\nu_6\nu_{14}.\\
&\Rightarrow~ 2^ra_r\beta_{17}^0+(a_rb_{\theta}+b_r')\beta^0_6\zeta_6+(a_rc_{\theta}+c_r')\beta^0_6\bar\nu_6\nu_{14}=2^{2r}a_0\beta^0_{17}+2^{2r}b'_0\beta^0_6\zeta_6.
\end{align*} 
Comparing the coefficients of the torsion free element $\beta_{17}^0$, we get $a_r=2^ra_0$. Then $2^ra_0b_{\theta}+b_r'=2^{2r}b_0'\in \Z_8$ and  $2^ra_0c_{\theta}+c_r'=0\in \Z_2$ imply $b'_r=2^rb_r$, $c'_r=2c''_r$. Hence 
\begin{align*}
	\gamma_3^r=2^ra_0\beta_{17}+2^rb_r\beta_6\zeta_6+2c''_r\beta_6\bar\nu_6\nu_{14}.
\end{align*}
By the equations (\ref{Equ:gamma30,3r}) and  (\ref{equ2  g^s_t})  for $k=6, s=r\geq 2, t=1$, we get 
\begin{align*}
&	\bar g_1^r(2^ra_0\beta^r_{17}+2^rb_r\beta^r_6\zeta_6+2c''_r\beta^r_6\bar\nu_6\nu_{14})=2^{2(r-1)}(2a_0\beta^1_{17}+2b_1\beta^1_6\zeta_6+2c''_1\beta^1_6\bar\nu_6\nu_{14}).
\end{align*}
	Comparing the coefficients of $\beta^r_6\bar\nu_6\nu_{14}$ on both sides of above equation, $2c''_r=0 (r\geq 2)$. 
\end{proof}
\begin{lemma}\label{lem: pi_{17}(J3P7)}
	\begin{align*}
		&	\pi_{17}(J_{3,6}^r)=\left\{
		\begin{array}{ll}
			\Z_2\{a_0\check{\beta}_{17}+b_1\check{\beta}_{6}\zeta_6+c_1\check{\beta}_{6}\bar\nu_6\nu_{14}\}\oplus	\Z_8\{\check{\beta}_6\zeta_6\}\oplus \Z_4\{\check{\beta}_6\bar\nu_6\nu_{14}\}, & \hbox{$r=1$;} \\
			\Z_4\{a_0\check{\beta}_{17}+b_2\check{\beta}_{6}\zeta_6\}\oplus	\Z_8\{\check{\beta}_6\zeta_6\}\oplus \Z_4\{\check{\beta}_6\bar\nu_6\nu_{14}\}, & \hbox{$r=2$;} \\
			\Z_{2^r}\{a_0\check{\beta}_{17}\}\oplus	\Z_8\{\check{\beta}_6\zeta_6\}\oplus \Z_4\{\check{\beta}_6\bar\nu_6\nu_{14}\},& \hbox{$r\geq 3$.}\\
		\end{array}
		\right.\\
		 & 	\pi_{18}(J_{3,6}^r)=\Z_{2^{m_{r}^4}}\{\check\beta_6\Delta(\sigma_{13})\}\oplus \Z_{2}\{\check\beta_{17}\eta_{17}\}\oplus \Z_{2}\{I_2^r\widehat{\nu_{12}^2}\}.\\
		 &\pi_{19}(J_{3,6}^{r})=\left\{
		 \begin{array}{ll}
			 \Z_{2}\{\check\beta_6\nu_6\sigma_9\nu_{16}\}\oplus  \Z_{2}\{I_2^1\widehat{8\sigma_{12}}\}\oplus \Z_4\{\widearc{\eta_{18}}\},~~ 2\widearc{\eta_{18}}=\check\beta^1_{17}\eta_{17}^2 & \hbox{$r=1$;} \\
		 \Z_{2}\{\check\beta_6\nu_6\sigma_9\nu_{16}\}\oplus \Z_{2}\{\check\beta_{17}\eta_{17}^2\}\oplus \Z_{2^{m_{r}^4}}\{I_2^r\widehat{\delta_{r}^{16}\sigma_{12}}\}\oplus \Z_2\{\widearc{\eta_{18}}\}, & \hbox{$r\geq 2$.} \\
		 	\end{array}
		 \right.
	\end{align*}
\end{lemma}
\begin{proof}
	Consider the Lemma \ref{lem: exact seq pi_m(J3(2rl4))} for $m=17,18,19$, $k=6$. 
	
$\pi_{m}(J_{3,6}^r)$ for $m=17,18$ are easily obtained from the isomorphisms  
	$$	\frac{\pi_{17}(J_{2,6}^r)}{\lr{\gamma^r_{3}}}\xrightarrow[\cong]{I_{2\ast}^r }\pi_{17}(J_{3,6}^r); ~~\pi_{18}(J_{2,6}^r)\xrightarrow[\cong]{I_{2\ast}^r}\pi_{18}(J_{3,6}^r).$$
For $\pi_{19}(J_{3,6}^r)$, we get $	0\rightarrow \pi_{19}(J_{2,6}^r)\xrightarrow{I_{2\ast}^r}\pi_{19}(J_{3,6}^r)\xrightarrow{\check p_{18\ast}} \pi_{19}(S^{18})=\Z_2\{\eta_{18}\}\rightarrow 0.$
\newline
 $ \exists -\mathbf{x}_r\in \{I^r_2,\gamma_3^r,\eta_{17}\}\subset \pi_{19}(J_{3,6}^{r})$, such that $\check p_{12\ast}(\mathbf{x}_r)=\eta_{18}$ and $2\mathbf{x}_r\in I^r_2\{\gamma_3^r,\eta_{17},2\iota_{18}\}$. 

	From Lemma \ref{lem:Toda bracket}, Lemma \ref{lem: gamma3 for P7} and by  $\zeta_6\eta_{17}\in \lr{8\Delta(\sigma_{13})}\subset \pi_{18}(S^6)=\Z_{16}\{\Delta(\sigma_{13})\}$, 
	\begin{align*}
		&\{\gamma_3^r,\eta_{17},2\iota_{18}\}\ni 2^{r-1}(a_0\beta_{17}+b_r\beta_6\zeta_6+c_r\beta_6\bar\nu_6\nu_{14})\eta_{17}^2=2^{r-1}\beta_{17}\eta_{17}^2 ~\text{mod}\lr{2\pi_{19}(J_{2,6}^r)}.
	\end{align*}
Hence $2\mathbf{x}_r=2^{r-1}\check\beta_{17}\eta_{17}^2+2I_2^r\theta_r$ for some $\theta_r\in \pi_{19}(J_{2,6}^r)$. Take $\mathbf{x}'_r:=\mathbf{x}_r-I_2^r\theta_r$, then $o(\mathbf{x}'_1)=4$ with $2\mathbf{x}'_1=\check\beta_{17}\eta_{17}^2$ and $o(\mathbf{x}'_r)=2$ for $r\geq 2$. Thus we get the $\pi_{19}(J_{3,6}^{r})$. 
	
\end{proof}

\begin{remark}
   In fact,	by our calculation we find that the groups $\pi_{i}(J_{3,6}^r)$ and $\pi_{i}(P^{7}(2^r))$ for $i\leq 22$, under isomorphism,  have nothing to do with	the undetermined coefficients $b_r,c_r$ in Lemma \ref{lem: gamma3 for P7} .
\end{remark}	
	
$\bullet~~\pi_{m}(P^{7}(2^r)), m=11,12$. 

Consider diagram (\ref{exact:hgps piPk+1(2r)}) for $k=6, m=11,12$, we get
the isomorphism $\pi_{m}(J_{2,6}^r)\cong \pi_{i}(F_7)\xrightarrow[\cong]{\tau_{7\ast}} \pi_{m}(P^{7}(2^r))$. Hence (1) (2) of Lemma \ref{lem pi(J2(2rl6))} implies
\begin{align*}
	&\pi_{11}(P^{7}(2^r))\cong \Z_{2^r}, r\geq 1; ~~\pi_{12}(P^{7}(2^r))\cong \Z_{2}, r\geq 1. 
\end{align*}

$\bullet~~\pi_{13}(P^{7}(2^r))$. 

There is the exact sequence 
\begin{align*}
	\Z_8\{\sigma'\}=\pi_{14}(S^7)\xrightarrow{\partial_{13,\ast}^{7,r}} \pi_{13}(F_7)\xrightarrow{\tau_{7\ast}}\pi_{13}(P^7(2^{r}))\xrightarrow{p_{7\ast}} \Z_2\{\nu_{7}^2\}\rightarrow 0. 
\end{align*}
By (3) of 	Lemma \ref{lem pi(J2(2rl6))}, assume that $\partial_{13,\ast}^{7,r}(\sigma')=x_r\beta_6\sigma''+y_r\widehat{\eta_{12}}$, $x_r\in \Z_4$, $y_r\in \Z_2$.
	\newline
  Lemma \ref{Lem: compute H2} and Corollary \ref{cor: suspen Fp Moore} implie the commutative diagrams
	\begin{align*}
			&\xymatrix{
				\pi_{14}(S^{7})\ar[d]_{H_2} \ar[r]^-{\partial^{7,r}_{13\ast}} & \pi_{13}(F_{7})\ar[d]_{H'_2 }&\pi_{13}(J_{2,6}^r)\!\ar[l]_-{\check I^r_{2\ast}}^-{\cong}\ar[r]^-{\bar p_{12\ast}}&\!\!\pi_{13}(S^{12})\ar[d]_{\Sigma}\\
				\pi_{14}(S^{13})\ar[r]^-{(2^r\iota_{13})_{\ast}} & \pi_{14}(S^{13})\ar@{=}[rr]&&\! \pi_{14}(S^{13})};~	\xymatrix{
			\pi_{14}(S^{7}) \ar[r]^-{\partial^{7,r}_{13\ast}}\ar[d]_{\Sigma^{\infty} } & \pi_{13}(F_7) \ar[d]^{E^{\infty}}\\
			\Z_{16}\{\sigma\}\!=\!\pi^s_{13}(S^{6})\ar[r]^-{(2^r\iota)_{\ast}} & 	\pi^s_{13}(S^{6})} \\
				&0=H_{2}'\partial^{7,r}_{13\ast}(\sigma')=H_{2}'( x_r\beta_6\sigma''+y_r\widehat{\eta_{12}})=y_r\eta_{13}~~\Rightarrow~y_r=0.
	\end{align*}
	$(2^r\iota)_{\ast}\Sigma^{\infty}(\sigma')=2^{r+1}\sigma~\text{\cite[Lemma 5.14]{Toda}}$ and the commutative triangle in Corollary \ref{cor: suspen Fp Moore} implies  $E^{\infty}\partial^{7,r}_{13\ast}(\sigma')= E^{\infty}(x_r\beta_6\sigma'')=k_rx_r\Sigma^{\infty}(\sigma'')=4x_rk_r \sigma$ where $k_r$ is odd. So $4x_rk_r=2^{r+1}\in \Z_{16}$, i.e., $x_r=\pm 2^{r-1}\in \Z_4$. We get $\partial_{13,\ast}^{7,r}(\sigma')=\pm 2^{r-1}\beta_6\sigma''$. So
	\begin{align*}
		&0\rightarrow \Z_{2^{m_{r-1}^2}}\{\beta_6\sigma''\}\oplus \Z_2\{\widehat{\eta_{12}}\}\xrightarrow{\tau_{7\ast}}\pi_{13}(P^7(2^{r}))\xrightarrow{p_{7\ast}} \Z_2\{\nu_{7}^2\}\rightarrow 0. 
	\end{align*}
	There is $-\mathbf{x}\in \{i_6, 2^r\iota_6, \nu_6^2\}$ such that $p_{7\ast}(\mathbf{x})=\nu_7^2$ and $2\mathbf{x}\in i_6\{ 2^r\iota_6, \nu_6^2, 2\iota_{12}\}$. Note that $(2^r\iota_6)\sigma''=2^r\sigma''\pm \binom{2^{r}}{2}[\iota_6,\iota_6]H_2(\sigma'')=2^r\sigma''\pm \binom{2^{r}}{2}\Delta(\iota_{13})\eta_{11}^2=2^r\sigma''$. Then 
	\begin{align*}
		\{ 2^r\iota_6, \nu_6^2, 2\iota_{12}\}\ni 2^{r-1}\nu_6^2\eta_{12}=0~ \text{mod}~ \lr{2\sigma''}.
	\end{align*}
So $2\mathbf{x}=2ti_6\sigma''$. This implies $\pi_{13}(P^7(2^{r}))\stackrel{p_{7\ast}}\twoheadrightarrow \Z_2\{\nu_{7}^2\}$ is split onto. Thus 
\begin{align*}
	\pi_{13}(P^7(2^{r}))\cong \Z_2^{\oplus 2}\oplus \Z_{2^{m_{r-1}^2}}, r\geq 1.
\end{align*}
	
	$\bullet~~\pi_{14}(P^{7}(2^r))$. 
	
Consider diagram (\ref{exact:hgps piPk+1(2r)}) for $k=6, m=14$ implies the exact sequence 
\begin{align*}
	\Z_2\{\sigma'\eta_{14}\}\!\oplus \!	\Z_2\{\bar{\nu}_7\}\!\oplus\! \Z_2\{\varepsilon_7\} \!=\!\pi_{15}(S^7)\!\xrightarrow{\partial_{14\ast}^{7,r}}\! \pi_{14}(F_7)\!\xrightarrow{\tau_{7\ast}}\!\pi_{14}(P^7(2^{r}))\!\xrightarrow{p_{7\ast}}\! \Z_{2^{m_r^3}}\{\delta_r\sigma'\}\!\rightarrow\! 0. 
\end{align*}
$\partial_{14\ast}^{7,r}(\varepsilon_7)=0$ and $ \partial_{14\ast}^{7,r}(\bar{\nu}_7)=\beta_{6}(2^r\iota_6)\bar{\nu}_6=2^{2r}\beta_{6}\bar{\nu}_6=0$ (Lemma \ref{lemA: Sigma nu'sigam'}).
By Lemma \ref{lem: partial(asigamb)}, 
\begin{align*}
&	\partial_{14\ast}^{7,r}(\sigma'\eta_{14})=\partial_{13\ast}^{7,r}(\sigma')\eta_{13}=\pm 2^{r-1}\beta_6\sigma''\eta_{13}=2^{r+1}\beta_6\bar\nu_6=0~~ (\sigma''\eta_{13}=4\bar\nu_6~\text{by \cite[(7.4)]{Toda}}).\\
&\Rightarrow ~ Coker(\partial_{14\ast}^{7,r})=\pi_{14}(F_7)=\Z_{2^{m_{r+1}^3}}\{\beta_6\bar\nu_6\}\oplus\Z_2\{\beta_{6}\varepsilon_{6}\}\oplus\Z_2\{\widehat{\eta_{12}}\eta_{13}\}.
\end{align*}
For $r\leq 2$, $\delta_r\sigma'=2^{2-r}\Sigma\sigma''$. There is $-\mathbf{x}_r\in\{i_6, 2^r\iota_6, 2^{2-r}\sigma''\}$ such that $p_{7\ast}(\mathbf{x}_r)=\delta_r\sigma'$ and $2^r\mathbf{x}_r\in i_6\{2^r\iota_6, 2^{2-r}\sigma'', 2^r\iota_{13}\}$. By (vii) of  Lemma  \ref{lemA: Sigma nu'sigam'}, $2^r\mathbf{x}_r=2^rt_ri_6\bar\nu_6$ for some integer $t_r$. Then 
$\mathbf{x}_r-t_ri_6\bar\nu_6$ is the order $2^r$ lift of $\delta_r\sigma'$ by $p_7$, i.e., $\pi_{14}(P^7(2^{r}))\stackrel{p_{7\ast}}\twoheadrightarrow \Z_{2^{m_r^3}}\{\delta_r\sigma'\}$ is split onto for $r\leq 2$.

For $r\geq 3$,  the diagram (\ref{diam big}) for $k=6,s=r,t=2$ induces the following  diagram
\begin{align*}
	\small{\xymatrix{
			\Z_8\{\beta_6^r\bar\nu_6\}\ar@{^{(}->}[r]\ar@{=}[d]&  	\pi_{14}(J_{2,6}^r) \cong 	\pi_{14}(F_{7}^r)\ar[d]^{\chi_{2\ast}^{r}}\ar[r]^-{\tau_{7\ast}^r} &\pi_{14}(P^{7}(2^r)) \ar[d]_{\bar \chi^r_{2\ast}}\\
				\Z_8\{\beta_6^2\bar\nu_6\}\ar@{^{(}->}[r]& \pi_{14}(J_{2,6}^2) \cong \pi_{14}(	F_{7}^2)\ar[r]^-{\tau_{7\ast}^2}&  \pi_{14}(P^{7}(4))
	} }.
\end{align*}
Since $(\tau_{7\ast}^2)|_{res}:\Z_8\{\beta_6^2\bar\nu_6\}\hookrightarrow  \pi_{14}(P^{7}(4)) $ is split into, so is $(\tau_{7\ast}^r)|_{res}:\Z_8\{\beta_6^r\bar\nu_6\}\hookrightarrow  \pi_{14}(P^{7}(2^r)) $, i.e., there is an epimorphism  $\pi_{14}(P^{7}(2^r))\stackrel{P_r}\twoheadrightarrow 	\Z_8\{\beta_6^r\bar\nu_6\}$ such that 
$P_r(\tau_{7\ast}^r)|_{res}=id$, i.e., $P_r(i_6^r\bar\nu_6)=\beta_6^r\bar\nu_6$.  Let 
$\mathbf{y}_r\in \pi_{14}(P^{7}(2^r))$ such that $p_{7\ast}(\mathbf{y}_r)=\sigma'$. Note that the order $4$ element $\bar\chi_r^2\mathbf{x}_2$ satisfies  $p_{7\ast}(\bar\chi_r^2\mathbf{x}_2)=2\sigma'$, where $\mathbf{x}_2\in \pi_{14}(P^{7}(4))$ has been given before which satisfies $p_{7\ast}(\mathbf{x}_2)=2\sigma'$. Thus $p_{7\ast}(2\mathbf{y}_r-\bar\chi_r^2\mathbf{x}_2)=0$, i.e., $Ker(p_{7\ast})\ni  2\mathbf{y}_r-\bar\chi_r^2\mathbf{x}_2=\tau_{6\ast}(a\beta_6\bar\nu_6+b\beta_{6}\varepsilon_{6}+c\widehat{\eta_{12}}\eta_{13})=ai_6\bar\nu_6+bi_{6}\varepsilon_{6}+c\tau_{6}\widehat{\eta_{12}}\eta_{13}$, $a\in \Z_8, b,c\in \Z_2$. Hence $8\mathbf{y}_r=4ai_6\bar\nu_6$. So $4a\beta_6^r\bar\nu_6= P_r(4ai_6\bar\nu_6)=8P_r(\mathbf{y}_r)=0\in 	\Z_8\{\beta_6^r\bar\nu_6\}$. We get $4a=0\in \Z_8$, hence $8\mathbf{y}_r=0$, which implies $o(\mathbf{y}_r)=8$. Thus $\pi_{14}(P^7(2^{r}))\stackrel{p_{7\ast}}\twoheadrightarrow \Z_{2^{m_r^3}}\{\delta_r\sigma'\}$ is also split onto for $r\geq 3$. Now we have 
\begin{align*}
	\pi_{14}(P^7(2^{r}))\cong \Z_2^{\oplus 2}\oplus \Z_{2^{m_{r+1}^3}}\oplus \Z_{2^{m_{r}^3}} (r\geq 1).
\end{align*}

	$\bullet~~\pi_{15}(P^{7}(2^r))$. 

Consider diagram (\ref{exact:hgps piPk+1(2r)}) for $k=6, m=15$ implies the exact sequence 
\begin{align}
&\pi_{16}(S^7)\!\xrightarrow{\partial_{15\ast}^{7,r}}\! \pi_{15}(F_7)\!\xrightarrow{\tau_{7\ast}}\!\pi_{15}(P^7(2^{r}))\!\xrightarrow{p_{7\ast}}\! \Z_2\{\sigma'\eta_{14}\}\!\oplus \!	\Z_2\{\bar{\nu}_7\}\!\oplus\! \Z_2\{\varepsilon_7\}\!\rightarrow\! 0, \nonumber\\
&\text{where}~\pi_{16}(S^7)=\Z_2\{\sigma'\eta^2_{14}\}\!\oplus \!	\Z_2\{\nu^3_7\}\!\oplus\! \Z_2\{\mu_7\}\oplus \Z_2\{\eta_7\varepsilon_8\}~\text{and}~\nonumber\\
& \partial_{15\ast}^{7,r}(\nu^3_7)=\partial_{15\ast}^{7,r}(\mu_7)=\partial_{15\ast}^{7,r}(\eta_7\varepsilon_8)=0; \partial_{15\ast}^{7,r}(\sigma'\eta^2_{14})=\partial_{14\ast}^{7,r}(\sigma'\eta_{14})\eta_{14}=0.\nonumber\\
&\Rightarrow ~0 \rightarrow \pi_{15}(F_7)\xrightarrow{\tau_{7\ast}}\!\pi_{15}(P^7(2^{r}))\!\xrightarrow{p_{7\ast}}\! \Z_2\{\sigma'\eta_{14}\}\!\oplus \!	\Z_2\{\bar{\nu}_7\}\!\oplus\! \Z_2\{\varepsilon_7\}\!\rightarrow\! 0 . \label{exact: pi15(P7)}
\end{align}

For $r=1$, $\pi_{15}(P^7(2))\cong \Z_2\oplus \Z_4^{\oplus 2}\oplus \Z_8$ \cite[Theorem 5.13]{WJ Proj plane}.

Corollary \ref{cor: suspen Fp Moore} implies the commutative diagram 
\begin{align}
	&\!\!\!\!\!\!\xymatrix{
		\pi_{15}(J_{2,6}^r) \cong  \pi_{15}(F_7) \ar[d]^{E^{\infty}}\ar@{^{(}->}[r]^{\tau_{7\ast}}&\pi_{15}(P^7(2^{r}))\ar[d]^{\Sigma^{\infty}}\ar@{->>}[r]^-{p_{7\ast}} & \pi_{15}(S^{7})\ar[d]^{\Sigma^{\infty}}\\
		\Z_2\{\nu_3\}\!\oplus\! \Z_2\{\mu\}\!\oplus\! \Z_2\{\eta\varepsilon\}\!=\!\pi^s_{15}(S^{6})\ar@{^{(}->}[r]^-{i_{\ast}}&\!\pi^s_{15}(P^7(2^{r}))\ar@{->>}[r]^-{p_{\ast}}& \pi^s_{15}(S^{7})\!=\!\Z_2\{\nu\}\!\oplus\!\Z_2\{\varepsilon\}} \label{diam: P15(P7) to stable}
\end{align}
For $r\geq 2$, 
By Lemma \ref{lem split}, it is easy to get the bottom sequence is split. By the right commutative square in Lemma \ref{lemm Suspen Fp}, we get  $\Z_{2}\{\beta_6\nu^3_6\}\!\oplus\!\Z_2\{\beta_{6}\mu_{6}\}\!\oplus\!\Z_2\{\beta_{6}\eta_{6}\varepsilon_6\}\xrightarrow{(\tau_{7\ast})|_{res}} \pi_{15}(P^7(2^{r}))$ is split into. 
By the first proposition given in \cite{Barratt}, we get $2^{r+1}\pi_{15}(P^7(2^{r}))=0 (r\geq 2)$. Thus for $r=2$, $8\pi_{15}(P^7(4))=0$ implies that $\Z_{8}\{\widehat{\delta_{r+1}\nu_{12}}\}
\xrightarrow{(\tau_{7\ast})|_{res}} \pi_{15}(P^7(2^{r}))$ is split into. So the sequence  (\ref{exact: pi15(P7)}) is split.  For $r\geq 3$, the diagram (\ref{diam Moore s to t}) for $s=r,t=2$ induces the following commutative diagram
\begin{align*}
	\small{\xymatrix{
		0\ar[r]& \pi_{15}(F_{7}^2)\ar[d]^{\chi_{r\ast}^{2}}\ar[r]^-{\tau_{7\ast}^2} &\pi_{15}(P^{7}(4)) \ar[d]_{\bar \chi^2_{r\ast}}\ar[r]& \pi_{15}(S^7)\ar@{=}[d]\ar[r]& 0\\
				0\ar[r] &\pi_{15}(F_{7}^r)\ar[r]^-{\tau_{7\ast}^r}&  \pi_{15}(P^{7}(2^r))\ar[r]& \pi_{15}(S^7)\ar[r]& 0
	} }.
\end{align*}
Since the top sequence in the above diagram is split, so is the bottom one. Now we have 
\begin{align*}
	\pi_{15}(P^{7}(2^r))\cong \Z_2^{\oplus 6}\oplus \Z_{2^{m_{r+1}^3}}, r\geq 2. 
\end{align*}

$\bullet~~\pi_{16}(P^{7}(2^r))$.

Consider diagram (\ref{exact:hgps piPk+1(2r)}) for $k=6, m=16$ implies the exact sequence 
\begin{align}
	&0\rightarrow Coker(\partial_{16\ast}^{7,r})\!\xrightarrow{\tau_{7\ast}}\!\pi_{16}(P^7(2^{r}))\!\xrightarrow{p_{7\ast}}\! \Z_2\{\sigma'\eta_{14}^2\}\!\oplus \!	\Z_2\{\nu_7^3\}\!\oplus\!	\Z_2\{\mu_7\}\!\oplus\! \Z_2\{\eta_7\varepsilon_8\}\!\rightarrow\! 0 \nonumber \\
	&Coker(\partial_{16\ast}^{7,r})=\Z_{2^{m_r^3}}\{\beta_6\nu_6\sigma_9\}\oplus \Z_2\{\beta_6\eta_6\mu_7\}.\nonumber
\end{align}
Since $\pi_{15}(P^7(2^{r}))\!\xrightarrow{p_{7\ast}}\! \Z_2\{\sigma'\eta_{14}\}\!\oplus \!	\Z_2\{\bar{\nu}_7\}\!\oplus\! \Z_2\{\varepsilon_7\}$ is an epimorphism, let $\mathbf{x}, \mathbf{y}$ and $\mathbf{z}\in \pi_{15}(P^7(2^{r}))$  be the lifts of $\sigma'\eta_{14}$, $\bar{\nu}_7$ and $\varepsilon_7$ by $p_7$ respectively. By $\bar\nu_7\eta_{15}=\nu_7^3$ and $\varepsilon_7\eta_{15}=\eta_7\varepsilon_8$ \cite[(7.3), (7.5)]{Toda}, $\mathbf{x}\eta_{15}, \mathbf{y}\eta_{15}$ and $\mathbf{z}\eta_{15}\in \pi_{15}(P^7(2^{r}))$ are the lifts of 
  $\sigma'\eta^2_{14}$, $\nu_7^3$ and $\eta_7\varepsilon_8$ respectively.  For $\mu_7$, there is $-\mathbf{x}_r^{\mu}\in \{i_6, 2^r\iota_6, \mu_6\}$ such that $p_{7\ast}(\mathbf{x}_r^{\mu})=\mu_7$ and $2\mathbf{x}_r^{\mu}\in i_6\{2^r\iota_6, \mu_6, 2\iota_{15}\}$, where $\{2^r\iota_6, \mu_6, 2\iota_{15}\}\ni 2^{r-1}\mu_6\eta_{15}=2^{r-1}\eta_{6}\mu_7$ mod $(2^r\iota_{6})\fhe \pi_{16}(S^6)+\pi_{16}(S^6)\fhe (2\iota_{15})=\lr{2\nu_6\sigma_9}$. Hence there are integers $t_r$ such that  $2\mathbf{x}_1^{\mu}=i_6\eta_{6}\mu_7+2t_1i_6\nu_6\sigma_9$ and $2\mathbf{x}_r^{\mu}=2t_ri_6\nu_6\sigma_9, r\geq 2$.  Then  $\mathbf{y}_r^{\mu}:=\mathbf{x}_r^{\mu}-t_ri_6\nu_6\sigma_9$ is a lift of $\mu_7$ by $p_7$ with $o(\mathbf{y}_1^{\mu})=4$ and $o(\mathbf{y}_r^{\mu})=2 (r\geq 2)$. Hence 
  \begin{align*}
  	\pi_{16}(P^7(2^{r}))\cong \left\{
  	\begin{array}{ll}
  		\Z_2^{\oplus 4}\oplus \Z_4, & \hbox{$r=1$;} \\
  		\Z_{2}^{\oplus 5}\oplus \Z_{2^{m_r^3}}, & \hbox{$r\geq 2$.}
  	\end{array}
  	\right.
  \end{align*}

$\bullet~~\pi_{17}(P^{7}(2^r))$.

By diagram (\ref{exact:hgps piPk+1(2r)}) for $k=6, m=17$ and $\pi_{18}(S^7)=\Z_8\{\zeta_7\}\oplus \Z_{2}\{\bar\nu_7\nu_{15}\}$
\begin{align*}
	&\partial_{16\ast}^{7,r}(\zeta_7)=\beta_{6}(2^r\iota_6)(\zeta_6)=2^r\beta_{6}\zeta_6;\\
	&\partial_{16\ast}^{7,r}(\bar\nu_7\nu_{15})=\beta_{6}(2^r\iota_6)(\bar\nu_6\nu_{14})=\beta_{6}(2^{2r}\bar\nu_6)\nu_{14}=0  ~~\text{by Lemma \ref{lemA: Sigma nu'sigam'}}.\\
 \Rightarrow~	&0\rightarrow Coker(\partial_{17\ast}^{7,r})\!\xrightarrow{\tau_{7\ast}}\!\pi_{17}(P^7(2^{r}))\!\xrightarrow{p_{7\ast}}\! \Z_{2^{m_{r}^3}}\{\delta_r\nu_7\sigma_{10}\}\!\oplus \!	\Z_2\{\eta_7\mu_8\}\rightarrow\! 0  \\
\text{where}~ &Coker(\partial_{17\ast}^{7,r})\cong \Z_{2^{m_r^3}}\oplus \Z_{2^{m_r^3}}\oplus \Z_4,  ~\text{by Corollary \ref{lem: pi_{17}(J3P7)}}. 
\end{align*}
Since $\pi_{16}(P^{6}(2^r))\stackrel{p_{6\ast}}\twoheadrightarrow\Z_{2^{m_{r}^3}}\{\delta_r\nu_6\sigma_9\}\oplus  \Z_2\{\eta_6\mu_7\}$  is split onto, so is the epimorphism  $\pi_{17}(P^{7}(2^r))\stackrel{p_{7\ast}}\twoheadrightarrow\Z_{2^{m_{r}^3}}\{\delta_r\nu_7\sigma_{10}\}\oplus  \Z_2\{\eta_7\mu_8\}$ by  Lemma \ref{lem: split Suspen Cf}.  Hence 
\begin{align*}
		\pi_{17}(P^7(2^{r}))\cong  \Z_2\oplus \Z_4 \oplus \Z_{2^{m_r^3}}^{\oplus 3}, r\geq 1. 
\end{align*}

$\bullet~~\pi_{18}(P^{7}(2^r))$.

By diagram (\ref{exact:hgps piPk+1(2r)}) for $k=6, m=18$ and $\pi_{19}(S^7)=0$, we get 
\begin{align*}
	0\rightarrow \pi_{18}(J_{3,6}^r)\cong \pi_{18}(F_7)\!\xrightarrow{\tau_{7\ast}}\!\pi_{17}(P^7(2^{r}))\!\xrightarrow{p_{7\ast}}\! \Z_{2^{m_{r}^3}}\{\zeta_7\}\!\oplus \!	\Z_2\{\nu_7\nu_{15}\}\rightarrow\! 0  \\
\end{align*}
Since $\pi_{17}(P^{6}(2^r))\stackrel{q_1p_{6\ast}}\twoheadrightarrow \Z_{2^{m_{r}^3}}\{\delta_r\zeta_6\}$ is split onto, there is  $\widetilde{\delta_r\zeta_6}\in \pi_{17}(P^{6}(2^r))$ with $o(\widetilde{\delta_r\zeta_6})=2^{m_r^3}$ such that $p_{6\ast}(\widetilde{\delta_r\zeta_6})=\delta_r\zeta_6$. Then $\Sigma \widetilde{\delta_r\zeta_6}$ is an order $2^{m_r^3}$ element such that 
$p_{7\ast}(\Sigma\widetilde{\delta_r\zeta_6})=\delta_r\zeta_7$. 

For $r=1$, consider commutative diagram 
(\ref{diam: P15(P7) to stable}) for $r=1$. 

In the bottom exact sequence of (\ref{diam: P15(P7) to stable}), by the Toda bracket in stable case \cite[Page 32]{Toda}, $\exists \widetilde{\bar\nu}\in <i, 2\iota, \bar\nu>$ such that $p_{\ast}(\widetilde{\bar\nu})=\bar\nu\in \pi^s_{15}(S^{7})$ and $2\widetilde{\bar\nu}\in i<2\iota, \bar\nu, 2\iota>$. $<2\iota, \bar\nu, 2\iota>\ni \bar\nu\eta=i\nu^3$ mod $2\pi^s_{9}(S^0)=\{0\}$. Thus $2\widetilde{\bar\nu}=\nu^3$ and $o(\widetilde{\bar\nu})=4$.  Since $\pi_{15}(P^{7}(2))\cong \Z_2\oplus \Z_4^{\oplus 2}\oplus \Z_8$ and  $\sigma'\eta_{14}\in \pi_{15}(S^7)$ has an order $8$ lift in $\pi_{15}(P^{7}(2)) $by $p_7$ \cite[Page 82]{WJ Proj plane},  there is an order $4$ lift $\widetilde{\bar\nu_7}\in \pi_{15}(P^{7}(2))$ of $\bar\nu_7$ by $p_7$. Hence $p_{\ast}(\Sigma^{\infty}\widetilde{\bar\nu_7})=\bar\nu$, which implies that $\Sigma^{\infty}\widetilde{\bar\nu_7}-\widetilde{\bar\nu}\in Ker (p_{\ast})=Im(i_{\ast})\cong \Z_2^{\oplus 3}$. So $2\Sigma^{\infty}\widetilde{\bar\nu_7}=2\widetilde{\bar\nu}=i\nu^3$. 
Since  $2\widetilde{\bar\nu_7}\in Ker(p_{7\ast})=Im(\tau_{7\ast})$, assume 
\begin{align*}
	&2\widetilde{\bar\nu_7}=a_1\tau_7\beta_6\nu^3_6+a_2\tau_7\beta_{6}\mu_{6}+a_3\tau_7\beta_{6}\eta_{6}\varepsilon_7+2b\tau_7 \widehat{\delta_{r+1}\nu_{12}}, a_1,a_2,a_3,b\in \Z_2.\\
	\text{Then}~&i\nu^3=\Sigma^{\infty}(2\widetilde{\bar\nu_7})=\Sigma^{\infty}(a_1\tau_7\beta_6\nu^3_6+a_2\tau_7\beta_{6}\mu_{6}+a_3\tau_7\beta_{6}\eta_{6}\varepsilon_7+2b\tau_7 \widehat{\delta_{r+1}\nu_{12}})\\
	&=a_1i\nu^3+a_2i\mu+a_3i\eta\varepsilon~\text{(by Corollary \ref{cor: suspen Fp Moore})}~. \\
	\Rightarrow ~& a_1=1, a_2=a_3=0. \Rightarrow  2\widetilde{\bar\nu_7}=i_6\nu^3_6+2b\tau_7 \widehat{\delta_{r+1}\nu_{12}}.
\end{align*}
Let  $\mathbf{v}=\widetilde{\bar\nu_7}-b\tau_7 \widehat{\delta_{r+1}\nu_{12}}$. Then $2\mathbf{v}=i_6\nu^3_6$ and $p_{7\ast}(\mathbf{v})=\bar\nu_7$. Then $p_{7\ast}(\mathbf{v}\nu_{15})=\bar\nu_7\nu_{15}$ and $2\mathbf{v}\nu_{15}=(2\mathbf{v})\nu_{15}=i_6\nu^4_6=i_6\bar\nu_6\eta_{14}\nu_{15}=0$ \cite[Lemma 6.3, Proposition 5.8]{Toda}.

For $r\geq 2$, from the sequence $(\ref{exact: pi15(P7)})$, there is an order $2$ element $\widetilde{\bar\nu_7}\in \pi_{15}(P^7(2^r))$ such that $p_{7\ast}(\widetilde{\bar\nu_7})=\bar\nu_7$. Then the order $2$ element $\widetilde{\bar\nu_7}\nu_{15}\in \pi_{18}(P^7(2^r))$ satisfies $p_{7\ast}(\widetilde{\bar\nu_7}\nu_{15})=\bar\nu_7\nu_{15}$.

 Hence $\pi_{17}(P^7(2^{r}))\stackrel{p_{7\ast}}\twoheadrightarrow \Z_{2^{m_{r}^3}}\{\zeta_7\}\oplus\Z_2\{\nu_7\nu_{15}\}$ is split onto for $r\geq 1$. 
$$ \pi_{18}(P^7(2^r))\cong \Z_2^{\oplus 3}\oplus \Z_{2^{m_{r}^3}}\oplus \Z_{2^{m_{r}^4}}(r\geq 1).$$

$\bullet~~\pi_{19}(P^{7}(2^r))$.

By diagram (\ref{exact:hgps piPk+1(2r)}) for $k=6, m=19$ and $\pi_{20}(S^7)=\Z_2\{\nu_7\sigma_{10}\nu_{17}\}$, we get 
$$\pi_{19}(P^{7}(2^r))\cong \pi_{19}(F_7)\cong \pi_{19}(J_{3,6}^r)\cong \left\{
\begin{array}{ll}
	\Z_{2}^{\oplus 2}\oplus \Z_4, & \hbox{$r=1$;} \\
	\Z_{2}^{\oplus 3} \oplus\Z_{2^{m_{r}^4}}, & \hbox{$r\geq 2$.} \\
\end{array}
\right.$$

\subsection{Homotopy groups of $P^{8}(2^r)$}
\label{subsec: P8(2^r)}

There are cofibration and fibration sequences 
\begin{align*}
	&	S^7\xrightarrow{2^{r}\iota_7} S^{7}\xrightarrow{i_{7}} P^{8}(2^r)\xrightarrow{p_{8}}  S^{8};
	~~\Omega S^8\xrightarrow{\partial^{8}}F_{8}\xrightarrow{\tau_8} P^{8}(2^r)\xrightarrow{p_{8}}  S^{8}
\end{align*}
From  Lemma \ref{lem:gamma3 k odd}, $Sk_{20}(F_{8})\simeq J_{2,7}^{r}\simeq S^7\vee S^{14}$, hence here $\beta_7=j_1^7$ and $\beta_{14}=j_{2}^{14}$.
\begin{align*}
	 ~~Sk_{27}(F_{8})\simeq J_{3,7}^{r}\simeq (S^7\vee S^{14})\cup_{\gamma^r_3=\pm 2^r[j_1^7, j_{2}^{14}]}CS^{20}. 
\end{align*}

\begin{lemma}\label{lem:pi(J3 P8)}
	\begin{align*}
		(1)~& \pi_{20}(J_{3,7}^{r})=\Z_2\{\check\beta_7\nu_7\sigma_{10}\nu_{17}\}\oplus\Z_{2}\{\check\beta_{14}\nu_{14}^2\}\oplus \Z_{2^r}\{[\check\beta_7,\check\beta_{14}]\}.\\
		(2)~& \pi_{21}(J_{3,7}^{r})=\Z_8\{\check\beta_7\sigma'\sigma_{14}\}\oplus \Z_4\{\check\beta_7\kappa_7\}\oplus\Z_{16}\{\check\beta_{14}\sigma_{14}\}\oplus \Z_2\{[\check\beta_7,\check\beta_{14}]\eta_{20}\}.
	\end{align*}
\end{lemma}
\begin{proof}
	$(1)$ comes from the isomorphism
	 $\pi_{20}(J_{2,7}^{r})/\lr{\gamma_3}\xrightarrow[\cong]{I_{2\ast}^r}\pi_{20}(J_{3,7}^{r})$.
	 
   $(2)$ comes from the following  isomorphism  by Lemma \ref{lem: exact seq pi_m(J3(2rl4))} for $m=21$, $k=7$
	$$\pi_{21}(S^7\vee S^{14})\cong \pi_{21}(J_{2,7}^{r})\xrightarrow[\cong]{I_{2\ast}^r}\pi_{21}(J_{3,7}^{r}).$$
\end{proof}

$\bullet~~\pi_{13}(P^{8}(2^r))=\Z_2\{i_7\nu_7^2\}$, which is easily obtained by  the diagram (\ref{exact:hgps piPk+1(2r)}) for $k=7, m=13$.

$\bullet~~\pi_{14}(P^{8}(2^r))$.

Consider the diagram (\ref{exact:hgps piPk+1(2r)}) for $k=7, m=14$.
\begin{align*}
	&\Z_{(2)}\{\sigma_8\}\oplus \Z_8\{\Sigma\sigma'\}=\pi_{15}(S^8)\!\xrightarrow{\partial_{14\ast}^{8,r}}\! \pi_{14}(F_8)\!\xrightarrow{\tau_{8\ast}}\!\pi_{14}(P^8(2^{r}))\!\xrightarrow{p_{8\ast}}\! 	\Z_2\{\nu^2_8\}\rightarrow\! 0\\
	&\pi_{14}(J_{2,7}^r)\cong \pi_{14}(F_8)=\Z_{(2)}\{\beta_{14}\}\oplus\Z_8\{\beta_7\sigma'\},  \partial_{14\ast}^{8,r}(\Sigma\sigma')=\beta_7(2^r\iota_7)\sigma'=2^r\beta_7\sigma'.
\end{align*}
By Lemma \ref{lem:partial^{8,s}_{14}(sigma8)}, $\partial_{14\ast}^{8,r}(\sigma_8)=2^r\beta_{14}+y_r\beta_7\sigma'$, where  $\small{\begin{tabular}{r|c|c|c|c|}
		\cline{2-5}
		& $ r=1$&$r=2$&$r=3$& $r\geq 4$ \\
		\cline{2-5}
		$\Z_8\ni y_r=$  &  odd & $\pm 2$ &$4$ & $0$\\
		\cline{2-5}
\end{tabular}}$. 
Thus $Coker(\partial_{14\ast}^{8,r})=\left\{
\begin{array}{ll}
	\Z_{2^{r+1}}\{\beta_{14}\}\oplus 	\Z_{2^{r-1}}\{2\beta_{14}+\beta_7\sigma'\}, & \hbox{$r\leq 3$;} \\
	\Z_{2^{r}}\{\beta_{14}\}\oplus 	\Z_{8}\{\beta_7\sigma'\}, & \hbox{$r\geq 4$.}
\end{array}
\right.$

Since $\pi_{13}(P^7(2^{r}))\stackrel{p_{7\ast}}\twoheadrightarrow \Z_2\{\nu^2_7\}$ is split onto, so is the epimorphism $\pi_{14}(P^8(2^{r}))\stackrel{p_{8\ast}}\twoheadrightarrow \Z_2\{\nu^2_8\}$ by Lemma \ref{lem: split Suspen Cf}. Hence 
\begin{align*}
		\pi_{14}(P^{8}(2^r))\cong \left\{
	\begin{array}{ll}
		\Z_2\oplus \Z_{2^{r-1}}\oplus \Z_{2^{r+1}}, & \hbox{$r\leq 3$;} \\
		\Z_2\oplus \Z_{8}\oplus \Z_{2^{r}}, & \hbox{$r\geq 4$.}
	\end{array}
	\right.
\end{align*}

$\bullet~~\pi_{15}(P^{8}(2^r))$.

The diagram (\ref{exact:hgps piPk+1(2r)}) for $k=7, m=15$ implies the exact sequence 
\begin{align*}
&\pi_{16}(S^8)\!\xrightarrow{\partial_{15\ast}^{8,r}}\!\pi_{15}(F_8)\!\xrightarrow{\tau_{8\ast}}\!\pi_{15}(P^8(2^{r}))\!\xrightarrow{p_{8\ast}}	\Z_{2^{m_{r}^3}}\{\delta_{r}\Sigma\sigma'\} \rightarrow 0\\
\text{where}~&\pi_{16}(S^8)=\Z_2\{\sigma_8\eta_{15}\}\!\oplus \!	\Z_2\{\Sigma\sigma'\eta_{15}\}\!\oplus \!	\Z_2\{\bar{\nu}_8\}\!\oplus\! \Z_2\{\varepsilon_8\},\\
&\pi_{15}(F_8)=\Z_2\{\beta_7\sigma'\eta_{14}\}\oplus \Z_2\{\beta_7\bar\nu_7\}\oplus \Z_2\{\beta_7\varepsilon_7\}\oplus \Z_2\{\beta_{14}\eta_{14}\},\\
&\partial_{15\ast}^{8,r}(\Sigma\sigma'\eta_{15})=\partial_{15\ast}^{8,r}(\bar\nu_8)	=\partial_{15\ast}^{8,r}(\varepsilon_8)=0, \\
&\partial_{15\ast}^{8,r}(\sigma_8\eta_{15})=\partial_{14\ast}^{8,r}(\sigma_8)\eta_{14}=(2^r\beta_{14}+y_r\beta_7\sigma')\eta_{14}
=2^{r-1}\beta_7\sigma'\eta_{14}. \\
\Rightarrow~& Coker(\partial_{15\ast}^{8,r})=\vartheta_r\Z_2\{\beta_7\sigma'\eta_{14}\}\oplus 	\Z_2\{\beta_7\bar\nu_7\}\oplus  \Z_2\{\beta_7\varepsilon_7\}\oplus \Z_2\{\beta_{14}\eta_{14}\}.
\end{align*}
Since $\pi_{14}(P^7(2^{r}))\stackrel{p_{7\ast}}\twoheadrightarrow \Z_{2^{m_{r}^3}}\{\delta_{r}\sigma'\}$ is split onto, so is the epimorphism $\pi_{15}(P^8(2^{r}))\stackrel{p_{8\ast}}\twoheadrightarrow \Z_{2^{m_{r}^3}}\{\delta_{r}\Sigma\sigma'\}$ by Lemma \ref{lem: split Suspen Cf}. Hence 
\begin{align*}
	\pi_{15}(P^{8}(2^r))\cong \vartheta_r\Z_2\oplus  \Z_2^{\oplus 3}\oplus \Z_{2^{m_{r}^3}}, r\geq 1.
\end{align*}

$\bullet~~\pi_{16}(P^{8}(2^r))$.

Consider the diagram (\ref{exact:hgps piPk+1(2r)}) for $k=7, m=16$.
\begin{align*}
	&\pi_{17}(S^8)\!\xrightarrow{\partial_{16\ast}^{8,r}}\! \pi_{16}(F_8)\!\xrightarrow{\tau_{8\ast}}\!\pi_{16}(P^8(2^{r}))\!\xrightarrow{p_{8\ast}}\! \vartheta_r\Z_2\{\sigma_8\eta_{15}\}\!\oplus \!	\Z_2\{\Sigma\sigma'\eta_{15}\}\!\oplus \!	\Z_2\{\bar{\nu}_8\}\!\oplus\! \Z_2\{\varepsilon_8\}\!\rightarrow\! 0\\
	& \pi_{17}(S^8)=\Z_2\{\sigma_8\eta_{15}^2\}\!\oplus \!	\Z_2\{\Sigma\sigma'\eta_{15}^2\}\oplus 	\Z_2\{\nu_8^3\}\oplus 	\Z_2\{\mu_8\}\oplus \Z_2\{\eta_8\varepsilon_9\},\\
	&\pi_{16}(J_{2,7}^r)\cong \pi_{16}(F_8)=\Z_2\{\beta_7\sigma'\eta^2_{14}\}\oplus 	\Z_2\{\beta_7\nu_7^3\}\oplus \Z_2\{\beta_7\mu_7\}\oplus \Z_2\{\beta_7\eta_7\varepsilon_8\}\oplus \Z_2\{\beta_{14}\eta_{14}^2\}\\
	&\partial_{16\ast}^{8,r}(\Sigma\sigma'\eta_{15}^2)=\partial_{16\ast}^{8,r}(\nu_8^3)=\partial_{16\ast}^{8,r}(\mu_8)	=\partial_{16\ast}^{8,r}(\eta_8\varepsilon_9)=0,\\
	&\partial_{16\ast}^{8,r}(\sigma_8\eta^2_{15})=\partial_{15\ast}^{8,r}(\sigma_8\eta_{15})\eta_{15}=2^{r-1}\beta_7\sigma'\eta^2_{14}. \\
	&\Rightarrow~ Coker(\partial_{16\ast}^{8,r})=\vartheta_r\Z_2\{\beta_7\sigma'\eta^2_{14}\}\oplus 	\Z_2\{\beta_7\nu_7^3\}\oplus \Z_2\{\beta_7\mu_7\}\oplus \Z_2\{\beta_7\eta_7\varepsilon_8\}\oplus \Z_2\{\beta_{14}\eta_{14}^2\}.
\end{align*}

For $\alpha=\sigma'\eta_{14}, \bar\nu_7, \varepsilon_7$, there is $-\mathbf{x}\in\{i_7,2^r\iota_7, \alpha\}$ such that $p_{8\ast}(\mathbf{x})=\Sigma \alpha$ and 
$2\mathbf{x}\in\{2^r\iota_7, \alpha, 2\iota_{15}\}$ which is a coset of subgroup $(2^r\iota_7)\fhe\pi_{16}(S^7)+\pi_{16}(S^7)\fhe (2^r\iota_{16})=\{0\}$. 

For $\alpha=\sigma'\eta_{14}$, since $i_7\sigma'\eta_{14}^2=0$ for $r=1$, by Lemma \ref{lemA: Todabraket } $(viii)$, $2\mathbf{x}=0$ for $r\geq 1$. 

For $\alpha=\bar\nu_7$, since  $\pi_{17}(S^8)\cong \Z_2^{\oplus 5}$, by \cite[Proposition 1.3, Lemma 6.3]{Toda} and Lemma \ref{lem:Toda bracket}
\newline
 $\Sigma \{2^r\iota_7, \bar\nu_7, 2\iota_{15}\}\!=\!-\Sigma \{2^r\iota_7, \bar\nu_7, 2\iota_{15}\}\!\supset\! \{2^r\iota_8, \bar\nu_8, 2\iota_{16}\}_1\!\ni\! \left\{
 \!\!\begin{array}{ll}
 \bar\nu_8\eta_{16}=\nu_8^3~\text{mod}\lr{\Sigma\sigma'\eta_{15}^2}, \!&\! \hbox{$r=1$;} \\
 0 ~\text{mod}~\{0\}, \!&\! \hbox{$r\geq 2$.}
 \end{array}
 \right.$
The injection $\Sigma:\pi_{16}(S^7)\rightarrow \pi_{17}(S^8)$  implies that $2\mathbf{x}=i_7\nu_7^3$ for $r=1$ and $2\mathbf{x}=0$ for $r\geq 2$. 

For $\alpha=\varepsilon_7$,  $\{2^r\iota_7, \varepsilon_7, 2\iota_{15}\}\ni 2^{r-1}\eta_7\varepsilon_8$ mod $\{0\}$. Hence $2\mathbf{x}=i_7\eta_7\varepsilon_8$ for $r=1$ and $2\mathbf{x}=0$ for $r\geq 2$. 

For $\sigma_8\eta_{15}\in Ker(\partial_{15\ast}^{8,r})$ with $r\geq 2$, by the same proof as that of Lemma \ref{lem:P7(2r-1),P4(2r)}, we have the following lemma 
\begin{lemma}\label{lem:exist tidle sigma8}
	There is a $\tilde{\sigma}_8\in[P^{15}(2^{r-1}), P^{8}(2^{r})]$ such that 
	the following square commutes 
		$$	\footnotesize{\xymatrix{
		P^{15}(2^{r-1})\ar[r]^{p_{15}^{r-1}}\ar[d]^{\tilde{\sigma}_8} & S^{15} \ar[d]^{\sigma_8}\\
	P^{8}(2^r)	\ar[r]^{p_{8}^r} & S^8 ~~.} }
	$$
	Moreover,	$p_{8\ast}(\tilde{\sigma}_8\widetilde{\eta_{15}})=\sigma_8\eta_{15}$
	where $\widetilde{\eta_{15}}\in \pi_{16}(P^{15}(2^{r-1}))$ is a lift of $\eta_{15}$ by $p^{r-1}_{15}$ and $o(\tilde{\sigma}_8\widetilde{\eta_{15}})=4$ and $2$  for $r=2$ and $r\geq 3$ respectively. 
\end{lemma}

Now we get 
\begin{align*}
	\pi_{16}(P^{8}(2^r))\cong \left\{
	\begin{array}{ll}
		\Z_2^{\oplus 3}\oplus \Z_4^{\oplus 2}, & \hbox{$r=1$;} \\
	\Z_{2}^{\oplus 7}\oplus \Z_{4}, & \hbox{$r=2$;}\\
				\Z_{2}^{\oplus 9}, & \hbox{$r\geq 3$;}\\
	\end{array}
	\right.
\end{align*}

$\bullet~~\pi_{17}(P^{8}(2^r))$.

The diagram (\ref{exact:hgps piPk+1(2r)}) for $k=7, m=17$ implies the exact sequence 
\newline
$\pi_{18}(S^8)\!\xrightarrow{\partial_{17\ast}^{8,r}}\!\pi_{17}(F_8)\!\xrightarrow{\tau_{8\ast}}\!\pi_{17}(P^8(2^{r}))\!\xrightarrow{p_{8\ast}}	Ker(\partial_{16\ast}^{8,r}) \rightarrow 0$ where 
\newline
$Ker(\partial_{16\ast}^{8,r})=\vartheta_r\Z_2\{\sigma_8\eta_{15}^2\}\!\oplus \!	\Z_2\{\Sigma\sigma'\eta_{15}^2\}\oplus 	\Z_2\{\nu_8^3\}\oplus 	\Z_2\{\mu_8\}\oplus \Z_2\{\eta_8\varepsilon_9\}$;
\newline
$\pi_{18}(S^8)\!=\!\Z_8\{\sigma_8\nu_{15}\}\!\oplus \!	\Z_8\{\nu_8\sigma_{11}\}\!\oplus \!	\Z_2\{\eta_8\mu_9\}$; 
\newline
$\pi_{17}(F_8)\!=\!\Z_8\{\beta_7\nu_7\sigma_{10}\}\!\oplus\! \Z_2\{\beta_7\eta_7\mu_8\}\!\oplus\! \Z_8\{\beta_{14}\nu_{14}\}$.
\newline $\partial_{17\ast}^{8,r}(\nu_8\sigma_{11})=2^r\beta_7\nu_7\sigma_{10},~\partial_{17\ast}^{8,r}(\eta_8\mu_9)=0$.

By \cite[(7.19)]{Toda}, $\sigma'\nu_{14}=x\nu_7\sigma_{10}$, $x$ is odd integer. 
\begin{align*}
	&\partial_{17\ast}^{8,r}(\sigma_8\nu_{15})=\partial_{14\ast}^{8,r}(\sigma_8)\nu_{14}=(2^r\beta_{14}+y_r\beta_7\sigma')\nu_{14}
	=2^{r}\beta_{14}\nu_{14}+xy_r\beta_7\nu_7\sigma_{10}.\\
	\Rightarrow~& Coker(\partial_{17\ast}^{8,r})= \left\{
	\begin{array}{ll}
 \Z_2\{\beta_7\eta_7\mu_8\}	\oplus 	\Z_4\{\beta_{14}\nu_{14}\}, & \hbox{$r=1$;} \\
 \Z_2\{\beta_7\nu_7\sigma_{10}+2\beta_{14}\nu_{14}\}\oplus	 \Z_2\{\beta_7\eta_7\mu_8\}	\oplus 	\Z_8\{\beta_{14}\nu_{14}\}, & \hbox{$r=2$;}\\ \Z_{2^{m_{r-1}^3}}\{\beta_7\nu_7\sigma_{10}\}\oplus	 \Z_2\{\beta_7\eta_7\mu_8\}	\oplus 	\Z_8\{\beta_{14}\nu_{14}\}, & \hbox{$r\geq 3$.}
	\end{array}
	\right.\\
 \text{i.e.}~	&  Coker(\partial_{17\ast}^{8,r})\cong \Z_{2^{m_{r-1}^3}}\oplus	 \Z_2	\oplus 	\Z_{2^{m_{r+1}^3}}, r\geq 1.
 \end{align*}
 Since $\pi_{16}(P^8(2^{r}))\!\xrightarrow{p_{8\ast}}\! \vartheta_r\Z_2\{\sigma_8\eta_{15}\}\!\oplus \!	\Z_2\{\Sigma\sigma'\eta_{15}\}\!\oplus \!	\Z_2\{\bar{\nu}_8\}\!\oplus\! \Z_2\{\varepsilon_8\}$ is  epimorphic,  let $\widetilde{\alpha}$ be the lift of $\alpha=\sigma_8\eta_{15} (r\geq 2), \Sigma\sigma'\eta_{15}, \bar{\nu}_8, \varepsilon_8$ by $p_8$. Then the order $2$ element $\widetilde{\alpha}\eta_{16}$ is the lift of $\sigma_8\eta^2_{15}, \Sigma\sigma'\eta^2_{15}, \nu^3_8, \eta_8\varepsilon_9$ by $p_8$.
 
  For $\mu_7$, there is $-\mathbf{x}_r^{\mu}\in \{i_7, 2^r\iota_7, \mu_7\}$ such that $p_{8\ast}(\mathbf{x}_r^{\mu})=\mu_8$ and $2\mathbf{x}_r^{\mu}\in i_7\{2^r\iota_7, \mu_7, 2\iota_{16}\}$, where $\{2^r\iota_7, \mu_7, 2\iota_{17}\}\ni 2^{r-1}\mu_7\eta_{16}=2^{r-1}\eta_{7}\mu_8$ mod $(2^r\iota_{7})\fhe \pi_{17}(S^7)+\pi_{17}(S^7)\fhe (2\iota_{16})=\lr{2\nu_7\sigma_{10}}$. Hence there are integers $t_r$ such that  $2\mathbf{x}_1^{\mu}=i_7\eta_{7}\mu_8+2t_1i_7\nu_7\sigma_{10}$ and $2\mathbf{x}_r^{\mu}=2t_ri_7\nu_7\sigma_{10}, r\geq 2$.  Then  $\mathbf{y}_r^{\mu}:=\mathbf{x}_r^{\mu}\mathbf{x}_r^{\mu}-t_ri_7\nu_7\sigma_{10}$ is a lift of $\mu_8$ by $p_8$ with $o(\mathbf{y}_1^{\mu})=4$ and $o(\mathbf{y}_r^{\mu})=2 (r\geq 2)$. Hence 
\begin{align*}
	\pi_{17}(P^{8}(2^r))\cong \left\{
	\begin{array}{ll}
		\Z_2^{\oplus 3}\oplus \Z_4^{\oplus 2}, & \hbox{$r=1$;} \\
		\Z_{2}^{\oplus 6} \oplus \Z_{2^{m_{r-1}^3}}	\oplus 	\Z_{8}, & \hbox{$r\geq 2$.}
	\end{array}
	\right.
\end{align*}

$\bullet~~\pi_{18}(P^{8}(2^r))$.

The diagram (\ref{exact:hgps piPk+1(2r)}) for $k=7, m=18$ implies the exact sequence 
\newline
$\Z_8\{\zeta_8\}\oplus \Z_2\{\bar\nu_8\nu_{16}\}=\pi_{19}(S^8)\!\xrightarrow{\partial_{18\ast}^{8,r}}\!\pi_{18}(F_8)\!\xrightarrow{\tau_{8\ast}}\!\pi_{18}(P^8(2^{r}))\!\xrightarrow{p_{8\ast}}	Ker(\partial_{17\ast}^{8,r}) \rightarrow 0$ where 
\newline
$Ker(\partial_{17\ast}^{8,r})=\left\{
\begin{array}{ll}
\Z_4\{4\sigma_8\nu_{15}+2\nu_8\sigma_{11}\}\oplus \Z_2\{\eta_8\mu_9\}, & \hbox{$r=1$;} \\
	\Z_8\{2\sigma_8\nu_{15}+\nu_8\sigma_{11}\}\oplus\Z_2\{4\sigma_8\nu_{15}\}\oplus \Z_2\{\eta_8\mu_9\}, & \hbox{$r=2$;}\\
	\Z_8\{\nu_8\sigma_{11}\}\oplus\Z_4\{2\sigma_8\nu_{15}\}\oplus \Z_2\{\eta_8\mu_9\}, & \hbox{$r=3$;}\\	
	\Z_8\{\nu_8\sigma_{11}\}\oplus\Z_8\{\sigma_8\nu_{15}\}\oplus \Z_2\{\eta_8\mu_9\}, & \hbox{$r\geq 4$.}
\end{array}
\right.$
\newline
$\pi_{18}(F_8)=\Z_8\{\beta_7\zeta_7\}\oplus \Z_2\{\beta_7\bar\nu_7\nu_{15}\}$.
 $\partial_{18\ast}^{8,r}(\zeta_8)=2^r\beta_7\zeta_7,~\partial_{18\ast}^{8,r}(\bar\nu_8\nu_{16})=0$.
 \newline
$\Rightarrow$ $0\rightarrow Coker(\partial_{18\ast}^{8,r})=\Z_{2^{m_r^3}}\{\beta_7\zeta_7\}\oplus \Z_2\{\beta_7\bar\nu_7\nu_{15}\}\xrightarrow{\tau_{8\ast}}\!\pi_{18}(P^8(2^{r}))\!\xrightarrow{p_{8\ast}}	Ker(\partial_{17\ast}^{8,r}) \rightarrow 0$. 

In order to solve the extension problem of the above  short exact sequence, we firstly compute the $\pi_{20}(P^{10}(2^r))$. 

From  $(i)$ of Lemma \ref{lem:gamma3 k odd}, 	$Sk_{26}(F_{10})\simeq J_{2,9}^{r}\simeq S^9\vee S^{18}$. Then by the diagram (\ref{exact:hgps piPk+1(2r)}) for $k=9, m=20$, we get the short exact sequence 
\begin{align*}
 \Z_{2^{m_{r}^3}}\{\beta_9\zeta_9\}\!\oplus\! \Z_2\{\beta_9\bar\nu_9\nu_{17}\}\!\oplus\! \Z_2\{\beta_{18}\eta_{18}^2\}\stackrel{\tau_{{10}\ast}}\hookrightarrow\!\pi_{20}(P^{10}(2^{r}))\!\stackrel{p_{10\ast}}\twoheadrightarrow	\Z_{2^{r-1}}\{\delta_r\sigma_{10}\nu_{11}\}\!\oplus\! \Z_2\{\eta_{10}\mu_{11}\}.
\end{align*}
We have the following commutative squares
\begin{align*}
	\small{\xymatrix{
		 \Z_2\{\beta^r_9\bar\nu_9\nu_{17}\}\ar@{^{(}->}[r]^-{(\tau^r_{{10}\ast})|_{res}}\ar@{=}[d]&  \pi_{20}(P^{10}(2^r))  \ar[d]^{\bar\chi_{1\ast}^{r}}\\
		 \Z_2\{\beta^1_9\bar\nu_9\nu_{17}\}\ar@{^{(}->}[r]^-{(\tau^1_{{10}\ast})|_{res}}\ar@{^{(}->}[r]&  \pi_{20}(P^{10}(2))
	} };~~	\small{\xymatrix{
 \Z_{2^{m_{r}^3}}\{\beta_9^r\zeta_9\}\ar@{^{(}->}[r]^-{(\tau^r_{{10}\ast})|_{res}}\ar[d]^{\cong}&  \pi_{20}(P^{10}(2^r))  \ar[d]^{\Sigma^{\infty}}\\
 \Z_{2^{m_{r}^3}}\{\zeta\}\ar@{^{(}->}[r]^-{i_{\ast}}\ar@{^{(}->}[r]&  \pi_{20}^s(P^{10}(2^r))\ar@{->>}[r]^-{p_{\ast}}& \Z_{2}\{\eta\mu\}.
} }.
\end{align*}
where the left commutative square is induced by the  diagram (\ref{diam Moore s to t}) for $s=r\geq 2$, $t=1$ and the right commutative square is from Corollary \ref{cor: suspen Fp Moore}. 

For $\eta_8\mu_9\in Ker(\partial_{17\ast}^{8,r})$, there is an order $2$ element $\widetilde{\mu_8}\eta_{17}\in \pi_{18}(P^8(2^{r}))$ such that $p_{8\ast}(\widetilde{\mu_8}\eta_{17})=\mu_8\eta_{17}=\eta_8\mu_{9}$, where $\widetilde{\mu_8}\in \pi_{17}(P^8(2^{r}))$ is the lift of $\mu_8$ by $p_{8}$. Then the order $2$ elements 
$\Sigma^2(\widetilde{\mu_8}\eta_{17})$ and $\Sigma^{\infty}(\widetilde{\mu_8}\eta_{17})$ are the lifts of $\eta_{10}\mu_{11}$ and $\eta\mu$ by $p_{10}$ and $p$ (in stable case) respectively. 
Hence $\pi_{20}(P^{10}(2))\stackrel{p_{10\ast}}\twoheadrightarrow \Z_2\{\eta_{10}\mu_{11}\}$ and $\pi_{20}^s(P^{10}(2^r))\stackrel{p_{\ast}}\twoheadrightarrow \Z_2\{\eta\mu\}$ are split onto, which imply the monomorphisms  $  \Z_2\{\beta_9\bar\nu_9\nu_{17}\}\stackrel{(\tau^1_{{10}\ast})|_{res}}\hookrightarrow\pi_{20}(P^{10}(2))$ and $ \Z_{2^{m_{r}^3}}\{\zeta\}\stackrel{i_{\ast}}\hookrightarrow \pi_{20}^s(P^{10}(2^r))$  are split into. Hence  the monomorphisms \newline
$\Z_2\{\beta^r_9\bar\nu_9\nu_{17}\}\stackrel{(\tau^r_{{10}\ast})|_{res}}\hookrightarrow\pi_{20}(P^{10}(2^r))$ and $ \Z_{2^{m_{r}^3}}\{\beta_9\zeta_9\}\stackrel{(\tau^r_{{10}\ast})|_{res}}\hookrightarrow\pi_{20}(P^{10}(2^r))$ are also split into.  Consider the following commutative diagram where the left map is isomorphic by Lemma \ref{lemm Suspen Fp}
\begin{align*}
	\small{\xymatrix{
			\Z_{2^{m_{r}^3}}\{\beta_7\zeta_7\}\oplus \Z_2\{\beta_7\bar\nu_7\nu_{15}\}\ar@{^{(}->}[r]^-{\tau_{8\ast}}\ar[d]^{E^{2}~\cong}&  \pi_{18}(P^{8}(2^r))  \ar[d]^{\Sigma^{2}}\\
				\Z_{2^{m_{r}^3}}\{\beta_9\zeta_9\}\oplus\Z_2\{\beta_9\bar\nu_9\nu_{17}\}\ar@{^{(}->}[r]^-{(\tau_{{10}\ast})|_{res}}\ar@{^{(}->}[r]&  \pi_{20}(P^{10}(2^r))
	} }.
\end{align*}
Since the bottom monomorphism is split into, so is the top one.  Hence 
\begin{align*}
	\pi_{18}(P^{8}(2^r))\cong \Z_2^{\oplus 2}\oplus \Z_{2^{m_{r-1}^3}}\oplus\Z_{2^{m_{r}^3}}\oplus \Z_{2^{m_{r+1}^3}}.
\end{align*}

$\bullet~~\pi_{19}(P^{8}(2^r))\cong Ker(\partial_{18\ast}^{8,r})\cong \Z_{2^{m_{r}^3}}\oplus \Z_2$, which is easily obtained by $\pi_{19}(F_{8})\cong \pi_{19}(S^7\vee S^{14})=0$. 

$\bullet~~\pi_{20}(P^{8}(2^r))$

The diagram (\ref{exact:hgps piPk+1(2r)}) for $k=7, m=20$ implies the exact sequence 
\newline
$\Z_2\{\sigma_8\nu_{15}^2\}\oplus \Z_2\{\nu_8\sigma_{11}\nu_{18}\}=\pi_{21}(S^8)\!\xrightarrow{\partial_{20\ast}^{8,r}}\!\pi_{20}(F_8)\!\xrightarrow{\tau_{8\ast}}\!\pi_{20}(P^8(2^{r}))\!\xrightarrow{p_{8\ast}}\pi_{20}(S^8)=0$ where 
$\pi_{20}(F_8)=\Z_2\{\beta_7\nu_7\sigma_{10}\nu_{17}\}\oplus\Z_{2}\{\beta_{14}\nu_{14}^2\}\oplus \Z_{2^r}\{[\beta_7,\beta_{14}]\}$;
\newline
$\partial_{20\ast}^{8,r}(\sigma_8\nu_{15}^2)=\partial_{17\ast}^{8,r}(\sigma_8\nu_{15})\nu_{17}=(2^r\beta_{14}\nu_{14}+xy_r\beta_7\nu_7\sigma_{10})\nu_{17}=(1-\vartheta_r)\beta_7\nu_7\sigma_{10}\nu_{17}$, with
\newline    integers $x$ and $y_r$ from the previous computation of $\pi_{14}(P^{8}(2^r))$ and $\pi_{17}(P^{8}(2^r))$ respectively;
$\partial_{20\ast}^{8,r}(\nu_8\sigma_{11}\nu_{18})=\beta_7(2^r\iota_7)(\nu_7\sigma_{10}\nu_{17})=0$. Hence
\begin{align*}
	\pi_{20}(P^{8}(2^r))\cong Coker(\partial_{20\ast}^{8,r})\cong \vartheta_r\Z_2\oplus \Z_2\oplus \Z_{2^r}.
\end{align*}

$\bullet~~\pi_{21}(P^{8}(2^r))$

The diagram (\ref{exact:hgps piPk+1(2r)}) for $k=7, m=21$ implies the exact sequence 
\newline
$\Z_{16}\{\sigma^2_8\}\!\oplus \!\Z_8\{\Sigma\sigma'\!\sigma_{15}\}\!\oplus\! \Z_4\{\kappa_8\}\!=\!\pi_{22}(S^8)\!\xrightarrow{\partial_{21\ast}^{8,r}}\!\pi_{21}(F_8)\!\xrightarrow{\tau_{8\ast}}\!\pi_{21}(P^8(2^{r}))\!\!\xrightarrow{p_{8\ast}}\! Ker(\partial_{20\ast}^{8,r})\!\!\rightarrow\! \!0$
\newline
where $Ker(\partial_{20\ast}^{8,r})=\vartheta_r\Z_2\{\sigma_8\nu_{15}^2\}\oplus \Z_2\{\nu_8\sigma_{11}\nu_{18}\}$;
\newline
$\pi_{21}(J_{3,7}^{r})\cong \pi_{21}(F_8)=\Z_8\{\beta_7\sigma'\sigma_{14}\}\oplus \Z_4\{\beta_7\kappa_7\}\oplus\Z_{16}\{\beta_{14}\sigma_{14}\}\oplus \Z_2\{[\beta_7,\beta_{14}]\eta_{20}\}$;
\newline
$\partial_{21\ast}^{8,r}(\sigma_8^2)=\partial_{14\ast}^{8,r}(\sigma_8)\sigma_{14}=(2^r\beta_{14}+y_r\beta_7\sigma')\sigma_{14}=2^r\beta_{14}\sigma_{14}+y_r\beta_7\sigma'\sigma_{14}$;
\newline
$\partial_{21\ast}^{8,r}(\Sigma\sigma'\sigma_{15})=\beta_7(2^r\iota_7)(\sigma'\sigma_{14})=2^r\beta_7\sigma'\sigma_{14}$;~~ $\partial_{21\ast}^{8,r}(\kappa_8)=\beta_7(2^r\iota_7)(\kappa_7)=2^r\beta_7\kappa_7$. 
\newline
Hence $Coker(\partial_{21\ast}^{8,r})=A_r\oplus \Z_{2^{m_{r}^2}}\{\beta_7\kappa_7\}\oplus \Z_2\{[\beta_7,\beta_{14}]\eta_{20}\}$ with

 $\Z_{2^{m^4_{r+1}}}\oplus \Z_{2^{m^3_{r-1}}}\cong A_r=\left\{
\begin{array}{ll}
	\Z_4\{\beta_{14}\sigma_{14}\}, & \hbox{$r=1$;} \\
	\Z_8\{\beta_{14}\sigma_{14}\}\oplus\Z_2\{2\beta_{14}\sigma_{14}+\beta_7\sigma'\sigma_{14}\}, & \hbox{$r=2$;}\\
	\Z_{16}\{\beta_{14}\sigma_{14}\}\oplus\Z_4\{2\beta_{14}\sigma_{14}+\beta_7\sigma'\sigma_{14}\}, & \hbox{$r=3$;}\\	
	\Z_{16}\{\beta_{14}\sigma_{14}\}\oplus\Z_8\{2\beta_{14}\sigma_{14}\}, & \hbox{$r\geq 4$.}
\end{array}
\right.$

For $\nu_8\sigma_{11}\nu_{18}$, there is $-\mathbf{x}\in \{i_7, 2^r\iota_7, \nu_7\sigma_{10}\nu_{17}\}$ such that $p_{8\ast}(\mathbf{x})=\nu_8\sigma_{11}\nu_{18}$ and $2\mathbf{x}\in i_7\{2^r\iota_7, \nu_7\sigma_{10}\nu_{17}, 2\iota_{20}\}$, where $\{2^r\iota_7, \nu_7\sigma_{10}\nu_{17}, 2\iota_{20}\}\ni 2^{r-1} \nu_7\sigma_{10}\nu_{17}\eta_{20}=0$ mod $(2^r\iota_{7})\fhe \pi_{21}(S^7)+\pi_{21}(S^7)\fhe (2\iota_{21})=2\pi_{21}(S^7)$ by Lemma \ref{lem:Toda bracket}. Hence  $2\mathbf{x}=2i_7\theta$ for some $\theta\in \pi_{21}(S^7)$. Let  $\mathbf{x}'=\mathbf{x}-i_7\theta$. Then $\mathbf{x}'$ is a lift of $\nu_8\sigma_{11}\nu_{18}$ by $p_8$ with $o(\mathbf{x}')=2$. 

For $r\geq 2$ and $\sigma_8\nu_{15}^2\in Ker(\partial_{20\ast}^{8,r})$, the commutative square in Lemma \ref{lem:exist tidle sigma8} induces the following commutative square
\begin{align*}
	\footnotesize{\xymatrix{
		\pi_{21}	(P^{15}(2^{r-1}))\ar[r]^-{p_{15\ast}^{r-1}}\ar[d]^{\tilde{\sigma}_{8\ast}} & \pi_{21}(S^{15})=\Z_2\{\nu_{15}^2\} \ar[d]^{\sigma_{8\ast}}\ar[r]&0\\
			\pi_{21}(P^{8}(2^r))	\ar[r]^-{p_{8\ast}^r} & 	\pi_{21}(S^8) ~~.} }
\end{align*}

 Since $\pi_{13}(P^7(2^{r-1}))\stackrel{p_{7\ast}}\twoheadrightarrow \Z_2\{\nu^2_7\}$ is split onto, so is the epimorphism $\pi_{21}(P^{15}(2^{r-1}))\stackrel{p^{r-1}_{15\ast}}\twoheadrightarrow \Z_2\{\nu^2_{15}\}$ by Lemma \ref{lem: split Suspen Cf}, i.e., $\exists$ order $2$ element $\mathbf{y}\in \pi_{21}(P^{15}(2^{r-1}))$ such that $p^{r-1}_{15\ast}(\mathbf{y})=\nu^2_{15}$. Hence  $\tilde{\sigma}_8\mathbf{y}\in \pi_{21}(P^{8}(2^{r}))$  is the lift of $\sigma_8\nu_{15}^2$ by $p^r_8$ with $o(\tilde{\sigma}_8\mathbf{y})=2$.
 
 Hence 
 \begin{align*}
 	\pi_{21}(P^{8}(2^r))\cong \Z_2^{\oplus 3}\oplus \Z_{2^{m_{r}^2}}\oplus\Z_{2^{m^3_{r-1}}}\oplus \Z_{2^{m^4_{r+1}}}.
 \end{align*}
 
Because of the limitation of length, we stop the computation of homotopy groups of Moore spaces here. In the next section we show that we are able to compute $\gamma_3$ for some other $2$-cell complexes.

\section{Some remarks}
\label{sec:some remarks}
In the vast majority of cases, the generators of the above homotopy groups of Moore spaces are easily obtained  according to the calculation process. 

For example, the generators of $\pi_{14}(P^{6}(2^r))$ are given as follows
$$\pi_{14}(P^{6}(2))=\Z_{2}\{i_{5}\nu_5^3\}\oplus \Z_{2}\{i_{5}\mu_5\}\oplus \Z_{2^r}\{[i_{5}, \tau_6\beta_{10}]\}\oplus \Z_{2}\{\widetilde{4\bar\nu_6}\}\oplus \Z_{4}\{\widetilde{\varepsilon_6}\}$$
with $2\widetilde{\varepsilon_6}=i_{5}\eta_5\varepsilon_6$. For $r\geq 2$
$$\pi_{14}(P^{6}(2^r))=\Z_{2}\{i_{5}\nu_5^3\}\oplus \Z_{2}\{i_{5}\mu_5\}\oplus\Z_2\{i_{5}\eta_5\varepsilon_6\}\oplus \Z_{2^r}\{[i_{5}, \tau_6\beta_{10}]\}\oplus \Z_{2^{m_r^3}}\{\widetilde{\delta_r\bar\nu_6}\}\oplus \Z_{2}\{\widetilde{\varepsilon_6}\}.$$

The following theorem generalizes the results of  $\pi_{i}(P^{n}(2))$ for $n=9,10,11$ in \cite{WJ Proj plane} to  $\pi_{i}(P^{n}(2^r))$ for any $r\geq 1$. The calculation  is not difficult since it only needs the homotopy of $J_2(M_{2^r\iota_{n-1}}, S^{n-1})$ and we omit the proof here.
\begin{theorem}\label{thm: pi(P9,10,11)}
	\begin{align*}
		&\pi_{15}(P^{9}(2^r))\cong \Z_2\oplus \Z_{2^{m^3_{r-1}}}\oplus \Z_{2^{r+1}}.\\
		&\pi_{16}(P^{9}(2^r))\cong \vartheta_r\Z_2\oplus \Z_2^{\oplus 3} \oplus \Z_{2^{m^4_{r}}}.\\
		&\pi_{17}(P^{9}(2^r))\cong  \left\{
		\begin{array}{ll}
			\Z_2^{\oplus 3}\oplus \Z_4^{\oplus 2}, & \hbox{$r=1$;} \\
			\Z_{2}^{\oplus 7} \oplus \Z_{4}, & \hbox{$r= 2$;}\\
			\Z_{2}^{\oplus 9}  & \hbox{$r\geq 3$.}
		\end{array}
		\right.\\	
			&\pi_{17}(P^{10}(2^r))\cong \Z_2^{\oplus 3}\oplus \Z_{2^{m^4_{r}}}.\\	
			&\pi_{18}(P^{10}(2^r))\cong  \left\{
		\begin{array}{ll}
			\Z_2\oplus \Z_4^{\oplus 2}\oplus \Z_8, & \hbox{$r=1$;} \\
			\Z_{2}^{\oplus 6} \oplus \Z_{2^{r+1}} & \hbox{$r\geq 2$.}
		\end{array}
		\right.\\	
		&\pi_{19}(P^{11}(2^r))\cong  \left\{
	\begin{array}{ll}
		\Z_2^{\oplus 2}\oplus \Z_4^{\oplus 2}, & \hbox{$r=1$;} \\
		\Z_{2}^{\oplus 5} \oplus \Z_{2^{r}} & \hbox{$r\geq 2$.}
	\end{array}
	\right.
	\end{align*}
\end{theorem}

At last,  I would like to mention that our methods are also useful to calculate higher dimensional  homotopy groups of other $2$-cell complexes.
For the suspended complex   projective
plane  $\Sigma \mathbb{C}P^2$ and  quaternionic projective
plane  $\Sigma \mathbb{H}P^2$, we have the cofibration sequences 
\begin{align*}
	S^{4}\xrightarrow{\eta_3}S^3\xrightarrow{i_3^{\eta}} \Sigma \mathbb{C}P^2\xrightarrow{p_{5}^{\eta}}S^5;~~	S^{8}\xrightarrow{\nu_5}S^5\xrightarrow{i_5^{\nu}} \Sigma \mathbb{H}P^2\xrightarrow{p_{9}^{\nu}}S^9
\end{align*}
$J_2(M_{\eta_3}, S^4)\simeq S^3\cup_{\gamma_{2}^{\eta}=[\iota_3, \eta_3]}e^{7}\simeq S^3\vee S^7$; $J_2(M_{\nu_5}, S^8)\simeq S^5\cup_{\gamma_{2}^{\nu}=[\iota_5, \nu_5]}e^{13}\simeq S^5\vee S^{13}$;
\newline
By the following lemma we can get the homotopy types of $Sk_{14}(F_{p_{5}^{\eta}})$ and $Sk_{28}(F_{p_{9}^{\nu}})$
\begin{align*}
i.e.,~~	J_3(M_{\eta_3}, S^4)\simeq (S^3\vee S^7)\cup_{\gamma_{3}^{\eta}}e^{11}~~\text{and}~~J_3(M_{\nu_5}, S^4)\simeq (S^5\vee S^{13})\cup_{\gamma_{3}^{\nu}}e^{21}
\end{align*}
which enables us to calculate the $\pi_{i}(\Sigma \mathbb{C}P^2)$ and $\pi_{i}(\Sigma \mathbb{H}P^2)$ for $i\leq $ $13$ and $27$ respectively if the relevant homotopy groups of spheres are known. 
\begin{lemma}
	$\gamma_{3}^{\eta}=[j_1^3, j_2^7]\eta_{9}$; 	$\gamma_{3}^{\nu}=a[j_1^5, j_2^{13}]\nu_{17}$, where $a$ is an odd integer.
\end{lemma}
\begin{proof}
	We only prove the second conclusion of this lemma, while the first is easier to obtain by the similar proof.  
	
	Consider the following commutative diagrams of cofbirations and fibrations
	\begin{align*}
		\small{\xymatrix{
			S^8\ar@{=}[d] \ar[r]^-{\iota_8}&S^8 \ar[d]^{\nu_5}\ar[r] & \ast \ar[d]_{\bar{\chi}^{0}_{\nu}}\ar[r]^{p^{0}_{9}}& S^{9}\ar@{=}[d] \\
			S^8 \ar[r]^-{\nu_5}&S^5 \ar[r]& \Sigma\mathbb{H}P^2\ar[r]^-{p_9^{\nu}}& S^{9}
	} };~~	\small{\xymatrix{
	\Omega S^9\ar@{=}[d] \ar[r]^-{\partial^{9,0}}&F_9^0 \ar[d]^{\chi^{0}_{\nu}}\ar[r] & \ast \ar[d]_{\bar{\chi}^{0}_{\nu}}\ar[r]^{p^{0}_{9}}& S^{9}\ar@{=}[d] \\
\Omega S^9 \ar[r]^-{\partial^{9,\nu}}&F_{p_9^{\nu}} \ar[r]& \Sigma\mathbb{H}P^2\ar[r]^-{p_9^{\nu}}& S^{9}
} }.~~
	\end{align*}
where $\bar{\chi}^{0}_{\nu}$ is induced by the left commutative square of the first diagram and the $\chi^{0}_{\nu}$ is induced by the right commutative square of the second diagram.

Let $\bar g^{0}_{\nu}:=(\chi^{0}_{\nu})|_{J_2}$, $J_{2,\nu}:=J_2(M_{\nu_5}, S^8)$, from Lemma \ref{Lem: natural J(Mf,X)} and Lemma \ref{lem: Natural gamma_n}, we get
\begin{align}
	&\small{\xymatrix{
		S^8\ar[d]^{\nu_5}\ar[r]^-{\beta_8} & J_{2,8}^0\simeq S^8\cup_{2\sigma_{8}-\Sigma\sigma'}e^{16} \ar[d]_-{\bar{g}^{0}_{\nu}}\ar[r]^-{\bar p^{0}_{16}}& S^{16}\ar[d]^{\nu_{13}} \\
		S^{5} \ar[r]^-{j_1^5}&  J_{2,\nu}\simeq S^5\vee S^{13}\ar[r]^-{p_9^{\nu}}& S^{9}
	} };~~ \small{\xymatrix{
	S^{23}\ar[d]^{\nu_{20}}\ar[r]^{\gamma_3^0} & J_{2,8}^0 \ar[d]_{\bar{g}^{0}_{\nu}} \\
	S^{20} \ar[r]^{\gamma_3^\nu}&  J_{2,\nu}
} }. \label{diam: g^{0}_{nu}}\\
&\Rightarrow~~ \bar{g}^{0}_{\nu}\beta_8\simeq j_1^5\nu_5~~\text{and}~~\bar{g}^{0}_{\nu}\gamma_3^0\simeq \gamma_3^\nu\nu_{20}. \label{equ: g^{0}_{nu}}
\end{align}
From Theorem \ref{Theorem of ZhuJin} and Lemma \ref{lemma for [a,b,c]}, $\gamma^{\nu}_{3}\in [j_1^5, j_1^5\nu_5, j_1^5\nu_5]$ which is a coset of the subgroup
\begin{align*}
A^{\nu}_{3}:=	[\pi_{16}(S^5\vee S^{13}), j_1^5]+[\pi_{13}(S^5\vee S^{13}), j_1^5\nu_5]+[\pi_{13}(S^5\vee S^{13}), j_1^5\nu_5]
\end{align*}
By $\pi_{16}(S^5\vee S^{13})=\Z_2\{j_1^5\zeta_5\}\oplus \Z_2\{j_1^5\nu_5\bar\nu_8\}\oplus \Z_2\{j_1^5\nu_5\bar\varepsilon_8\}\oplus \Z_8\{j_2^{13}\nu_{13}\}$ and $\pi_{13}(S^5\vee S^{13})=\Z_2\{j_1^5\varepsilon_5\}\oplus \Z_{(2)}\{j_2^{13}\}$, we have 
$A^{\nu}_{3}=\lr{[j_1^5,j_2^{13}]\nu_{17}}$. Since $\Sigma [\iota_5,\nu_5, \nu_5]=0$ and $\Sigma: \pi_{20}(S^5)\rightarrow \pi_{21}(S^6)$ is an injection \cite[(10.14), (10.15)]{Toda}, by Proposition \ref{Prop: Naturality HOWP},
\begin{align}
 &	[j_1^5, j_1^5\nu_5, j_1^5\nu_5]\supset j_1^5[\iota_5,\nu_5, \nu_5]=\{0\}~\Rightarrow~[j_1^5, j_1^5\nu_5, j_1^5\nu_5]=\lr{[j_1^5,j_2^{13}]\nu_{17}}.\nonumber\\
 & \Rightarrow~\gamma^{\nu}_{3}=a[j_1^5,j_2^{13}]\nu_{17}, ~\text{for some}~a\in \Z_8. \label{equ:gamma3nu}
\end{align}
There is  the cofibration sequence $S^{15}\xrightarrow{2\sigma_{8}-\Sigma\sigma'}S^8\xrightarrow{\beta_8} J_{2,8}^0\xrightarrow{\bar p_{16}^0} S^{16}$. From (\ref{equ: Sk_5k-3 barFp2k}),  we get  $Sk_{24}(F_{ p_{16}^0})\simeq J_{2,8}^{\gamma,0}\simeq S^8\vee S^{23}$. Then Lemma \ref{lem: exact seq pi_m(J3(2rl4))} for $k=8,r=0,m=23$ gives 
\begin{align*}
&	0\rightarrow Coker(\bar\partial_{23\ast}^{16,0})\xrightarrow{\bar\tau_{16}^0}\pi_{23}(J_{2,8}^0)\xrightarrow{\bar p_{16\ast}^0} Ker(\bar\partial_{22\ast}^{16,0})=\Z_2\{8\sigma_{16}\}\rightarrow 0, \\
& Coker(\bar\partial_{23\ast}^{16,0})=\Z_2\{\bar\beta_8\sigma_8\bar\nu_{15}\}\oplus\Z_2\{\bar\beta_8\sigma_8\varepsilon_{15}\}\oplus \Z_8\{\bar\beta_8\Sigma\rho''\}\oplus \Z_2\{\bar\beta_8\bar\varepsilon_8\}\oplus\Z_{(2)}\{\bar\beta_{23}\}
\end{align*}
where $\sigma_8\bar\nu_{15}$, $\sigma_8\varepsilon_{15}$, $\Sigma\rho''$ and $\bar\varepsilon_8$ are the generators of $\pi_{23}(S^8)$. 

 From Lemma \ref{lem:pinch gamma3}, $\bar p_{16}^0\gamma_3^0\simeq 0$. Hence $\gamma_3^0\in Ker(\bar p_{16\ast}^0)$ implies
 \begin{align*}
 	\gamma_3^0=\bar\tau_{16}^0(\bar\beta_8\theta_0+a_0\bar\beta_{23})=\beta_8\theta_0+a_0\beta_{23} ~\text{for}~\theta_0\in\pi_{23}(S^8)  ~\text{and some integer}~a_0\in \Z_{(2)}. 
 \end{align*}
Moreover, $\pi_{23}(J_{2,8}^{0})/\lr{\gamma_3^0}\cong \pi_{23}(J_{3,8}^{0})\cong \pi_{23}(F_9^{0})\cong \pi_{24}(S^9)\cong \Z_{16}\oplus \Z_{2}^{\oplus 3}$. 
\newline Let $a_0=2^{t_0}a_0'$ where integer $t_0\geq 0$ and integer $a_0'$ is odd. By comparing the order of the groups of the first and last two items
in above isomorphisms, we get $2^{t_0+7}=2^{7}$. Hence $t_0=0$, i.e.,
$a_0$ is odd. 

The first diagram in (\ref{diam: g^{0}_{nu}}) induces the following diagram
\begin{align*}
	&\footnotesize{\xymatrix{
			S^{23}\ar[r]^-{j_2^{23}}&J_{2,8}^{\gamma,0}\ar[r]^-{I_2^{\gamma,0}}\ar[d]^{( g^{0}_{\nu})|_{J_2}}&F_{\bar p_{16}^0}\ar[d]^{ g^s_t}\ar[r]^-{\bar\tau_{16}^0}&J_{2,8}^0\ar[d]^{\bar g^0_{\nu}}\ar[r]^-{\bar{p}_{16}^0}& S^{16}\ar[d]^{\nu_{13}}\\
				S^{17}\ar[r]^-{j_2^{17}}&J_{2}(M_{\gamma_{2}^{\nu}},S^8)\simeq S^5\vee S^{17}\ar[r]^-{I_2^{\gamma,\nu}} &F_{q_2^{13}}\ar[r]^-{\bar\tau_{13}^{\nu}}&J_{2,\nu} \ar[r]^-{q_2^{13}}& S^{13}
	} } \\
\text{with} ~&  ( g^{0}_{\nu})|_{J_2}=j_1^5\nu_5q_1^8+j_2^{17}\nu_{17}^2q_2^{23}+j_1^5\theta_1q_2^{23}, ~\theta_1\in \pi_{23}(S^5),
\end{align*}
where $I_2^{\gamma,\nu}$ and $\bar\tau_{13}^{\nu}$ are canonical inclusions. By (\ref{equ:taujp=i}) and the definition of $\beta_{23}$ in the Section \ref{sec:attaching maps}, we  have 
\begin{align*}
    \bar\tau_{13}^{\nu}I_2^{\gamma,\nu}j_1^5\simeq j_1^5 ~~\text{and}~~	\beta_{23}=\bar\tau_{16}^0I_2^{\gamma,0}j_2^{23}. 
\end{align*}
Since $\pi_{17}(J_{2,\nu})=\pi_{17}(S^5\vee S^{13})=\Z\{[j_1^5, j_2^{13}]\}\oplus j_1^5\fhe\pi_{17}(S^5)$ and on the other hand  $\pi_{17}(J_{2,\nu})=\Z\{I_2^{\gamma,\nu}\bar\tau_{13}^{\nu}j_2^{17}\}\oplus j_1^5\fhe\pi_{17}(S^5)$ which is obtained  by the exact sequence of $\pi_{17}(-)$ induced by the fibration sequence  $\Omega S^{13}\rightarrow F_{q_2^{13}}\xrightarrow{\bar\tau_{13}^{\nu}} J_{2,\nu} \xrightarrow{q_2^{13}} S^{13}$, 
  we get 
\begin{align*}
	I_2^{\gamma,\nu}\bar\tau_{13}^{\nu}j_2^{17}=\pm [j_1^5, j_2^{13}]+j_1^5\varrho, \varrho\in \pi_{17}(S^5). 
\end{align*}
So
\begin{align*}
	& \bar g^{0}_{\nu}\beta_{23}\simeq  \bar g^{0}_{\nu}\bar\tau_{16}^0I_2^{\gamma,0}j_2^{23}\simeq 	I_2^{\gamma,\nu}\bar\tau_{13}^{\nu}( g^{0}_{\nu})|_{J_2}j_2^{23}\simeq I_2^{\gamma,\nu}\bar\tau_{13}^{\nu}( j_1^5\nu_5q_1^8+j_2^{17}\nu_{17}^2q_2^{23}+j_1^5\theta_1q_2^{23})j_2^{23}\\
	& \simeq I_2^{\gamma,\nu}\bar\tau_{13}^{\nu}j_2^{17}\nu_{17}^2+ I_2^{\gamma,\nu}\bar\tau_{13}^{\nu}j_1^5\theta_1 \simeq \pm [j_1^5, j_2^{13}]\nu_{17}^2+j_1^5\theta'_1, ~\text{where}~ \theta'_1=\varrho\nu_{17}^2+\theta_1.\\
	& \Rightarrow~ \bar{g}^{0}_{\nu}\gamma_3^0=\bar{g}^{0}_{\nu}(\beta_8\theta_0+a_0\beta_{23})\simeq j_1^5\nu_5\theta_0+a_0(\pm [j_1^5, j_2^{13}]\nu_{17}^2+j_1^5\theta'_1),  ~\text{by (\ref{equ: g^{0}_{nu}})}\\
	&\simeq \pm a_0 [j_1^5, j_2^{13}]\nu_{17}^2+j_1^{5}\theta_2\in \pi_{23}(J_{2,\nu}), ~\theta_2=\theta'_1+\nu_5\theta_0\in \pi_{23}(S^5). 
\end{align*}
Since $a_0$ is odd and $\pi_{23}(J_{2,\nu})=\pi_{23}(S^5\vee S^{13})=j_{1}^5\fhe \pi_{23}(S^5)\oplus j_{2}^{13}\fhe \pi_{23}(S^{13})\oplus \Z_8\{[j_1^5, j_2^{13}]\nu_{17}^2\}$, $\simeq \bar{g}^{0}_{\nu}\gamma_3^0\neq 0$.  From $(\ref{equ: g^{0}_{nu}})$ and $(\ref{equ:gamma3nu})$, 
\begin{align*}
0\neq \bar{g}^{0}_{\nu}\gamma_3^0\simeq \gamma_3^\nu\nu_{20}=a[j_1^5,j_2^{13}]\nu_{17}~~\Rightarrow~a~\text{is odd}. 
\end{align*}
We complete the proof of this  lemma. 
\end{proof}

\qquad

\qquad

\qquad

\noindent
{\bf Acknowledgement.}  The author was partially supported by National Natural Science Foundation of China (Grant No. 11701430).

\bibliographystyle{amsplain}

\begin{thebibliography}{10}

\bibitem{Arkowitz}
M. Arkowitz, \textit{The generalized Whitehead product}, Pacific J Math. 12(1962) 7-23.

\bibitem{Barratt}
M.G. Barratt, \textit{On spaces of finite characteristic}, Quart. J. Math. Oxford 11 (1960), 124-136.

\bibitem{Bau1985}
H.J. Baues,  \textit{On homotopy classification problems of J. H. C. Whitehead},  Algebraic topology, G$\ddot{o}$ttingen 1984 (1985) 17-55.

\bibitem{Bau n-1n+3}
H.J. Baues, M. Hennes, \textit{The homotopy classification of $(n-1)$-connected
$(n + 3)$-dimensional polyhedra, $n\geq 4$}, Topology 30 (1991), 373-408.

\bibitem{Hopfinvariants1967}
J.M. Boardman, B. Steer, \textit{On Hopf invariants}. Comment. Math. Helv. 42(1967) 180–221. 

\bibitem{ChenWu}
W.D. Chen, J. Wu,  \textit{Decomposition of loop spaces and periodic problem on $\pi_{\ast}$}, Algebr. Geom. Topol. 13 (2013), 3245–3260.

\bibitem{Cohen}
F.R. Cohen, \textit{Stable Homotopy}, Lecture Notes in Mathematics 165 Berlin-Heidelberg: Springer, 1970.

\bibitem{Cohen1987}
F.R. Cohen,  \textit{A course in some aspects of classical homotopy theory},  1987.DOI:10.1007/BFb0078738.

\bibitem{CohenTalor}
F.R. Cohen, L.R. Taylor,  \textit{Homology of Function Spaces}, Math. Z. 198(1988) 299-316. 

\bibitem{CohenWu95}
F.R. Cohen,  J. Wu, \textit{A remark on the homotopy groups of $\Sigma^n \mathbb{R}P^2$}, Contem. Math. 181(1995), 65-81.

\bibitem{So6dim}
T. Cutler, T. So, \textit{The homotopy type of a once-suspended 6-manifold and its applications}, Topology Appl. 318(2022) 108213.

\bibitem{Golasinski deMelo}
M. Golasi$\acute{n}$ski, T. de Melo,  \textit{On the higher Whitehead product},  J Homotopy and Related Stru. 11(2016) 825-845.

\bibitem{Gray}
B. Gray, \textit{ On the homotopy groups of mapping cones},  Proc. Lond. Math. Soc. 26(1973) 497-520.

\bibitem{Hardie}
K.A. Hardie,  \textit{ Higher Whitehead products}, Quart. J. Math. 12(1961) 241-249.

\bibitem{Hardie Zeeman}
K.A. Hardie, \textit{On a construction of E.C.Zeeman}, J. Lond.  Math. Soc. 35(1960) 452-464.

\bibitem{HatcherSpectr}
A. Hatcher, \textit{Spectral Sequences in Algebraic Topology}, (Unpublished) http://www.math. cornell.edu/~hatcher. 

\bibitem{R.Huang 1}
R.Z. Huang,  \textit{Homotopy of gauge groups over non-simply connected five dimensional manifolds}, Sci. China Math. 64(2021) 1061-1092.

\bibitem{R.Huang 2}
R.Z. Huang,  \textit{Suspension homotopy of 6-manifolds},  Alg. Geom.
Topol. 23(2023) 439-460.

\bibitem{R.Huang Li}
R.Z. Huang,  P.C. Li,  \textit{Suspension homotopy of simply connected 7-
manifolds}, Sci. China Math. (2023), in press, arXiv:2208.13145.

\bibitem{James I M}
I.M. James, \textit{On the homotopy groups of certain pairs and triads}, Quart. J. Math. Oxford 5(1954) 260-270.

\bibitem{JamesReduceProdu}
I.M. James, \textit{Reduced product spaces}, Ann. of Math. (2) 62 (1955), 170-179.

\bibitem{James cup-prod}
I.M. James,  \textit{Note on cup-products}, Proc. Amer. Math. Soc. 8(1957) 374-383.

\bibitem{JZhtpygps} T.Jin T, Z.J. Zhu,  \textit{The $n + 3$, $n + 4$ dimenisonal homotopy groups of $A_n^2$-complexes}, arXiv: 2402.11226 (2024).

\bibitem{X.G.Liu}
X.G. Liu,  \textit{On the Moore space $P^n(8)$ and its homotopy groups}(in Chinese), Chinese Ann. Math. Ser. A 28(2007) 305-318.

\bibitem{LiZhu}
P.C. Li, Z.J. Zhu, \textit{The homotopy decomposition of the suspension of a non-simply-connected five-manifold}, Proc. R. Soc. Edinb. A (2024) 1-29. doi:10.1017/prm.2024.49.

\bibitem{Morisugi Mukai lift}
K. Morisugi, J. Mukai,  \textit{Lifting to mod 2 Moore spaces}, J. Math. Soc. Japan 52(2000) 515-533.

\bibitem{Mukai}
J. Mukai, T. Shinpo, \textit{Some homotopy groups of the mod 4 Moore space}, J. Fac. Sci. Shinshu Univ. 34(1999) 1–14.

\bibitem{Mukai2}
J. Mukai,  \textit{A remark on Toda's result about the suspension order of the stunted real projective space}, Memo. of. Facu. of Kyushu Univ. Ser. A 42(1988) 87-94.

\bibitem{Mukai3}
J. Mukai,  \textit{Note on existence of the unsatable Adams map}, Kyushu J. Math. 49(1995) 271-279.

\bibitem{Mukai4}
J. Mukai,  \textit{ Generators of some homotopy groups of the mod2 Moore space of dimension 3 or 5}, Kyushu J. Math. 55(2001) 63-73.

\bibitem{MukaiS1}
J. Mukai,  \textit{The $S^1$-transfer map and homotopy groups of suspended complex projective spaces}, Math. J. Okayama Univ. 24(1982) 179-200.

\bibitem{NeisendorferBook}
J. Neisendorfer, \textit{Algebraic methods in unstable homotopy theory}, Cambridge University Press, 2010.

\bibitem{N.Oda}
N. Oda, \textit{Unstable homotopy groups of spheres}, Bull Inst Adv Res Fukuoka Univ, 44 (1979) 49-151.

\bibitem{Porter}
G.J. Porter,  \textit{Higher order Whitehead products}, Topology 3(1965) 123-135.

\bibitem{Porter Postnikov}
G.J. Porter,   \textit{Higher order Whitehead products and Postnikov systems}, Illinois J. Math. 11(1967) 414-416.

\bibitem{Amin}
P. Selick, J. Wu, \textit{The functor $A^{min}$ on $p$-local spaces}, Math. Z. 253 (3) (2006), 435-451.

\bibitem{Theriault}
T. So,  S. Theriault, \textit{The suspension of a 4-manifold and its applications}, Isr. J. Math. (2022), in press, arXiv: 1909.11129.

\bibitem{Steer1963}
B. Steer, \textit{Transgression in sphere-bundles}, Topology 2(1963), 1-10

\bibitem{Toda}
H. Toda, \textit{Composition Methods in Homotopy Groups of Spheres}, Annals of Mathematics Studies, Princeton University Press, Princeton, 1962.

\bibitem{Toda Whitehead Prod.}
H. Toda, \textit{Generalized Whitehead Products and Homotopy Groups of Spheres},  J. Inst. Polytech. Osaka City Univ. Ser. A 3(1952) 43-82.

\bibitem{GTM61}
G.W. Whitehead, \textit{Elements of Homotopy Theory}, Graduate Texts in Mathematics 61, Springer-Verlag, 1978.

\bibitem{Wucombina}
J. Wu, \textit{On combinatorial calculations for the James Hopf maps}, Topology 37 (1998), 1011-1023.

\bibitem{WJ Proj plane}
J. Wu, \textit{Homotopy Theory of the Suspensions of the Projective Plane},  Mem. Amer. Math. Soc. 162(2003), No. 769.

\bibitem{Yamaguchi}
K. Yamaguchi, \textit{The group of self-homotopy equivalences of $S^2$-bundles over $S^4$.I.}, Kodai Math. J. 9(1986) 308-326.

\bibitem{JXYang}
J.X. Yang, J. Mukai, J. Wu, \textit{ On the homotopy groups of the suspended quaternionic projective plane and the applications}, Preprint (2023) arXiv:2301.06776.


\bibitem{ZP}
Z.J. Zhu, J.Z. Pan,  \textit{The decomposability of smash product of $\mathbf{A}_n^2$
	complexes}, Homology Homotopy and Applications 19(2017) 293-318.

\bibitem{ZLP}
Z.J. Zhu, P.C. Li,  J.Z. Pan,  \textit{Periodic problem on homotopy groups of Chang complexes $C_{r}^{n+2,r}$},  Homology Homotopy and Applications 21.2 (2019) 363-375.

\bibitem{ZPhyperbolicity}
Z.J. Zhu,  J.Z. Pan, \textit{The local hyperbolicity of $\mathbf{A}_n^2$-complexes},  Homology Homotopy and Applications 23.1 (2021) 367-386.

\bibitem{ZP23}
Z.J. Zhu, J.Z. Pan, \textit{2-local unstable homotopy groups of $A^2_3$-Complexes}, Sci. China Math.  67 (2024) 607–626.	
	
\bibitem{ZJ}
Z.J. Zhu, T. Jin,   \textit{The relative James construction and its application to homotopy groups}, Topology and its Applications 356 (2024) 109043
\end{thebibliography}

	\section{Appendix}
	\label{sec:Appendix}
	
	\begin{lemma}\label{lemA: Sigma nu'sigam'}
		$(i)~\Sigma \nu'\sigma'=2\Sigma \varepsilon'$.~~ $(ii)~\mu_4\eta_{13}=\eta_4\mu_5$.~~
		
		$(iii)~\nu_5\eta_8^2\sigma_{10}=\nu_5^4+\nu_5\eta_8\varepsilon_9$.   $(iv)~(2^{r}\iota_5)\nu_5\sigma_8=2^r\nu_5\sigma_8$.
		
		$(v)~(2^{r}\iota_6)\bar\nu_6=2^{2r}\bar\nu_6$.
	\end{lemma}
	\begin{proof} (i)~~
		By $\Sigma \nu'\sigma'\in \pi_{14}(S^4)=\Z_8\{\nu_4\sigma'\}\oplus \Z_4\{\Sigma \varepsilon'\}\oplus \Z_2\{\eta_4\mu_5\} $, assume that
		\begin{align}
			\Sigma \nu'\sigma'=x\nu_4\sigma'+y\Sigma \varepsilon'+z\eta_4\mu_5, ~~x\in \Z_8, y\in \Z_4, z\in \Z_2. \label{equ: appendix assume}
		\end{align}
		Then $2\nu_5\Sigma \sigma'=\Sigma(\Sigma \nu'\sigma')=x\nu_5\Sigma \sigma'+y\Sigma^2\varepsilon'+z\eta_5\mu_6$. From (7.10), (7.16) of \cite{Toda}, $2\nu_5\sigma_8=\nu_5\Sigma\sigma'=\pm \Sigma^2\varepsilon'$ implies that $4\nu_5\sigma_8=(2x\pm 2y)\nu_5\sigma_8+z\eta_5\mu_6\in \pi_{15}(S^5)=\Z_8\{\nu_5\sigma_8\}\oplus \Z_2\{\eta_5\mu_6\}$. Hence $z=0$, $2(x\pm y)=4$. Multiply both sides of equation (\ref{equ: appendix assume}) by $2$, 
		\begin{align}
			2\Sigma \nu'\sigma'=2x\nu_4\sigma'+2y\Sigma \varepsilon'. \label{equ:2times appendix assume}
		\end{align}
		$2\Sigma \nu'\sigma'=\Sigma \nu'(2\sigma')=\Sigma(\nu'\sigma'')$ since $2\sigma'=\Sigma\sigma''$ by Lemma 5.14 of \cite{Toda}. Consider Hopf invariant $H_2$ on both sides of (\ref{equ:2times appendix assume}),  $0=H_2(2\Sigma \nu'\sigma')=H_2(2x\nu_4\sigma')=xH_2(\nu_4\Sigma\sigma'')=xH_2(\nu_4)\Sigma\sigma''=x\Sigma\sigma''=2x\sigma'\in \pi_{14}(S^7)=\Z_8\{\sigma'\}$. So $x=4t (t=0,1)$ implies $y=2$ and (\ref{equ: appendix assume}) becomes $\Sigma \nu'\sigma'=4t\nu_4\sigma'+2\Sigma \varepsilon'=2t\nu_4\Sigma\sigma''+2\Sigma \varepsilon'$. By Proposition 2.2, Lemma 3.1 and Lemma 5.14 of \cite{Toda},
		$H_2(\Sigma \nu'\sigma')=(\Sigma \nu'\wedge \nu')H_{2}(\sigma')=\Sigma(-\Sigma^3\nu'\Sigma^6\nu')\eta_{13}=-4\nu_7^2\eta_{13}=0$. On the other hand, $H_2(\Sigma \nu'\sigma')=2tH_2(\nu_4\Sigma\sigma'')=2tH_2(\nu_4)\Sigma\sigma''=2t\Sigma\sigma''=4t\sigma'$.
		Thus $4t\sigma'=0$ implies $t=2t'$ and $x=8t'=0\in \Z_8$. So $\Sigma \nu'\sigma'=2\Sigma \varepsilon'$.
		
		(ii)~Let $\mu_4\eta_{13}=4x'\nu_4\sigma'+2y'\Sigma \varepsilon'+z'\eta_4\mu_5, ~~x', y', z'\in \Z$.  By page.66 of \cite{Toda}, 
		$\eta_5\mu_6=\Sigma \mu_4\eta_{13}=\pm 4y' \nu_5\sigma_8+z'\eta_5\mu_6$, implies $2|y'$, $z'=1$. Hence $\mu_4\eta_{13}=4x'\nu_4\sigma'+\eta_4\mu_5$. By
		$0=H(\mu_4\eta_{13})=H_2(4x'\nu_4\sigma'+\eta_4\mu_5)=4x'\sigma'$, we get $2|x'$. 
		So $\mu_4\eta_{13}=\eta_4\mu_5$. 
		
		(iii)~$\nu_5\eta_8^2\sigma_{10}=\nu_5\eta_8((\Sigma^2\sigma')\eta_{16}+\bar\nu_9+ \varepsilon_9)=\nu_5\eta_8(\bar\nu_9+\varepsilon_9)=\nu_5^3+\nu_5\eta_8\varepsilon_9$.
		where the first, second and the last equations are from (7.4), Lemma 5.14 and  (7.3) of \cite{Toda} respectively.
		
		(iv)~ By the same proof as that of  Lemma A.1 of \cite{ZP23}, $(2^{r}\iota_8)\sigma_8=2^{2r}\sigma_8-2^{r-1}(2^r-1)\Sigma\sigma'$.

		$(2^{r}\iota_5)\nu_5\sigma_8=\nu_5((2^{r}\iota_8)\sigma_8)=2^{2r}\nu_5\sigma_8-2^{r-1}(2^r-1)\nu_5\Sigma\sigma'=2^{2r}\nu_5\sigma_8-2^{r}(2^r-1)\nu_5\sigma_8=2^r\nu_5\sigma_8$. 
		
		(v)~By \cite[Proposition 2.10]{N.Oda}, $H_2(\bar\nu_6)=l\nu_{11}$, $l$ is odd, and $\Delta(\nu_{13})=\pm (\iota_{13})\nu_{11}=\pm 2\bar\nu_6$ \cite[Lemma 6.2]{Toda}, we have 
		\newline
		$(2^{r}\iota_6)\bar\nu_6=2^{r}\bar\nu_6\pm \binom{2^{r}}{2}[\iota_6,\iota_6]H_2(\bar\nu_6)=2^{r}\bar\nu_6\pm \binom{2^{r}}{2}l\Delta(\iota_{13})\nu_{11}=2^{r}\bar\nu_6\pm \binom{2^{r}}{2}l\Delta(\nu_{13})=2^{r}\bar\nu_6\pm 2^r(2^r-1)\bar\nu_6=2^{2r}\bar\nu_6$ or $(2^{r+1}-2^{2r})\bar\nu_6$. Moreover $H_2((2^{r}\iota_6)\bar\nu_6)=\Sigma(2^r\iota_5\wedge 2^r\iota_5)H_{2}(\bar\nu_6)=2^{2r}\nu_{11}=H_{2}(2^{2r}\bar\nu_6)$. We get $(2^{r}\iota_6)\bar\nu_6=2^{2r}\bar\nu_6$. 
	\end{proof}

	\begin{lemma}\label{lemA: Todabraket }
		\begin{align*}
				(i)~~&	\{2\iota_3, \nu'\eta_6,  2\iota_7\}=\{\nu'\eta_6^2\};~~ 	\{2^r\iota_3, \nu'\eta_6,  2\iota_7\}=\{0\} (r\geq 2).\\
			(ii)~~& 	\{4\iota_4, \nu_4\eta_7, 2\iota_8\}=\{\Sigma \nu'\eta_7^2\};~~	\{2^r\iota_4, \nu_4\eta_7, 2\iota_8\}=\{0\} (r\geq 3).\\
			(iii)~~& \{2\iota_4, \varepsilon_4, 2\iota_{12}\}=\{\eta_4\varepsilon_5\};~~\{2^r\iota_4, \varepsilon_4, 2\iota_{12}\}=\{0\} (r\geq 2).\\
			(iv)~~& \{2\iota_4, \mu_4, 2\iota_{13}\}= \eta_4\mu_5+\lr{2\nu_4\sigma', 2\Sigma\varepsilon'}; \{2^r\iota_4, \mu_4, 2\iota_{13}\}= \lr{2\nu_4\sigma', 2\Sigma\varepsilon'}(r\geq 2).\\
					(v)~~&	\{2\iota_5, \sigma''', 2\iota_{12}\}=\{0\}.\\
			(vi)~~& 	\{2\iota_5, \varepsilon_5, 2\iota_{13}\}=\{\eta_5\varepsilon_6\}; ~~	\{2^r\iota_5, \varepsilon_5, 2\iota_{13}\}=\{0\}, r\geq 2.\\
		    (vii)~~& \{2\iota_6, 2\sigma'', 2\iota_{13}\}=\lr{2\bar\nu_6};~~\{4\iota_6, \sigma'', 4\iota_{13}\}=\lr{4\bar\nu_6}.\\
		    (viii)~~& \{2\iota_7, \sigma'\eta_{14}, 2\iota_{15}\}=\{\sigma'\eta^2_{14}\}; \{2^r\iota_7, \sigma'\eta_{14}, 2\iota_{15}\}=\{0\} (r\geq 2);
		\end{align*}
	\end{lemma}
	\begin{proof}
		$(i)$~
		For $r\geq 1$, 	$\{2^r\iota_3, \nu'\eta_6,  2\iota_7\}$ is a coset of subgroup $(2^r\iota_3)\fhe \pi_{8}(S^3)+\pi_{8}(S^3)\fhe(2\iota_8)=\{0\}$, implies that $\{2^r\iota_3, \nu'\eta_6,  2\iota_7\}$ has only one element. Then (i) of this lemma comes from the following arguments
		\begin{align*}
			&\{2^r\iota_3, \nu'\eta_6,  2\iota_7\}\supset \{2^r\nu', \eta_6,  2\iota_7\}\supset	(2^{r-1}\nu')\{2\iota_6, \eta_6,  2\iota_7\} ~\text{by \cite[Prop. 1.2]{Toda} }\\
			\supset &(2^{r-1}\nu')\{2\iota_6, \eta_6,  2\iota_7\}_1\ni 2^{r-1}\nu'\eta_6^2, ~\text{by \cite[(1.15) and Corollary 3.7 ]{Toda}}.
		\end{align*}
	
		$(ii)$~ For $r\geq 2$, 	$\{2^r\iota_4, \nu_4\eta_7, 2\iota_8\}$ is a coset of subgroup $\pi_9(S^4)\fhe(2\iota_9)+(2^r\iota_4)\fhe\pi_9(S^4)=\{0\}$ which implies that $\{2^r\iota_4, \nu_4\eta_7, 2\iota_8\}$ has only one element.
		\begin{align*}
			& \{2^r\iota_4, \nu_4\eta_7, 2\iota_8\}\supset  \{(2^r\iota_4)\nu_4, \eta_7, 2\iota_8\}= \{2^{2r}\nu_4-2^{r-1}(2^r-1)\Sigma\nu', \eta_7, 2\iota_8\}\\
			\subset& \{2^{2r}\nu_4, \eta_7, 2\iota_8\}+\{2^{r-1}(2^r-1)\Sigma\nu', \eta_7, 2\iota_8\}=\{2^{r-1}\Sigma\nu', \eta_7, 2\iota_8\}.
		\end{align*}
		Thus $(ii)$ of Lemma \ref{lemA: Todabraket } is obtained by $(ii)$ of this lemma.
		\newline
		$(iii)$~ By Lemma \ref{lem:Toda bracket} and \cite[(7.5)]{Toda}, $\{2^r\iota_4,\varepsilon_4, 2\iota_{12}\}\ni 2^{r-1}\varepsilon_4\eta_{12}=2^{r-1}\eta_4\varepsilon_5$ mod  $\pi_{13}(S^4)\fhe (2\iota_{13})+(2^r\iota_4)\fhe \pi_{13}(S^4)=\{0\}$ which implies the $(iii)$ of Lemma \ref{lemA: Todabraket }.
		\newline
		$(iv)$~$\{2^r\iota_4, \mu_4, 2\iota_{13}\}$ is a coset of subgroup 
		$\pi_{14}(S^4)\fhe (2\iota_{14})+(2^r\iota_4)\fhe \pi_{14}(S^4)=
		\lr{2\nu_4\sigma', 2\Sigma\varepsilon'}$ $\subset\pi_{14}(S^4)$. 
		For $r=1$, by Corollary 3.7 of \cite{Toda} and (ii) of Lemma \ref{lemA: Sigma nu'sigam'}, $\mu_4\eta_{13}=\eta_4\mu_5\in \{2\iota_4, \mu_4, 2\iota_{13}\}$. Hence $\{2\iota_4, \mu_4, 2\iota_{13}\}= \eta_4\mu_5+\lr{2\nu_4\sigma', 2\Sigma\varepsilon'}$.  For $r\geq 2$, $\{2^r\iota_4, \mu_4, 2\iota_{13}\}\supset
		(2^{r-1}\iota_4)\fhe \{2\iota_4, \mu_4, 2\iota_{13}\}\ni (2^{r-1}\iota_4)\eta_4\mu_5=0$ implies $\{2^r\iota_4, \mu_4, 2\iota_{13}\}=	\lr{2\nu_4\sigma', 2\Sigma\varepsilon'}$.
		\newline
			$(v)$~$\{2\iota_5, \sigma''', 2\iota_{12}\}$ is a coset of subgroup
		$(2\iota_{5})\fhe \pi_{13}(S^5)+ \pi_{13}(S^5)\fhe(2\iota_{13})=\{0\}$.
		\newline $-\Sigma\{2\iota_5, \sigma''', 2\iota_{12}\}\subset \{2\iota_6, \Sigma\sigma''', 2\iota_{13}\}_1\ni \Sigma\sigma'''\eta_{13}=(2\sigma'')\eta_{13}=0 $ mod $\lr{2\bar\nu_6}$. 	\newline
		Thus $\Sigma\{2\iota_5, \sigma''', 2\iota_{12}\}\subset \lr{2\bar\nu_6}$ implies $\{2\iota_5, \sigma''', 2\iota_{12}\}=0$ by \cite[Theorem 7.1]{Toda}.
			\newline
		$(vi)$~$\{2^r\iota_5, \varepsilon_5, 2\iota_{13}\}$ is a coset of $(2^r\iota_5)\fhe \pi_{14}(S^5)+ \pi_{14}(S^5)\fhe( 2\iota_{14})=\{0\}$. Moreover $\{2^r\iota_5, \varepsilon_5, 2\iota_{13}\}\ni (2^{r-1}\iota_5)\varepsilon_5\eta_{13}=2^{r-1}\eta_5\varepsilon_6$ by Lemma \ref{lem:Toda bracket}. This implies (vi) of this lemma.
		\newline
		$(vii)$~ For $r\leq 2$, $\{2^r\iota_6, 2^{2-r}\sigma'', 2^r\iota_{13}\}\subset \{2\iota_6, (2^{r-1}\iota_6)( 2^{2-r}\sigma''), 2^r\iota_{13}\}=\{2\iota_6, 2\sigma'', 2^r\iota_{13}\}=\{2\iota_6, \Sigma\sigma''', 2^r\iota_{13}\}\supset \{2\iota_6, \Sigma\sigma''', 2\iota_{13}\}(2^{r-1}\iota_{14})\ni 2^{r-1} \Sigma\sigma'''\eta_{13}=0$.  
		$\{2\iota_6, \Sigma\sigma''', 2^r\iota_{13}\}$ is a coset of subgroup
	 $(2\iota_6)\fhe\pi_{14}(S^6)+\pi_{14}(S^6)\fhe (2^r\iota_{14})=\lr{2^r\bar\nu_6} (r\leq 2)$  by (v) of Lemma \ref{lemA: Sigma nu'sigam'} and so is $\{2^r\iota_6, 2^{2-r}\sigma'', 2^r\iota_{13}\}$. Hence $\{2^r\iota_6, 2^{2-r}\sigma'', 2^r\iota_{13}\}=\lr{2^r\bar\nu_6}$.
	 	\newline
	 $(viii)$  	$\{2^r\iota_7, \sigma'\eta_{14}, 2\iota_{15}\}$ is a coset of subgroup $(2^r\iota_7)\fhe\pi_{16}(S^7)+\pi_{16}(S^7)\fhe (2^r\iota_{16})=\{0\}$ for $r\geq 1$.
	 Then from \cite[Propostion 1.2]{Toda} and Lemma \ref{lem:Toda bracket}, we get
		$\{2^r\iota_7, \sigma'\eta_{14}, 2\iota_{15}\}\supset \{2^r\sigma', \eta_{14}, 2\iota_{15}\}\ni 2^{r-1}\sigma'\eta^2_{14}$, which implies $(viii)$ of this lemma. 
	
	\end{proof}

\end{document}